\documentclass[12pt]{amsart}

\usepackage{graphicx}
\usepackage{latexsym}
\usepackage{amsfonts}
\usepackage{amscd}
\usepackage{amssymb}
\usepackage{amsmath}
\usepackage{amsthm}
\usepackage{bm}
\usepackage[all]{xy}
\usepackage[lmargin=1.25in,rmargin=1.25in,tmargin=1.25in,bmargin=1in,dvips]{geometry}
\usepackage{pinlabel}
\usepackage{paralist}
\usepackage{mathtools}
\usepackage{enumitem}
\usepackage[normalem]{ulem}
\usepackage[unicode, linktocpage, bookmarksnumbered=true, backref=true, bookmarksdepth=3]{hyperref}
\usepackage{tikz}
\usepackage{subfigure}

\usepackage[textsize=scriptsize]{todonotes}

\newtheorem{theorem}{Theorem}[section]
\newtheorem{lemma}[theorem]{Lemma}

\newtheorem{corollary}[theorem]{Corollary}
\newtheorem{proposition}[theorem]{Proposition}

\theoremstyle{definition}
\newtheorem{definition}[theorem]{Definition}

\newtheorem{remark}[theorem]{Remark}

\newtheorem{example}[theorem]{Example}

\def\N{\mathbb{N}}

\def\Z{\mathbb{Z}}
\def\Q{\mathbb{Q}}
\def\F{\mathbb{F}}

\def\CF {\operatorname{CF}}

\def\CFKa {\widehat{\operatorname{CFK}}}
\def\HFKa {\widehat{\operatorname{HFK}}}

\def\CFp {\operatorname{CF}^+}
\def\CFm {\operatorname{CF}^-}
\def\CFi {\operatorname{CF}^\infty}
\def\CFa {\operatorname{\widehat{CF}}}
\def\HFp {\operatorname{HF}^+}
\def\HFm {\operatorname{HF}^-}
\def\HFi {\operatorname{HF}^\infty}
\def\HFa {\operatorname{\widehat{HF}}}

\def\CFpt {\ul{\operatorname{CF}}^+}

\def\CFit {\ul{\operatorname{CF}}^\infty}

\def\CFic {\mathbf{CF}^\infty}

\def\CFmc {\mathbf{CF}^-}

\def\CFK {\operatorname{CFK}}
\def\HFK {\operatorname{HFK}}
\def\CFKm {\operatorname{CFK}^-}

\def\CFKi {\operatorname{CFK}^\infty}

\def\CFKp {\operatorname{CFK}^+}

\def\CFLi {\operatorname{CFL}^\infty}

\def\spincs {\mathfrak{s}}
\def\spinct {\mathfrak{t}}
\def\spincu {\mathfrak{u}}
\def\spincv {\mathfrak{v}}
\def\spincx {\mathfrak{x}}
\def\spincy {\mathfrak{y}}

\def\DD {\mathcal{D}}
\def\FF {\mathcal{F}}

\def\II {\mathcal{I}}
\def\MM {\mathcal{M}}

\def\T{\mathbb{T}}

\def\a {\mathbf{a}}
\def\b {\mathbf{b}}
\def\q {\mathbf{q}}
\def\r {\mathbf{r}}
\def\x {\mathbf{x}}
\def\y {\mathbf{y}}

\def\ul {\underline}

\newcommand{\abs}[1] {\left\lvert #1 \right\rvert}

\newcommand{\gen}[1] {\langle #1 \rangle}
\newcommand{\floor}[1] {\left\lfloor #1 \right\rfloor}

\def\Th{^{\text{th}}}
\def\minus{\smallsetminus}
\def\co{\colon\thinspace}

 \DeclareMathOperator{\id}{id} 
  \DeclareMathOperator{\sign}{sign}
\DeclareMathOperator{\Sym}{Sym} \DeclareMathOperator{\Spin}{Spin}
\DeclareMathOperator{\Span}{Span} \DeclareMathOperator{\nbd}{nbd}
\DeclareMathOperator{\PD}{PD} \DeclareMathOperator{\gr}{gr}
 \DeclareMathOperator{\Cone}{Cone}
\DeclareMathOperator{\Gr}{Gr}

\def\absgr{\operatorname{\widetilde{gr}}}

\def\conn{\mathbin{\#}}
\def\bconn{\mathbin{\natural}}

\renewcommand{\MR}[1]{}

\definecolor{darkgreen}{rgb}{0,0.5,0}
\definecolor{purple}{rgb}{0.5,0,0.5}

\def\Al {\tilde A}
\def\AlNorm {A}

\def\JJ {\mathcal{J}}

\def\MultComb {\mathcal{N}}
\def\Gens {\mathfrak{S}}

\numberwithin{equation}{section}

\begin{document}

\title{A surgery formula for knot Floer homology}

\author{Matthew Hedden}
\thanks{MH was partially supported by NSF CAREER grant DMS-1150872, NSF grant DMS-1709016 and an Alfred P. Sloan Research Fellowship.}
\address{Department of Mathematics, Michigan State University, East Lansing, MI 48824}
\email{mhedden@math.msu.edu}

\author{Adam Simon Levine}
\thanks{ASL was partially supported by NSF grant DMS-1806437.}
\address{Department of Mathematics, Duke University, Durham, NC 27708}
\email{alevine@math.duke.edu}

\allowdisplaybreaks

\begin{abstract}
Let $K$ be a rationally null-homologous knot in a $3$-manifold $Y$, equipped with a nonzero framing $\lambda$, and let $Y_\lambda(K)$ denote the result of $\lambda$-framed surgery on $Y$. Ozsv\'ath and Szab\'o gave a formula for the Heegaard Floer homology groups of $Y_\lambda(K)$ in terms of the knot Floer complex of $(Y,K)$. We strengthen this formula by adding a second filtration that computes the knot Floer complex of the dual knot $K_\lambda$ in $Y_\lambda$, i.e., the core circle of the surgery solid torus.   In the course of proving our refinement we derive a combinatorial formula for the Alexander grading  which may be of independent interest.
\end{abstract}

\maketitle


\section{Introduction} \label{sec: intro}

Let $K$ be a rationally null-homologous knot in a $3$-manifold $Y$. Let $\lambda$ be any framing on $K$, and let $Y_\lambda(K)$ denote the result of $\lambda$-framed surgery along $K$. In \cite{OSzSurgery, OSzRational}, Ozsv\'ath and Szab\'o gave a formula for the Heegaard Floer homology groups of $Y_\lambda(K)$ in terms of the knot Floer complex $\CFKi(Y,K)$. This formula has been one of the most important tools in the Heegaard Floer toolkit.  Not only has it has been the primary method of  computation for many specific examples of Floer homology groups \cite{CochranHarveyHornFiltering, HeddenKimLivingston, HeddenWatsonUnknot, HomLevineLidmanPL, JabukaMarkSurface, LevineLidmanSpineless, MeierDoublySlice}, but the existence of the formula indicates that the knot Floer homology invariants tightly constrain the Floer invariants of manifolds obtained by surgery, and conversely.  This interplay between the two invariants, coupled with the rich geometric content of both, has led to striking new applications in Dehn surgery.  For instance, it has given rise to interesting new surgery obstructions  \cite{JabukaSurgeryGenus, HomKarakurtLidman, Rasmussen2Simple} and led to significant progress on the cosmetic surgery conjecture \cite{WangGenusOne, WuCosmetic, NiWuCosmetic}, exceptional surgeries \cite{DoigFinite, HomLidmanZufelt, LiNi, MeierCabled, NiZhangFinite, WuSeifert}, and the Berge Conjecture \cite{BakerGrigsbyHedden, HeddenBerge, RasmussenLens}. The surgery formula was subsequently generalized by Manolescu and Ozsv\'ath \cite{ManolescuOzsvathLink} to surgeries on links, which results in a combinatorial (albeit largely impractical) algorithm for computing all versions of Heegaard Floer homology for any $3$-manifold \cite{ManolescuOzsvathThurston}.

Let $K_\lambda \subset Y_\lambda(K)$ denote the core circle of the surgery solid torus, often called the \emph{dual knot}. In this paper, we strengthen Ozsv\'ath and Szab\'o's results to provide a formula for $\CFKi(Y_\lambda(K), K_\lambda)$, provided the framing is nonzero. Specifically, we define a second filtration on the chain complex defined by Ozsv\'ath and Szab\'o, and we show that it agrees with the Alexander filtration induced by $K_\lambda$.

Some special cases of our formula are already known and have had numerous applications. In \cite{HeddenWhitehead}, the first author established a limited version of our formula, addressing the significantly easier computation of the ``hat'' knot Floer homology groups of the dual knot in sufficiently large surgery, and used this computation to derive a formula for the knot Floer homology of Whitehead doubles in terms of the complex of the companion knot. In \cite{HeddenBerge}, the same formula was used to derive an obstruction to lens space surgeries in terms of the dual knot, namely that the dual knot must have simple Floer homology (c.f. \cite{RasmussenLens}); this result is central to Baker--Grigsby--Hedden's approach to the Berge conjecture \cite{BakerGrigsbyHedden}. Also, in joint work with Plamenevskaya \cite{HeddenPlamenevskayaRational}, the ``hat'' formula was used to provide criteria for manifolds obtained by Dehn surgery on fibered knots to admit tight contact structures. Subsequently, Kim, Livingston, and the first author \cite{HeddenKimLivingston} extended the preceding result to describe the full complex $\CFKi(Y_\lambda(K), K_\lambda)$ for sufficiently large surgeries, established that a framing coefficient that is twice the genus of $K$ is ``sufficiently large,'' and used the surgery formula as the key tool in $d$-invariant computations that verified the existence of $2$--torsion in the subgroup of smooth concordance generated by topologically slice knots.

Most recently, Hom, Lidman, and the second author \cite{HomLevineLidmanPL} have used our main theorem  (Theorem \ref{thm: mapping-cone}) to provide an example of a knot in a homology sphere which has infinite order in the non-locally-flat piecewise-linear concordance group. The reader is encouraged to refer to that paper for a detailed computation using this formula, which illustrates the general technique.

\subsection{Statement of the theorem}

In order to state the main theorem, we start by quickly establishing some terminology and notation. We will fill in more details in Section \ref{sec: mapping-cone}.

Assume that $K$ represents a class of order $d>0$ in $H_1(Y;\Z)$. Fix a tubular neighborhood $\nbd(K)$. Let $\mu \subset \partial(\nbd(K))$ be a right-handed meridian of $K$.

A \emph{relative spin$^c$ structure} is a homology class of nowhere-vanishing vector fields on $Y \minus \nbd(K)$ which is tangent to the boundary along $\partial (\nbd(K))$. The set of relative spin$^c$ structures is denoted $\ul\Spin^c(Y,K)$ and is an affine set for $H^2(Y,K)$. (This set does not depend on the orientation of $K$.) Given an orientation, Ozsv\'ath and Szab\'o define a map
\[
G_{Y,K} \co \ul\Spin^c(Y,K) \to \Spin^c(Y),
\]
which is equivariant with respect to the restriction map
\[
H^2(Y,K;\Z) \to H^2(Y;\Z).
\]
The fibers of $G_{Y,K}$ are precisely the orbits of $\ul\Spin^c(Y,K)$ under the action of $\gen{\PD[\mu]} \subset H^2(Y,K;\Z)$.

The \emph{Alexander grading} of each $\xi \in \ul\Spin^c(Y,K)$ is defined as
\begin{equation} \label{eq: spinc-alex}
\AlNorm_{Y,K}(\xi) = \frac{\gen{c_1(\xi), [F]} + [\mu] \cdot [F] }{2 [\mu] \cdot [F]} \in \frac{1}{2d} \Z,
\end{equation}
where $F$ is a rational Seifert surface for $K$, and $\cdot$ denotes the intersection pairing between $H_1(Y \minus K)$ and $H_2(Y,K)$.   Note that the relative Chern class in the above equation depends on a choice of vector field along the boundary torus of the knot complement; for this, we take a nowhere-vanishing vector field tangent to the torus.
For each $\spincs \in \Spin^c(Y)$, the values of $\AlNorm_{Y,K}(\xi)$, taken over all $\xi \in G_{Y,K}^{-1}(\spincs)$, form a single coset in $\Q/\Z$, which we denote by $\AlNorm_{Y, K}(\spincs)$. Indeed, any $\xi \in \Spin^c(Y,K)$ is uniquely determined by the pair $(G_{Y,K}(\xi), \AlNorm_{Y,K}(\xi))$.
Let $\F$ denote the field of two elements. The \emph{knot Floer complex} of $(Y,K)$ is a doubly-filtered chain complex $\CFKi(Y,K)$, defined over $\F[U,U^{-1}]$, which is invariant up to doubly-filtered chain homotopy equivalence, with a decomposition
\[
\CFKi(Y,K) = \bigoplus_{\spincs \in \Spin^c(Y)} \CFKi(Y,K,\spincs).
\]
The two filtrations are denoted by $i$ and $j$. Our conventions are slightly different from Ozsv\'ath and Szab\'o's: on each summand $\CFKi(Y,K,\spincs)$, $i$ is an integer, while $j$ takes values in $\Z + \AlNorm_{Y, K}(\spincs)$. The action of $U$ decreases both filtrations by $1$. By ignoring the $j$ filtration, we have $\CFKi(Y,K, \spincs) = \CFi(Y,\spincs)$; in particular, the groups $\HFm(Y,\spincs)$, $\HFp(Y,\spincs)$, and $\HFa(Y,\spincs)$ are the homologies of the $i<0$ subcomplex, the $i \ge 0$ quotient, and the $i=0$ subquotient, respectively. If $\spincs$ is a torsion spin$^c$ structure, then $\CFKi(Y,K,\spincs)$ also comes equipped with an absolute $\Q$-grading $\absgr$ that lifts a relative $\Z$-grading; the differential has grading $-1$, and multiplication by $U$ has grading $-2$.

For each $\xi \in G_{Y,K}^{-1}(\spincs)$, there is a ``flip map''
\[
\Psi^\infty_\xi \co \CFKi(Y,K, \spincs) \to \CFKi(Y,K, \spincs+ \PD[K]),
\]
a filtered chain homotopy equivalence that is an invariant of the knot $K$ up to filtered chain homotopy. (See Lemma \ref{lemma: Psi-filtered} for the precise sense in which $\Psi^\infty_\xi$ is filtered.)

An (integral) framing on $K$ is specified by a choice of longitude $\lambda$, which we may view as a curve in $\partial (\nbd(K))$. As elements of $H_1(\partial (\nbd(K)))$, we have $\partial F = d\lambda - k\mu$ for some $k\in\Z$; the framing determines and is determined by $k$. Let $Y_\lambda = Y_\lambda(K)$ denote the manifold obtained by $\lambda$-framed surgery on $K$. The meridian $\mu$ is isotopic to a core circle of the glued-in solid torus. Let $K_\lambda$ denote this core circle, with the orientation inherited from the \emph{left-handed} meridian $-\mu$. The sets $\ul\Spin^c(Y,K)$ and $\ul\Spin^c(Y_\lambda, K_\lambda)$ are canonically identified, since they depend only on the complement. The orientation of $K_\lambda$ induces a map $G_{Y_\lambda, K_\lambda} \co \ul\Spin^c(Y_\lambda, K_\lambda) \to \Spin^c(Y_\lambda)$ whose fibers are the orbits of the action of $\PD[\lambda]$.

Assume henceforth that $k \ne 0$. Choose a spin$^c$ structure $\spinct$ on $Y_\lambda(K)$. 
Let us index the elements of $G_{Y_\lambda, K_\lambda}^{-1}(\spinct)$ by $(\xi_l)_{l \in \Z}$, where $\xi_{l+1} = \xi_l + \PD[\lambda]$. Let $\spincs_l = G_{Y,K}(\xi_l)$ and $s_l = \AlNorm_{Y,K}(\xi_l)$. Then $\spincs_{l+1} = \spincs_l + \PD[K]$ (so that the sequence $(\spincs_l)_{l \in \Z}$ repeats with period $d$), while $s_{l+1} = s_l + \frac{k}{d}$. We pin down the indexing by the conventions
\begin{gather}
\label{eq: xil-bound-pos}
\frac{(2l-1)k}{2d}  < \AlNorm_{Y,K}(\xi_l) \le \frac{(2l+1)k}{2d} \qquad \text{if } k>0, \\
\label{eq: xil-bound-neg}
\frac{(2l+1)k}{2d} \le  \AlNorm_{Y,K}(\xi_l) < \frac{(2l-1)k}{2d}   \qquad \text{if } k<0.
\end{gather}
Moreover, it is easy to see that
\begin{equation}\label{eq: A(xi-l)}
\AlNorm_{Y_\lambda, K_\lambda}(\xi_l)  = \frac{2d \AlNorm_{Y, K}(\xi_l) + k-d}{2k} = \frac{2ds_l+k-d}{2k}.
\end{equation}

For each $l \in \Z$, let $A^\infty_{\xi_l}$ and $B^\infty_{\xi_l}$ each denote a copy of $\CFKi(Y,K,\spincs_l)$. Define a pair of filtrations $\II_\spinct$ and $\JJ_\spinct$ and an absolute grading $\gr_\spinct$ on these complexes as follows:
\begin{align}
\intertext{For $[\x,i,j] \in A^\infty_{\xi_l}$,}
\label{eq: It-def-A} \II_\spinct([\x,i,j]) &= \max\{i,j-s_l\} \\
\label{eq: Jt-def-A} \JJ_\spinct([\x,i,j]) &= \max\{i-1,j-s_l\} + \frac{2ds_l+k-d}{2k} \\
\label{eq: grt-def-A} \gr_\spinct([\x,i,j]) &= \absgr([\x,i,j]) + \frac{(2ds_l - k)^2 }{4dk} + \frac{2-3\sign(k)}{4} \\
\intertext{For $[\x,i,j] \in B^\infty_{\xi_l}$,}
\label{eq: It-def-B} \II_\spinct([\x,i,j]) &= i \\
\label{eq: Jt-def-B} \JJ_\spinct([\x,i,j]) &= i-1 + \frac{2ds_l+k-d}{2k} \\
\label{eq: grt-def-B} \gr_\spinct([\x,i,j]) &= \absgr([\x,i,j]) + \frac{(2ds_l - k)^2 }{4dk} + \frac{-2-3\sign(k)}{4}
\end{align}
The values of $\II_\spinct$ are integers, while the values of $\JJ_\spinct$ live in the coset $\AlNorm_{Y_\lambda,K_\lambda}(\spinct)$. Let $A^-_{\xi_l}$ (resp.~$B^-_{\xi_l}$) denote the subcomplex of $A^\infty_{\xi_l}$ (resp.~$B^\infty_{\xi_l}$) generated by elements with $\II<0$, and let $A^+_{\xi_l}$ (resp.~$B^+_{\xi_l}$) denote the quotient by this subcomplex; these agree with the definitions from \cite{OSzRational}.

Let $v^\infty_{\xi_l} \co A^\infty_{\xi_l} \to B^\infty_{\xi_l}$ denote the identity map, and let $h^\infty_{\xi_l} \co A^\infty_{\xi_l} \to B^\infty_{\xi_{l+1}}$ denote the ``flip map'' $\Psi^\infty_{\xi_l}$ described above. Both $v^\infty_{\xi_l}$ and $h^\infty_{\xi_l}$ are filtered with respect to both $\II_\spinct$ and $\JJ_\spinct$ and homogeneous of degree $-1$ with respect to $\gr_\spinct$; this is obvious for $v^\infty_{\xi_l}$, and for $h^\infty_{\xi_l}$ it is Lemma \ref{lemma: h-filt} below. It is simple to check that $v^\infty_{\xi_l}$ and $h^\infty_{\xi_l}$ are homogeneous of degree $-1$ with respect to $\gr_\spinct$.

%

If $k>0$, then for any integers $a \le b$, define a map
\begin{equation}\label{eq: D-infty-k-pos}
D^\infty_{\lambda,\spinct,a,b} \co \bigoplus_{l=a}^b A^\infty_{\xi_l} \to \bigoplus_{l=a+1}^b B^\infty_{\xi_l}
\end{equation}
which is the sum of all the terms $v^\infty_{\xi_l}$ (for $l=a+1, \dots, b$) and $h^\infty_{\xi_l}$ (for $l=a, \dots, b-1$). If $k<0$, we likewise define
\begin{equation}\label{eq: D-infty-k-neg}
D^\infty_{\lambda,\spinct,a,b} \co \bigoplus_{l=a}^b A^\infty_{\xi_l} \to \bigoplus_{l=a}^{b+1} B^\infty_{\xi_l}
\end{equation}
to be the sum of all terms $v^\infty_{\xi_l}$ and $h^\infty_{\xi_l}$ for $l=a, \dots, b$. In either case, $D^\infty_{\lambda,\spinct,a,b}$ is a doubly-filtered chain map. Let $X^\infty_{\lambda, \spinct, a,b}$ denote the mapping cone of $D^\infty_{\lambda,\spinct,a,b}$, which inherits the structure of a doubly-filtered chain complex that is finitely generated over $\F[U,U^{-1}]$.
We will see below (Lemma \ref{lemma: Xab-indep-ab}) that for all $a$ sufficiently negative and all $b$ sufficiently positive, the doubly-filtered chain homotopy type of $X^\infty_{\lambda, \spinct, a, b}$ is independent of $a$ and $b$.

We are now  able to state the main theorem:

\begin{theorem} \label{thm: mapping-cone}
Let $K$ be a rationally null-homologous knot in a $3$-manifold $Y$, let $\lambda$ be a nonzero framing on $K$, and let $\spinct$ be any torsion spin$^c$ structure on $Y_\lambda(K)$. Then for all $a \ll 0$ and $b \gg 0$, the chain complex $X^\infty_{\lambda,\spinct,a,b}$, equipped with the filtrations $\II_\spinct$ and $\JJ_\spinct$, is doubly-filtered chain homotopy equivalent to $\CFKi(Y_\lambda, K_\lambda, \spinct)$.
\end{theorem}

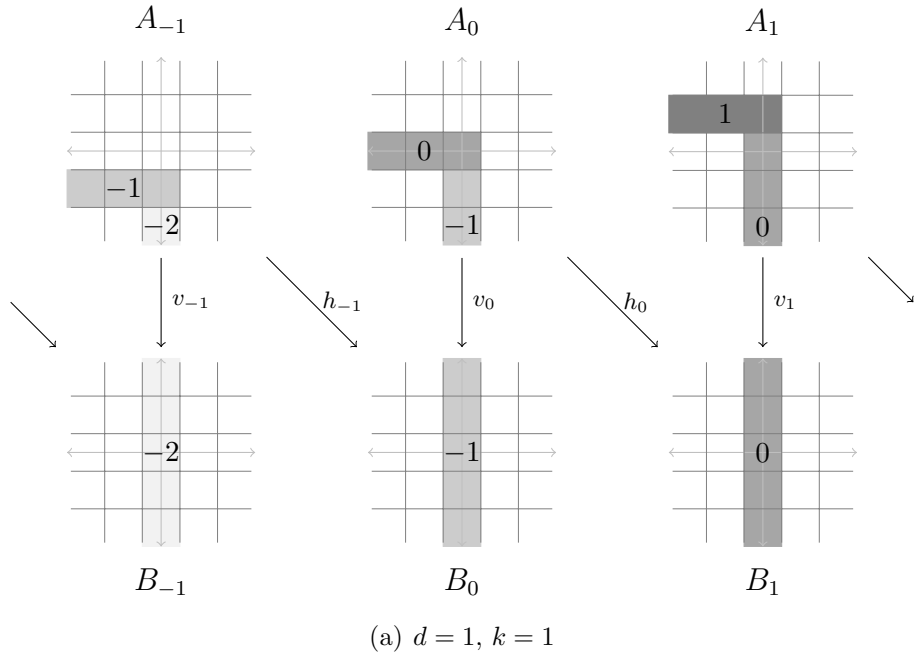
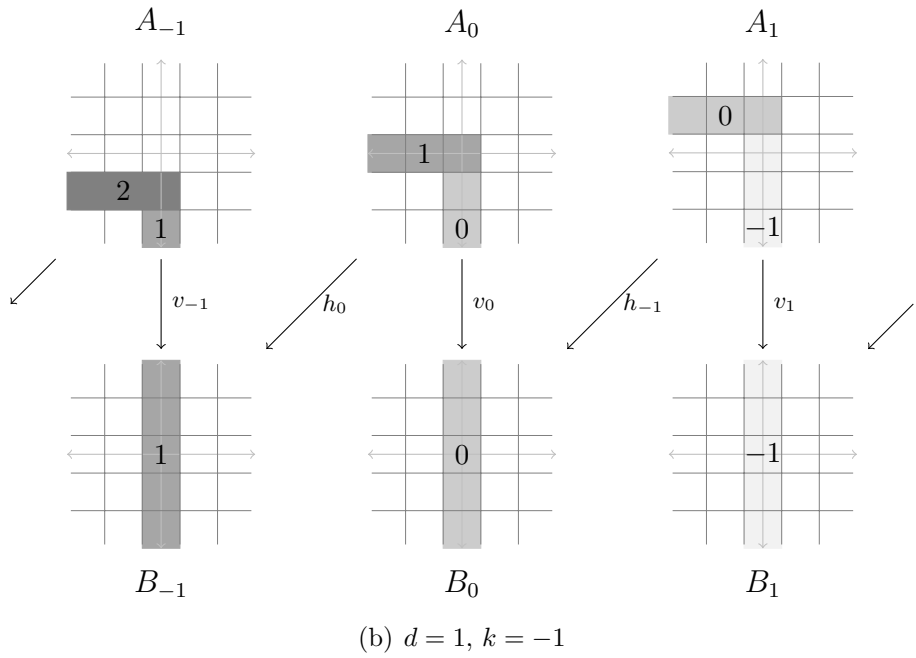
\begin{figure}
\subfigure[][$d=1$, $k=1$]{
\begin{tikzpicture}
\node[label=north:$A_{-1}$](A-1) at (-4,0) {
\begin{tikzpicture}[scale=0.5]
\filldraw[black!5!white] (0, -2) rectangle (1, -1);
\filldraw[black!20!white] (-2, -1) rectangle (1, 0);
	\begin{scope}[thin, black!25!white]
		\draw [<->] (-2, 0.5) -- (3, 0.5);
		\draw [<->] (0.5, -2) -- (0.5, 3);
	\end{scope}
	\draw[step=1, black!50!white, very thin] (-1.9, -1.9) grid (2.9, 2.9);
\node[] at (-0.5,-0.5) {\small $-1$};
\node[] at (0.5,-1.5) {\small $-2$};
\end{tikzpicture}};

\node[label=south:$B_{-1}$](B-1) at (-4,-4) {
\begin{tikzpicture}[scale=0.5]
\filldraw[black!5!white] (0, -2) rectangle (1, 3);
	\begin{scope}[thin, black!25!white]
		\draw [<->] (-2, 0.5) -- (3, 0.5);
		\draw [<->] (0.5, -2) -- (0.5, 3);
	\end{scope}
	\draw[step=1, black!50!white, very thin] (-1.9, -1.9) grid (2.9, 2.9);
\node[] at (0.5,0.5) {\small $-2$};
\end{tikzpicture}};

\node[label=north:$A_0$](A0) at (0,0) {
\begin{tikzpicture}[scale=0.5]
\filldraw[black!20!white] (0, -2) rectangle (1, 0);
\filldraw[black!35!white] (-2, 0) rectangle (1, 1);
	\begin{scope}[thin, black!25!white]
		\draw [<->] (-2, 0.5) -- (3, 0.5);
		\draw [<->] (0.5, -2) -- (0.5, 3);
	\end{scope}
	\draw[step=1, black!50!white, very thin] (-1.9, -1.9) grid (2.9, 2.9);
\node[] at (-0.5,0.5) {\small $0$};
\node[] at (0.5,-1.5) {\small $-1$};
\end{tikzpicture}};

\node[label=south:$B_0$](B0) at (0,-4) {
\begin{tikzpicture}[scale=0.5]
\filldraw[black!20!white] (0, -2) rectangle (1, 3);
	\begin{scope}[thin, black!25!white]
		\draw [<->] (-2, 0.5) -- (3, 0.5);
		\draw [<->] (0.5, -2) -- (0.5, 3);
	\end{scope}
	\draw[step=1, black!50!white, very thin] (-1.9, -1.9) grid (2.9, 2.9);
\node[] at (0.5,0.5) {\small $-1$};
\end{tikzpicture}};

\node[label=north:$A_1$](A1) at (4,0) {
\begin{tikzpicture}[scale=0.5]
\filldraw[black!35!white] (0, -2) rectangle (1, 1);
\filldraw[black!50!white] (-2, 1) rectangle (1, 2);
	\begin{scope}[thin, black!25!white]
		\draw [<->] (-2, 0.5) -- (3, 0.5);
		\draw [<->] (0.5, -2) -- (0.5, 3);
	\end{scope}
	\draw[step=1, black!50!white, very thin] (-1.9, -1.9) grid (2.9, 2.9);
\node[] at (-0.5,1.5) {\small $1$};
\node[] at (0.5,-1.5) {\small $0$};
\end{tikzpicture}};

\node[label=south:$B_1$](B1) at (4,-4) {
\begin{tikzpicture}[scale=0.5]
\filldraw[black!35!white] (0, -2) rectangle (1, 3);
	\begin{scope}[thin, black!25!white]
		\draw [<->] (-2, 0.5) -- (3, 0.5);
		\draw [<->] (0.5, -2) -- (0.5, 3);
	\end{scope}
	\draw[step=1, black!50!white, very thin] (-1.9, -1.9) grid (2.9, 2.9);
\node[] at (0.5,0.5) {\small $0$};
\end{tikzpicture}};

\draw [->] (-6,-2) -- (B-1);
\draw [->] (A-1) -- (B-1) node[midway,anchor=west] {\scriptsize $v_{-1}$};
\draw [->] (A-1) -- (B0) node[midway,anchor=west] {\scriptsize $h_{-1}$};
\draw [->] (A0) -- (B0) node[midway,anchor=west] {\scriptsize $v_{0}$};
\draw [->] (A0) -- (B1) node[midway,anchor=west] {\scriptsize $h_{0}$};
\draw [->] (A1) -- (B1) node[midway,anchor=west] {\scriptsize $v_{1}$};
\draw [->] (A1) -- (6,-2);
\end{tikzpicture}  \label{sfig: mapping-cone-ex-pos}
}

\subfigure[][$d=1$, $k=-1$]{
\begin{tikzpicture}
\node[label=north:$A_{-1}$](A-1) at (-4,0) {
\begin{tikzpicture}[scale=0.5]
\filldraw[black!35!white] (0, -2) rectangle (1, -1);
\filldraw[black!50!white] (-2, -1) rectangle (1, 0);
	\begin{scope}[thin, black!25!white]
		\draw [<->] (-2, 0.5) -- (3, 0.5);
		\draw [<->] (0.5, -2) -- (0.5, 3);
	\end{scope}
	\draw[step=1, black!50!white, very thin] (-1.9, -1.9) grid (2.9, 2.9);
\node[] at (-0.5,-0.5) {\small $2$};
\node[] at (0.5,-1.5) {\small $1$};
\end{tikzpicture}};

\node[label=south:$B_{-1}$](B-1) at (-4,-4) {
\begin{tikzpicture}[scale=0.5]
\filldraw[black!35!white] (0, -2) rectangle (1, 3);
	\begin{scope}[thin, black!25!white]
		\draw [<->] (-2, 0.5) -- (3, 0.5);
		\draw [<->] (0.5, -2) -- (0.5, 3);
	\end{scope}
	\draw[step=1, black!50!white, very thin] (-1.9, -1.9) grid (2.9, 2.9);
\node[] at (0.5,0.5) {\small $1$};
\end{tikzpicture}};

\node[label=north:$A_0$](A0) at (0,0) {
\begin{tikzpicture}[scale=0.5]
\filldraw[black!20!white] (0, -2) rectangle (1, 0);
\filldraw[black!35!white] (-2, 0) rectangle (1, 1);
	\begin{scope}[thin, black!25!white]
		\draw [<->] (-2, 0.5) -- (3, 0.5);
		\draw [<->] (0.5, -2) -- (0.5, 3);
	\end{scope}
	\draw[step=1, black!50!white, very thin] (-1.9, -1.9) grid (2.9, 2.9);
\node[] at (-0.5,0.5) {\small $1$};
\node[] at (0.5,-1.5) {\small $0$};
\end{tikzpicture}};

\node[label=south:$B_0$](B0) at (0,-4) {
\begin{tikzpicture}[scale=0.5]
\filldraw[black!20!white] (0, -2) rectangle (1, 3);
	\begin{scope}[thin, black!25!white]
		\draw [<->] (-2, 0.5) -- (3, 0.5);
		\draw [<->] (0.5, -2) -- (0.5, 3);
	\end{scope}
	\draw[step=1, black!50!white, very thin] (-1.9, -1.9) grid (2.9, 2.9);
\node[] at (0.5,0.5) {\small $0$};
\end{tikzpicture}};

\node[label=north:$A_1$](A1) at (4,0) {
\begin{tikzpicture}[scale=0.5]
\filldraw[black!5!white] (0, -2) rectangle (1, 1);
\filldraw[black!20!white] (-2, 1) rectangle (1, 2);
	\begin{scope}[thin, black!25!white]
		\draw [<->] (-2, 0.5) -- (3, 0.5);
		\draw [<->] (0.5, -2) -- (0.5, 3);
	\end{scope}
	\draw[step=1, black!50!white, very thin] (-1.9, -1.9) grid (2.9, 2.9);
\node[] at (-0.5,1.5) {\small $0$};
\node[] at (0.5,-1.5) {\small $-1$};
\end{tikzpicture}};

\node[label=south:$B_1$](B1) at (4,-4) {
\begin{tikzpicture}[scale=0.5]
\filldraw[black!5!white] (0, -2) rectangle (1, 3);
	\begin{scope}[thin, black!25!white]
		\draw [<->] (-2, 0.5) -- (3, 0.5);
		\draw [<->] (0.5, -2) -- (0.5, 3);
	\end{scope}
	\draw[step=1, black!50!white, very thin] (-1.9, -1.9) grid (2.9, 2.9);
\node[] at (0.5,0.5) {\small $-1$};
\end{tikzpicture}};

\draw [->] (A-1) -- (-6,-2);
\draw [->] (A-1) -- (B-1) node[midway,anchor=west] {\scriptsize $v_{-1}$};
\draw [->] (A0) -- (B-1) node[midway,anchor=west] {\scriptsize $h_{0}$};
\draw [->] (A0) -- (B0) node[midway,anchor=west] {\scriptsize $v_{0}$};
\draw [->] (A1) -- (B0) node[midway,anchor=west] {\scriptsize $h_{-1}$};
\draw [->] (A1) -- (B1) node[midway,anchor=west] {\scriptsize $v_{1}$};
\draw [->] (6,-2) -- (B1);
\end{tikzpicture} \label{sfig: mapping-cone-ex-neg}
}

\caption{A portion of the filtered mapping cone for $d=1$ and $k=\pm 1$. The shaded regions represent the portion of each complex $A_s^\infty$ or $B_s^\infty$ with $\II_\spinct=0$, or in other words the subquotients $A_s^0$ and $B_s^0$. The level of shading indicates the $\JJ_\spinct$ filtration. The values of both $\II_\spinct$ and $\JJ_\spinct$ on the remainder of each complex are determined by the fact that multiplication by $U$ (i.e. translation down and to the left) decreases both filtrations by $1$.  Taking the homology of the mapping cone of the (two) maps relating shaded regions with a fixed integer label  yields $\HFKa(Y_{\pm1}(K), K_{\pm1})$ in the Alexander grading corresponding to the label.}
\label{fig: mapping-cone-ex}

\end{figure}

\begin{example}
If the knot $K$ is null-homologous (i.e. $d=1$), the Alexander grading for $(Y,K)$ is integer-valued, so the values of $s_l$ are integers, and the spin$^c$ structures $\spincs_l$ are all the same. In particular, when $k=\pm 1$, the bounds \eqref{eq: xil-bound-pos} and \eqref{eq: xil-bound-neg} imply that $s_l = \pm l$. A portion of the mapping cone complex in the $k = \pm 1 $ cases is shown in Figure \ref{fig: mapping-cone-ex}, where we index the $A$ and $B$ complexes by the integers $s_l$, as in \cite{OSzSurgery}.
\end{example}

The proof of Theorem \ref{thm: mapping-cone} follows the same template as Ozsv\'ath and Szab\'o's original proof in \cite{OSzSurgery, OSzRational}, with some modifications. Specifically, we examine the construction of an exact triangle relating the Heegaard Floer homologies of $Y$, $Y_\lambda(K)$, and $Y_{\lambda+m\mu}(K)$ where $m$ is large. The main new ingredient is the behavior of the maps in the exact triangle with respect to the Alexander gradings, which turns out to be quite subtle. Specifically, we must show that the chain maps and chain homotopies used in the exact triangle detection lemma are filtered with respect to the relevant Alexander filtrations. This turns out to be true only for the subquotients of the Heegaard Floer complexes consisting of generators with bounded powers of $U$, and it requires modification of the construction of the maps by considering only certain spin$^c$ structures on the various cobordisms involved.

In an unpublished preprint from 2006 \cite{EftekharyMorse},  Eftekhary described a similar mapping cone formula for $\CFKi(Y_\lambda(K), K_\lambda)$. Although there are certain technical problems with that formula, primarily related to the behavior of the flip maps $\Psi^\infty_\xi$ and the filtration issues discussed in the previous paragraph, the overarching ideas are similar. Moreover, the ``hat'' version of our formula (i.e., the associated graded complex, which computes $\HFKa(Y_\lambda(K), K_\lambda)$) coincides with Eftekhary's description in \cite{EftekharyIncompressible}; see Corollary \ref{cor: mapping-cone-hat} below.

A key technical tool which allows us to get a handle on the grading subtleties is a formula for the rational Alexander grading of knot Floer homology generators in terms of data on the Heegaard diagram. This formula is analogous to Ozsv\'ath and Szab\'o's formula for the evaluation of the Chern class of a spin$^c$ structure associated to a Floer complex generator on  the homology class of a periodic domain.  We expect this formula to be a useful addition to the Heegaard Floer tool-box, independent of the present paper. (For instance, it was recently used by Raoux \cite{RaouxTau}). For that reason, we take the time to state it here.  Recall that a \emph{relative periodic domain} on a doubly-pointed Heegaard diagram is a linear combination of its regions whose boundary consists of multiples of the $\alpha$ and $\beta$ curves and a longitude for the knot, drawn as a union of arcs connecting the basepoints. (See \cite[Definition 2.1]{HeddenPlamenevskayaRational}).

\begin{proposition} \label{prop: abs-alex} Let $(\Sigma, \alpha,\beta,w,z)$ be a doubly-pointed Heegaard diagram for a knot $(Y,K)$  representing a class in $H_1(Y)$ of order $d$, and let $P$ be a relative periodic domain specifying a homology class $[P]\in H_2(Y,K)$.  Then
the absolute Alexander grading of a generator $\x \in \T_\alpha \cap \T_\beta$, taken with respect to $[P]$, is given by
\begin{align}
\label{eq: abs-alex} \AlNorm_{w,z}(\x) &= \frac{1}{2d} \left( \hat\chi(P) + 2n_\x(P) - n_{z}(P) - n_{w}(P) \right)
\end{align}
Here $\hat\chi(P)$ denotes the Euler measure of $P$, and $n_\x(P)$ denotes the sum, taken over all $x_i \in \x$, of the average of the four local multiplicities of $P$ in the regions abutting $x_i$. Finally, $n_{w}(P)$ (resp.~$n_{z}(P)$) denotes the average of the multiplicities of $P$ on either side of the longitude at $w$ (resp.~$z$).

\end{proposition}

\subsection{Future directions}

Before turning to the proof of Theorem \ref{thm: mapping-cone}, we discuss a few potential applications and directions for future investigation.

Our formula should  be useful for computing the Heegaard Floer homology of a splice of knot complements in terms of their knot Floer homology. Indeed, let $K$ and $K'$ be knots in $S^3$, and let $M$ be the manifold obtained by gluing the exteriors of $K$ and $K'$, where the gluing identifies the meridian of $K$ (resp.~$K'$) with a longitude $\lambda'$ (resp.~$\lambda$) of $K'$ (resp.~$K$). Then $M$ can be viewed as Dehn surgery on the knot $K_\lambda \conn K' \subset S^3_\lambda(K)$. Thus, we can determine the Heegaard Floer homology of $M$ (and, better yet, the knot Floer complex of a certain knot in $M$) as follows: use Theorem \ref{thm: mapping-cone} to determine $\CFKi(Y_\lambda(K), K_\lambda)$, take the tensor product with $\CFKi(S^3,K')$ to obtain $\CFKi(Y_\lambda(K), K_\lambda \conn K')$, and then use the surgery formula again to determine the Heegaard Floer homology of $M$. The only difficulty is that for the second application, we need to understand the flip map on $\CFKi(Y_\lambda(K), K_\lambda \conn K')$. This can always be done explicitly if $Y_\lambda(K)$ is an L-space; see Lemma \ref{lemma: Psi-Lspace} below. The general case would be tractable if we could compute the flip map on $\CFKi(Y_\lambda(K), K_\lambda)$ in terms of the mapping cone formula, but at present we do not know of such a description.

%
%

In another direction, the knot Floer homology of fibered  knots carries geometric information about their associated contact structures.  As mentioned above, this idea was used in \cite{HeddenPlamenevskayaRational} in conjunction with the ``large surgery'' version of our formula to give conditions for surgeries on a fibered knot to admit tight contact structures.  The present work allows us to extend the scope of these results.  In particular, the formula for the knot Floer homology of the dual knot to $\pm 1$ surgery on a fibered knot significantly extends the potential scope of applications to contact geometry.  This is because the dual knot to  $-1$ (resp. $+1$) surgery on a fibered knot $K$ is a fibered knot whose monodromy differs from that of $K$ by a right-handed (resp. left-handed) boundary Dehn twist.  As an application,  coupling our formula with the strong detection by knot Floer homology of the identity mapping class (\cite[Theorem 4]{HeddenWatsonGeography}) should allow us to prove that knot Floer homology determines whether a knot is fibered with monodromy consisting of a boundary Dehn twist; we plan to return to this question in a future paper.

In the same vein, our formula will allow us to derive conditions on the knot Floer homology of a fibered knot which determine whether adding a left-handed (respectively right-handed) Dehn twist along the boundary will kill the contact invariant (respectively have non-trivial contact invariant).  This understanding, in turn, should lead to restrictions on the fractional Dehn twist coefficient of the monodromy of a fibered knot in terms of its Floer homology and its flip map. To the latter end, it would be quite useful to have a formula for the Floer homology of the dual knot to a rational surgery. This would allow for a determination of the integral part of the fractional Dehn twist in terms of knot Floer homology. (In Section \ref{sec: rational-surgery}, we describe $\CFKi$ of the knot in a rational surgery obtained from the meridian of the surgery curve; however, when the surgery slope is not integral, the meridian is not isotopic to the core of the surgery solid torus.)

In another direction, it may be possible to generalize our formula to a much broader situation. Let $L \subset Y$ be a link of $n$ components with a framing $\Lambda$. Manolescu and Ozsv\'ath \cite{ManolescuOzsvathLink} give a formula for determining the Heegaard Floer homology of any surgery $Y_\Lambda(L)$ on $L$ in terms of the link Floer complex $\CFLi(Y,L)$. If $K$ is any knot in $Y \minus L$, and $K^*$ is the induced knot in $Y_\Lambda(L)$, it might be possible to obtain a similar formula for $\CFKi(Y_\Lambda(L), K^*)$ in terms of $\CFLi(Y, L \cup K)$, by tracing through the Manolescu--Ozsv\'ath's proof while keeping track of an additional Alexander grading corresponding to $K$, along the lines of our argument below. Carrying out this proof seems like a duanting task, given the technical issues involved in our substantially simpler situation.   If one could, then one could likely recover Theorem \ref{thm: mapping-cone} by applying this more general result along with following two observations. First, for any knot $K \subset Y$ with framing $\lambda$ and meridian $\mu$, the knot $K_\lambda \subset Y_\lambda$ is isotopic to $\mu^*$ (in the terminology of this section). Second, the link Floer complex $\CFLi(Y, K \cup \mu)$ can be determined from $\CFKi(Y,K)$ since $K \cup \mu$ is the connected sum of $L$ with a Hopf link. This would then lead to a description of $\CFKi(Y_\lambda, K_\lambda)$ in terms of $\CFKi(Y,K)$, which presumably would agree with Theorem \ref{thm: mapping-cone}.

More abstractly, our main theorem can be viewed as a stand-in for the infinity or minus version of the bordered Floer homology of a knot complement in a general $3$-manifold.  More precisely, the bordered invariant of a manifold with torus boundary will admit a splitting with respect to  idempotents corresponding to a basis for its first homology.  The basis can be taken as a meridian and longitude for a knot in any $3$-manifold obtained by Dehn filling.  In these terms, our formula allows us to compute the invariant gotten by projection to one of the idempotents in terms of the invariant gotten by projection to the other.    In principle, higher structure maps will be necessary to understand the full $A_\infty$ module associated to a knot complement by any minus version of bordered Floer homology, but in practice it should be feasible to work solely with our formula. For many applications this should prove easier.

\subsection*{Organization}
In Section \ref{sec: preliminaries}, we collect various preliminary results, many of which can be described as ``Heegaard Folkloer,'' and prove Proposition \ref{prop: abs-alex}. In Section \ref{sec: mapping-cone}, we provide more details about the mapping cone formula, outline the proof, and discuss an example. In Section \ref{sec: alex-surgery}, we study the behavior of the Alexander grading under $2$-handle cobordisms. In Section \ref{sec: large}, we look at the Heegaard quadruple diagrams relating $Y$, $Y_\lambda(K)$, and $Y_{\lambda+m\mu}(K)$, and give a formula for $\CFKi$ of the dual knot in large surgery. The most technical part of the paper is Section \ref{sec: exact-sequence}, where we go through the construction of the surgery exact sequence relating $Y$, $Y_\lambda(K)$, and $Y_{\lambda+m\mu}(K)$ and study the behavior of each map with respect to the Alexander gradings. In Section \ref{sec: proof}, we assemble the pieces to prove Theorem \ref{thm: mapping-cone}. Finally, in Section \ref{sec: rational-surgery}, we discuss rational surgeries.

\subsection*{Acknowledgments} We are deeply grateful for helpful discussions (over the course of many years) with Peter Ozsv\'ath, Zolt\'an Szab\'o, Eaman Eftekhary, Jen Hom,  Tye Lidman, Tom Mark, Olga Plamenevskaya, and Andr\'as Stipsicz.

\section{Preliminaries} \label{sec: preliminaries}

\subsection{Homological algebra} \label{ssec: algebra}

We begin by stating a few basic facts about filtered chain complexes that will be useful later on.

\begin{definition} \label{def: filt}
Let $S$ be a partially ordered set. An \emph{$S$-filtered chain complex} is a chain complex $C$ (over any ring) equipped with an exhausting family of subcomplexes $\{\FF_s C \mid s \in S\}$, such that $\FF_s C \subset \FF_{s'} C$ whenever $s \le s'$. The \emph{associated graded complex} of $C$ is
\[
\Gr(C) = \bigoplus_{s \in S}  \Gr_s(C), \quad \text{where} \quad \Gr_s(C) = \frac{\FF_s C }{\sum_{s'< s} \FF_{s'} C},
\] 
with the induced differential. Given two $S$-filtered chain complexes $B$ and $C$, a chain map $f \co B \to C$ is called a \emph{filtered chain map} if $f(\FF_s B) \subset \FF_s C$ for all $s \in S$. Two filtered chain maps $f, g \co B \to C$ are \emph{filtered homotopic} if they are related by a chain homotopy $h$ such that $h(\FF_s B) \subset \FF_s C$ for all $s \in S$. We call $f$ a \emph{filtered homotopy equivalence} if there is a filtered chain map $g \co C \to B$ such that $g \circ f$ and $f \circ g$ are each filtered homotopic to the respective identity maps. We call $f$ a \emph{filtered quasi-isomorphism} if it induces an isomorphism on the homology of the associated graded complexes.
(We emphasize that a homotopy equivalence that is filtered is not necessarily a filtered homotopy equivalence, and a quasi-isomorphism that is filtered is not necessarily a filtered quasi-isomorphism. This terminology is unfortunately fairly standard in the literature.)
\end{definition}

A filtered chain homotopy equivalence is immediately seen to be a filtered quasi-isomorphism, but the converse does not hold in general, even over a field. Indeed, considerable caution is required when working with filtrations by an arbitrary partially ordered set as opposed to $\Z$.  For instance, suppose $C$ is generated over $\F$ by a single element $x$, with vanishing differential, and we define a $\Z \times \Z$ filtration on $C$ by
\[
\FF_{(i,j)}C =
\begin{cases}
  C & i \ge 0 \\
  0 & i < 0.
\end{cases}
\]
According to Definition \ref{def: filt}, we then have $\Gr_{(i,j)} (C) = 0$ for all $(i,j) \in \Z \times \Z$, even though $C$ has nontrivial homology! Moreover, $C$ is filtered quasi-isomorphic to the trivial complex, even though it is clearly not homotopy equivalent to this complex.\footnote{We are grateful to the anonymous referee of \cite{HomLevineLidmanPL} for pointing out this example.}

We therefore specialize to a particular type of filtration, as follows. Let $(C, \FF)$ be a filtered complex over a field $\F$. We call $(C,\FF)$ \emph{special} if there exists a basis $\mathcal{B}_i$ for each chain group $C_i$, and a function $\JJ \co \mathcal{B}_i \to S$, such that
\begin{equation} \label{eq: filt-function}
\FF_s C_i = \Span \{x \in \mathcal{B}_i \mid \JJ(x) \le s\}.
\end{equation}
Such a basis is called a \emph{filtered basis}. Given a complex $C$, we may describe a special filtration by simply specifying a function $\JJ$ defined on some basis, provided that the differential of each basis element $x$ is a linear combination of basis elements whose $\JJ$ values are less than or equal to $\JJ(x)$, and taking \eqref{eq: filt-function} as the definition of the filtration. If $(C,\FF)$ is special, then any choice of filtered basis induces an isomorphism of vector spaces from $C$ to $\Gr(C)$. (Thus, the filtered complex in the previous paragraph is not special.)

The following lemma is immediate:
\begin{lemma} \label{lemma: assoc-gr-isom}
If $(C,\FF)$ is a special $S$-filtered chain complex over a field $\F$, then $C$ and $\Gr(C)$ are isomorphic as graded vector spaces over $\F$ (although not necessarily as chain complexes), where we take the homological grading on both complexes. Moreover, if $f \co B \to C$ is a filtered chain map between special $S$-filtered complexes, and $\Gr(f) \co \Gr(B) \to \Gr(C)$ is the induced map on associated graded objects, then we may choose isomorphisms to make the square
\begin{equation} \label{eq: assoc-gr-isom}
\xymatrix{ B \ar[r]^f \ar[d]^\cong & C \ar[d]^\cong \\ \Gr(B) \ar[r]^{\Gr(f)} & \Gr(C) }
\end{equation}
commute. \qed
\end{lemma}

A filtered chain complex is called \emph{reduced} if the induced differential on the associated graded complex vanishes, or equivalently if $\partial(\FF_s C) \subset \sum_{s' < s} \FF_{s'} C$. As an immediate consequence of Lemma \ref{lemma: assoc-gr-isom}, note that any filtered quasi-isomorphism between reduced special complexes is an isomorphism. Moreover, we have:
\begin{lemma} \label{lemma: reduction-special}
Any finitely generated, special $S$-filtered chain complex is filtered homotopy equivalent to a reduced complex.
\end{lemma}

\begin{proof}
Follow the discussion in \cite[Section 4.1]{HeddenNiUnlink}, using the filtration function $\JJ$ in place of the notion of ``filtration level.''
\end{proof}

\begin{lemma} \label{lemma: filt-QI-special}
If $B$ and $C$ are finitely generated, special $S$-filtered chain complexes, then any filtered quasi-isomorphism from $B$ to $C$ is a filtered homotopy equivalence.
\end{lemma}

\begin{proof}
By Lemma \ref{lemma: reduction-special}, we may find reduced complexes $B'$ and $C'$ which are filtered homotopy equivalent to $B$ and $C$, respectively. The composition $B' \xrightarrow{\simeq} B \xrightarrow{f} C \xrightarrow{\simeq} C'$ is a filtered quasi-isomorphism, and therefore a filtered chain {\em isomorphism} of complexes of $\F[U]$-modules. It then follows that $f$ is a filtered homotopy equivalence.
\end{proof}

Because Lemmas \ref{lemma: reduction-special} and \ref{lemma: filt-QI-special} are stated for finitely generated chain complexes, we need a slightly modified version for the types of complexes that arise in Heegaard Floer homology. Let $\F$ be any field. Analogous to \cite[Definition 2.6]{OSzSurgery} and \cite[Definition 10.2]{ManolescuOzsvathLink}, we say that a \emph{chain complex of torsion $\CFm$ type} is a finitely generated, free module $C$ over $\F[U]$, equipped with an absolute $\Q$-grading that lifts a relative $\Z$-grading, for which multiplication by $U$ has degree $-2$, and a differential $\partial$ with degree $-1$.\footnote{Unlike in loc.~cit., we restrict our attention to complexes that are actually finitely generated, free modules, rather than complexes that are quasi-isomorphic to such complexes.} Given a basis $\{x_1, \dots, x_k\}$ for $C$ consisting of homogeneous elements, note that if the coefficient of $U^n x_j$ in $\partial(x_i)$ is nonzero, then $n = (\gr(x_j) - \gr(x_i) + 1)/2$.

Next, we discuss filtrations. Suppose $C$ is a complex of torsion $\CFm$ type, and let $x_1, \dots, x_k$ be homogeneous elements which form an $\F[U]$--basis for $C$. Let $\JJ \co \{x_1, \dots, x_k\} \to \Q$ be a function whose values are congruent mod $\Z$, and extend this function to the set of all translates $\{U^n x_i\}$ by declaring $\JJ(U^n x_i) = \JJ(x_i) - n$. Suppose that whenever $U^n x_j$ appears in $\partial(x_i)$, we have $\JJ(U^n x_i) \le \JJ(x_j)$. Then the subspaces of $C$ spanned by the sublevel sets of $\JJ$ give a filtration of $C$ by $\F[U]$--subcomplexes. A filtration obtained in this way is said to be \emph{of Alexander type}. We will typically just refer to $\JJ$ as the filtration.

Note that the filtration of $C$ by the subcomplexes $U^n C$ is itself of Alexander type, defined via a function $\II$ that is identically $0$ on the elements of any basis for $C$. We call this the \emph{trivial filtration}. Any $\F[U]$--equivariant quasi-isomorphism between complexes of $\CFm$ type is a filtered quasi-isomorphism with respect to the trivial filtration. If we are given a second filtration $\JJ$ of Alexander type, $C$ acquires the structure of a special $\Z \times \Z$--filtered complex, using the above terminology. Moreover, $C$ is reduced if there are no terms in the differential that preserve both the $\II$ and $\JJ$ filtrations; in other words, if $U^n x_j$ appears in $\partial(x_i)$, then either $n>0$ or $\JJ(U^n x_j) < \JJ(x_i)$. Note that a reduced complex is isomorphic to its associated graded complex as an $\F[U]$--module, not just as an $\F$--vector space. Moreover, the analogue of Lemma \ref{lemma: reduction-special} also holds here:
\begin{lemma} \label{lemma: reduction}
Let $C$ be a complex of $\CFm$ type equipped, equipped with an Alexander-type filtration $\JJ$. Then $C$ is filtered homotopy equivalent (over $\F[U]$) to a reduced complex.  \qed
\end{lemma}
See, e.g., \cite[Section 4.1]{HeddenNiUnlink}, \cite[Reduction Lemma, p.~1005]{HeddenWatsonGeography}, or \cite[Proposition 11.57]{LOTBordered} for a proof; the key point is that the cancellations taking $C$ to a reduced complex can be performed $U$-equivariantly. Likewise, akin to Proposition \ref{lemma: filt-QI-special} above, we have:
\begin{proposition} \label{prop: filt-quasi}
Let $B$ and $C$ be complexes of $\CFm$ type, each equipped with an Alexander-type filtration. Then any filtered quasi-isomorphism $f \co B \to C$ (over $\F[U]$) is a filtered homotopy equivalence.
\end{proposition}

%

Next, we introduce the machinery of ``vertical truncation.'' Given a chain complex $C$ of torsion $\CFm$ type, for any $t \in \N$, let $C^t$ denote the quotient $C / U^{t+1} C$. Any filtration of $C$ of Alexander type descends to a filtration of $C^t$. Note that $C^t$ is a free module over $\F[U]/(U^{t+1})$, and any basis for $C$ descends to a basis for $C^t$. Moreover, for any $t<t'$, we have natural isomorphisms $U^{t'-t} \cdot C^{t'} \cong C^t$ (with a grading shift of $2(t'-t)$). The following lemmas imply that a (filtered) complex $C$ of torsion $\CFm$ type is determined up to (filtered) quasi-isomorphism by the complexes $C^t$ for large $t$. (Compare \cite[Lemma 2.7]{OSzSurgery} and \cite[Lemma 10.4]{ManolescuOzsvathLink}.)

\begin{lemma} \label{lemma: trunc-QI}
Let $C$ be a complex of torsion $\CFm$ type, equipped with a filtration of Alexander type. Then for large $t$, $C$ is filtered quasi-isomorphic to $C^t$ in sufficiently large gradings. To be precise, for any $\delta \in \Q$, there exists $T \in \N$ such that for all $t \ge T$, all gradings $d \ge \delta$, and all filtration levels $s$, the projection map $C \to C^t$ induces isomorphisms $H_d(C) \to H_d(C^t)$ and $H_d(\Gr_s(C)) \to H_d (\Gr_s(C^t))$.
\end{lemma}

\begin{proof}
Given $\delta$, for all $t$ sufficiently large, the projection $C$ to $C^t$ is simply the identity map in all gradings $d \ge \delta -1$, and the filtrations on those portions of $C$ and $C^t$ agree by construction. The result then follows immediately.
\end{proof}

\begin{lemma} \label{lemma: filt-QI}
Let $B$ and $C$ be chain complexes of torsion $\CFm$ type, each equipped with a filtration of Alexander type. Suppose that for all $t \ge 0$, the complexes $B^t$ and $C^t$ are $\F[U]$--equivariantly filtered quasi-isomorphic. Then $B$ and $C$ are $\F[U]$--equivariantly filtered quasi-isomorphic.
\end{lemma}

\begin{proof}
To begin, we show that there is a single chain map $f \co B^t \to C^t$ that induces filtered quasi-isomorphisms $B^t \to C^t$ for all $t$ simultaneously. (A priori, the hypotheses of the lemma only stipulate that there exist such maps for each $t$, without requiring them to be related in any way.)

Let $\{x_1, \dots, x_k\}$ and $\{y_1, \dots, y_l\}$ be bases for $B$ and $C$, respectively, on which we have functions $\JJ_B$ and $\JJ_C$ specifying the filtrations. Choose some $t_0$ large enough that for all $z, z' \in \{x_1, \dots, x_k, y_1, \dots, y_l\}$, we have $t_0 > (\gr(z') - \gr(z) + 1)/2$. Let $\partial_B$ and $\partial_C$ denote the differentials on $B$ and $C$ respectively, and $\partial_B^{t_0}$ and $\partial_C^{t_0}$ the induced differentials on $B^{t_0}$ and $C^{t_0}$. Choose a filtered quasi-isomorphism $f^{t_0} \co B^{t_0} \to C^{t_0}$.

Let us write
\[
\partial_B^{t_0}([x_i]) = \sum_j p_{i,j}[x_j], \qquad
\partial_C^{t_0}([y_i]) = \sum_j q_{i,j} [y_j], \quad \text{and} \quad
f^{t_0}([x_i]) = \sum_j r_{i,j} [y_j],
\]
where each coefficient $p_{i,j}$, $q_{i,j}$, and $r_{i,j}$ is either $0$ or a multiple of $U^n$ for some $0 \le n \le t$. We claim that the differential $\partial_B$ must be given by precisely the same formula: $\partial_B(x_i) = \sum_j p_{i,j} x_j$. Indeed, every nonzero term in $\partial_B^{t_0}$ must be induced from the corresponding term in $\partial_B$. The only possible additional terms would have to involve powers of $U$ that vanish in $\F[U]/U^{t_0}$; however, this contradicts our hypothesis on $t$. The same applies identically to $\partial_C$. Likewise, the map $f \co B \to C$ defined by $f(x_i) = \sum_j r_{i,j} y_j$ is a chain map: any nonzero term in $\partial_C \circ f$ must also occur in $\partial_C^{t_0} \circ f^{t_0}$, and hence be cancelled by a term in $f^{t_0} \circ \partial_B^{t_0}$, which then also occurs in $f \circ \partial_B$.

Next, we claim that $f$ induces a filtered quasi-isomorphism $B^t \to C^t$ for all $t$. For $t<t_0$, this follows by restricting $f^{t_0}$ to the kernel of $U^{t+1}$, while for $t \ge t_0$, it then follows by induction using the five-lemma.

By the previous lemma, for any grading $d$ and filtration level $s$, we may find $t$ for which the map induced by $f$ factors as
\[ \xymatrix{
H_d(\Gr_s (B)) \ar[r]_{\cong} & H_d(\Gr_s(B^t)) \ar[r]^{f^t_*} _{\cong} & H_d(\Gr_s(C^t))  \ar[r]_{\cong}  &H_d(\Gr_s (C)) }.
\]
Thus, $f$ is a filtered quasi-isomorphism, as required.
\end{proof}

The reason for dwelling on the distinction between filtered quasi-isomorphism and filtered homotopy equivalence is that the proof our main theorem relies on a filtered version of the mapping cone detection lemma \cite[Lemma 4.2]{OSzDouble}, which takes place in the filtered derived category. Although we will mainly work over $\Z/2\Z$, we state the lemma with signs for completeness:
\begin{lemma} \label{lemma: mapping-cone}
Let $S$ be a partially ordered set, and let $(C_i, \partial_i)_{i \in \Z}$ be a family of $S$-filtered chain complexes (over any ring). Suppose we have filtered maps $f_i \co C_i \to C_{i+1}$ and $h_i \co C_i \to C_{i+2}$ so that:
\begin{enumerate}
    \item $f_i$ is an anti-chain map, i.e., $f_i \circ \partial_i + \partial_{i+1} \circ f = 0$.
    \item $h_i$ is a null-homotopy of $f_{i+1} \circ f_i$, i.e., $f_{i+1} \circ f_i + h_i \circ \partial_i + \partial_{i+2} \circ h_i = 0$;
    \item $h_{i+1} \circ f_i + f_{i+2} \circ h_i$ is a filtered quasi-isomorphism from $C_i$ to $C_{i+3}$.
\end{enumerate}
Then the anti-chain map
\[
\begin{pmatrix} f_i \\ h_i \end{pmatrix}\co C_i \to \Cone(f_{i+1})
\]
is a filtered quasi-isomorphism (and hence a filtered homotopy equivalence when working over a field).
\end{lemma}

(Note that the sign convention follows \cite[Lemma 7.1]{KronheimerMrowkaUnknot}, which we have verified independently.)

\begin{proof}
For each $s \in S$, the maps $f_i$ and $h_i$ induce maps $\Gr_s(f_i) \co \Gr_s(C_i) \to \Gr_s(C_{i+1})$ and $\Gr_s(h_i) \co \Gr_s(C_i) \to \Gr_s(C_{i+2})$, which satisfy the hypotheses of \cite[Lemma 4.2]{OSzDouble}. Therefore,
\[
\begin{pmatrix} \Gr_s(f_i) \\ \Gr_s(h_i) \end{pmatrix}\co \Gr_s(C_i) \to \Cone(\Gr_s(f_{i+1})),
\]
is a quasi-isomorphism. Moreover, there is a natural identification of $\Cone(\Gr_s(f_{i+1}))$ with $\Gr_s(\Cone(f_{i+1}))$ under which $\begin{pmatrix} \Gr_s(f_i) \\ \Gr_s(h_i) \end{pmatrix} = \Gr_s \begin{pmatrix} f_i \\ h_i \end{pmatrix}$. We thus deduce that $\begin{pmatrix} f_i \\ h_i \end{pmatrix}$ is a filtered quasi-isomorphism.
\end{proof}

Even over an arbitrary ring, one can also prove a version of the (filtered) mapping cone detection lemma in the (filtered) homotopy category, but it requires a stronger set of hypotheses.  We state it here for posterity:

\begin{lemma} \label{lemma: homotopy-mapping-cone}
Let $S$ be a partially ordered set, and let $(C_i, \partial_i)_{i \in \Z}$ be a $3$-periodic family of $S$-filtered chain complexes. Suppose we have filtered maps $f_i \co C_i \to C_{i+1}$, $h_i \co C_i \to C_{i+2}$, $g_i \co C_i \to C_{i+3}$, and $r_i \co C_i \to C_{i+4}$ so that:
\begin{enumerate}
  \item $f_i$ is an anti-chain map;
  \item $h_i$ is a null-homotopy of $f_{i+1} \circ f_i$;
  \item $h_{i+1} \circ f_i + f_{i+2}\circ h_i$ is filtered homotopic to the identity, i.e.,
\[\id_{C_i}+ h_{i+1} \circ f_i + f_{i+2} \circ h_i= \partial_i \circ g_i + g_i\circ\partial_i.\]
  \item $ h_i\circ h_{i+1}+ f_i\circ g_i + g_{i+1}\circ f_i = \partial_{i+1}\circ r_i+ r_i\circ\partial_i$
\end{enumerate}
Then the map
\[
\begin{pmatrix} f_i \\ h_i \end{pmatrix}\co C_i \to \Cone(f_{i+1})
\]
is a filtered homotopy equivalence,  with homotopy inverse given by
\[
\begin{pmatrix} h_{i+1} & f_{i+2} \end{pmatrix} \co \Cone(f_{i+1}) \to C_{i+3} = C_i.
\]
\end{lemma}

Oddly enough, Lemma \ref{lemma: homotopy-mapping-cone} is easier to derive than Lemma \ref{lemma: mapping-cone} (and is hence left to the reader as an exercise), but its hypotheses are clearly more difficult to verify. In the context of surgery exact triangles in Floer theory, in particular,  the families of complexes considered are not $3$-periodic;  it is only their isomorphism type which is $3$-periodic.  This fact makes Lemma \ref{lemma: homotopy-mapping-cone} rather unwieldy for our purposes and forces us to rely on Lemma \ref{lemma: mapping-cone} instead. However, by Lemma \ref{lemma: filt-QI}, we are then able to deduce filtered homotopy equivalence without

Henceforth, we always set $\F = \Z/2\Z$.

\subsection{Relative \texorpdfstring{spin$^c$}{spinc} structures and Alexander gradings} \label{ssec: rel-spinc}

We now discuss some more details about relative spin$^c$ structures and Alexander gradings.

As above, let $Y$ be a closed, oriented $3$-manifold, and let $K$ be an oriented, rationally null-homologous knot in $Y$, representing a class of order $d>0$ in $H_1(Y)$.\footnote{In \cite{OSzRational}, the notation $\overline{K}$ is used when the orientation of $K$ is relevant; here, we dispense with that convention and always treat $K$ as oriented.} For any class $P \in H_2(Y,K)$, the intersection number $[\mu] \cdot P$ is divisible by $d$. In particular, if $P = [F]$, where $F$ is a rational Seifert surface for $K$, then $[\mu] \cdot P = d$.

As in \eqref{eq: spinc-alex}, for any $P$ with $P \cdot \mu \ne 0$, and any relative spin$^c$ structure $\xi$, the \emph{Alexander grading of $\xi$ with respect to $P$} is defined as
\begin{align}
\label{eq: alex-gr-spinc-P}
\AlNorm_{Y,K, P}(\xi) &= \frac{\gen{c_1(\xi), P} + [\mu] \cdot P }{2 [\mu] \cdot P} \in \frac{1}{2d} \Z.
\end{align}
By construction, $\AlNorm_{Y,K, P}(\xi)$ is unchanged under multiplying $P$ by a nonzero scalar; in particular, if $Y$ is a rational homology sphere, then the Alexander grading is independent of $P$.\footnote{\label{fn: alex-norm} Some authors (e.g. Ni \cite{NiThurstonNorm}) normalize the Alexander grading differently, without the factor of $[\mu] \cdot P$ in the denominator. The disadvantage of that convention is that the independence of scaling $P$ no longer holds.}
More generally, suppose $P,P'$ are nonzero classes in $H_2(Y,K)$ whose restrictions to $K$ agree; after scaling, assume that $[\mu] \cdot P = [\mu] \cdot P' = d$. Then $P-P'$ is the image of a class $Q \in H_2(Y)$. For any $\xi \in \ul\Spin^c(Y,K)$, we have
\[
\AlNorm_{Y,K,P}(\xi) - \AlNorm_{Y,K,P'}(\xi)  = \frac{1}{2d} \gen{c_1(\xi), P-P'} = \frac{1}{2d} \gen{c_1(G_{Y,K}(\xi)), Q}.
\]
In particular, if $G_{Y,K}(\xi)$ is a torsion spin$^c$ structure on $Y$, then $\AlNorm_{Y,K,P}(\xi)$ is completely independent of the choice of $P$. We will henceforth drop $P$ from the notation and just denote the Alexander grading by $\AlNorm_{Y,K}$.

A framing for $K$ is determined by the choice of a longitude $\lambda$, which we view as an oriented curve in $\partial(\nbd(K))$. Let $F$ be a rational Seifert surface for $K$. As elements of $H_1(\partial (\nbd(K));\Z)$, we have $\partial F = d\lambda - k\mu$ for some $k\in\Z$. For any other framing $\lambda' = \lambda+m\mu$, we have $\partial F = d(\lambda + m\mu) - (k+dm)\mu$. Thus, the framing determines and is determined by $k$, and the class of $k$ mod $d$ is independent of the choice of framing. The \emph{rational self-linking} of $K$ is $[\frac{k}{d}] \in \Q/\Z$; it depends only on the homology class of $K$.

Let $Y_\lambda = Y_\lambda(K)$ denote the manifold obtained by $\lambda$-framed surgery on $K$. The meridian $\mu$ is isotopic to a core circle of the glued-in solid torus. Let $K_\lambda$ denote this core circle, with the orientation inherited from the \emph{left-handed} meridian $-\mu$. The curve $\lambda \subset \partial (\nbd(K_\lambda)) = \partial (\nbd(K))$ then serves as a right-handed meridian for $K_\lambda$, since $\#(-\mu \cdot \lambda) = 1$ when $\partial (\nbd(K_\lambda))$ is given its boundary orientation.

The sets $\ul\Spin^c(Y,K)$ and $\ul\Spin^c(Y_\lambda, K_\lambda)$ are canonically identified, since they depend only on the complement. Viewing $[F]$ as an element of $H_2(Y_\lambda(K), K_\lambda)$, we have $[\lambda] \cdot [F] = k$. We thus have
\[
\AlNorm_{Y_\lambda, K_\lambda}(\xi) = \frac{ \gen{c_1(\xi), [F]} + k }{2k} =  \frac{2d \AlNorm_{Y,K}(\xi) - d +k}{2k}.
\]
This justifies \eqref{eq: A(xi-l)}.

As shown in \cite{OSzKnot}, a doubly-pointed Heegaard diagram $(\Sigma, \bm\alpha, \bm\beta, w, z)$ determines a $3$-manifold $Y$ and an oriented knot $K \subset Y$. To be precise, let $H_\alpha$ and $H_\beta$ be the two handlebodies in the Heegaard splitting; recall that $\Sigma$ is oriented as the boundary of the $\alpha$ handlebody. Let $\lambda$ be an immersed curve in $\Sigma$ obtained as $\lambda = t_\alpha + t_\beta$, where $t_\alpha$ is an embedded arc in $\Sigma \minus \bm\alpha$ from $z$ to $w$, and $t_\beta$ is an embedded arc in $\Sigma \minus \bm\beta$ from $w$ to $z$. We obtain $K \subset Y$ by pushing $t_\alpha$ into $H_\alpha$ and $t_\beta$ into $H_\beta$. Thus, $K$ intersects $\Sigma$ positively at $w$ and negatively at $z$. In other words, we may write $K = \gamma_w - \gamma_z$, where $\gamma_w$ (resp.~$\gamma_z$) is the upward-oriented flowline from the index-$0$ critical point $p_0$ to the index-$3$ critical point $p_3$ of the a function associated with the Heegaard diagram.  The meridian $\mu$ can be realized as a counterclockwise circle in $\Sigma$ around $w$.

Note that both possible conventions for how to orient $K$ exist in the literature, leading to some sign confusions. Our convention agrees with \cite{OSzKnot}, but not with \cite{OSzLink, OSzRational}.


Ozsv\'ath and Szab\'o show how to associate to each generator $\x \in \T_\alpha \cap \T_\beta$ a relative spin$^c$ structure $\ul\spincs_{w,z}(\x) \in \ul\Spin^c(Y,K)$, with the property that $G_{Y,K} (\ul\spincs_{w,z}(\x)) = \spincs_w(\x)$. The Alexander grading of $\x$ is defined as
\begin{align}
\label{eq: alex-def}
\AlNorm_{w,z}(\x) &= \AlNorm_{Y,K}(\ul\spincs_{w,z}(\x)) = \frac{1}{2d} \gen{c_1(\ul\spincs_{w,z}(\x)), [F]} + \frac{1}{2}
\end{align}
where $F$ is a rational Seifert surface for $K$.

For any generators $\x, \y$ with $\spincs_w(\x) = \spincs_w(\y)$, and any disk $\phi \in \pi_2(\x,\y)$, we have the familiar formula
\begin{equation} \label{eq: rel-alex}
\AlNorm_{w,z}(\x) - \AlNorm_{w,z}(\y) = n_z(\phi) - n_w(\phi).
\end{equation}
We will verify this formula below.

More generally, given any $\x, \y \in \T_\alpha \cap \T_\beta$, let $a$ and $b$ be $1$-chains in $\bm\alpha$ and $\bm\beta$, respectively, with $\partial a = \partial b = \y - \x$, and let $\ul\epsilon(\x,\y)$ be the $1$-cycle $a - b$. (That is, $\ul\epsilon(\x,\y)$ goes from $\x$ to $\y$ along $\bm\alpha$ and from $\y$ to $\x$ along $\bm\beta$.) This is well-defined up to adding multiples of the $\alpha$ and $\beta$ circles. Note that $\ul\epsilon(\x,\y)$ is homologous in $H_1(Y \minus K)$ to the difference $\gamma_\y - \gamma_\x$, where $\gamma_\x$ (resp.~$\gamma_\y$) is the sum of the upward gradient flow lines through $\x$ (resp.~$\y$) of the Morse function on $Y$ associated to the Heegaard diagram. (We see this by pushing the interior of $a$ into $H_\alpha$ and the interior of $b$ into $H_\beta$.) This $1$-cycle represents a class in $H_1(Y \minus K)$ which is independent of the choices of $a$ and $b$ (that is, up to adding multiples of $\alpha$ and $\beta$ circles). By \cite[Lemma 3.11]{OSzLink} and \cite[Lemma 2.1]{OSzRational}, we have:\footnote{Formula \ref{eq: rel-spinc-epsilon} was stated with signs reversed in \cite[Lemma 3.11]{OSzLink}, on account of $\ul\epsilon(\x,\y)$ implicitly being oriented the wrong way. However, the proof of \cite[Lemma 2.19]{OSz3Manifold} shows that our statement has the correct signs.}
\begin{equation} \label{eq: rel-spinc-epsilon}
\ul\spincs_{w,z}(\x) - \ul\spincs_{w,z}(\y) = -\PD[\ul\epsilon(\x,\y)].
\end{equation}
Therefore,
\begin{align}
\nonumber \AlNorm_{w,z}(\x) - \AlNorm_{w,z}(\y) &= \frac{1}{2d} \gen{c_1(\ul\spincs_{w,z}(\x)) - c_1(\ul\spincs_{w,z}(\y)), [F]} \\
\nonumber &= -\frac{\gen{\PD[\ul\epsilon(\x,\y)], [F]}}{d} \\
\label{eq: rel-alex-epsilon} &= -\frac{[\ul\epsilon(\x,\y)] \cdot [F]}{d}.
\end{align}
This formula completely characterizes the Alexander grading up to an overall shift, even when $Y$ is not a rational homology sphere.

If $\x$ and $\y$ represent the same (absolute) spin$^c$ structure, and $D$ is the domain of a disk $\phi \in \pi_2(\x,\y)$, then $\partial D = \ul\epsilon(\x,\y)$. More generally, suppose the image of $\ul\epsilon(\x,\y)$ in $H_1(Y)$ has finite order $n$. (If $Y$ is a rational homology sphere, this is true for all $\x$ and $\y$.) Then there is a domain $D$ in $\Sigma$ (that is, an integral linear combination of regions) with $\partial D = n \ul\epsilon(\x,\y)$.
We may interpret the intersection number $[\ul\epsilon(\x,\y)] \cdot [F]$ from \eqref{eq: rel-alex-epsilon} as the linking number between the disjoint $1$-cycles $\epsilon(\x,\y)$ and $d K$. Symmetry of the linking number then implies that
\[
\AlNorm_{w,z}(\x) - \AlNorm_{w,z}(\y) = - \frac{[D] \cdot [K]}{n}
\]
Since $K$ meets $\Sigma$ positively at $w$ and negatively at $z$, we deduce that
\begin{equation} \label{eq: rel-alex-fractional}
\AlNorm_{w,z}(\x) - \AlNorm_{w,z}(\y) = \frac{n_z(D) - n_w(D)}{n}.
\end{equation}
The $n=1$ case is \eqref{eq: rel-alex}, as claimed above.

Next, we explain conjugation symmetry of knot Floer homology, which motivates the second term in the numerator of \eqref{eq: spinc-alex}. It is shown in \cite[Lemma 3.12 and Proposition 8.2]{OSzLink} that for each $\xi \in \ul\Spin^c(Y,K)$, we have
\begin{equation} \label{eq: HFK-sym-spinc}
\HFKa(Y,K, \xi) \cong \HFKa(Y,K, J(\xi)-\PD[\mu]),
\end{equation}
with an appropriate shift in the Maslov grading, where $J$ denotes spin$^c$ conjugation.\footnote{Ozsv\'ath and Szab\'o \cite{OSzLink} state this formula with $+\PD[\mu]$ instead of $-$. However, as noted above, their orientation convention for $K$ is opposite ours, so the sign of the meridian is reversed as well. Ni's definition of the Alexander grading  \cite[Section 4.4]{NiThurstonNorm} follows the same convention as \cite{OSzLink}; this explains the sign discrepancy between our definition \eqref{eq: spinc-alex} and Ni's.} By our definition \eqref{eq: spinc-alex}, we have:
\begin{align*}
\AlNorm_{Y,K}(J(\xi) - \PD[\mu]) &= \frac{1}{2d} \left(\gen{c_1(J(\xi) - \PD[\mu]), [F]} + [\mu] \cdot [F] \right) \\
&= \frac{1}{2d} \left(\gen{c_1(J(\xi)) - 2\PD[\mu], [F]} + [\mu] \cdot [F] \right) \\
&= \frac{1}{2d} \left(\gen{-c_1(\xi), [F]} - [\mu] \cdot [F] \right) \\
&= - \AlNorm_{Y,K}(\xi).
\end{align*}
Therefore, if we define (for any rational number $r$)
\begin{equation} \label{eq: HFK-alex-decomp}
\HFKa(Y,K,  r) = \bigoplus_{\{\xi \in \ul\Spin^c(Y,K) \mid \AlNorm_{Y,K}(\xi) = r\}} \HFKa(Y,K,\xi),
\end{equation}
we have the symmetrization property
\begin{equation} \label{eq: HFK-sym-alex}
\HFKa(Y,K, r) \cong \HFKa(Y,K, -r).
\end{equation}
This property together with \eqref{eq: rel-alex-epsilon} completely determines the function $\AlNorm_{w,z}$. Note that the sum in \eqref{eq: HFK-alex-decomp} may range over relative spin$^c$ structures which induce different absolute spin$^c$ structures on $Y$. Note, too, that the symmetry does not necessarily hold within each individual absolute spin$^c$ structure. (However, it does hold if we sum over all $\xi \in \ul\Spin^c(Y,K)$ which map to all the torsion spin$^c$ structures on $Y$; this is relevant for Lemma \ref{lemma: wz-maslov} below.)

\subsection{Relative periodic domain formula} \label{ssec: rel-per-domain}

We now prove Proposition \ref{prop: abs-alex}, which shows how the absolute Alexander grading can be computed directly from a Heegaard diagram.

\begin{proof}[Proof of Proposition \ref{prop: abs-alex}]
It is possible to give an explicit topological proof of \eqref{eq: abs-alex} along the lines of the first Chern class formula from \cite[Proposition 7.5]{OSzProperties}, taking into account both basepoints. Here, we take a more indirect approach. As noted in the previous section, the function $\AlNorm_{w,z}\co \T_\alpha \cap \T_\beta \to \Q$ is completely determined by the properties \eqref{eq: rel-alex-epsilon} and \eqref{eq: HFK-sym-alex}. It thus suffices to show that the function
\[
\AlNorm'_{w,z}(\x) := \frac{1}{2d} \left( \hat\chi(P) + 2n_\x(P) - n_z(P) - n_w(P) \right)
\]
satisfies the same two properties.

To check that $\AlNorm'_{w,z}$ satisfies the analogue of \eqref{eq: rel-alex-epsilon}, it suffices to see that
\[
[\ul\epsilon(\x,\y)] \cdot [P] = n_\y(P)-n_\x(P).
\]
This is immediate from the description of $[\epsilon(\x,\y)]$ as $[\gamma_\y-\gamma_\x]$ as above, together with the construction of a relative $2$-cycle representing $[F]$ from the relative periodic domain $P$.  Details are provided in \cite[Lemma 2.1]{HeddenPlamenevskayaRational}. Thus, $\AlNorm'_{w,z}$ agrees with $\AlNorm_{w,z}$ up to adding an overall constant.

Verifying the symmetry
\begin{equation} \label{eq: HFK-sym-alex2}
\HFKa(Y,K,\AlNorm'_{w,z} = r) \cong \HFKa(Y,K, \AlNorm'_{w,z}=-r)
\end{equation}
is somewhat more involved, though straightforward.  The first step is to check that the absolute grading on $\CFKa(\Sigma, \bm\alpha, \bm\beta, w, z)$ induced by $\AlNorm'_{w,z}$ does not depend on the Heegaard diagram or auxiliary choices. This entails several verifications, whose details are left as an exercise:
\begin{itemize}
\item
If we leave $\lambda$ fixed, any other relative periodic domain representing $[F]$ differs from $P$ by adding a multiple of $\Sigma$. Note that
\begin{align*}
\hat\chi(P + \Sigma) &= \hat\chi(P) + 2-2g  && \\
n_\x(P + \Sigma) &= n_\x(P) + g & n_\y(P + \Sigma) &= n_\y(P) + g \\
n_{w}(P + \Sigma) &= n_{w}(P) + 1 & n_{z}(P + \Sigma) &= n_{z}(P) + 1.
\end{align*}
Therefore, $\AlNorm'_{w,z}$ is unchanged under replacing $P$ by $P+\Sigma$ in the definition.

\item
Any two choices of the arc $t_\alpha$ differ by isotopy rel endpoints or by a handleslide over the $\alpha$ circles. Either operation may introduce new intersections between $t_\alpha$ and either the $\beta$ circles or $t_\beta$. By looking at how the local multiplicities of $P$ change under each operation, one can verify that $\AlNorm'_{w,z}$ is unchanged. An analogous argument holds for $t_\beta$.

\item
Finally, if we modify the Heegaard diagram by an isotopy, handleslide, or (de)stabilization away from both $w$ and $z$, the induced homotopy equivalence on $\CFKa$ preserves $\AlNorm_{w,z}$. If this map takes a generator $\x$ of the old diagram to a generator $\y$ of the new one, then by looking at an appropriately defined $1$-cycle $\ul\epsilon(\x,\y)$ and its intersection with the Seifert surface as above, one can verify that
\[
\AlNorm_{w,z}'(\x) - \AlNorm_{w,z}'(\y) = \AlNorm_{w,z}(\x) - \AlNorm_{w,z}(\y) = 0.
\]
Hence, the homotopy equivalence preserves $\AlNorm'_{w,z}$ as well.  For instance, in the map associated to a handleslide, such a 1-cycle is provided by the obstruction class for finding a Whitney triangle connecting $\x$ to $\y$ which misses the basepoints.
\end{itemize}

Next, recall that the Heegaard diagrams $(\Sigma, \bm\alpha, \bm\beta, w, z)$ and $(-\Sigma, \bm\beta, \bm\alpha, z, w)$ both present $(Y,K)$ with the same orientations and have isomorphic $\CFKa$, which gives the spin$^c$ conjugation symmetry \eqref{eq: HFK-sym-spinc}. Because we swap $w$ and $z$,  $-P$ plays the role of the relative periodic domain in the latter Heegaard diagram; this has the effect of negating each term on the right side of \eqref{eq: abs-alex}. For each $\x \in \T_\alpha \cap \T_\beta$, we thus have $\AlNorm'_{w,z}(\x)= -\AlNorm'_{z,w}(\x) $, where the former refers to the proposed grading on  $\CFKa(\Sigma, \bm\alpha, \bm\beta, w, z)$ and the latter refers to its values on $\CFKa(-\Sigma, \bm\beta, \bm\alpha, z, w)$. Thus, we have an isomorphism
\[
\HFKa(\Sigma, \bm\alpha, \bm \beta, w,z, \AlNorm'_{w,z} = r)  \cong \HFKa(-\Sigma,\bm\beta, \bm\alpha,z,w, \AlNorm'_{z,w}=-r).
\]
Combining this with the isomorphism
\[
\HFKa(-\Sigma, \bm\beta, \bm\alpha,z,w, \AlNorm'_{z,w}=-r) \cong \HFKa(\Sigma, \bm\alpha, \bm\beta,w,z, \AlNorm'_{w,z} = -r)
\]
induced by the Heegaard moves taking $(-\Sigma, \bm\beta, \bm\alpha, z, w)$  to $(\Sigma, \bm\alpha, \bm\beta, z, w)$, followed by the map induced by the half-twist diffeomorphism of pointed knots moving the basepoints half-way around $K$ (to yield  $(\Sigma, \bm\alpha, \bm\beta, w, z)$),  we deduce that the symmetry \eqref{eq: HFK-sym-alex2} holds, as required.
 \end{proof}

\begin{figure}
\labellist \small
 \pinlabel $\bullet$ at 68.5 84
 \pinlabel $z$ [l] at 69 84
 \pinlabel $\bullet$ at 68.5 76
 \pinlabel $w$ [l] at 69 75
 \pinlabel $-1$ at 55 82
 \pinlabel $0$ at 81 82
 \pinlabel $4$ at 61 68
 \pinlabel $5$ at 76 68
 \pinlabel $2$ at 97 61
 \pinlabel $3$ at 118 61
 \pinlabel $1$ at 149 61
 \pinlabel $1$ at 18 61
 \pinlabel $2$ at 18 22
 \pinlabel $2$ at 18 150
 \pinlabel $3$ at 38 137
 \pinlabel $1$ at 58 114
 \pinlabel $2$ at 79 114
 \pinlabel $4$ at 98 99
 \pinlabel $5$ at 119 99
 \pinlabel $3$ at 149 99
 \pinlabel $2$ at 160 150
 \pinlabel $2$ at 160 22
 \pinlabel $a$ [Bl] at 48 90
 \pinlabel $b$ [Bl] at 88 90
 \pinlabel $c$ [Bl] at 128 90
 \pinlabel $x$ [tl] at 67 57
 \pinlabel $\bullet$ at 56.5 68
 \pinlabel $w'$ [r] at 57 69
 \pinlabel $\bullet$ at 80.5 68
 \pinlabel $z'$ [l] at 80 69
 \pinlabel $p$ [Bl] at 68 90
 \pinlabel $q$ [Bl] at 108 90
 \pinlabel $r$ [br] at 29 37
 \pinlabel $s$ [t] at 64 49
 \pinlabel $t$ [b] at 115 127
 \pinlabel {{\color{red} $\alpha_1$}} [b] at 29 88
 \pinlabel {{\color{red} $\alpha_2$}} [l] at 28 155
 \pinlabel {{\color{blue} $\beta_1$}} [r] at 50 114
 \pinlabel {{\color{blue} $\beta_2$}} [tl] at 75 62
 \pinlabel {{\color{green} $\lambda$}} [t] at 56 36
\endlabellist
\includegraphics[scale=1.5]{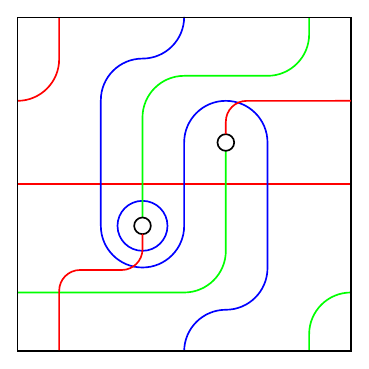}
\caption{Heegaard diagram $(\Sigma, \bm\alpha, \bm\beta, w, z)$ for the right-handed trefoil. The green curve $\lambda$ represents a $+5$-framed longitude.}
\label{fig: trefoil}
\end{figure}
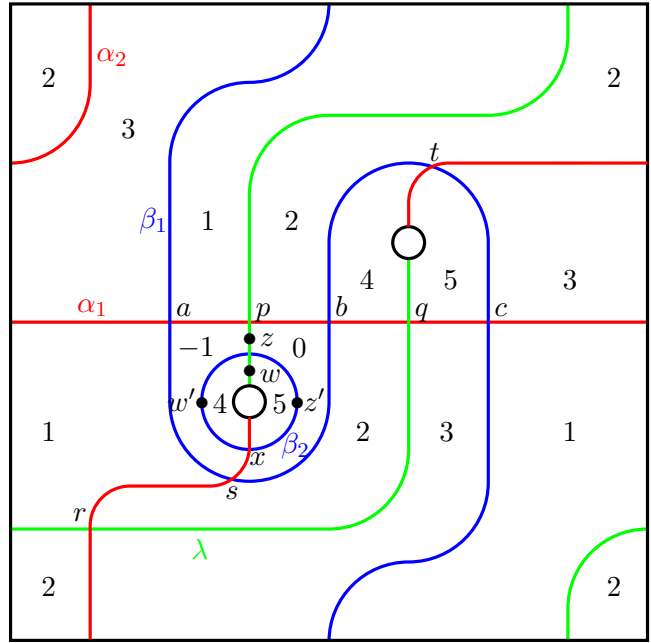

\begin{remark}
The above discussion provides an alternative proof of Ozsv{\'a}th and Szab{\'o}'s Chern class evaluation formula \cite[Proposition 7.5]{OSzProperties}.  Given a generator $\x\in \T_\alpha\cap\T_\beta$, their formula expresses the evaluation of the  Chern class of the absolute spin$^c$ structure $\spincs_w(\x)$ structure on the absolute homology class of a periodic domain $P$:
\begin{equation}\label{eq:OSchern} \langle c_1(\spincs_w(\x)),[P]\rangle =\hat\chi(P) + 2n_\x(P) - 2n_{w}(P).  \end{equation}

To recover this formula from Proposition \ref{prop: abs-alex}, we consider an unknot $U$ bounding  a disk  in a coordinate  ball of $Y$.  Given a pointed Heegaard for $Y$, we obtain a doubly pointed Heegaard diagram for $(Y,U)$ by placing $z$ in the same region as $w$, and choose $\lambda$ to be the boundary of a small disk $D$ contained inside this region. Given an absolute homology class $[P] \in H_2(Y)$ represented by a periodic domain $P$, we consider the relative periodic domain $P' = P+D$, which represents a relative homology class  $[P']\in H_2(Y,U)$. We have $i_*[P] = [P']$, where $i\co Y\rightarrow (Y,K)$ is the inclusion map. Naturality of relative and absolute Chern classes  \cite{KervaireRelative} implies that the left hand side of \ref{eq:OSchern} can be identified with $\gen{c_1(\ul\spincs_{w,z}(\x)),i_*[P]}$.

On the other hand, the formula for the Alexander grading, taken with respect to the relative homology class $[P']$, shows that
\begin{equation}\label{eq:AlexP'} 
\langle c_1(\ul\spincs_{w,z}(\x)),[P']\rangle + [\mu]\cdot P'= \hat\chi(P') + 2n_\x(P') - n_{w}(P')-n_{z}(P')  
\end{equation}
Observe that $n_w(P')=n_z(P')=n_w(P)+\frac{1}{2}$, that  $\hat\chi(P')=\hat\chi(P)+1$, and that $2n_\x(P') = 2n_\x(P)$. Therefore, the right-hand side of \ref{eq:AlexP'} equals the right-hand side of \ref{eq:OSchern}.  Next, observe that
 \[
 \gen{ c_1(\ul\spincs_{w,z}(\x)),[P']} + [\mu]\cdot P' = \gen{ c_1(\ul\spincs_{w,z}(\x)),i_*[P]} +  \gen{c_1(\ul\spincs_{w,z}(\x)),[D]} + [\mu]\cdot D.\]
Since the Alexander grading for an unknot relative to a disk Seifert surface is identically zero, we have $\gen{ c_1(\ul\spincs_{w,z}(\x)),[D]} + [\mu]\cdot [D]=0$, which concludes the proof of \eqref{eq:OSchern}.
\end{remark}

\begin{example} \label{ex: trefoil}
Consider the genus-2 Heegaard diagram for the right-handed trefoil shown in Figure \ref{fig: trefoil}. We have $\T_\alpha \cap \T_\beta = \{ax,bx,cx\}$. The green curve $\lambda$, which passes through the basepoints $w$ and $z$, represents a $+5$-framed longitude. The coefficients in Figure \ref{fig: trefoil} represent a relative periodic domain $P$. We have
\begin{align*}
\hat\chi(P)
&= -4 &
n_z(P) &= -\tfrac12 &
n_w(P) &= \tfrac92 && \\
n_a(P) &= 1 &
n_b(P) &= 2 &
n_c(P) &= 3 &
n_x(P) &= 2
\end{align*}
and therefore
\begin{align*}
\AlNorm_{w,z}(ax) &= -1 & \AlNorm_{w,z}(bx) &= 0 & \AlNorm_{w,z}(cx) &= 1.
\end{align*}
This is consistent with both \eqref{eq: rel-alex} and \eqref{eq: HFK-sym-alex}.

For a second example, note that the Heegaard diagram $(\Sigma, \bm\alpha, \bm\gamma)$ presents $\lambda$-framed surgery on $K$, where $\bm\gamma = \{\beta_1, \lambda\}$. Moreover, the curve $\beta_2$ determines the knot $K_\lambda$ induced by the surgery, so we can represent $K_\lambda$ with basepoints $w'$, $z' \in \beta_2$ as shown. The ordering of $w'$ and $ z'$ is chosen to be consistent with the orientation on $\beta_2$ that makes it occur in the boundary of $P$ with positive coefficient. The reader can check that the Alexander gradings of generators of $\CFKa(\Sigma, \bm\alpha, \bm\gamma, w', z')$ are:
\begin{gather*}
\AlNorm_{w',z'}(ar) = -\frac15  \qquad \qquad
\AlNorm_{w',z'}(ps) = -\frac35 \\
\AlNorm_{w',z'}(pt) = \AlNorm_{w',z'}(qs) = \AlNorm_{w',z'}(br) = 0 \\
\AlNorm_{w',z'}(qt) = \frac35   \qquad \qquad
\AlNorm_{w',z'}(cr) = \frac15
\end{gather*}
Once again, the symmetry \eqref{eq: HFK-sym-alex} is satisfied.

In the complex $\CFKa(\Sigma, \bm\alpha, \bm\gamma, w', z')$, we have $\partial(pt) = \partial (qs) = br$, which implies that the knot $K_\lambda$ is Floer simple. As a sanity check, since $\abs{\alpha_2 \cap \lambda} = 1$, we can destabilize this pair of curves to produce a standard genus-$1$ Heegaard diagram for a simple knot in a lens space, which is consistent with known results about $+5$ surgery on the trefoil.

For additional examples, see \cite[Section 6]{RaouxTau}.
\end{example}

\begin{remark} \label{rmk: per-domain-sign}
Suppose we choose a Heegaard diagram $(\Sigma, \bm\alpha, \bm\beta, w, z)$ for an oriented knot $K$, but consider a relative periodic domain $P$ that represents $-[F]$ rather than $[F]$; in other words, we assume $\partial P = -d\lambda + \cdots$, where $d>0$. Then \eqref{eq: abs-alex} still holds, provided that we take $-d$ in place of $d$ in the denominator. In other words, the denominator is simply the coefficient of $\lambda$ in $\partial P$, whether positive or negative. This is one of the reasons we prefer our normalization for the Alexander grading.
\end{remark}

We conclude this section with another helpful fact about the relationship between the Alexander and Maslov gradings (c.f. \cite[Proof of Lemma 4.10]{BakerGrigsbyHedden}, \cite[Proof of Theorem 3.3]{NiWuRational}). For any generator $\x \in \T_\alpha \cap \T_\beta$, we have $\spincs_z(\x) = \spincs_w(\x) + \PD[K]$. Since we assume that the knot $K$ is rationally null-homologous, this implies that $\spincs_w(\x)$ is a torsion spin$^c$ structure iff $\spincs_z(\x)$ is. If so, then $\x$ admits two separate absolute Maslov gradings when viewed as an element of $\CFa(\Sigma, \bm\alpha, \bm\beta, w)$ and $\CFa(\Sigma, \bm\alpha, \bm\beta, z)$; we denote these by $\absgr_w$ and $\absgr_z$ respectively.

\begin{lemma} \label{lemma: wz-maslov}
For any $\x \in \T_\alpha \cap \T_\beta$ for which $\spincs_w(\x)$ is torsion, we have
\[
\absgr_w(\x) - \absgr_z(\x) = 2 \AlNorm_{w,z}(\x).
\]
\end{lemma}

\begin{proof}
Define $\AlNorm_{w,z}'(\x) = \frac12(\absgr_w(\x) - \absgr_z(\x))$. Suppose $\x$ and $\y$ are generators representing (possibly different) torsion spin$^c$ structures. Choose a domain $D$ with $\partial(D) = n \epsilon(\x,\y)$. By the Lee--Lipshitz relative grading formula \cite[Proposition 2.13]{LeeLipshitz} together with \eqref{eq: rel-alex-fractional}, we compute:
\begin{align*}
\AlNorm_{w,z}'(\x) - \AlNorm'_{w,z}(\y) &= \frac{1}{2n} \left( (\mu(D) - 2n_w(D)) - (\mu(D) - 2n_z(D)) \right) \\
&= \frac{1}{n} \left( n_z(D) - n_w(D) \right) \\
&= \AlNorm_{w,z}(\x) - \AlNorm_{w,z}(\y).
\end{align*}
Thus, $\AlNorm_{w,z}'$ agrees with $\AlNorm_{w,z}$ (on all generators representing torsion spin$^c$ structures) up to an overall constant. To pin down the constant, note that $\AlNorm_{z,w}'(\x) = -\AlNorm'_{w,z}(\x)$, where the former refers to the grading on $\CFK(-\Sigma, \bm\beta, \bm\alpha, z, w)$, and proceed just as in the proof of Proposition \ref{prop: abs-alex}.
\end{proof}

\subsection{The knot Floer complex for rationally null-homologous knots} \label{ssec: CFKi}

For any Heegaard diagram $(\Sigma, \bm\alpha, \bm\beta, w)$, the complex $\CFi(\Sigma, \bm\alpha, \bm\beta, w)$ is generated (over $\F$) by all pairs $[\x, i]$, where $\x \in \T_\alpha \cap \T_\beta$ and $i \in \Z$. The differential on the chain complex $\CFi(\Sigma, \bm\alpha, \bm\beta, w)$ is given by
\begin{equation} \label{eq: CFi-diff}
\partial [\x,i] = \sum_{\y \in \T_\alpha \cap \T_\beta} \sum_{\substack{\phi \in \pi_2(\x,\y) \\ \mu(\phi)=1}} \# \widehat{\MM}(\phi) [\y, i-n_w(\phi)].
\end{equation}
For each spin$^c$ structure $\spincs$, the summand $\CFi(\Sigma, \bm\alpha, \bm\beta, w, \spincs)$ is generated by all $[\x,i]$ with $\spincs_w(\x) = \spincs$. The action of $U$ is given by $U \cdot [\x,i] = [\x, i-1]$. Let $\CFm(\Sigma, \bm\alpha, \bm\beta, w, \spincs)$ denote the subcomplex generated by all $[\x,i]$ with $i<0$, and $\CFp(\Sigma, \bm\alpha, \bm\beta, w, \spincs)$ the quotient of $\CFi$ by this complex. For $t \in \N$, let $\CF^t(\Sigma, \bm\alpha, \bm\beta, w, \spincs)$ denote the kernel of the action of $U^{t+1}$ on $\CFp$; concretely, it is generated by all $[\x,i]$ with $0 \le i \le t$.\footnote{In \cite{OSzSurgery, OSzRational} and elsewhere in the literature, the notation $\CF^\delta$ is used; we have chosen to use $t$ to avoid confusion with the $\delta$ curves used in Sections \ref{sec: large} and \ref{sec: exact-sequence}.} Note that $\CF^t(\Sigma, \bm\alpha, \bm\beta, w, \spincs)$ is isomorphic (up to a grading shift) to the quotient
\[
\CFm(\Sigma, \bm\alpha, \bm\beta, w, \spincs) / (U^{t+1} \cdot \CFm(\Sigma, \bm\alpha, \bm\beta, w, \spincs)),
\]
and we sometimes use this perspective instead.


Let $(\Sigma, \bm\alpha, \bm\beta, w, z)$ be a doubly pointed Heegaard diagram for a rationally null-homologous knot $K \subset Y$.  For each $\spincs \in \Spin^c(Y)$, let $\CFKi(\Sigma, \bm\alpha, \bm\beta, w, z, \spincs)$ be generated by all $[\x, i, j]$, where $\spincs_w(\x) = \spincs$, $i \in \Z$, and $j-i = \AlNorm_{w,z}(\x)$. (Note that $j$ need not be an integer!) The action of $U$ is given by $U \cdot [\x,i,j] = [\x,i-1,j-1]$, and the differential is given by
\begin{equation} \label{eq: CFKi-diff}
\partial [\x,i,j] = \sum_{\y \in \T_\alpha \cap \T_\beta} \sum_{\substack{\phi \in \pi_2(\x,\y) \\ \mu(\phi)=1}} \# \widehat{\MM}(\phi) [\y, i-n_w(\phi), j-n_z(\phi)],
\end{equation}
which is valid by \eqref{eq: rel-alex}. There is a canonical isomorphism
\[
\Omega_w \co \CFKi(\Sigma, \bm\alpha, \bm\beta, w, z, \spincs) \to \CFi(\Sigma, \bm\alpha, \bm\beta, w, \spincs)
\]
given by $\Omega_w([\x,i,j]) = [\x,i]$. In other words, the $j$ coordinate can be seen as giving an extra filtration on $\CFi(\Sigma, \bm\alpha, \bm\beta, w)$, which we call the \emph{Alexander filtration}. Using the terminology of Section \ref{ssec: algebra}, this is a filtration of Alexander type, given by the function $\AlNorm_{w,z}([\x,i]) = \AlNorm_{w,z}(\x) + i$.

This filtration descends to the other flavors; when thinking of them as doubly-filtered objects, we will sometimes denote them by $\CFKm$, $\CFKp$, and $\CFK^t$. In particular,  $\CFK^0(\Sigma, \bm\alpha, \bm\beta, w, z)$ is simply $\CFa(\Sigma, \bm\alpha, \bm\beta, w)$, equipped with its Alexander filtration. The associated graded complex of the latter is $\CFKa(\Sigma, \bm\alpha, \bm\beta, w, z)$, whose homology is the knot Floer homology $\HFKa(Y,K)$.

Each of these complexes is a topological invariant of $(Y,K)$ up to doubly-filtered chain homotopy equivalence; as in the introduction, we sometimes denote them by $\CFKi(Y,K, \spincs)$, etc.

\begin{remark} \label{rmk: Cxi}
In \cite{OSzRational}, Ozsv\'ath and Szab\'o define a separate doubly filtered complex for each relative spin$^c$ structure. Specifically, they define $\CFKi(\Sigma, \bm\alpha, \bm\beta, w, z, \xi)$ to be generated by all $[\x, i, j] \in (\T_\alpha \cap \T_\beta) \times \Z \times \Z$ with
\[
\ul\spincs_{w,z}(\x) + (i-j) \PD[\mu] = \xi.
\]
For all $\xi$ within a given fiber $G_{Y,K}^{-1}(\spincs)$ (for $\spincs \in \Spin^c(Y)$), the resulting complexes are isomorphic by a shift in $j$. To translate between the Ozsv\'ath--Szab\'o description of $\CFKi$ and ours, for each $\xi \in G_{Y,K}^{-1}(\spincs)$, there is an isomorphism
\begin{align}
\label{eq: Cxi-isom} \CFKi(\Sigma, \bm\alpha, \bm\beta, w, z, \xi)  &\to \CFKi(\Sigma, \bm\alpha, \bm\beta, w, z, \spincs) \\
\nonumber [\x,i,j] &\mapsto [\x,i,j + \AlNorm_{Y,K}(\xi) ].
\end{align}
\end{remark}

We now describe the so-called ``flip map'' alluded to in the introduction. To begin, note that for any $\x \in \T_\alpha \cap \T_\beta$, we have $\spincs_z(\x) = \spincs_w(\x) + \PD[K]$. Thus, for each $s\in \Z+ \AlNorm_{Y,K}(\spincs)$, there is an isomorphism
\[
\Omega_{z,s} \co \CFKi(\Sigma, \bm\alpha, \bm\beta, w, z, \spincs) \to \CFi(\Sigma, \bm\alpha, \bm\beta, z, \spincs + \PD[K])
\]
given by $\Omega_{z,s}([\x,i,j]) = [\x, j-s]$. Let
\[
\Gamma \co \CFi(\Sigma, \bm\alpha, \bm\beta, z, \spincs + \PD[K]) \to \CFi(\Sigma, \bm\alpha, \bm\beta, w, \spincs + \PD[K])
\]
be the $\F[U]$-equivariant chain homotopy equivalence induced by Heegaard moves taking the diagram $(\Sigma, \bm\alpha, \bm\beta, z)$ to the diagram $(\Sigma, \bm\alpha, \bm\beta, w)$.

The naturality theorem of Juh\'asz, Thurston, and Zemke \cite{JuhaszThurstonZemkeNaturality} addresses the dependence of $\Gamma$ (up to $\F[U]$-equivariant chain homotopy) on the choice of Heegaard moves.  To describe this, note that the main theorem of \cite{JuhaszThurstonZemkeNaturality} assigns to a {\em pointed} three-manifold a transitive system of chain complexes.  In this framework, a diagram adapted to  the surgery should be regarded as  embedded in a fixed pointed 3-manifold.  Changing the basepoint from $w$ to $z$, however,  does not leave us with the same pointed 3-manifold.  To work with a fixed pointed 3-manifold and its associated transitive system, we view the effect of switching $w$ to $z$ instead as an isotopy of the meridian of $K$ over $w$ so that it lies on the other side (and then calling it $z$, but noting that is still the same point in the same three-manifold). We can realize the new diagram as the ``pushforward" of the original diagram under a diffeomorphism of the pointed 3-manifold.  This diffeomorphism is taken as the  time one map of an ambient extension  of the isotopy of the Heegaard surface which pushes the meridian over the baspoint.  While the resulting diffeomorphism is isotopic to the identity, the isotopy does not preserve the basepoint.     The definition of the pointed mapping class group action on Floer homology now associates to this diffeomorphism the chain homotopy equivalence gotten by composing the canonical ``pushforward" identification of Floer complexes induced by the diffeomorphism, with any set of {\em pointed} Heegaard moves bringing the meridian back to its original position (the homotopy class of $\Gamma$ described above).  The content of the Juh\'asz--Thurston--Zemke theorem is that the homotopy class of this  homotopy equivalence is a well-defined invariant of the pointed mapping class used to define it.    Note, however, that there was a choice involved in the construction of our pointed mapping class: it can be distilled down to the homotopy class of arc connecting the basepoint $w$ to $z$ on the Heegaard diagram.  In the end, this dependence is essentially equivalent to the $\pi_1(Y,p)$ action on Heegaard Floer homology, and so it is indeed important to specify an arc; this latter action, however, is homotopic to the identity for rational homology spheres by work of Zemke \cite{ZemkeGraph}, and so in those cases the homotopy class of the homotopy equivalence $\Gamma$ is independent of all choices involved.

Let
\[
\Psi^\infty_{\spincs, s} \co \CFKi(\Sigma, \bm\alpha, \bm\beta, w, z, \spincs) \to \CFKi(\Sigma, \bm\alpha, \bm\beta, w, z, \spincs + \PD[K])
\]
denote the composition $(\Omega_w)^{-1} \circ \Gamma \circ \Omega_{z, s}$, which is a chain-homotopy equivalence. Since the pair $(\spincs, s)$ determines, and is determined by, a relative spin$^c$ structure $\xi$, we may also denote this map by $\Psi^\infty_\xi$. For varying $s$, the maps $\Psi_{\spincs, s}$  are related by:
\[
\Psi_{\spincs, s+1} = U \circ \Psi_{\spincs, s} = \Psi_{\spincs, s} \circ U.
\]
Thus, it really suffices to know only one of them. When $K$ is null-homologous, so that the Alexander grading is integer-valued, it is most convenient to take $s=0$.

\begin{lemma} \label{lemma: Psi-filtered}
The map $\Psi^\infty_{\spincs, s}$ is a filtered homotopy equivalence with respect to the $j$ filtration on the domain and the $i$ filtration (shifted) on the range, in the following sense: for any $t \in \AlNorm_{Y,K}(\spincs)$, $\Psi^\infty_{\spincs, s}$  restricts to a homotopy equivalence from the $j \le t$ subcomplex of $\CFKi(\Sigma, \bm\alpha, \bm\beta, w, z, \spincs)$ to the $i \le t-s$ subcomplex of $\CFKi(\Sigma, \bm\alpha, \bm\beta, w, z, \spincs + \PD[K])$. Moreover, $\Psi^\infty_{\spincs,s}$ is homogeneous of degree $-2s$ with respect to the Maslov grading $\absgr$.
\end{lemma}

\begin{proof}
For any $[\x,i,j] \in \CFKi(\Sigma, \bm\alpha, \bm\beta, w, z, \spincs)$, we have:
\begin{align*}
\Psi^\infty_{\spincs, s} ([\x,i,j]) &= (\Omega_w)^{-1} (\Gamma([\x, j-s])) \\
&= (\Omega_w)^{-1} \left(  \sum_p [\y_p, i'_p] \right) \\
&= \sum_p [\y_p, i'_p, i'_p + \AlNorm_{w,z}(\y_p)],
\end{align*}
where $p$ is taken from some finite indexing set and each $i'_p$ is an integer $\le j-s$. This fact follows from the definition of $\Gamma$, which is a composition of maps which are filtered homotopy equivalences with respect to the basepoint filtration.

Finally, for the statement about the Maslov grading, Lemma \ref{lemma: wz-maslov} implies that $\Omega_{z,s}$ is homogeneous of degree $-2s$ (using $\absgr_w$ on the domain and $\absgr_z$ on the target), while $\Gamma$ and $\Omega$ are grading-preserving.
\end{proof}

However, we emphasize that $\Psi^\infty_{\spincs, s}$ is \emph{not} necessarily filtered with respect to the other filtration on the domain and target; see Section \ref{ssec: trefoil-surgery} for an example. In particular, in the case of a null-homologous knot, the complex $\CFKi(Y,K,\spincs)$ is symmetric (up to isomorphism) under interchanging $i$ and $j$, but the map $\Psi^\infty_{\spincs,0}$ does not necessarily realize that symmetry.

The maps $\Psi^\infty_{\spincs, s}$ are actually invariants of the pointed knot $K$, in the following sense:

\begin{lemma}
Let $(\Sigma, \bm\alpha, \bm\beta, w, z)$ and $(\Sigma', \bm\alpha', \bm\beta', w, z)$ be two doubly-pointed Heegaard diagrams diagram which present the pointed knot $(Y,K)$. Then for each pair $(\spincs, s)$ as above, the following diagram commutes up to homotopy:
\begin{equation}
\xymatrix{
\CFKi(\Sigma, \bm\alpha, \bm\beta, w, z, \spincs) \ar[r]^-{\Psi^\infty_{\spincs, s}} \ar[d]^{\Phi_\spincs}_\simeq & \CFKi (\Sigma, \bm\alpha, \bm\beta, w, z, \spincs+\PD[K]) \ar[d]^{\Phi_{\spincs+\PD[K]}}_\simeq \\
\CFKi(\Sigma', \bm\alpha', \bm\beta', w, z, \spincs) \ar[r]^-{\Psi^{\infty\prime}_{\spincs, s}} & \CFKi(\Sigma', \bm\alpha', \bm\beta', w, z, \spincs+\PD[K]) }
\end{equation}
where $\Phi_\spincs$ and $\Phi_{\spincs+\PD[K]}$ are the doubly-filtered chain homotopy equivalences induced by a sequence of Heegaard moves taking $(\Sigma, \bm\alpha, \bm\beta, w, z)$ to $(\Sigma', \bm\alpha', \bm\beta', w, z)$, and the homotopy can be assumed to satisfy the same filteredness property as $\Psi^\infty_{\spincs, s}$ (see Lemma \ref{lemma: Psi-filtered}).
\end{lemma}

\begin{proof}
This follows from Juh\'asz--Thurston--Zemke's naturality theorem \cite[Theorem 1.8]{JuhaszThurstonZemkeNaturality}.  More precisely, the proof of their theorem shows that there is a transitive system of doubly-filtered chain complexes associated to a pointed knot $(Y,K)$, and that diffeomorphisms between pointed knots induce  maps of transitive systems.  This means, in particular, that there are canonical filtered homotopy equivalences induced by sequences of  pointed Heegaard moves relating (based, embedded) Heegaard diagrams for a pointed knot, $(Y,K)$.  These homotopy equivalences are represented by the vertical arrows.  Naturality with respect to pointed diffeomorphisms implies the existence of the horizontal chain maps and the homotopy commutativity of the diagram.
\end{proof}

In general, the maps $\Psi^\infty$ are extremely difficult to determine from the definition, since they require understanding the homotopy equivalences induced by a series of Heegaard moves. However, there is a special case in which they can be determined explicitly:

\begin{lemma} \label{lemma: Psi-Lspace}
Let $Y$ be an L-space and $K$ a knot in $Y$. Let
\[
\Psi^\infty, \Psi'{}^\infty \co \CFKi(Y,K, \spincs) \to \CFKi(Y,K, \spincs + \PD[K])
\]
be any two maps which are filtered chain homotopy equivalences (in the sense of Lemma \ref{lemma: Psi-filtered}). Then
$\Psi$ and $\Psi'$ are filtered chain-homotopic.
\end{lemma}

\begin{proof}
By construction, $\Psi$ and $\Psi'$ each restrict to filtered quasi-isomorphisms
\[
\Psi^-, \Psi'{}^- \co \CFKi(Y,K, \spincs)\{j \le s\} \to \CFKi(Y,K, \spincs + \PD[K])\{i \le 0\}
\]
(for some fixed $s \in \Z + \Al_{w,z}(\spincs)$). Since $Y$ is an L-space, the homology of each of these complexes is isomorphic to $\F[U]$, so $\Psi^-$ and $\Psi'{}^-$ induce the same map on homology. Therefore, $\Psi^- - \Psi'{}^-$ is filtered null-homotopic (with respect to the filtration by $U$ powers), via an $\F[U]$-linear null-homotopy
\[
H \co \CFKi(Y,K, \spincs)\{j \le s\} \to \CFKi(Y,K, \spincs + \PD[K])\{i \le 0\}.
\]
By $U$-equivariance, we can then extend $H$ over all of $\CFKi(Y,K, \spincs)$ to be a filtered null-homotopy of $\Psi^\infty - \Psi'{}^\infty$.
\end{proof}

Thus, when $Y$ is an L-space, it suffices to guess any chain map $\Psi^\infty$ which is a filtered quasi-isomorphism (in the sense of Lemma \ref{lemma: Psi-filtered}); Lemma \ref{lemma: Psi-Lspace} then guarantees that this map is the actual map. In particular, for null-homologous knots in any L-space (e.g. knots in $S^3$), any map realizing the $i \leftrightarrow j$ symmetry suffices. (This principle has been used, implicitly or explicitly, by many authors; see, e.g., \cite[Section 6]{HomLevineLidmanPL}.)

\section{More on the mapping cone formula} \label{sec: mapping-cone}

We now discuss a few more details concerning the mapping cone formula from the introduction, and outline the proof.

\begin{lemma} \label{lemma: h-filt}
For each $l \in \Z$, the map $h^\infty_{\xi_l}$ is filtered with respect to both $\II_\spinct$ and $\JJ_\spinct$ and homogeneous of degree $-1$ with respect to $\gr_\spinct$.
\end{lemma}

\begin{proof}
This is a straightforward exercise using Lemma \ref{lemma: Psi-filtered}.
\end{proof}

\begin{lemma} \label{lemma: Xab-indep-ab}
For all $a$ sufficiently negative and all $b$ sufficiently positive, the doubly-filtered chain homotopy type of $X^\infty_{\lambda,\spinct,a,b}$ is independent of $a$ and $b$, and likewise for $X^-_{\lambda,\spinct,a,b}$, $X^+_{\lambda,\spinct,a,b}$, and $X^t_{\lambda,\spinct,a,b}$.
\end{lemma}

\begin{proof}
The values of $j-i$ for all nonzero elements of $\CFKi(Y,K)$ are bounded above and below by constants. Therefore, when $s_l \gg 0$ (which holds for $l \gg 0$ if $k>0$ and for $l \ll 0$ if $k < 0$), the filtrations on $A^\infty_{\xi_l}$ both agree precisely with \eqref{eq: It-def-B} and \eqref{eq: Jt-def-B}, so $v^\infty_{\xi_l}$ is a doubly-filtered isomorphism. Similarly, when $s_l \ll 0$ (which holds for $l \gg 0 $ if $k<0$ and for $l \ll 0$ if $k>0$),
the filtrations on $A^\infty_{\xi_l}$ are given by
\[
\II_\spinct([\x,i,j]) = j-s_l \quad \text{and} \quad \JJ_\spinct([\x,i,j]) = j-s_l + \frac{2ds_l+k-d}{2k},
\]
which is just the vertical ($z$ basepoint) filtration, shifted appropriately. It follows from Lemma \ref{lemma: Psi-filtered} that $h^\infty_{\xi_l}$ is a doubly-filtered quasi-isomorphism.

Now, suppose $k>0$; the other case is handled similarly. If $b$ is large enough, then in the complex $X^\infty_{\lambda,\spinct,a,b+1}$, we can cancel the filtered-acyclic subcomplex $A^\infty_{\xi_{b+1}} \xrightarrow{v^\infty_{\xi_{b+1}}} B^\infty_{\xi_{b+1}}$, and the resulting complex is filtered isomorphic to $X^\infty_{\lambda,\spinct,a,b}$. Likewise, if $a$ is negative enough, then in $X^\infty_{\lambda,\spinct,a-1,b}$, we can cancel the filtered-acyclic subcomplex $A^\infty_{\xi_{a-1}} \xrightarrow{h^\infty_{\xi_{a-1}}} B^\infty_{\xi_{a}}$.
\end{proof}

\begin{remark} \label{rmk: suff-large-ab}
The range of values of $a$ and $b$ for which the conclusion of Lemma \ref{lemma: Xab-indep-ab} holds (i.e., how large is sufficiently large) depends on the spread of the Alexander grading on $\HFKa(Y,K)$, and thus on the genus of $K$. For simplicity, suppose that $K$ is null-homologous, so that $d=1$, and let $g= g(K)$. If we used a reduced complex for $\CFKi(Y,K)$, then all nonzero elements of $\CFKi(Y,K)$ satisfy $-g \le j-i \le g$. By examining \eqref{eq: It-def-A} and \eqref{eq: Jt-def-A}, we see that $v^\infty_{\xi_l}$ is a doubly-filtered isomorphism when $s_l \ge g+1$, and $h^\infty_{\xi_l}$ is a doubly-filtered quasi-isomorphism when $s_l \le -g$. Thus, for example, when $k=\pm 1$, we may compute $\CFKi(Y_\lambda(K), K_\lambda)$ using the complex $X^\infty_{\lambda, \spinct, 1-g, g}$.
\end{remark}

Let $A^-_{\xi_l}$ (resp.~$B^-_{\xi_l}$) denote the subcomplex of $A^\infty_{\xi_l}$ (resp.~$B^\infty_{\xi_l}$) generated by elements with $\II<0$, and let $A^+_{\xi_l}$ (resp.~$B^+_{\xi_l}$) denote the quotient by this subcomplex. As a result, these maps $v^\infty_{\xi_l}$, $h^\infty_{\xi_l}$ descend to give maps
\[
v^\pm_{\xi_l} \co A^\pm_{\xi_l} \to B^\pm_{\xi_l} \quad \text{and} \quad h^\pm_{\xi_l} \co A^\pm_{\xi_l} \to B^\pm_{\xi_{l+1}}
\]
on the plus and minus versions of the complexes. Under the isomorphisms from Remark \ref{rmk: Cxi}, our construction of the $A^+$ and $B^+$ complexes and the $v^+$ and $h^+$ maps agrees with Ozsv\'ath and Szab\'o's description in \cite[Section 4]{OSzRational}. Additionally, for any $t\in \N$, let $A^t_{\xi_l}$ (resp.~$B^t_{\xi_l}$) denote the kernel of $U^{t+1}$ on $A^+_{\xi_l}$ (resp.~$B^+_{\xi_l}$); concretely, these are generated by all generators with $0 \le \II \le t$. We also define the $U$-completed versions of the minus and infinity complexes:
\begin{align*}
\mathbf{A}^-_{\xi_l} &= A^-_{\xi_l} \otimes_{\F[U]} \F[[U]] &
\mathbf{B}^-_{\xi_l} &= B^-_{\xi_l} \otimes_{\F[U]} \F[[U]] \\
\mathbf{A}^\infty_{\xi_l} &= A^\infty_{\xi_l} \otimes_{\F[U,U^{-1}]} \F[[U,U^{-1}] &
\mathbf{B}^\infty_{\xi_l} &= B^\infty_{\xi_l} \otimes_{\F[U,U^{-1}]} \F[[U,U^{-1}].
\end{align*}
Concretely, the elements of each group are countably infinite sums $\sum_\alpha [\x_\alpha,i_\alpha, i_\alpha + \AlNorm(\x_\alpha)]$ such that for each $I \in \Z$, there are at most finitely terms with $i_\alpha \ge I$. As such, the $\II_\spinct$ and $\JJ_\spinct$ filtrations still make sense. We may thus define corresponding versions of the mapping cone, which we denote by $X^-_{\lambda, \spinct, a,b}$, $X^+_{\lambda, \spinct, a,b}$, $X^t_{\lambda, \spinct, a,b}$, $\mathbf{X}^-_{\lambda, \spinct, a,b}$, and $\mathbf{X}^\infty_{\lambda, \spinct, a,b}$. In particular, the finite $U$-power versions $X^t_{\lambda, \spinct, a,b}$ will play a crucial role in the proof.

\begin{remark}
Ozsv\'ath and Szab\'o originally stated the surgery formula only for $\HFp$, not for $\HFi$, and they made use of an infinite version of the mapping cone. Specifically, let
\begin{equation}\label{eq: D-plus}
D^+_{\lambda,\spinct} \co \bigoplus_{l\in\Z} A^+_{\xi_l} \to \bigoplus_{l\in\Z} B^+_{\xi_l}
\end{equation}
be the sum of all the $v^+_{\xi_l}$ and $h^+_{\xi_l}$ maps, and let $X^+_{\lambda,\spinct}$ be the mapping cone of $D^+_{\lambda,\spinct}$. Ozsv\'ath and Szab\'o proved that $X^+_{\lambda,\spinct}$ is quasi-isomorphic to $\CFp(Y_\lambda(K), \spinct)$. Manolescu and Ozsv\'ath \cite{ManolescuOzsvathLink} showed that the analogous results for $\HFm$ and $\HFi$ hold if one uses the $U$-completed versions and infinite direct products: that is, $\mathbf{HF}^-(Y_\lambda(K))$ and $\mathbf{HF}^\infty(Y_\lambda(K))$ are respectively isomorphic to the mapping cones of
\[
\mathbf{D}^-_{\lambda,\spinct} \co \prod_{l\in\Z} \mathbf{A}^-_{\xi_l} \to \prod_{l\in\Z} \mathbf{B}^-_{\xi_l}  \quad \text{and} \quad
\mathbf{D}^\infty_{\lambda,\spinct} \co \prod_{l\in\Z} \mathbf{A}^\infty_{\xi_l} \to \prod_{l\in\Z} \mathbf{B}^\infty_{\xi_l}.
\]
(See \cite[Section 4.3]{ManolescuOzsvathLink} for a discussion of why direct products rather than direct sums are needed.) The technique of ``horizontal truncation'' from \cite[Section 10.1]{ManolescuOzsvathLink} shows that the finite and infinite versions yield filtered quasi-isomorphic complexes. We find it preferable to avoid using infinite direct sums and products entirely, at the cost of being more explicit about the roles of $a$ and $b$.
\end{remark}

We now discuss the proof of Theorem \ref{thm: mapping-cone}. The proof follows the same basic outline as Ozsv\'ath and Szab\'o's \cite{OSzSurgery, OSzRational}, with a few modifications. Our main technical result, which occupies most of the remainder of the paper, is the following:
\begin{proposition} \label{prop: mapping-cone-CFt}
Let $\spinct \in \Spin^c(Y_{\lambda}(K))$. Then for any $a \ll 0$ and $b \gg 0$ and any $t \in \N$, $\CFK^t(Y_\lambda(K), K_\lambda, \spinct)$ is filtered homotopy equivalent to the mapping cone $X^t_{\lambda, \spinct, a, b}$, equipped with the filtrations $\II_\spinct$ and $\JJ_\spinct$.
\end{proposition}
The new (but surprisingly subtle) ingredient in this result is that the equivalence respects the second filtrations; the rest was shown by Ozsv\'ath and Szab\'o. Assuming Proposition \ref{prop: mapping-cone-CFt}, the rest of the main theorem follows immediately:

\begin{proof}[Proof of Theorem \ref{thm: mapping-cone}]
In the terminology of Section \ref{ssec: algebra}, $X^-_{\lambda, \spinct, a, b}$ is a complex of torsion $\CFm$ type equipped with a filtration of Alexander type, and $X^t_{\lambda, \spinct, a, b}$ is filtered isomorphic (with a shift in the grading) to  $X^-_{\lambda, \spinct, a, b} / U^{t+1}  X^-_{\lambda, \spinct, a, b}$. Therefore, Lemma \ref{lemma: filt-QI} and Proposition \ref{prop: mapping-cone-CFt} imply that $X^-_{\lambda, \spinct, a, b}$ is filtered quasi-isomorphic to $\CFKm(Y_\lambda, K_\lambda, \spinct)$. By taking the tensor product of each complex with $\F[U,U^{-1}]$, we then see that $X^\infty_{\lambda, \spinct, a, b}$ is filtered quasi-isomorphic to $\CFKi(Y_\lambda, K_\lambda, \spinct)$, as required.
\end{proof}

Next, we describe the version of the mapping cone which computes $\HFKa(Y_\lambda,K_\lambda)$. For knots in homology spheres, this agrees with Eftekhary's results in \cite{EftekharyIncompressible}.

\begin{corollary} \label{cor: mapping-cone-hat}
For any $\spinct \in \Spin^c(Y_\lambda)$, and any $\xi_l \in G_{Y_\lambda, K_\lambda}^{-1}(\spinct)$, $\HFKa(Y_\lambda, K_\lambda, \xi_l)$ is isomorphic to the homology of the mapping cone of
\begin{equation} \label{eq: mapping-cone-hat}
(h_{\xi_l}, v_{\xi_{l+1}}) \co A_{\xi_l} \{i\le 0, j = s_l\} \oplus A_{\xi_{l+1}} \{i=0, j \le s_{l+1}-1\} \to B_{\xi_{l+1}} \{i=0\}.
\end{equation}
\end{corollary}

\begin{proof}
The complex $X^0_{\lambda, \spinct}$ (or $X^0_{\lambda, \spinct, a,b}$ for $a \ll 0$ and $b \gg 0$) computes $\CFa(Y_\lambda)$ with its Alexander filtration, so its associated graded complex computes $\CFKa(Y_\lambda, K_\lambda)$. In particular, for each $\xi_l \in G_{Y_\lambda, K_\lambda}^{-1}(\spinct)$, $\CFKa(Y_\lambda, K_\lambda, \xi_l)$ is given by the subquotient of $X^\infty_{\lambda,\spinct}$ with
\[
\II_\spinct=0 \quad \text{and} \quad \JJ_\spinct = \AlNorm_{Y_\lambda, K_\lambda}(\xi_l) = \frac{2ds_l + k -d}{2k}.
\]
Using the definitions, we may verify that the only portions of the complex for which both of these conditions hold are the three terms listed in \eqref{eq: mapping-cone-hat}.
\end{proof}

\begin{remark}
The mapping cone formula in \cite{OSzSurgery, OSzRational} is stated with coefficients in $\Z$, not just in $\F$. Our proof should go through with coefficients in $\Z$ as well, but this requires understanding signed counts of holomorphic rectangles and pentagons, which is a technical headache and not fully spelled out in the literature. (One particular difficulty that arises is described below in Remark \ref{rmk: pentagon-signs}.) Therefore, we have chosen to work over $\F$ for simplicity.
\end{remark}

\subsection{Example: Surgery on the trefoil} \label{ssec: trefoil-surgery}

In Lemma \ref{lemma: Psi-filtered}, we saw that the flip map on $\CFKi$ is filtered with respect to the vertical ($j$) filtration on the domain and the horizontal ($i$) filtration on the target. Using the mapping cone formula, we now show an example illustrating that the map can be quite badly behaved with respect to the second filtration on each complex. (Another example can be found in \cite[Section 3.2]{JabukaMarkSurface}, although the pathologies there  become apparent only when using $\Z$ coefficients.)

Let $K \subset S^3$ denote the right-handed trefoil. The complex $\CFKi(S^3,K)$ can be generated (over $\F[U,U^{-1}]$) by generators $a,b,c$ in $(i,j)$ filtration levels $(0,-1)$, $(0,0)$, $(0,1)$ and Maslov gradings $-2,-1,0$ respectively. The differential is given by $\partial(b) = a + Uc$ and $\partial(a) = \partial(c) = 0$, and the flip map is an involution which fixes $b$ and interchanges $a$ and $UC$. This complex is shown in Figure \ref{fig: trefoil-CFK}.

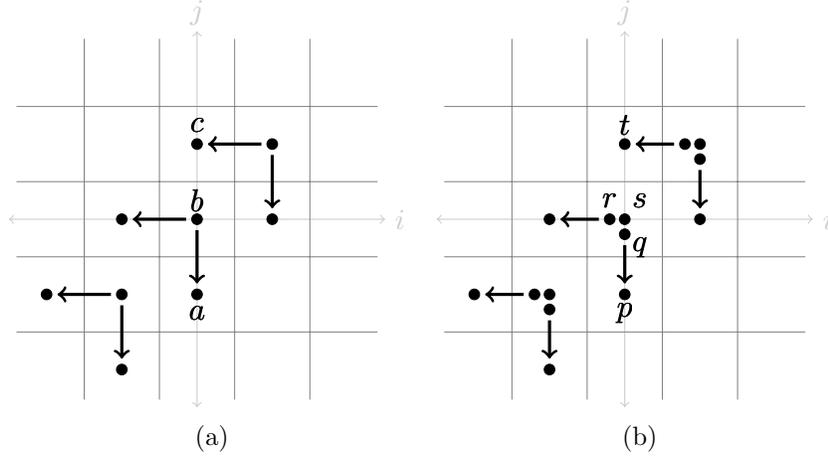
\begin{figure}
\subfigure[]{\begin{tikzpicture}
	\begin{scope}[thin, black!20!white]
		\draw [<->] (-2, 0.5) -- (3, 0.5);
		\draw [<->] (0.5, -2) -- (0.5, 3);
        \node[] at (3.2,0.5) {\small $i$};
        \node[] at (0.5,3.25) {\small $j$};
	\end{scope}
	\draw[step=1, black!50!white, very thin] (-1.9, -1.9) grid (2.9, 2.9);

\foreach \x in {-2,...,0}
{	
	\filldraw (\x+0.5, \x+1.5) circle (2pt) node[] (c){};
	\filldraw (\x+1.5, \x+1.5) circle (2pt) node[] (b){};
	\filldraw (\x+1.5, \x+0.5) circle (2pt) node[] (a){};
	\draw [very thick, ->] (b) -- (a);
	\draw [very thick, ->] (b) -- (c);
\node[] at (0.5, 1.75) {$c$};
\node[] at (0.5, 0.75) {$b$};
\node[] at (0.5, -0.75) {$a$};
}	
\end{tikzpicture} \label{subfig: trefoil-CFK}}
\subfigure[]{\begin{tikzpicture}
	\begin{scope}[thin, black!20!white]
		\draw [<->] (-2, 0.5) -- (3, 0.5);
		\draw [<->] (0.5, -2) -- (0.5, 3);
        \node[] at (3.2,0.5) {\small $i$};
        \node[] at (0.5,3.25) {\small $j$};
	\end{scope}
	\draw[step=1, black!50!white, very thin] (-1.9, -1.9) grid (2.9, 2.9);

\foreach \x in {-2,...,0}
{	
	\filldraw (\x+0.5, \x+1.5) circle (2pt) node[] (t){};
	\filldraw (\x+1.5, \x+1.5) circle (2pt) node[] (s){};
	\filldraw (\x+1.3, \x+1.5) circle (2pt) node[] (r){};
	\filldraw (\x+1.5, \x+1.3) circle (2pt) node[] (q){};
	\filldraw (\x+1.5, \x+0.5) circle (2pt) node[] (p){};
	\draw [very thick, ->] (q) -- (p);
	\draw [very thick, ->] (r) -- (t);
\node[] at (0.5, 1.75) {$t$};
\node[] at (0.7, 0.75) {$s$};
\node[] at (0.7, 0.125) {$q$};
\node[] at (0.3, 0.75) {$r$};
\node[] at (0.5, -0.75) {$p$};
}	
\end{tikzpicture} \label{subfig: trefoil-surgery-CFK}}

\caption{The complex $\CFKi$ for (a) the right-handed trefoil and (b) the dual knot in its $-1$ surgery. (The pattern repeats infinitely in both directions.)}
\label{fig: trefoil-CFK}
\end{figure}

Let $(Y,J) = (S^3_{-1}(K), K_{-1})$. Let us first apply Theorem \ref{thm: mapping-cone} to compute $\CFKi(Y,J)$. Since $g(K)=1$, it suffices to look at the mapping cone
\[(A^\infty_0 \oplus A^\infty_1) \to (B^\infty_{-1} \oplus B^\infty_0 \oplus B^\infty_1),\]
which we denote by $X$. Let us write $a_s, b_s, c_s$ for the copies of $a, b,c$ in $A_s^\infty$ (for $s=0,1$), and $a'_s, b'_s, c'_s$ for the copies in $B_s^\infty$ (for $s=-1,0,1$). By canceling differentials which preserve both the $\II$ and $\JJ$ filtrations, it is not hard to check that $X$ can be reduced to the complex generated (over $\F[U,U^{-1}]$) by generators $\bar p, \bar q, \bar r, \bar s, \bar t$ as in the following table:
\begin{center}
\begin{tabular}{|l|c|c|c|c|} \hline
Generator & $\II$ & $\JJ$ & $\absgr$ & $\partial$ \\ \hline
$\bar p  = c'_1$ & $0$ & $-1$ & $-2$ & $0$ \\
$\bar q = a_0+c_1$ & $0$ & $0$ & $-1$ & $\bar p + U \bar t$\\
$\bar r = a_0 + b'_0$ & $0$ & $0$ & $-1$ & $U \bar s + U \bar t$ \\
$\bar s=c'_0$ & $0$ & $0$ & $0$ & $0$ \\
$\bar t = c'_{-1}$ & $0$ & $1$ & $0$ & $0$ \\ \hline
\end{tabular}
\end{center}
We can then make a filtered change of basis to simplify the differential further: set $p = \bar p + U \bar s$, $q = \bar q + \bar r$, $r = \bar r$, $s = \bar s$, and $t = \bar t + \bar s$, so that $\partial(q) = p$, $\partial(r) = Ut$, and $\partial(p) = \partial(s) = \partial(t) = 0$. The complex $\CFKi(Y,J)$ is shown with respect to this basis in Figure \ref{subfig: trefoil-surgery-CFK}. (See \cite[Section 6]{HomLevineLidmanPL} for a more extensive computation that illustrates the technique in more detail.)

We now study the flip map $\Psi^\infty$ on $C = \CFKi(Y,J)$. Let us just consider the induced map $\hat \Psi \co C\{j=0\} \to C\{i=0\}$, which is necessarily a quasi-isomorphism. The complexes $C\{j=0\}$ and $C\{i=0\}$ are each filtered (the former by $i$, the latter by $j$) and are filtered quasi-isomorphic, but we claim that $\hat \Psi$ cannot be a filtered map. The grading requires that $\hat\Psi(q) = r$ and that $\hat\Psi(s)$ is a nonzero linear combination of $s$ and $t$. Suppose, toward a contradiction, that $\hat \Psi$ is filtered; then $\hat\Psi(s) = s$. However, observe that $(Y_{+1}(J), J_{+1}) = (S^3,K)$. Consider the associated graded complex of the filtered mapping cone formula for $+1$ surgery, as described in Corollary \ref{cor: mapping-cone-hat}. The part in Alexander grading $0$ has the form
\[
\xy
(0,0)*{A_0\{i\le 0,j=0\}}="A0";
(40,0)*{B_1\{i=0\}}="B1";
(80,0)*{A_1\{i=0,j \le 0\}}="A1";
{\ar^-{h_0} "A0";"B1"};
{\ar_-{v_1} "A1";"B1"};
\endxy
\]
which is the following complex:
\[
\xy
(0,25)*{s_{0}}="A0s";
(0,20)*{r_{-1}}="A0r";
(0,15)*{q_{-1}}="A0q";
(-25,20)*{Ut_{-2}}="A0t";
(50,40)*{t_{-1}}="B1t";
(50,25)*{s_{-1}}="B1s";
(50,20)*{r_{-2}}="B1r";
(50,15)*{q_{-2}}="B1q";
(50,0)*{p_{-3}}="B1p";
(100,25)*{s_{0}}="A1s";
(100,20)*{r_{-1}}="A1r";
(100,15)*{q_{-1}}="A1q";
(100,0)*{p_{-2}}="A1p";
{\ar "A0r";"A0t"};
{\ar "B1q";"B1p"};
{\ar "A1q";"A1p"};
{\ar "A1p";"B1p"};
{\ar "A1q";"B1q"};
{\ar "A1r";"B1r"};
{\ar "A1s";"B1s"};
{\ar@{-->} "A0r";"B1r"};
{\ar@{-->} "A0r";"B1q"};
{\ar "A0q";"B1r"};
{\ar "A0s";"B1s"};
{\ar@{-->}@/_2pc/ "A0t";"B1p"};
\endxy
\]
(Here, the subscripts indicate the Maslov grading on the mapping cone, given by \eqref{eq: grt-def-A} and \eqref{eq: grt-def-B}, and the dashed arrows indicate possible additional terms in $\hat\Psi$.) Examining this complex, we see that its homology has rank $3$, which contradicts the fact that $\HFK(S^3,K,0) \cong \F$. The only way to remedy this issue is to add a component taking $s \in A_0\{i\le 0\}$ to $t \in B_1\{i=0\}$, which means that $\hat \Psi$ is not filtered with respect to the second grading. (With further work, one can then use this information to completely pin down $\Psi^\infty$ up to chain homotopy.)

\section{Alexander gradings and surgery cobordisms} \label{sec: alex-surgery}

In this section, we study the relationship between the Alexander grading and spin$^c$ structures on the $2$-handle cobordism associated to a framed knot.

As above, assume that $Y$ is an oriented $3$-manifold and that $K$ is an oriented, rationally null-homologous knot representing a class of order $d > 0$ in $H_1(Y; \Z)$. Let $\lambda$ be a nonzero framing for $K$, and let $W = W_\lambda(K)$ be the corresponding $2$-handle cobordism from $Y$ to $Y_\lambda(K)$.

Let $C\subset W_\lambda(K)$ denote the core disk of the $2$-handle together with $K \times I$, and let $C^* \subset W_\lambda(K)$ denote the cocore disk. We assume these are oriented to intersect positively. Then $[C]$ and $[C^*]$ generate $H_2(W,Y)$ and $H_2(W,Y_\lambda)$, respectively. Consider the Poincar\'e duals $\PD[C] \in H^2(W, Y_\lambda(K))$ and $\PD[C^*] \in H^2(W,Y)$; by a slight abuse of notation, we will also use $\PD[C]$ and $\PD[C^*]$ to denote the images of these classes in $H^2(W)$. Then $\PD[C]$ restricts to $\PD[K] \in H^2(Y)$, and it generates the kernel of $H^2(W_\lambda(K)) \to H^2(Y_\lambda(K))$. In particular, if $\spinct$  and $\spinct'$ are spin$^c$ structures on $W_\lambda(K)$ whose restrictions to $Y_\lambda(K)$ (resp.~$Y$) are the same, then they differ by a multiple of $\PD[C]$ (resp.~$\PD[C^*]$).

Let $F$ be a rational Seifert surface for $K$, and assume that $[\partial F] = d\lambda - k\mu$ in $H_1(\partial(Y \minus \nbd(K)))$. We can cap off $F$ in $W_\lambda(K)$ to obtain a closed surface $\hat F$.\footnote{In \cite{OSzRational}, the notation $\hat F$ is used for what we call $C$.} To understand this surface, it helps to imagine attaching the $2$-handle in two steps: First, attach $S^1 \times D^2 \times I$ to $Y \times I$, gluing $S^1 \times D^2 \times \{0\}$ to $\nbd(K) \times \{1\}$; and then attach the $2$-handle along $S^1 \times D^2 \times \{1\}$. Inside $S^1 \times D^2 \times I$, $\partial F \times \{0\}$ is homologous to $(d \text{ parallel copies of } \lambda) \times \{1\}$; let $G$ be a surface joining them, and let $\hat F = F \cup G \cup (d \text{ parallel copies of the core of the 2-handle})$. The homology class $[\hat F] \in H_2(W)$ does not depend on the choice of $G$. Since $[\hat F]$ maps to $d[C]$ in $H_2(W,Y)$ and to $k[C^*]$ in $H_2(W, Y_\lambda(K))$, it follows that $[\hat F]^2=dk$.

We may represent $W$ by a doubly pointed Heegaard triple diagram $(\Sigma, \bm\alpha, \bm\beta, \bm\gamma, w, z)$ with the following properties:
\begin{itemize}
\item
The diagram $(\Sigma, \bm\alpha, \bm\beta, w, z)$ represents $(Y,K)$, as above. Moreover, there is an arc $t_\alpha$ from $z$ to $w$ that meets $\beta_g$ in a single point and is disjoint from all other $\alpha$ and $\beta$ curves.

\item
The curve $\beta_g$ meets $\alpha_g$ in a single point $x_0$ and is disjoint from the remaining $\alpha$ curves.

\item
The curve $\gamma_g$ is a $\lambda$-framed longitude that meets $\beta_g$ once and is disjoint from the remaining $\beta$ curves; it is oriented with the same orientation as $K$. For $i=1, \dots, g-1$, $\gamma_i$ is a small pushoff of $\beta_i$, meeting $\beta_i$ in two points.

\item
The points $w$ and $z$ lie to the right of $\gamma_g$ (with its specified orientation).
\end{itemize}
We say that $(\Sigma, \bm\alpha, \bm\beta, \bm\gamma, w, z)$ is \emph{adapted} to $(Y, K, \lambda)$.

\begin{remark} \label{rmk: ab-PD}
If $b_1(Y)>0$, we will further assume that $(\Sigma, \bm\alpha, \bm\beta, w)$ is admissible for all torsion spin$^c$ structures on $Y$. Indeed, let $\Pi_{\alpha\beta}$ denote the group of $(\alpha,\beta)$ periodic domains satisfying $n_w=0$, and define $\Pi_{\alpha\gamma}$ analogously. Then $\Pi_{\alpha\beta} \cong H_2(Y)$ and $\Pi_{\alpha\gamma} \cong H_2(Y_\lambda(K))$. Because $K$ is rationally null-homologous, every element of $\Pi_{\alpha\beta}$ must have $n_w = n_z$, so the multiplicity of $\beta_g$ in its boundary is $0$. Furthermore, if $k \ne 0$, there is a natural isomorphism $\Pi_{\alpha\beta} \cong \Pi_{\alpha\gamma}$, given by adding thin $(\beta,\gamma)$ periodic domains; thus, the multiplicity of $\gamma_g$ in the boundary of any element of $\Pi_{\alpha\gamma}$ is also $0$. (If $k = 0$, then $\Pi_{\alpha\gamma} \cong \Pi_{\alpha\beta} \oplus \Z$, where the generator of the $\Z$ factor is given by $P$ plus appropriate thin domains, but we will rarely need to consider this case.)
\end{remark}


Orient the curves $\alpha_g, \beta_g, \gamma_g$ so that $\#(\alpha_g \cap \beta_g) = \#(\gamma_g \cap \beta_g) = \#(t_\alpha \cap \beta_g) = 1 $ and $\gamma_g$. Orient the remaining $\alpha$, $\beta$, and $\gamma$ curves arbitrarily, except that $\beta_i$ and $\gamma_i$ are assumed to be oriented parallel to each other for $i=1, \dots, g-1$. There is a triply periodic domain $P$ with $n_z(P) = -k$, $n_w(P) = 0$, and
\[
\partial P = -d\alpha_g  -k \beta_g + d\gamma_g + \sum_{i=1}^{g-1} (a_i \alpha_i + b_i \beta_i)
\]
for some integers $a_i, b_i$ (using the specified orientations). This periodic domain represents the class of a capped-off Seifert surface in $H_2(W_\lambda(K))$.  We may also view $P$ as a relative periodic domain in the sense of the previous subsection. There is a slight caveat: To compute Alexander gradings using Proposition \ref{prop: abs-alex}, we let $\bar w$ and $\bar z$ denote the points on $\gamma_g$ closest to $w$ and $z$ respectively; we then use $n_{\bar w}(P)$ and $n_{\bar z}(P)$ in place of $n_w(P)$ and $n_z(P)$ in \ref{eq: abs-alex}.

If $k>0$, then the diagram $(\Sigma, \bm\alpha, \bm\beta, \bm\gamma, w)$ is admissible since $P$ has both positive and negative coefficients. If $k<0$, an adapted diagram is not necessarily admissible. We can achieve admissibility by winding, as discussed below.

The self-intersection number of the homology class represented by $P$ is given by
\begin{equation} \label{eq: PD-self-int}
[P]^2 = (\partial_\alpha P \cdot \partial_\beta P) = (\partial_\beta P \cdot \partial_\gamma P) = (\partial_\gamma P \cdot \partial_\alpha P).
\end{equation}
In this case, this formula gives
\[
[P]^2 = (-dk) (\beta_g \cdot \gamma_g) = dk,
\]
as expected.

As discussed above, let $K_\lambda \subset Y_\lambda(K)$ be obtained from a left-handed meridian of $K$. Let $z'$ be a basepoint on the other side of $\gamma_g$ from $w$. The Heegaard diagram $(\Sigma, \bm\alpha, \bm\gamma, w, z')$ then represents $K_\lambda$, with the specified orientation.

We now show how to relate the Alexander gradings for $(Y,K)$ and $(Y_\lambda, K_\lambda)$ in terms of Heegaard diagrams. Let $\Theta_{\beta\gamma} \in \T_\beta \cap \T_\gamma$ denote the standard top-dimensional cycle in  $\CF(\Sigma, \bm\beta, \bm\gamma, w)$.

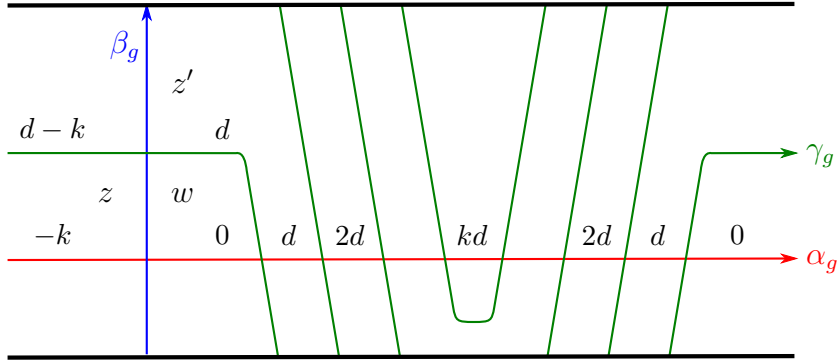
\begin{figure}
\labellist
 \pinlabel {{\color{red} $\alpha_g$}} [l] at 306 45
 \pinlabel {{\color{blue} $\beta_g$}} [r] at 61 126
 \pinlabel {{\color{darkgreen} $\gamma_g$}} [l] at 306 85
 \pinlabel $z$ at 45 70
 \pinlabel $w$ at 74 70
 \pinlabel $z'$ at 74 113
 \small
 \pinlabel $0$ at 89 55
 \pinlabel $-k$ at 25 55
 \pinlabel $d$ at 89 95
 \pinlabel $d-k$ at 25 95
 \pinlabel $d$ at 114 55
 \pinlabel $2d$ at 137 55
 \pinlabel $kd$ at 183 55
 \pinlabel $2d$ at 230 55
 \pinlabel $d$ at 253 55
 \pinlabel $0$ at 283 55
\endlabellist
\includegraphics{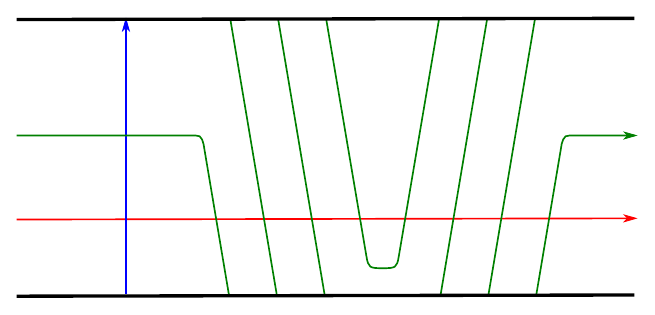}
\caption{Winding $\gamma_g$ in the case where $k>0$. The numbers represent the local multiplicities of the triply periodic domain $P$.}
\label{fig: winding-gamma}
\end{figure}

\begin{lemma} \label{lemma: c1-alex}
For any $\x \in \T_\alpha \cap \T_\beta$, $\q \in \T_\alpha \cap \T_\gamma$, and $\psi\in\pi_2(\x,\Theta_{\beta\gamma},\q)$, we have:
\begin{align}
\label{eq: alex-shift}
d\AlNorm_{w,z}(\x) - k\AlNorm_{w,z'}(\q) &= d n_z(\psi) + k n_{z'}(\psi) - (k+d) n_w(\psi) +\frac{k+d}{2} \\
\label{eq: c1-alex-x}
\gen{c_1(\spincs_w(\psi)), [P]} &= 2d \AlNorm_{w,z}(\x)  + 2d n_w(\psi) - 2d n_z(\psi) - k \\
\label{eq: c1-alex-q}
&= 2k \AlNorm_{w,z'}(\q) + 2k n_{z'}(\psi) - 2k n_w(\psi) + d.
\end{align}
\end{lemma}

\begin{proof}
To begin, we may assume that the Heegaard diagram contains a ``winding region'' in a tubular neighborhood of $\beta_g$, shown in Figure \ref{fig: winding-gamma} in the case where $k$ is positive. Specifically, we wind the $\gamma$ curve $\abs{k}$ times in a direction specified by the sign of $k$, so that every spin$^c$ structure on $Y_\lambda(K)$ is represented by a generator that uses a point in the winding region. When $k<0$, this further guarantees that the triple diagram $(\Sigma, \bm\alpha, \bm\beta, \bm\gamma, w)$ is admissible, since $n_w(P)=0$ and $P$ has both positive and negative coefficients. The general case will then follow by tracing through the proof of isotopy invariance.

Up to permuting the indices of the $\beta$ curves, let us assume that $\x$ consists of points $x_j \in \alpha_j \cap \beta_j$ for $j=1, \dots, g$. In particular, $x_g$ is the unique point in $\alpha_g \cap \beta_g$, which is located in the winding region. For $j=1, \dots, g-1$, the local multiplicities of $P$ around $x_j$ are $c_j, c_j+a_j, c_j+a_j+b_j, c_j+b_j$ in some order (for some $c_j$), while the local multiplicities at $x_g$ are $0, d, d-k, -k$ as in Figure \ref{fig: winding-gamma}. Hence, we have
\[
n_\x(P) = \frac{d-k}{2} + \sum_{j=1}^{g-1} \left(c_j + \frac{a_j + b_j}{2} \right) \quad \text{and} \quad n_{\bar wq}(P) + n_{\bar z}(P) = d-k.
\]

For each $\x \in \T_\alpha \cap \T_\beta$ and each $i = 1, \dots, \abs{k}$, let $\x_i \in \T_\alpha \cap \T_\gamma$ be the generator consisting of the $i\Th$ point of $\alpha_g \cap \gamma_g$ over from $x_g$, together with the points of $\alpha_j \cap \gamma_j$ that are ``nearest'' to $x_j$ for $j=1, \dots, g-1$. These generators all represent different spin$^c$ structures on $Y_\lambda(K)$. Let $\psi_{\x,i} \in \pi_2(\x, \Theta_{\beta\gamma}, \x_i)$ be the class whose domain consists of $g$ small triangles $\psi_{\x,i}^j$, where $\psi_{\x,i}^g$ is supported in the winding region (having positive coefficients if $k>0$ and negative coefficients if $k<0$), and for $j=1, \dots, g-1$, $\psi_{\x,i}^j$ connects $x_j$ to its ``nearest point'' in $\alpha_j \cap \gamma_j$ (and is independent of $i$). It is easy to check that
\[
(n_w(\psi_{\x,i}), n_z(\psi_{\x,i}), n_{z'}(\psi_{\x,i})) =
\begin{cases}
  (i, 0, i-1) & \text{if } k > 0 \\
  (1-i, 0, -i) & \text{if } k < 0.
\end{cases}
\]
In particular, $n_w(\psi_{\x,i}) - n_{z'}(\psi_{\x,i}) = 1$ in all cases.

We begin by showing that \eqref{eq: alex-shift} and \eqref{eq: c1-alex-x} hold when $\psi = \psi_{\x,i}$.

Let $P'$ be obtained from $P$ by adding copies of the small periodic domains bounded by $\beta_i - \gamma_i$ for $i=1, \dots, g-1$. Then $P'$ is a relative periodic domain for $K_\lambda$, in the sense of Section \ref{ssec: rel-per-domain}. We consider each of the terms in \eqref{eq: abs-alex}. We have $\hat\chi(P') = \hat\chi(P)$. Let $\tilde w$ and $\tilde z$ be the points on $\beta_g$ closest to $w$ and $z'$, respectively; then
\[
n_{\tilde w}(P') + n_{\tilde z}(P') = n_{\bar w}(P) + n_{\bar z}(P) = d-k .
\]
Finally, for $\x \in \T_\alpha \cap \T_\beta$ and $i=1, \dots, k$, we have:
\begin{align*}
n_{\x_i}(P') &=
\begin{cases}
n_\x(P) - \frac{d-k}{2} + di & \text{if } k>0 \\
n_\x(P) - \frac{d-k}{2} + d(1-i) & \text{if } k<0
\end{cases} \\
&= n_\x(P) - \frac{d-k}{2} + dn_w(\psi_{\x,i}).
\end{align*}
For $j=1, \dots, g-1$, the local multiplicities of $P$ at $x_j$ are the same as those of $P'$ at the nearest point. Combining these facts, we see that
\begin{align*}
d\AlNorm_{w,z}(\x) - k\AlNorm_{w,z'}(\x_i) &= -d n_w(\psi_{\x,i}) + \frac{d-k}{2} \\
&= - k -d n_w(\psi_{\x,i}) + \frac{d+k}{2} \\
&= d n_z(\psi_{\x,i}) + kn_{z'}(\psi_{\x,i}) - (d+k) n_w(\psi_{\x,i}) + \frac{d+k}{2}
\end{align*}
as required.

To prove \eqref{eq: c1-alex-q}, we use the first Chern class formula from \cite[Proposition 6.3]{OSz4Manifold}.\footnote {There is a sign inconsistency in the definition of the dual spider number in \cite[Section 6.1]{OSz4Manifold}: if we compute intersection numbers in the usual way, it should be
\[
\sigma(\psi, P) = n_{u(x)}(P) - \#(a \cap \partial'_\alpha P) - \#(b \cap \partial'_\beta P) - \#(c \cap \partial'_\gamma P),
\]
rather than with $+$ signs throughout. Also, the term $\#(\partial P)$ is a \emph{signed} count of the curves in $\partial P$ relative to some fixed orientations (which are the ones used to define the parallel pushoffs $\partial'_\alpha P$, etc.)}
The local contribution of $\psi_{\x,i}^j$ to the dual spider number $\sigma(\psi_{\x,i}, P)$ is $c_j$ for $j=1, \dots, g-1$ and either $di$ (if $k>0$) or $d(1-i)$  (if $k<0$) for $j=g$. Note that the latter equals $n_w(\psi_{\x,i})$ in either case. Therefore,
\begin{align*}
\gen{c_1(\spincs_w(\psi_{\x,i})), [P]} &= \hat\chi(P) + \#(\partial P) - 2n_w(P) + 2\sigma(\psi_{\x,i}, P) \\
&= \hat\chi(P) + \left(-k + \sum_{j=1}^{g-1} (a_j + b_j) \right) - 2 \cdot 0 + 2 \left( d n_w(\psi_{\x,i}) + \sum_{j=1}^{g-1} c_j \right) \\
&= \hat\chi(P) - k + \sum_{j=1}^{g-1} (a_j + b_j + 2c_j) + 2d n_w(\psi_{\x,i}) \\
&= \hat\chi(P) + 2 n_\x(P) -d + 2d n_w(\psi_{\x,i}) \\
&= \hat\chi(P) + 2 n_\x(P) - n_{\bar w}(P) - n_{\bar z}(P) -k + 2d n_w(\psi_{\x,i}) \\
&= 2d \AlNorm_{w,z}(\x) + 2d n_w(\psi_{\x,i}) - 2d n_z(\psi_{\x,i}) - k
\end{align*}
as required.

Now, we consider an arbitrary triangle $\psi \in \pi_2(\x, \Theta_{\beta\gamma}, \q)$. (Assume that $k>0$; the other case follows similarly.) Choose $r \in \Z$ and $i \in \{1, \dots, k\}$ such that
\[
n_w(\psi) - n_z(\psi)  = rk+i.
\]
Let $\psi' = \psi - rP \in \pi_2(\x, \Theta_{\beta\gamma}, \q)$; then $n_w(\psi') - n_z(\psi') = i$. The composite domain $\phi$ with $\DD (\phi)= \DD(\psi') - \DD(\psi_{\x,i})$ is a disk in $\pi_2(\x_i, \q)$, so $\spincs_w(\psi') = \spincs_w(\psi_{\x,i})$.  We then compute:
\begin{align*}
\AlNorm_{w,z'}(\x_i) - \AlNorm_{w,z'}(\q) &= n_{z'}(\phi) - n_w(\phi) \\
&= n_{z'}(\psi') - n_{z'}(\psi_{\x,i}) - n_w(\psi') + n_w(\psi_{\x,i}) \\
&= n_{z'}(\psi) - n_w(\psi) -rd + 1 \\
d\AlNorm_{w,z}(\x) - k\AlNorm_{w,z'}(\q)
&= (d\AlNorm_{w,z}(\x) - k\AlNorm_{w,z'}(\x_i)) + k(\AlNorm_{w,z'}(\x_i) - \AlNorm_{w,z'}(\q)) \\
&= -di + \frac{d-k}{2} + k (n_{z'}(\psi) - n_w(\psi) -rd +1) \\
&= -di + \frac{d+k}{2} + k (n_{z'}(\psi) - n_w(\psi)) - rdk \\
&= k (n_{z'}(\psi) -  n_w(\psi)) - d(n_w(\psi) - n_z(\psi)) + \frac{d+k}{2}   \\
&= d n_z(\psi) + k n_{z'}(\psi) - (k+d) n_w(\psi) +\frac{d+k}{2}
\end{align*}
as required. Similarly, we have:
\begin{align*}
\gen{c_1(\spincs_w(\psi)), [P]}
&= \gen{c_1(\spincs_w(\psi' + rP)), [P]} \\
&= \gen{c_1(\spincs_w(\psi') + r \PD[P]), [P]} \\
&= \gen{c_1(\spincs_w)(\psi'), [P]} +   2r [P]^2 \\
&= \gen{c_1(\spincs_w)(\psi_{\x,i}), [P]} +   2r kd \\
&= 2d \AlNorm_{w,z}(\x) - k + 2di +   2r kd \\
&= 2d \AlNorm_{w,z}(\x) + 2d n_w(\psi) - 2d n_z(\psi) - k
\end{align*}
as required.

Finally, \eqref{eq: c1-alex-q} follows immediately from \eqref{eq: alex-shift} and \eqref{eq: c1-alex-x}.
\end{proof}

\begin{remark}
In \cite[Proposition 2.2]{OSzRational}, Ozsv\'ath and Szab\'o construct a bijection
\[
E_{Y, \lambda, K} \co \Spin^c(W_\lambda(K)) \to \ul\Spin^c(Y,K)
\]
characterized by the property that for $\psi \in \pi_2(\x, \Theta_{\beta\gamma}, \q)$,
\[
E_{Y,\lambda, K}(\spincs_w(\psi)) = \ul\spincs_{w,z}(\x) + (n_z(\psi) - n_w(\psi))\PD[\mu]
\]
(where, as above, our $\mu$ is the negative of theirs). Lemma \ref{lemma: c1-alex} gives an explicit and diagram-independent description of $E_{Y, \lambda, K}$: for any $\spinct \in \Spin^c(W_\lambda(K))$, $E_{Y,\lambda,K}(\spinct)$ is the relative spin$^c$ structure that satisfies
\[
G_{Y,K}(E_{Y,\lambda,K}(\spinct)) = \spinct|_Y \quad \text{and} \quad
\AlNorm_{Y,K}(E_{Y,\lambda, K}(\spinct)) = \frac{\gen{c_1(\spinct), [\hat F]} + k}{2d} .
\]
%
%
\end{remark}

Let $W'_\lambda = W'_\lambda(K)$ denote $W_\lambda$ with orientation reversed, viewed as a cobordism from $Y_\lambda(K)$ to $Y$. This cobordism can be represented by the triple diagram $(\Sigma, \bm\alpha, \bm\gamma, \bm\beta)$. The periodic domain $P$ still generates $H_2(W'_\lambda)$; with respect to the reversed orientation, we have $[P]^2 = -dk$. Let $\Theta_{\gamma\beta}$ be the corresponding top generator. The analogue of Lemma \ref{lemma: c1-alex} for $(\alpha, \gamma, \beta)$ triangles then states:

\begin{lemma} \label{lemma: c1-alex-rev}
For any $\q \in \T_\alpha \cap \T_\gamma$, $\x \in \T_\alpha \cap \T_\beta$, and $\psi \in \pi_2(\q, \Theta_{\gamma\beta}, \x)$, we have:
\begin{align}
\label{eq: alex-shift-rev}
k\AlNorm_{w,z'}(\q) - d\AlNorm_{w,z}(\x) &= d n_z(\psi) + k n_{z'}(\psi) - (k+d) n_w(\psi) -\frac{k+d}{2} \\
\label{eq: c1-alex-q-rev}
\gen{c_1(\spincs_w(\psi)), [\hat F]} &=   2k\AlNorm_{w,z'}(\q)  +2k n_w(\psi) -2k n_{z'}(\psi)  + d \\
\label{eq: c1-alex-x-rev}
&= 2d\AlNorm_{w,z}(\x) + 2d n_z(\psi) - 2d n_w(\psi) - k
\end{align}
\end{lemma}

\begin{proof}
There is a disk $\phi \in \pi_2(\Theta_{\beta\gamma}, \Theta_{\gamma\beta})$ consisting entirely of small bigons outside the winding region, with $n_w(\phi) = n_z(\phi) = n_{z'}(\phi) = 0$. Hence, for any $\q \in \T_\alpha \cap \T_\gamma$, $\x \in \T_\alpha \cap \T_\beta$, and $\psi \in \pi_2(\q, \Theta_{\gamma\beta}, \x)$, we have a class $\psi' = \overline{\psi * \phi} \in \pi_2(\x, \Theta_{\beta\gamma}, \q)$. We now apply Lemma \ref{lemma: c1-alex} to this class.
\end{proof}

As a consequence of either of the two previous lemmas, we see that the cosets in $\Q/\Z$ in which the Alexander gradings for $Y_\lambda(K)$ are contained is closely connected to spin$^c$ structures on $W_\lambda(K)$:

\begin{corollary} \label{cor: c1-alex-q-cong}
Let $\spinct \in \Spin^c(Y_\lambda(K))$, and let $\spincv$ be any spin$^c$ structure on $W_\lambda(K)$ extending $\spinct$. Then
\begin{equation}
\label{eq: c1-alex-q-cong}
\AlNorm_{Y_\lambda, K_\lambda}(\spinct) \equiv \frac{\gen{c_1(\spincv), [P]} - d}{2k} \pmod \Z
\end{equation}
\end{corollary}

\begin{proof}
Apply \eqref{eq: c1-alex-q} to any triangle representing $\spincv$.
\end{proof}

\section{Large surgeries} \label{sec: large}

In this section, we will restate the large surgery formulas from \cite[Section 4]{OSzRational} and \cite[Section 4.1]{HeddenPlamenevskayaRational} with more details about the Alexander and Maslov gradings, and prove some key lemmas that will be needed for studying the surgery exact triangle in Section \ref{sec: exact-sequence}. Throughout this section, let $\lambda$ denote a fixed longitude for $K$ as above, corresponding to some integer $k \ne 0$. We will be studying Heegaard diagrams for $Y_{\lambda + m \mu}(K)$, where $m$ is a large positive integer.

\subsection{Well-adapted Heegaard diagrams}

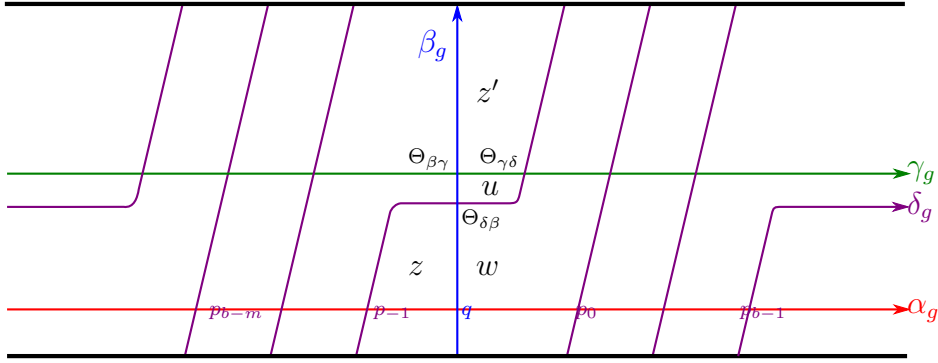
\begin{figure}
\labellist
 \pinlabel $w$ at 190 42
 \pinlabel $z$ at 163 42
 \pinlabel $z'$ at 190 109
 \pinlabel $u$ at 191 72
 \pinlabel {{\color{red} $\alpha_g$}} [l] at 345 26
 \pinlabel {{\color{blue} $\beta_g$}} [r] at 178 126
 \pinlabel {{\color{darkgreen} $\gamma_g$}} [l] at 345 78
 \pinlabel {{\color{purple} $\delta_g$}} [l] at 345 65
 \tiny
 \pinlabel $\Theta_{\beta\gamma}$ [br] at 178 77
 \pinlabel $\Theta_{\gamma\delta}$ [br] at 204 77
 \pinlabel $\Theta_{\delta\beta}$ [tl] at 178 66
 \pinlabel {{\color{blue} $q$}} [tl] at 178 30
 \pinlabel {{\color{purple} $p_{-1}$}} [tl] at 145 30
 \pinlabel {{\color{purple} $p_{b-m}$}} [tl] at 83 30
 \pinlabel {{\color{purple} $p_{0}$}} [tl] at 221 30
 \pinlabel {{\color{purple} $p_{b-1}$}} [tl] at 283 30
\endlabellist
\includegraphics{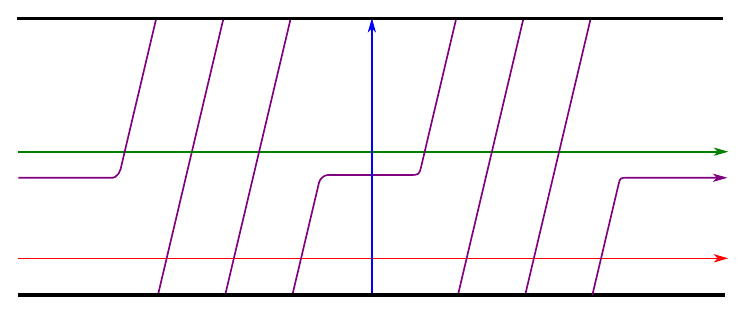}
\caption{The winding region of a well-adapted Heegaard diagram, in the case where $m=6$ and $b=3$.}
\label{fig: twist}
\end{figure}

Assume we have fixed a Heegaard diagram $(\Sigma, \bm\alpha, \bm\beta, \bm\gamma, w, z, z')$ adapted to $\lambda$-surgery on $K$. If $k<0$, we wind $\gamma_g$ to achieve admissibility as in the proof of Lemma \ref{lemma: c1-alex}. Let $A$ be an annular neighborhood of $\beta_g$ containing all three basepoints $w, z, z'$, and let $A' \subset A$ be a smaller such neighborhood. For any natural numbers $0 \le b \le m$, let $\bm\delta^{m,b} = (\delta_1^{m,b}, \dots, \delta_g^{m,b})$ be a tuple of curves obtained from $\bm\gamma$ as follows. For $i=1, \dots, g-1$, $\delta_i^{m,b}$ is a small translate of $\gamma_i$ meeting it in two points. The curve  $\delta_g$ is obtained from a parallel pushoff of $\gamma_g$ by performing $m$ left-handed Dehn twists parallel to $\beta_g$, where $b$ (resp. $m-b$) of these twists are performed in the component of $A \minus A'$ on the same side of $\beta_g$ as $w$ (resp. $z$). (See Figure \ref{fig: twist}.) We say that the Heegaard diagram $(\Sigma, \bm\alpha, \bm\delta^{m,b}, w,z,z')$ is \emph{well-adapted} to $(\lambda+m\mu)$-surgery on $K$. We call $A$ the \emph{winding region}.
If $m$ and $b$ are understood from context, we omit the superscripts from the $\delta$ curves.

The Heegaard triple diagram $(\Sigma, \bm\alpha, \bm\beta, \bm\delta^{m,b}, w, z,z')$ is adapted to $(\lambda+m\mu)$--surgery on $K$. Hence, all the results of the previous section apply, where we take $k+dm$ in place of $k$ throughout.

As shown in Figure \ref{fig: twist}, let $u$ be a basepoint located on the $w$ side of $\beta_g$, in between $\gamma_g$ and $\delta_g$. This will be needed later on to understand the effect of $(\alpha,\gamma,\delta)$ triangles on the Alexander grading.

We will typically use $\spincs$, $\spinct$, and $\spincu$ to refer to spin$^c$ structures on $Y_{\alpha\beta}$, $Y_{\alpha\gamma}$, and $Y_{\alpha\delta}$,  respectively, and use $\spincv$ for spin$^c$ structures on the various cobordisms between them. As a notational convenience, define $\Gens(\bm\alpha, \bm\beta) = \T_\alpha \cap \T_\beta$ and
\begin{equation} \label{eq: Gens}
\Gens(\bm\alpha, \bm\beta, w, \spincs) = \{\x \in \T_\alpha \cap \T_\beta \mid \spincs_w(\x) = \spincs\},
\end{equation}
and likewise for the spin$^c$ decompositions of the other complexes.

\subsection{Periodic domains}

We now discuss the periodic domains present in the Heegaard multi diagram $(\Sigma, \bm\alpha, \bm\beta, \bm\gamma, \bm\delta)$.

To begin, for any $j=1, \dots, g-1$, there are small periodic domains $S_{\beta\gamma}^j$ and $S_{\gamma\delta}^j$ with $\partial S_{\beta\gamma}^j = \beta_j - \gamma_j$ and $\partial S_{\gamma\delta}^j = \gamma_j - \delta_j$, supported in a small neighborhood of each pair of curves. We will refer to these as \emph{thin domains}. As in the previous section, the groups $\Pi_{\alpha\beta}$, $\Pi_{\alpha\gamma}$, and $\Pi_{\alpha\delta}$ are naturally isomorphic, by adding thin domains as needed.

Let $P_\gamma$ and $P_\delta$ be the triply periodic domains for $(\bm\alpha, \bm\beta, \bm\gamma)$ and $(\bm\alpha, \bm\beta, \bm\delta)$, respectively, which correspond to $P$ from Section \ref{sec: alex-surgery}. Specifically, we have $n_w(P_\gamma) = n_w(P_\delta) = 0$ and
\begin{align}
\label{eq: Pgamma-boundary}
\partial P_\gamma &= -d\alpha_g  -k \beta_g + d\gamma_g + \sum_{i=1}^{g-1} (a_i \alpha_i + b_i \gamma_i) \\
\label{eq: Pdelta-boundary}
\partial P_\delta &= -d\alpha_g  -(k+dm) \beta_g + d\delta_g + \sum_{i=1}^{g-1} (a_i \alpha_i + b_i \delta_i).
\end{align}
There is a $(\beta,\gamma,\delta)$ triply periodic domain $Q$ with
\[
\partial Q = m\beta_g + \gamma_g - \delta_g ,
\]
so that
\[
P_\gamma - P_\delta = dQ + \text{thin domains}.
\]
Finally, define the $(\alpha, \gamma, \delta)$ periodic domain
\[
R = \tfrac{m}{\nu} P_\gamma + \tfrac{k}{\nu} Q
\]
where $\nu = \gcd(k,m)$; it has
\[
\partial R = -\frac{dm}{\nu} \alpha_g  +  \frac{k+dm}{\nu}\gamma_g  - \frac{k}{\nu} \delta_g + \frac{m}{\nu} \sum_{i=1}^{g-1} (a_i \alpha_i + b_i \gamma_i).
\]

The multiplicities of the periodic domains at the various basepoints are as follows:
\begin{center}
\begin{tabular} {|c|c|c|c|c|} \hline
& $n_w$ & $n_z$ & $n_{z'}$ & $n_u$ \\ \hline
$P_\gamma$ & $0$ & $-k$ & $d$ & $0$ \\
$P_\delta$ & $0$ & $-k-dm$ & $d$ & $d$ \\
$Q$ & $0$ & $m$ & $0$ & $-1$ \\
$R$ & $0$ & $0$ & $\frac{dm}{\nu}$ & $-\frac{k}{\nu}$ \\ \hline
\end{tabular}
\end{center}

Let $\Pi_{\alpha\beta\gamma\delta}$ denote the group of integral $(\alpha, \beta, \gamma, \delta)$ periodic domains with $n_w=0$, and let $\bar \Pi_{\alpha\beta\gamma\delta}$ denote its quotient by the thin domains. Define $\Pi_{\alpha\beta\gamma}$, $\Pi_{\alpha\beta\delta}$, $\Pi_{\alpha\gamma\delta}$, and $\Pi_{\beta\gamma\delta}$ and their barred versions similarly. The following lemma is left as an exercise for the reader:

\begin{lemma}
The group $\bar \Pi_{\alpha\beta\gamma\delta}$ is free abelian of rank $2 + b_1(Y)$, generated (over $\Z$) by $P_\gamma$ and $Q$ together with any basis for $\Pi_{\alpha\beta} \cong H_2(Y)$.
\end{lemma}

For any domain $S$, define
\begin{equation}\label{eq: mult-comb}
\MultComb(S) = d n_z(S) + k n_{z'}(S) + dm n_u(S) - (k+dm+d)n_w(S).
\end{equation}
For any multi-periodic domain $S$ (including those with nonzero multiplicity at $w$), we have $\MultComb(S)=0$, since any such domain is a linear combination of $P_\gamma$, $Q$, $\Sigma$, thin domains, and elements of $\Pi_{\alpha\beta}$, and $\MultComb$ vanishes for each of these. Observe that for different types of domains, the formula for $\MultComb(S)$ simplifies considerably depending on which basepoints are in the same regions. These simplifications are noted in the following table:
\begin{center}
  \begin{tabular}{|c|c|}
     \hline
     Type of domain & $\MultComb(S)$ \\ \hline
     $(\alpha,\beta)$ & $d(n_z(S) - n_w(S))$ \\
     $(\alpha,\gamma)$ & $k(n_{z'}(S) - n_w(S))$ \\
     $(\alpha,\delta)$ & $(k+dm)(n_{z'}(S) - n_w(S))$ \\
     $(\beta,\gamma)$, $(\gamma,\delta)$, or $(\beta,\delta)$ & $0$ \\
     $(\alpha,\beta,\gamma)$  & $d n_z(S) + k n_{z'}(S) - (k+d)n_w(S)$  \\
     $(\alpha,\gamma,\delta)$ & $k n_{z'}(S) + dm n_u(S) - (k+dm)n_w(S)$ \\
     $(\alpha,\delta,\beta)$ & $d n_z(S) + (k+dm) n_{z'}(S) - (k+dm+d)n_w(S) $  \\
     $(\beta,\gamma,\delta)$ & $d n_z(S) + dm n_u(S) - (dm+d)n_w(S)$  \\
     \hline
   \end{tabular}
\end{center}

\subsection{Topology of the cobordisms} \label{ssec: cobordisms}

Let us consider the topology of the cobordisms associated with the quadruple diagram $(\Sigma,
\bm\alpha, \bm\beta, \bm\gamma, \bm\delta)$.

The construction from \cite[Section 8.1.5]{OSz3Manifold} gives rise to three separate $4$-manifolds $X_{\alpha\beta\gamma\delta}$, $X_{\alpha\gamma\delta\beta}$, and $X_{\alpha\delta\beta\gamma}$, with:
\begin{align*}
\partial X_{\alpha\beta\gamma\delta} &= {-Y_{\alpha\beta}} \sqcup {-Y_{\beta\gamma}} \sqcup
{-Y_{\gamma\delta}} \sqcup {Y_{\alpha\delta}} \\
\partial X_{\alpha\gamma\delta\beta} &= {-Y_{\alpha\gamma}} \sqcup {-Y_{\gamma\delta}} \sqcup {-Y_{\delta\beta}} \sqcup Y_{\alpha\beta} \\
\partial X_{\alpha\delta\beta\gamma} &= {-Y_{\alpha\delta}} \sqcup {-Y_{\delta\beta}} \sqcup {-Y_{\beta\gamma}}  \sqcup {Y_{\alpha\gamma}}
\end{align*}
Each one comes with a pair of decompositions:
\begin{align}
\label{eq: Xabgd-decomp}
X_{\alpha\beta\gamma\delta} &= X_{\alpha\beta\gamma} \cup_{Y_{\alpha\gamma}} X_{\alpha\gamma\delta} = X_{\alpha\beta\delta} \cup_{Y_{\beta\delta}} X_{\beta\gamma\delta} \\
\label{eq: Xagdb-decomp}
X_{\alpha\gamma\delta\beta} &= X_{\alpha\gamma\delta} \cup_{Y_{\alpha\delta}} X_{\alpha\delta\beta} =
X_{\alpha\gamma\beta} \cup_{Y_{\gamma\beta}} X_{\gamma\delta\beta} \\
\label{eq: Xadbg-decomp}
X_{\alpha\delta\beta\gamma} &= X_{\alpha\delta\beta} \cup_{Y_{\alpha\beta}} X_{\alpha\beta\gamma} =
X_{\alpha\delta\gamma} \cup_{Y_{\delta\gamma}} X_{\delta\beta\gamma}.
\end{align}
The $3$-manifolds in question are:
\begin{align*}
Y_{\alpha\beta} &= Y & Y_{\alpha\gamma}&= Y_\lambda(K) & Y_{\alpha\delta} &= Y_{\lambda+m\mu}(K) \\
Y_{\beta\gamma} &= \conn^{g-1}(S^1 \times S^2)  & Y_{\gamma\delta} &= L(m,1) \conn^{g-1}(S^1 \times S^2) & Y_{\delta\beta} &= \conn^{g-1}(S^1 \times S^2) \\
Y_{\gamma\beta} &= -Y_{\beta\gamma} & Y_{\delta\gamma} &= -Y_{\gamma\delta} & Y_{\beta\delta} &= -Y_{\delta\beta}
\end{align*}
Note also that $X_{\alpha\gamma\beta} = -X_{\alpha\beta\gamma}$, and so on.

Let $\bar X_{\alpha\beta\gamma}$, $\bar X_{\alpha\beta\gamma\delta}$, etc., be the manifolds obtained by attaching $3$-handles to kill all of the $S^1 \times S^2$ summands in $Y_{\beta\gamma}$, $Y_{\gamma\delta}$, and $Y_{\delta\beta}$, as appropriate; analogues of \eqref{eq: Xabgd-decomp}, \eqref{eq: Xagdb-decomp}, and \eqref{eq: Xadbg-decomp} hold for these manifolds as well. In each case, there are isomorphisms making the following diagram commute:
\[
\xymatrix{
H_2(X_*) \ar@{->>}[r] \ar[d]^\cong & H_2(\bar X_*) \ar[d]^\cong \\
\Pi_* \ar@{->>}[r] & \bar \Pi_*
}
\]
(where $*$ is any $3$- or $4$-element ordered subset of $\{\alpha,\beta,\gamma,\delta\}$). In particular, the periodic domains $\{P_\gamma, P_\delta, Q, R\}$ represent homology classes which survive in $H_2(\bar X_{\alpha\beta\gamma\delta})$ and satisfy the relations
\begin{align*}
[P_\delta] &= [P_\gamma] - d[Q], & [R] = \tfrac{m}{\nu} [P_\gamma] + \tfrac{k}{\nu} [Q].
\end{align*}
(Hence, we may also write $[R] = \tfrac{m}{\nu}[P_\delta] + \tfrac{k+dm}{\nu}[Q]$.) The same relations also hold in $H_2(\bar X_{\alpha\gamma\delta\beta})$ and $H_2(\bar X_{\alpha\delta\beta\gamma})$, which are defined analogously.

Let $W_{\alpha\beta\gamma}$ (resp.~$W_{\alpha\delta\beta}$) be obtained from $\bar X_{\alpha\beta\gamma}$ (resp.~$\bar X_{\alpha\delta\beta}$) by gluing a $4$-handle to the $S^3$ boundary component left over from $Y_{\beta\gamma}$ (resp. $Y_{\delta\beta}$). These cobordisms are simply the $2$-handle cobordisms $W_\lambda(K)$ and $W'_{\lambda+m\mu}(K)$, respectively. However, we cannot do this with $\bar X_{\alpha\gamma\delta}$, since the boundary component left over from $Y_{\gamma\delta}$ is $L(m,1)$ rather than $S^3$. Instead, let $W_{\alpha\gamma\delta}$ be obtained from $\bar X_{\alpha\gamma\delta}$ by deleting a neighborhood of an arc connecting $Y_{\alpha\gamma}$ and $L(m,1)$. This is a cobordism from $Y_\lambda(K) \conn L(m,1)$ to $Y_{\lambda + m \mu}$ given by a single $2$-handle attachment.

Let $\spincs^0_{\beta\gamma}$ denote the unique torsion spin$^c$ structure on $Y_{\beta\gamma}$. Let $\Theta_{\beta\gamma}$ and $\Theta_{\gamma\beta}$ be the standard top-dimensional generators for $\CF^{\le0}(\Sigma, \bm\beta, \bm\gamma, w)$ and $\CF^{\le0}(\Sigma, \bm\gamma, \bm\beta, w)$, both of which use the unique intersection point in $\beta_g \cap \gamma_g$ as shown in Figure \ref{fig: twist}. Define $\spincs^0_{\beta\delta}$, $\Theta_{\beta\delta}$, and $\Theta_{\delta\beta}$ analogously.

The situation for $Y_{\gamma\delta}$ is a bit more complicated. The triple diagram $(\Sigma, \bm\gamma, \bm\beta, \bm\delta)$ is an adapted diagram for $m$-framed surgery on the unknot in $\conn^{g-1}(S^1 \times S^2)$, where $\beta_g$ is the meridian and $\delta_g$ is the longitude, and $-Q$ plays the role of $P$ from Section \ref{sec: alex-surgery}; this confirms that $Y_{\gamma\delta}$ is indeed as describe above. Indeed, if we let $B_m$ denote the Euler number $m$ disk bundle over $S^2$, which has boundary $L(m,1)$, the $2$-handle cobordism associated to $(\Sigma, \bm\gamma, \bm\beta, \bm\delta)$ is diffeomorphic to $((\#^{g-1} S^1 \times S^2) \times I) \bconn B_m$, and $Q$ corresponds to the homology class of the zero section in $B_m$.

Let $\spincs^0_{\gamma\delta} \in \Spin^c(Y_{\gamma\delta})$ denote the canonical spin$^c$ structure from \cite[Definition 6.3]{OSzRational}; that is, $\spincs^0_{\gamma\delta}$ the unique spin$^c$ structure on $Y_{\gamma\delta}$ that is torsion and has an extension $\spinct$ to $X_{\gamma\beta\delta}$ which satisfies $\gen{c_1(\spinct), [S^2]} = \pm m$. The $m$ intersection points of $\gamma_g \cap \delta_g$ can be paired with the top-dimensional intersection points of $\gamma_j \cap \delta_j$ ($j=1, \dots, g-1$) to give $m$ canonical cycles in $\CF^{\le 0}(\Sigma, \bm\gamma, \bm\delta, w)$, each of which represents a different torsion spin$^c$ structure on $Y_{\gamma\delta}$. Let $\Theta_{\gamma\delta}$ denote the generator which uses the point of $\gamma_j \cap \delta_j$ that is adjacent to $w$, $z'$, and $u$, as shown in Figure \ref{fig: twist}. We have:
\begin{lemma} \label{lemma: Theta-gamma-delta}
The generator $\Theta_{\gamma\delta}$ represents $\spincs^0_{\gamma\delta}$.
\end{lemma}

We prove this by studying the diagram $(\Sigma, \bm\beta, \bm\gamma, \bm\delta)$, which describes the same $4$-manifold as $(\Sigma, \bm\gamma, \bm\beta, \bm\delta)$ but with reversed orientation. The following result is a simple adaptation of \cite[Section 6]{OSz4Manifold}; see also \cite[Section 5]{HeddenMarkFractional}.

\begin{lemma} \label{lemma: bgd-triangles}
For each integer $l \ge 0$, there are positive classes $\tau_l^+, \tau_l^- \in \pi_2(\Theta_{\beta\gamma}, \Theta_{\gamma\delta}, \Theta_{\beta\delta})$, which satisfy the following properties:
\begin{align*}
n_w(\tau_l^+) &= \frac{m l(l+1)}{2} &  n_z(\tau_l^+) &= \frac{m l(l-1)}{2} \\
n_w(\tau_l^-) &= \frac{m l(l+1)}{2} &  n_z(\tau_l^-) &= \frac{m (l+1)(l+2)}{2} 
\end{align*}
(In particular, the intersection of the domain of $\tau_0^+$ with the winding region is the small triangle containing $u$ in Figure \ref{fig: twist}.) Moreover, each of these classes has $\mu(\tau_l^\pm)=0$ and $\#\MM(\tau_l^\pm)=1$, and these are the only classes in $\pi_2(\Theta_{\beta\gamma}, \Theta_{\gamma\delta}, \Theta')$ (for any $\Theta' \in \T_\beta \cap \T_\delta$) with rigid holomorphic representatives. \qed
\end{lemma}

\begin{proof}[Proof of Lemma \ref{lemma: Theta-gamma-delta}]
A direct computation using \cite[Proposition 6.3]{OSz4Manifold} shows that
\[
\gen{c_1(\spincs_w(\tau_l^\pm)), [Q]} = \pm(2l+1)m.
\]
Since the restriction of $\spincs_w(\tau_l^\pm)$ to $Y_{\gamma\delta}$ is $\spincs_w(\Theta_{\gamma\delta})$, this shows that $\spincs_w(\Theta_{\gamma\delta})$ is the canonical spin$^c$ structure.
\end{proof}

For each of the $4$-manifolds $X_*$ described above, let $\Spin^c_0(X_*)$ denote the set of spin$^c$ structures that restrict to $\spincs^0_{\beta\gamma}$ on $Y_{\beta\gamma}$, $\spincs^0_{\gamma\delta}$ on $Y_{\gamma\delta}$, and $\spincs^0_{\delta\beta}$ on $Y_{\delta\beta}$, as applicable. Note that all such spin$^c$ structures extend uniquely to $\bar X_*$.

\begin{remark} \label{rmk: H2-evals}
Assuming that $k$ and $k+dm$ are both nonzero, the groups $H_2(Y_{\alpha\beta})$, $H_2(Y_{\alpha\gamma})$, and $H_2(Y_{\alpha\delta})$ are naturally isomorphic. Moreover, these isomorphisms are realized through the cobordisms $X_*$; that is, any element $S_{\alpha\beta} \in H_2(Y_{\alpha\beta})$ is homologous in $X_{\alpha\beta\gamma}$ to a unique element $S_{\alpha\gamma} \in H_2(Y_{\alpha\gamma})$, and so on. As a result, if $\spincs \in \Spin^c(Y_{\alpha\beta})$ and $\spinct \in \Spin^c(Y_{\alpha\gamma})$ are the restrictions of some $\spincv \in \Spin^c_0(X_{\alpha\beta\gamma})$, then $\gen{c_1(\spincs), S_{\alpha\beta}} = \gen{c_1(\spinct), S_{\alpha\gamma}}$. In particular, $\spincs$ is torsion iff $\spinct$ is torsion. (The same applies for the other cobordisms.)
\end{remark}

We conclude this section by discussing the intersection forms on the $4$-manifolds $X_*$. While $H_2(\bar X_{\alpha\beta\gamma\delta})$, $H_2(\bar X_{\alpha\gamma\delta\beta})$, and $H_2(\bar X_{\alpha\delta\beta\gamma})$ are all isomorphic groups, their intersection forms are quite different, as we now explain.

In $X_{\alpha\beta\gamma\delta}$, the classes $[P_\gamma]$, $[P_\delta]$, $[Q]$, and $[R]$ can be represented by surfaces which are contained in $X_{\alpha\beta\gamma}$, $X_{\alpha\beta\delta}$, $X_{\beta\gamma\delta}$, and $X_{\alpha\gamma\delta}$, respectively. 
Using the formula \eqref{eq: PD-self-int}, we can compute that
\begin{equation} \label{eq: abgd-PD-self-int}
\begin{aligned}
[P_\gamma]^2 &= dk & [P_\delta]^2 &= d(k+dm) \\ [Q]^2 &=-m &  [R]^2 &= -\frac{m k (k+dm)}{\nu^2} .
\end{aligned}
\end{equation}
The decomposition \eqref{eq: Xabgd-decomp} shows that
\begin{equation} \label{eq: abgd-int-form}
[P_\gamma] \cdot [R] = [P_\delta] \cdot [Q] = 0,
\end{equation}
since each pair can be represented by disjoint surfaces. The intersection numbers of the other pairs of generators can have nontrivial intersection numbers which can be worked out using bilinearity.

On the other hand, in $X_{\alpha\gamma\delta\beta}$, the above-mentioned classes have different self-intersection numbers (up to sign) and different pairs which are disjoint, namely:
\begin{gather*}
\begin{aligned}
[P_\gamma]^2 &= -dk & [P_\delta]^2 &= -d(k+dm) \\
[Q]^2 &=-m &  [R]^2 &= -\frac{m k (k+dm)}{\nu^2}
\end{aligned} \\
[R] \cdot [P_\delta] = [P_\gamma] \cdot [Q] = 0.
\end{gather*}
The signs of $[P_\gamma]^2$ and $[P_\delta]^2$ are reversed because they are contained in $X_{\alpha\gamma\beta}$ and $X_{\alpha\delta\beta}$, respectively, which are diffeomorphic to $-X_{\alpha\beta\gamma}$ and $-X_{\alpha\beta\delta}$. Note that these determine the reversed cobordisms $W'_\lambda(K)$ and $W'_{\lambda+m\mu}(K)$. Similar analysis applies to $X_{\alpha\delta\beta\gamma}$.

\subsection{Polygons, spin$^c$ structures, and Alexander gradings} \label{ssec: grading-shift}

We now describe the Alexander grading shifts and $c_1$ evaluations associated to Whitney triangles and rectangles in our Heegaard multi-diagram. If $b_1(Y)>0$, then the Alexander grading may depend on the choice of homology class of Seifert surface; if so, we fix such a choice for $K$, and use the corresponding choices for $K_\lambda$ and $K_{\lambda+m\mu}$.

Throughout the rest of the paper, we will generally refer to elements of $\T_\alpha \cap \T_\beta$ as $\x$ or $\y$, elements of $\T_\alpha \cap \T_\gamma$ as $\q$ or $\r$, and elements of $\T_\alpha \cap \T_\delta$ as $\a$ or $\b$. Also, as a notational convenience, let us define:
\begin{align*}
\AlNorm(\x) &= \AlNorm_{w,z}(\x) & \AlNorm(\q) &= \AlNorm_{w,z'}(\q) & \AlNorm(\a) &=  \AlNorm_{w,z'}(\a) \\
\Al(\x) &= d \AlNorm_{w,z}(\x) & \Al(\q) &= k \AlNorm_{w,z'}(\q) & \Al(\a) &= (k+dm) \AlNorm_{w,z'}(\a).
\end{align*}
(That is, $\Al$ denotes the other normalization convention for the Alexander grading, as discussed in footnote \ref{fn: alex-norm}.)

\begin{proposition} \label{prop: alex-triangle}
Let $\x \in \T_\alpha \cap \T_\beta$, $\q \in \T_\alpha \cap \T_\gamma$, and $\a \in \T_\alpha \cap \T_\delta$.
\begin{enumerate}
\item
For any $\psi \in \pi_2(\x, \Theta_{\beta\gamma}, \q)$,
\begin{align} \label{eq: abg-alex}
\Al(\x) - \Al(\q) &= d n_z(\psi) + k n_{z'}(\psi) - (k+d) n_w(\psi) +\frac{k+d}{2} \\
\label{eq: abg-c1-x} \gen{c_1(\spincs_w(\psi)), [P_\gamma]}
&= 2 \Al(\x) + 2d n_w(\psi) - 2d n_z(\psi) - k \\
\label{eq: abg-c1-q}
&= 2 \Al(\q) + 2k n_{z'}(\psi) - 2k n_w(\psi) + d
\end{align}

\item
For any $\psi \in \pi_2(\q, \Theta_{\gamma\delta}, \a)$,
\begin{align}
\label{eq: agd-alex}
\Al(\q) - \Al(\a) &= k n_{z'}(\psi) + dm n_u(\psi)  - (k+dm) n_w(\psi) - \frac{dm}{2} \\
\label{eq: agd-c1-a}
\langle c_1(\spincs_w(\psi)), [R]\rangle
&= \frac{m}{\nu} \left( 2 \Al(\a)  +  2(k+dm)(n_u(\psi) - n_w(\psi)) - k-dm+d \right) \\
\label{eq: agd-c1-q}
&= \frac{m}{\nu} \left( 2 \Al(\q) + 2k n_u(\psi) - 2k n_{z'}(\psi)  -  k  + d \right)
\end{align}

\item
For any $\psi \in \pi_2(\a, \Theta_{\delta\beta}, \x)$,
\begin{equation}  \label{eq: adb-alex}
\Al(\a) - \Al(\x) =  d n_z(\psi) + (k+dm) n_{z'}(\psi) - (k+dm+d) n_w(\psi)  -\frac{k+dm+ d}{2}
\end{equation}
\begin{align}
\label{eq: adb-c1-a}
\gen{c_1(\spincs_w(\psi)), [P_\delta]}
&= 2\Al(\a)  + 2(k+dm)n_w(\psi)  - 2(k+dm) n_{z'}(\psi)  + d \\
\label{eq: adb-c1-x}
&= 2\Al(\x) + 2d n_z(\psi) - 2d n_w(\psi)  - (k+dm)
\end{align}
\end{enumerate}
\end{proposition}

Note that the linear combinations of basepoint multiplicities in \eqref{eq: abg-alex}, \eqref{eq: agd-alex}, and \eqref{eq: adb-alex} are all specializations of $\MultComb(\psi)$ from \eqref{eq: mult-comb}.

\begin{proof}[Proof of Proposition \ref{prop: alex-triangle}]
The statements about $(\alpha,\beta,\gamma)$ and $(\alpha,\delta,\beta)$ triangles are given by Lemmas \ref{lemma: c1-alex} and \ref{lemma: c1-alex-rev}, respectively, where for the latter we take $k+dm$ in place of $k$. It remains to prove the three statements about $(\alpha,\gamma, \delta)$ triangles.

Let $\q \in \T_\alpha \cap \T_\gamma$, $\a \in \T_\alpha \cap \T_\delta$, and $\psi \in \pi_2(\q, \Theta_{\gamma\delta}, \a)$. Observe that $n_z(\psi) = n_w(\psi)$ since $w$ and $z$ are only separated by $\beta_g$.

Choose an arbitrary triangle $\phi \in \pi_2(\x, \Theta_{\beta\gamma}, \q)$ for some $\x \in \T_\alpha \cap \T_\beta$. By \eqref{eq: abg-alex},
\[
\Al(\x) - \Al(\q) = d n_z(\phi) + k n_{z'}(\phi) - (k+d) n_w(\phi) + \frac{k+d}{2}.
\]

Let $\tau = \tau_0^+ \in \pi_2(\Theta_{\beta\gamma}, \Theta_{\gamma\delta}, \Theta_{\beta\delta})$ be the class represented by the small triangle in the center of Figure \ref{fig: twist} (see Lemma \ref{lemma: bgd-triangles}); it has $n_w(\tau) = n_z(\tau) = n_{z'}(\tau) = 0$ and $n_u(\tau)=1$. Let $\sigma$ be the composite domain with $\DD(\sigma) = \DD(\phi) + \DD(\psi) - \DD(\tau)$. This is almost the domain of a triangle in $\pi_2(\x, \Theta_{\beta\delta}, \a)$, except that the boundary of $\sigma$ includes $\gamma_g$ with multiplicity
\begin{align*}
r &= n_{z'}(\sigma) - n_u(\sigma) \\
&= n_{z'}(\phi) + n_{z'}(\psi) - n_u(\phi) - n_u(\psi) + 1 \\
&= n_{z'}(\phi) + n_{z'}(\psi) - n_w(\phi) - n_u(\psi) + 1.
\end{align*}
Therefore, there is an actual triangle class $\sigma' \in \pi_2(\x, \Theta_{\beta\delta}, \a)$ with $\DD(\sigma') = \DD(\sigma) - rQ$. In other words, the composites $\rho_1 = \phi * \psi$ and $\rho_2 = \sigma' * \tau$ are each quadrilaterals in $\pi_2(\x, \Theta_{\beta\gamma}, \Theta_{\gamma\delta}, \a)$ whose domains satisfy $\DD(\rho_1) = \DD(\rho_2) + rQ$.

Using Lemma \ref{lemma: c1-alex} (with $k+dm$ in place of $k$), we now compute:
\begin{align*}
\Al(\x) - \Al(\a)
&= d n_z(\sigma') + (k+dm) n_{z'}(\sigma') - (k+dm+d) n_w(\sigma') + \frac{k+dm+d}{2} \\
&= d n_z(\sigma) - dmr + (k+dm) n_{z'}(\sigma) - (k+dm+d) n_w(\sigma) + \frac{k+dm+d}{2} \\
&= d n_z(\phi) + d n_z(\psi) - dm( n_{z'}(\phi) + n_{z'}(\psi) - n_w(\phi) - n_u(\psi) + 1) \\
  & \qquad + (k+dm) n_{z'}(\phi) + (k+dm) n_{z'}(\psi)\\
  & \qquad - (k+dm+d) n_w(\phi) - (k+dm+d) n_w(\psi) + \frac{k+dm+d}{2} \\
&= d n_z(\phi) + k n_{z'}(\phi) - (k+d) n_w(\phi) + k n_{z'}(\psi)  - (k+dm) n_w(\psi) \\
  &\qquad + dm n_u(\psi) + \frac{k-dm+d}{2}
\end{align*}
Subtracting, we have:
\begin{align*}
\Al(\q) - \Al(\a)
&=  k n_{z'}(\psi)  - (k+dm) n_w(\psi)  + dm n_u(\psi) - \frac{dm}{2}
\end{align*}
which proves \eqref{eq: agd-alex}. Likewise, using \eqref{eq: abgd-PD-self-int}, we have:
\begin{align*}
\gen{ c_1(\spincs_w&(\psi)), [R]} \\
&= \gen{c_1(\spincs_w(\rho_1)), [R]} \\
&= \gen{c_1(\spincs_w(\rho_2) + r\PD[Q]), [R]} \\
&= \gen{c_1(\spincs_w(\rho_2)) + 2r\PD[Q], \frac{m}{\nu} [P_\delta] + \frac{k+dm}{\nu}[Q]} \\
&= \frac{m}{\nu} \gen{c_1(\spincs_w(\sigma')), \  [P_\delta]} + \frac{k+dm}{\nu} \gen{c_1(\spincs_w(\tau)), [Q]} + \frac{2r(k+dm)}{\nu}[Q]^2   \\
&= \frac{m}{\nu} \left( 2 \Al(\a) + 2(k+dm)(n_{z'}(\sigma') - n_w(\sigma')) + d\right) + \frac{k+dm}{\nu}(m) +  \frac{2 r(k+dm)}{\nu}(-m)    \\
&= \frac{m}{\nu} \left( 2 \Al(\a) + 2(k+dm)(n_{z'}(\sigma) - n_w(\sigma)) + d\right) + \frac{m(k+dm)}{\nu}(1-2r) \\
&= \frac{m}{\nu} \left( 2 \Al(\a)  + (k+dm) \left( 2n_{z'}(\sigma) - 2n_w(\sigma) + 1 - 2r \right) + d \right) \\
&= \frac{m}{\nu} \left( 2 \Al(\a)  + (k+dm) \left( 2n_u(\sigma) - 2n_w(\sigma) + 1 \right) + d \right) \\
&= \frac{m}{\nu} \left( 2 \Al(\a)  + (k+dm) \left( 2n_u(\psi) - 2n_w(\psi) - 1 \right) + d \right)
\end{align*}
which proves \eqref{eq: agd-c1-a}. Finally, \eqref{eq: agd-c1-q} follows directly from \eqref{eq: agd-alex} and \eqref{eq: agd-c1-a}.
\end{proof}

Next, we turn to rectangles. Recall that for any rectangle $\rho$, we define
\[
\MultComb(\rho) = d n_z(\rho) + k n_{z'}(\rho) + dm n_u(\rho) - (k+dm+d)n_w(\rho).
\]

\begin{proposition} \label{prop: alex-rectangle}
Let $\x \in \T_\alpha \cap \T_\beta$, $\q \in \T_\alpha \cap \T_\gamma$, and $\a \in \T_\alpha \cap \T_\delta$.
\begin{enumerate}
\item
For any $\rho \in \pi_2(\x, \Theta_{\beta\gamma}, \Theta_{\gamma\delta}, \a)$,
\begin{align}
\label{eq: abgd-alex}
\Al(\x) - \Al(\a) &= \MultComb(\rho)  +\frac{k+d-dm}{2} \\
\label{eq: abgd-c1-Pg}
\gen{c_1(\spincs_w(\rho)), [P_\gamma]} &= 2 \Al(\x) + 2d n_w(\rho) - 2d n_z(\rho) -k \\
\label{eq: abgd-c1-R}
\gen{c_1(\spincs_w(\rho)), [R]}
&= \frac{m}{\nu} \left( 2 \Al(\a)  + 2 (k+dm) (n_u(\rho) - n_w(\rho)) - k-dm  +d \right) \\
\label{eq: abgd-c1-Q}
\gen{c_1(\spincs_w(\rho)), [Q]} &= m \left(  2n_u(\rho) - 2n_{z'}(\rho) - 1 \right).
\end{align}

\item
For any $\rho \in \pi_2(\q, \Theta_{\gamma\delta}, \Theta_{\delta\beta}, \x)$,
\begin{align}
\label{eq: agdb-alex}
\Al(\q) - \Al(\x) &= \MultComb(\rho) - \frac{k+d+2dm}{2} \\
\label{eq: agdb-c1-R}
\gen{c_1(\spincs_w(\rho)), [R]} &= \frac{m}{\nu} \left( 2 \Al(\q) + 2k n_u(\rho) - 2k n_{z'}(\rho)  - k  + d \right) \\
\label{eq: agdb-c1-Pd}
\gen{c_1(\spincs_w(\rho)), [P_\delta]} &= 2\Al(\x) + 2d n_z(\rho) - 2d n_w(\rho)  - (k+dm) \\
\label{eq: agdb-c1-Q}
\gen{c_1(\spincs_w(\rho)), [Q]} &= m \left(  2n_u(\rho) - 2n_w(\rho) - 1 \right).
\end{align}

\item
For any $\rho \in \pi_2(\a, \Theta_{\delta\beta}, \Theta_{\beta\gamma}, \q)$,
\begin{align}
\label{eq: adbg-alex}
\Al(\a) - \Al(\q) &=  \MultComb(\rho) - \frac{dm}{2} \\
\label{eq: adbg-c1-Pd}
\gen{c_1(\spincs_w(\rho)), [P_\delta]}
&= 2\Al(\a)  +2(k+dm) n_w(\rho) - 2(k+dm) n_{u}(\rho) + d \\
\label{eq: adbg-c1-Pg}
\gen{c_1(\spincs_w(\rho)), [P_\gamma]} &= 2 \Al(\q) + 2k n_{z'}(\rho) - 2k n_u(\rho) + d \\
\label{eq: adbg-c1-Q}
\gen{c_1(\spincs_w(\rho), [Q])} &= 2 n_w(\rho) - 2 n_z(\rho)  + m
\end{align}
\end{enumerate}
\end{proposition}

\begin{proof}
We consider the case of $(\alpha, \beta, \gamma, \delta)$ rectangles; the other two cases are similar. 

For any $\x \in \T_\alpha \cap \T_\beta$, $\a \in \T_\alpha \cap \T_\delta$, and $\rho \in \pi_2(\x, \Theta_{\beta\gamma}, \Theta_{\gamma\delta}, \a)$, we may choose $\q \in \T_\alpha \cap \T_\gamma$, $\psi_1 \in \pi_2(\x, \Theta_{\beta\gamma}, \q)$, and $\psi_2 \in \pi_2(\q, \Theta_{\gamma\delta}, \a)$ such that
$\psi_1 \in \pi_2(\x, \Theta_{\beta\gamma}, \q)$ so that $\spincs_w(\psi_1) = \spincs_w(\rho)|_{X_{\alpha\beta\gamma}}$ and $\spincs_w(\psi_2) = \spincs_w(\rho)|_{X_{\alpha\gamma\delta}}$. Moreover, by adding copies of $\Sigma$ to (say) $\psi_2$, which does not change the spin$^c$ structure condition, we may assume that $n_w(\rho) = n_w(\psi_1) + n_w(\psi_2)$. Hence, $S = \DD(\rho) - \DD(\psi_1 * \psi_2)$ is an (integral) quadruply periodic domain with $n_w(S)=0$. Since the function $\MultComb$ vanishes on all periodic domains, we have:
\begin{align*}
\Al(\x) - \Al(\a) &= (\Al(\x) - \Al(\q)) + (\Al(\q) - \Al(\a)) \\
&= \MultComb(\psi_1) + \frac{k+d}{2} + \MultComb(\psi_2) - \frac{dm}{2}
\end{align*}
which proves \eqref{eq: abgd-alex}.

Next, we consider the spin$^c$ evaluations. Up to thin domains, we have $S = x P_\gamma + y R$, where $x = - \frac{n_z(S)}{k}$ and $y = -\frac{\nu n_u(S)}{k}$. Note that $x$ and $y$ need not be integers. The decomposition $X_{\alpha\beta\gamma\delta} = X_{\alpha\beta\gamma} \cup_{Y_{\alpha\gamma} } X_{\alpha\gamma\delta}$ shows that classes $[P_\gamma]$ and $[R]$ can be represented by disjoint surfaces in $X_{\alpha\beta\gamma\delta}$, so $[P_\gamma] \cdot [R] = 0$ in the intersection form on $X_{\alpha\beta\gamma\delta}$. Using \eqref{eq: abgd-PD-self-int}, \eqref{eq: abgd-int-form}, and \eqref{eq: abg-c1-x}, we compute:
\begin{align*}
\gen{c_1(\spincs_w(\rho)), [P_\gamma]}
&= \gen{c_1(\spincs_w(\psi_1 * \psi_2) + \PD[S]), [P_\gamma]} \\
&= \gen{c_1(\spincs_w(\psi_1 * \psi_2)) + 2\PD[S], [P_\gamma]} \\
&= \gen{c_1(\spincs_w(\psi_1)), [P_\gamma]} + 2x [P_\gamma]^2 + 2y[R] \cdot [P_\gamma] \\
&= 2 \Al(\x) + 2d(n_w(\psi_1)-n_z(\psi_1)) - k  - 2dk\cdot \frac{ n_z(S)}{k} \\
&= 2 \Al(\x) + 2d(n_w(\psi_1)-n_z(\psi_1)) - k - 2d( n_z(\rho) - n_z(\psi_1) - n_z(\psi_2)) \\
&= 2 \Al(\x) + 2d(n_w(\psi_1) + n_z(\psi_2) -n_z(\rho)) -k  \\
&= 2 \Al(\x) + 2d(n_w(\psi_1) + n_w(\psi_2) -n_z(\rho)) -k  \\
&= 2 \Al(\x) + 2d(n_w(\rho) -n_z(\rho)) -k,
\end{align*}
which proves \eqref{eq: abgd-c1-Pg}. (Note the similarity between this equation and \eqref{eq: abg-c1-x}.) Formula \eqref{eq: abgd-c1-R} follows from \eqref{eq: agd-c1-a} in a similar manner. 
Finally, to prove \eqref{eq: abgd-c1-Q}, we have:
\begin{align*}
\gen{c_1(\spincs_w(\rho)), [Q]} &= \frac{\nu}{k} \gen{c_1(\spincs_w(\rho)), [R]} - \frac{m}{k}\gen{c_1(\spincs_w(\rho)), [P_\gamma]} \\
&= \frac{m}{k} \left( 2 \Al(\a)  + (k+dm) \left( 2n_u(\rho) - 2n_w(\rho) - 1 \right) +d \right) \\
& \qquad - \frac{m}{k} \left( 2 \Al(\x) + 2d(n_w(\rho) -n_z(\rho)) -k \right) \\
&= m \left(  2n_u(\rho) - 2n_{z'}(\rho) - 1 \right)
\end{align*}
as required.
\end{proof}

\subsection{The filtered large surgery formula} \label{ssec: large-surgery}

We now focus on the special Heegaard diagrams $(\Sigma, \bm\alpha, \bm\delta)$ associated to $Y_{\lambda+m\mu}(K)$, adding some additional details to the discussion from \cite[Section 4]{OSzKnot}, \cite[Section 4]{OSzRational}, \cite[Section 4.1]{HeddenPlamenevskayaRational}, and elsewhere. For ease of notation, let us write $W_m = W_{\lambda+m\mu}(K)$ and $W'_m = W'_{\lambda+m\mu}(K)$; these are the cobordisms induced by $(\Sigma, \bm\alpha, \bm\beta, \bm\delta)$ and $(\Sigma, \bm\alpha, \bm\delta, \bm\beta)$, respectively.

For any $\spincu \in \Spin^c(Y_{\lambda+m\mu}(K))$, the set of spin$^c$ structures on $W'_m$ which extend $\spincu$ form an orbit of the action of $\PD[C]$, where $C$ denotes the core of the $2$-handle attached to $Y$, extended across $Y \times I$. For each such spin$^c$ structure $\spincv$, we have
\[
\gen{c_1(\spincv + \text{PD}[C]), [\hat F]}  = \gen{c_1(\spincv), [\hat F]} + 2(k+dm),
\]
so the values of $\gen{c_1(\spincv), [\hat F]}$ taken over all such $\spincv$ form a single coset in $\Z/2(k+dm)$. Therefore, we can make the following definition.

\begin{definition} \label{def: xu}
For each $\spincu \in \Spin^c(Y_{\lambda+m\mu}(K))$, let $\spincx_\spincu$ denote the unique spin$^c$ structure on $W'_m$ extending $\spincu$ such that
\begin{equation} \label{eq: c1(xu)}
-2(k+dm) < \gen{c_1(\spincx_\spincu), [\hat F]} \le 0,
\end{equation}
and let $\spincs_\spincu = \spincx_\spincu|_Y$. Let $\spincy_\spincu = \spincx_\spincu + \PD[C]$, so that
\begin{equation} \label{eq: c1(yu)}
0 < \gen{c_1(\spincy_\spincu), [\hat F]} \le 2(k+dm).
\end{equation}
Define
\begin{equation} \label{eq: su-def}
s_\spincu = \frac{1}{2d}(\gen{c_1(\spincx_\spincu), [\hat F]} + k + dm),
\end{equation}
so that
\begin{equation} \label{eq: su-bound}
-\frac{k+dm}{2d} < s_\spincu \le \frac{k+dm}{2d}.
\end{equation}
Finally, define
\begin{equation} \label{eq: Delta-u-def}
\Delta_\spincu = \absgr(F^\infty_{W'_m, \spincx_u}) = -\frac{(2ds_\spincu-k-dm)^2}{4d(k+dm)} + \frac14.
\end{equation}
\end{definition}

Note that $\spincx_\spincu$ and $\spincy_\spincu$ achieve the two lowest values of $\abs{\gen{c_1(\spincv), [P_\delta]}}$ among all $\spincv \in \Spin^c(W'_m)$ restricting to $\spincu$. (These two values may be equal.) In the special case where $\gen{c_1(\spincx_\spincu), [P_\delta]}=0$ and $\gen{c_1(\spincy_\spincu), [P_\delta]} = 2(k+dm)$, there is a third spin$^c$ structure whose evaluation is $-2(k+dm)$, but this will not affect our arguments. Finally, Corollary \ref{cor: c1-alex-q-cong} implies that for any $\a \in \T_\alpha \cap \T_\delta$, we have the congruence
\begin{equation} \label{eq: su-A(a)-cong}
2 \Al(\a) + k + dm + d \equiv 2ds_\spincu \pmod {2(k+dm)}.
\end{equation}

Recall from Section \ref{ssec: CFKi} that $\CFKi(Y_{\lambda+m\mu}, K_{\lambda+m\mu}, \spincu)$
denotes the doubly-filtered knot Floer complex of $K_{\lambda+m\mu}$. By ignoring the second ($j$)
filtration, it is canonically identified with $\CFi(Y_{\lambda+m\mu}, \spincu)$.

We define a pair of filtrations $\II_\spincu, \JJ_\spincu$ on $\CFKi(Y,K,\spincs_\spincu)$ by the formula
\begin{align}
\label{eq: Iu-def} \II_\spincu([\x,i,j]) &= \max\{i, j-s_\spincu\} \\
\label{eq: Ju-def} \JJ_\spincu([\x,i,j]) &= \max\{i-1, j-s_\spincu\} + \frac{2ds_\spincu - d + k+ dm}{2(k+dm)} .
\end{align}
(Compare \eqref{eq: It-def-A} and \eqref{eq: Jt-def-A}.) It is clear from the definition that the
differential on $\CFKi(Y,K,\spincs_\spincu)$ is filtered with respect to both $\II_\spincu$ and
$\JJ_\spincu$. Observe that $\II_{\spincu}$ only depends on $\spincu$ via the number $s_{\spincu}$,
while $\JJ_{\spincu}$ also includes a shift that depends on $k+dm$.
\begin{theorem} \label{thm: large-surgery}
If $m$ is sufficiently large, then for every $\spincu \in \Spin^c(Y_{\lambda+m\mu})$, there is a grading-preserving, doubly-filtered quasi-isomorphism
\[
\Lambda_\spincu \co \CFKi(Y_{\lambda+m\mu}, K_{\lambda+m\mu}, \spincu) \to \CFKi(Y,K,\spincs_\spincu)[\Delta_\spincu],
\]
where the latter is equipped with the filtrations $\II_\spincu$ and $\JJ_\spincu$, making the diagrams
\[
\xymatrix{
\CFKi(Y_{\lambda+m\mu}(K), K_{\lambda+m\mu}, \spincu) \ar[r]^-{\Lambda_\spincu} \ar[d]^{=} & \CFKi(Y,K,\spincs_\spincu)[\Delta_{\spincu}] \ar[d]^{v^\infty} \\
\CFi(Y_{\lambda+m\mu}, \spincu) \ar[r]^-{F^\infty_{W'_m, \spincx_\spincu}} & \CFi(Y,\spincs_\spincu)
}
\]
and
\[
\xymatrix{
\CFKi(Y_{\lambda+m\mu}(K), K_{\lambda+m\mu}, \spincu) \ar[r]^-{\Lambda_\spincu} \ar[d]^{=} & \CFKi(Y,K,\spincs_\spincu) [\Delta_{\spincu}] \ar[d]^{h^\infty_{\spincs, s_\spincu}} \\
\CFi(Y_{\lambda+m\mu}, \spincu) \ar[r]^-{F^\infty_{W'_m, \spincy_\spincu}} & \CFi(Y,\spincs_\spincu+\PD[K])
}
\]
commute up to chain homotopy.
\end{theorem}

The bulk of Theorem \ref{thm: large-surgery} was proven by Ozsv\'ath and Szab\'o, apart from the presence of the second filtration and a slight technical issue regarding the definition of $\spincx_\spincu$ and $\spincy_\spincu$ (see Remark \ref{rmk: OSz-xu} below). We will follow through their proof while keeping track of the second filtration, clarifying a few details along the way. (An analogous result for large negative surgeries on knots in $S^3$, including the second filtration, was shown by Kim, Livingston, and the first author \cite[Theorem 4.2]{HeddenKimLivingston}.)

Consider any well-adapted Heegaard diagram $(\Sigma, \bm\alpha, \bm\beta, \bm\gamma, \bm\delta^{m,b}, w, z, z')$, as described above. The reader should refer to Figure \ref{fig: twist}.

We begin by discussing the generators in $\Gens(\bm\alpha, \bm\delta^{m,b})$ more carefully. Let $q$ be the unique point of $\alpha_g \cap \beta_g$. Label the $m$ points of $\alpha_g \cap \delta_g^{m,b}$ in the winding region $p_{b-m}, \dots, p_{b-1}$ according to the orientation of $\alpha_g$. Thus, $p_l$ is on the $z$ side of $\beta_g$ if $l < 0$ and on the $w$ side if $l \ge 0$, and $q$ lies between $p_{-1}$ and $p_0$. For any $\x \in \Gens(\bm\alpha, \bm\beta)$ and $l \in \{b-m, \dots, b-1\}$, let $\x_l^{m,b} \in \Gens(\bm\alpha, \bm\delta^{m,b})$ be the point obtained by replacing $q$ with $p_l$ and taking ``nearest points'' elsewhere; these generators are called \emph{interior generators}. (We sometimes omit the superscripts if they are understood from context.) There is a \emph{small triangle} $\psi_{\x,l}^{m,b} \in \pi_2(\x_l^{m,b}, \Theta_{\delta\beta}, \x)$ with
\begin{equation} \label{eq: small-triangle}
( n_w(\psi_{\x,l}^{m,b}), n_z(\psi_{\x,l}^{m,b}), n_{z'}(\psi_{\x,l}^{m,b})) =
\begin{cases}
(0,-l, 0) & l < 0 \\
(l,0, l+1) & l \ge 0.
\end{cases}
\end{equation}
The remaining elements of $\Gens(\bm\alpha, \bm\delta^{m,b})$, called \emph{exterior generators}, are naturally in $1$-to-$1$ correspondence with elements of $\Gens(\bm\alpha, \bm\gamma)$. For each $\q \in \Gens(\bm\alpha, \bm\gamma)$, let $\q^{m,b} \in \Gens(\bm\alpha, \bm\delta^{m,b})$ denote the nearest exterior generator.

The spin$^c$ structures represented by the two types of generators in $\Gens(\bm\alpha, \bm\delta^{m,b})$ are governed by the following lemma:
\begin{lemma} \label{lemma: ad-spinc}
For any $m >0$ and any $0 \le b \le m$, the following hold:
\begin{enumerate}
\item \label{it: ad-spinc-x}
For each $\x \in \Gens(\bm\alpha, \bm\beta)$ and each $l \in \{b-m, \dots, b-1\}$, $\spincs_w(\x_l^{m,b}) \in \Spin^c(Y_{\lambda+m\mu}(K))$ depends only on $\x$ and $l$ and not on $b$. Moreover, the Maslov grading of $\x_l^{m,b}$ is given by
\begin{equation}
\label{eq: xl-gr}
\absgr(\x_l^{m,b})
= \absgr(\x) + \frac{ ( 2 \Al(\x) - 2dl - k-dm )^2}{4d(k+dm)} - \frac14 - \max(0,2l). 
\end{equation}

\item \label{it: ad-spinc-q}
For each $\q \in \Gens(\bm\alpha, \bm\gamma)$, we have $\spincs_w(\q^{m,b+1}) = \spincs_w(\q^{m,b}) - \PD[K_{\lambda+m\mu}]$.
\end{enumerate}
\end{lemma}

\begin{proof}
For statement \eqref{it: ad-spinc-x}, equations \eqref{eq: adb-c1-x} and \eqref{eq: small-triangle} give:
\begin{align}
\label{eq: c1-small-triangle}
\gen{c_1(\spincs_w(\psi_{\x,l}^{m,b})), [P_\delta]} &= 2 \Al(\x) + 2d(n_z(\psi_{\x,l}^{m,b}) - n_w(\psi_{\x,l}^{m,b})) - (k+dm) \\
\nonumber &= 2\Al(\x) - 2dl - (k+dm).
\end{align}
Thus, $\spincs_w(\psi_{\x,l}^{m,b})$ is completely determined by $\x$ and $l$ and is independent of $b$. The same is therefore true of $\spincs_w(\x_l^{m,b})$. Equation \eqref{eq: xl-gr} then follows because
\[
\absgr(\x) - \absgr(\x_l^{m,b}) = \frac{c_1(\spincs_w(\psi_{\x,l}^{m,b}))^2 +1}{4} + 2 n_w(\psi_{x,l}^{m,b}).
\]

For statement \eqref{it: ad-spinc-q}, the curves $\delta_g^{m,b}$ are isotopic for all choices of $b$ if we allow crossing the basepoint; the only difference is the position of the basepoint $w$. To be precise, there is a diffeomorphism
\[
(\Sigma, \bm\alpha, \bm\delta^{m,b+1}, w) \cong (\Sigma, \bm\alpha, \bm\delta^{m,b}, z')
\]
taking $\q^{m,b+1} \to \q^{m,b}$. Thus, by \cite[Lemma 2.19]{OSz3Manifold}, we have
\[
\spincs_w(\q^{m,b+1}) = \spincs_{z'}(\q^{m,b}) = \spincs_w(\q^{m,b}) - \PD[K_{\lambda+m\mu}]
\]
as required.
\end{proof}

\begin{definition}
Following \cite[Definition 4.3]{OSzRational}, we say that $\spincu \in \Spin^c(Y_{\lambda+m\mu})$ is \emph{supported in the winding region of $(\Sigma, \bm\alpha, \bm\delta^{m,b})$} if every $\a \in \T_\alpha \cap \T_\delta$ with $\spincs_w(\a)=\spincu$ is of the form $\x_l^{m,b}$ for some $l$, and for every pair of such generators $\a, \b$ and any $\phi \in \pi_2(\a,\b)$, $\partial \DD(\phi) \cap \delta_g$ is contained in the winding region. (By Remark \ref{rmk: ab-PD}, the multiplicity of $\delta_g$ in the boundary of any $(\alpha,\delta)$ periodic domain is $0$, so this condition holds for a single $\phi \in \pi_2(\a,\b)$ iff it holds for every $\phi$.) We say that $\spincu$ is \emph{strongly supported in the winding region} if, additionally, for each generator $\x_l$ representing $\spincu$, we have
$c_1(\spincs_w(\psi^{m,b}_{\x,l})) = \spincx_\spincu$ (equivalently, $c_1(\spincs_z(\psi^{m,b}_{\x,l})) = \spincy_\spincu$).
\end{definition}

The following lemma gives a more explicit characterization of what it means for a spin$^c$ structure to be (strongly) supported in the winding region.

\begin{lemma} \label{lemma: same-spinc}
Consider the diagram $(\Sigma, \bm\alpha, \bm\delta^{m,b}, w)$.
\begin{enumerate}
\item \label{item: same-spinc}
For any $\x, \y \in \T_\alpha \cap \T_\beta$ and $i, j \in \{b-m, \dots, b-1\}$, the generators $\x_i^{m,b}$ and $\y_j^{m,b}$ represent the same spin$^c$ structure on $Y_{\lambda + m \mu}$ if and only if for some integer $r$,
\begin{equation} \label{eq: same-spinc}
\spincs_w(\x) - \spincs_w(\y) = -r[K] \quad \text{and} \quad i-j = \AlNorm_{w,z}(\x) - \AlNorm_{w,z}(\y) - \frac{r(k+dm)}{d}.
\end{equation}

\item \label{item: supported}
A spin$^c$ structure $\spincu$ is supported in the winding region of $(\Sigma, \bm\alpha,
\bm\delta^{m,b}, w)$ iff for some spin$^c$ structure $\spincs$ on $Y$ and some constant $L$, we
have
\begin{equation} \label{eq: supported}
\Gens(\bm\alpha, \bm\delta^{m,b}, w, \spincu) = \{\x_{\AlNorm(\x)+L}^{m,b}  \mid \x \in \Gens(\bm\alpha,
\bm\beta, w, \spincs)\}.
\end{equation}

\item \label{item: strongly-supported}
A spin$^c$ structure $\spincu$ is strongly supported in the winding region of $(\Sigma, \bm\alpha,
\bm\delta^{m,b}, w)$ iff
\begin{equation} \label{eq: strongly-supported}
\Gens(\bm\alpha, \bm\delta^{m,b}, w, \spincu) = \{\x_{\AlNorm(\x)-s_\spincu}^{m,b}  \mid \x \in
\Gens(\bm\alpha, \bm\beta, w, \spincs_\spincu)\}.
\end{equation}
\end{enumerate}
\end{lemma}

\begin{proof}
For the ``only if'' direction of statement \ref{item: same-spinc}, suppose $\spincs_w(\x_i^{m,b}) = \spincs_w(\y_j^{m,b})$, and choose any class $\phi \in \pi_2(\x_i^{m,b},\y_j^{m,b})$. Consider the concatenation $\rho = \overline{ \psi_{\x,i} } * \phi * \psi_{\y,j}$. The domain of $\rho$ has boundary equal to $\epsilon(\x,\y)$ (viewed as a $1$-chain in $\bm\alpha \cup \bm\beta$) together with $r$ copies of $\delta_g$, where $r = n_{z'}(\rho) - n_w(\rho) \in \Z$.
This shows that $\epsilon(\x,\y)$ is homologous to $-r[K]$ in $H_1(Y)$. Moreover, $B = d \DD(\rho) - r P_\delta$ (plus thin domains) is a domain whose boundary is $d \epsilon(\x,\y)$. Therefore,
\begin{align*}
\AlNorm(\x) - \AlNorm(\y) &= \frac{1}{d}( n_z(B) - n_w(B) ) \\
&= n_z(\rho) - n_w(\rho) - \frac{r}{d} \left(n_z(P_\delta) - n_w(P_\delta) \right) \\
&= -n_z(\psi_{\x,i}) + n_w(\psi_{\x,i}) + n_z(\psi_{\y,j}) - n_w(\psi_{\y,j}) + \frac{r(k+dm)}{d} \\
&= i-j + \frac{r(k+dm)}{d}
\end{align*}
as required. The ``if'' direction follows similarly, by applying the same construction in reverse.

For the ``only if'' direction of statement \ref{item: supported}, suppose $\spincu$ is supported
in the winding region. If $\x_i^{m,b}$ and $\y_j^{m,b}$ are generators representing $\spincu$, let $\phi \in
\pi_2(\x_i^{m,b}, \y_j^{m,b})$ be a class whose $\delta_g$ boundary segment is contained in the winding region.
Applying the above construction, we find that $r=0$. Thus, $\spincs_w(\x) = \spincs_w(\y)$ and
$\AlNorm(\x) - \AlNorm(\y) = i-j$, so both $\x_i^{m,b}$ and $\y_j^{m,b}$ are of the stated form. Moreover, given
any generators $\x, \y$ with $\spincs_w(\x) = \spincs_w(\y)$ and $\spincs_w(\x_i^{m,b}) = \spincu$ for
some $i$, we can find some $j$ for which $\spincs_w(\y_j^{m,b}) = \spincu$ as well, and therefore $i-j =
\AlNorm(\x) - \AlNorm(\y)$.  The converse follows similarly.

Finally, for statement \ref{item: strongly-supported}, we apply \eqref{eq: c1-small-triangle} and \eqref{eq: su-def}.
\end{proof}

\begin{lemma} \label{lemma: supported}
Let $(\Sigma, \bm\alpha, \bm\beta, \bm\gamma, w,z)$ be an adapted Heegaard diagram for
$\lambda$-surgery on $K$. Then there exists an $M$ such that for all $m \ge M$ and each spin$^c$
structure $\spincu \in \Spin^c(Y_{\lambda+m\mu}(K))$, $\spincu$ is strongly supported in the winding
region of $(\Sigma, \bm\alpha, \bm\delta^{m,b}, w)$ for some $b$. (Note, however, that $b$ depends on the choice of $\spincu$.)
\end{lemma}

\begin{remark} \label{rmk: OSz-xu}
In \cite[Lemmas 4.5 and 4.6]{OSzRational}, it is shown that each spin$^c$ structure $\spincu$ can
be supported in the winding region, in such a way that for each small triangle $\psi$, we have
\[
-C -2(k+dm) \le \gen{c_1(\spincs_w(\psi)), [P_\delta]} \le C
\]
for some constant $C \ge 0$ independent of $m$. However, this is not quite as strong as saying that $\spincu$
is strongly supported in the winding region, since these bounds do not uniquely determine $\spincs_w(\psi)$.
\end{remark}

\begin{proof}[Proof of Lemma \ref{lemma: supported}]
We begin by establishing a sufficient condition for $\Gens(\bm\alpha,
\bm\delta^{m,b}, w, \spincu)$ to contain all the generators called for by \eqref{eq:
strongly-supported}, i.e.,
\begin{equation} \label{eq: strongly-supported-contains}
\Gens(\bm\alpha, \bm\delta^{m,b}, w, \spincu) \supset \{\x^{m,b}_{\AlNorm(\x)-s_\spincu}  \mid \x \in
\Gens(\bm\alpha, \bm\beta, w, \spincs_\spincu)\}.
\end{equation}
By construction, we require that $0 \le b \le m$. Let $C$ be a constant such that for all $\x \in
\T_\alpha \cap \T_\beta$, $\abs{\AlNorm(\x)} \le C$. For any $\x \in \Gens(\bm\alpha, \bm\beta, w,
\spincs_\spincu)$, $\x^{m,b}_{\AlNorm(\x)-s_\spincu}$ is one of the elements of $\Gens(\bm\alpha,
\bm\delta^{m,b}, w, \spincu)$ iff
\[
 b-m \le \AlNorm(\x)-s_\spincu \le b-1
\]
or equivalently
\[
 \AlNorm(\x)-s_\spincu+1 \le b \le \AlNorm(\x)-s_\spincu+m.
\]
Thus, a sufficient condition for \eqref{eq: strongly-supported-contains} is that
\begin{equation} \label{eq: strongly-supported-b-bounds}
\max\{0, \, -C-s_\spincu+1\} \le b \le \min\{m, \, C-s_\spincu+m\}.
\end{equation}
Let $B(\spincu)$ denote the difference between the upper and lower bounds in \eqref{eq:
strongly-supported-b-bounds}; then
\[
B(\spincu) \ge \min\{m+C+s_\spincu-1, \, m+C-s_\spincu, \, m+2C-1,m\}.
\]
By \eqref{eq: su-bound}, we have
\[
m \pm s_\spincu \ge \frac{m}{2} - \frac{k}{2d}.
\]
Therefore, we have a lower bound $B(\spincu) \ge \frac{m}{2} + C'$, where $C'$ is a constant
independent of $m$ and $\spincu$. In other words, for each $\spincu$, there are at least
$\frac{m}{2}+C'$ consecutive values of $b$ for which \eqref{eq: strongly-supported-contains} holds.
(The ranges may differ for different choices of $\spincu$, of course.)

As noted by Ozsv\'ath and Szab\'o in the proof of \cite[Lemma 4.5]{OSzRational}, the number of exterior generators in $\Gens(\bm\alpha, \bm\delta^{m,b})$, and hence the number of spin$^c$ structures on $Y_{\lambda + m\mu}$ represented by the exterior generators for any particular $b$, is bounded by a constant independent of $m$. The set of such spin$^c$ structures varies with $b$ according to Lemma \ref{lemma: ad-spinc}\eqref{it: ad-spinc-q}. In particular, if $m$ is sufficiently large, then for any $\spincu$ and any $\frac{m}{2}$ consecutive values of $b \in \{0, \dots, m\}$, we find some $b$ within the specified range such that none of the exterior generators in $\Gens(\bm\alpha, \bm\delta^{m,b})$ represent $\spincu$. In particular, if we use the range of $b$ values specified by \eqref{eq: strongly-supported-b-bounds}, we see that $\spincu$ is strongly supported for some value of $b$.
\end{proof}

Next, we discuss the Alexander grading on $\CF(\bm\alpha, \bm\delta, w)$ induced by the knot $K_{\lambda+m\mu}$. (Henceforth, we omit the $m,b$ superscripts for conciseness.) For any interior generator $\x_l$, \eqref{eq: adb-alex} gives:
\begin{align*}
\Al(\x_l) - \Al(\x) &= d n_z(\psi_{\x,l}) +  (k+dm) n_{z'}(\psi_{\x,l}) - (k+dm + d) n_w(\psi_{\x,l}) -\frac{k+dm +d}{2} \\
&=
\begin{cases}
-dl -\frac{k+dm+d}{2} & l < 0 \\
-dl + (k+dm) -\frac{k+dm+d}{2} & l \ge 0
\end{cases} \\
&=
\begin{cases}
-dl - \frac{d}{2} - \frac{k+dm}{2} & l < 0 \\
-dl - \frac{d}{2} + \frac{k+dm}{2} & l \ge 0.
\end{cases}
\end{align*}
If we assume that $\x_l$ represents a spin$^c$ structure $\spincu$ which is strongly supported in
the winding region, then $l = \AlNorm(\x) - s_\spincu$, and therefore we have:
\begin{align}
\label{eq: xl-alex}
\Al(\x_l) &=
\begin{cases}
d s_\spincu - \frac{d}{2} - \frac{k+dm}{2} & l < 0 \\
d s_\spincu - \frac{d}{2} + \frac{k+dm}{2} & l \ge 0
\end{cases} \\
\label{eq: xl-alex-norm}
\AlNorm(\x_l) &=
\begin{cases}
\frac{d(2s_\spincu -1)}{2(k+dm)} - \frac{1}{2} & l < 0 \\
\frac{d(2s_\spincu -1)}{2(k+dm)} + \frac{1}{2} & l \ge 0.
\end{cases}
\end{align}
Thus, the Alexander grading for $\spincu$ takes exactly two values, which differ by $1$. (Cf.~\cite[Theorem 4.2]{HeddenPlamenevskayaRational}).

\begin{proof}[Proof of Theorem \ref{thm: large-surgery}]
Using Lemma \ref{lemma: supported}, we may assume that $\spincu$ is strongly supported in the winding region of $(\Sigma, \bm\alpha, \bm\delta^{m,b}, w, z')$ for some $b$. Define
\begin{equation} \label{eq: Lambda-u}
\Lambda^\infty_\spincu \co \CFi(\Sigma, \bm\alpha, \bm\delta, w, \spincu) \to \CFKi(\Sigma, \bm\alpha, \bm\beta, w, z, \spincs_\spincu)[\Delta_\spincu]
\end{equation}
by
\begin{equation} \label{eq: Lambda-u-def}
\Lambda^\infty_\spincu([\a,i]) =
\sum_{\substack{\x \in \T_\alpha \cap \T_\beta \\ \spincs_w(\x) = \spincs_\spincu}}
\sum_{\substack{\psi \in \pi_2(\a, \Theta_{\delta\beta}, \x) \\ \mu(\psi) = 0 \\ \spincs_w(\psi) = \spincx_\spincu}}
\#\MM(\psi) [\x, i-n_w(\psi), i-n_z(\psi) + s_\spincu].
\end{equation}
Standard arguments show that $\Lambda^\infty_\spincu$ is a chain map, and the shift in the Maslov grading by $\Delta_\spincu$ makes $\Lambda^\infty_\spincu$ grading-preserving. Using the standard identification of $\CFi(\Sigma, \bm\alpha, \bm\delta, w, \spincu)$ with $\CFKi(\Sigma, \bm\alpha, \bm\delta, w, z', \spincu)$, we may also think of $\Lambda_\spincu$ as being defined on the latter.

Because $\spincu$ is strongly supported in the winding region, every element of $\Gens(\bm\alpha, \bm\delta, w, \spincu)$ is of the form $\x_l$, where $\x \in \Gens(\bm\alpha,\bm\beta, w, \spincs_\spincu)$ and $l = \AlNorm(\x) - s_\spincu$. Denote the contributions to $\Lambda^\infty_\spincu$ coming from the small triangles $\psi_{\x,l}$ by $\tilde \Lambda^\infty_\spincu$. By \eqref{eq: small-triangle}, these terms are given by
\begin{equation} \label{eq: Lambda-u-tilde}
\tilde \Lambda^\infty_\spincu([\x_l,i]) =
\begin{cases}
[\x, i, i + s_\spincu + l] & l < 0 \\
[\x, i-l, i + s_\spincu] & l \ge 0.
\end{cases}
\end{equation}
Note that $\tilde \Lambda^\infty_\spincu$ is a $\F[U]$--module isomorphism but not necessarily a chain map. Indeed, it is easy to check that for $[\x,i,j]$ with $j=i+\AlNorm(\x)$,  the inverse of $\tilde \Lambda^\infty_\spincu$ is given by:
\begin{align} \label{eq: Lambda-u-tilde-inv}
(\tilde \Lambda^\infty_\spincu)^{-1}([\x,i,j]) &=
\begin{cases}
[\x_l,i] & \AlNorm(\x) \le s_\spincu \\
[\x_l,i + \AlNorm(\x) - s_\spincu] & \AlNorm(\x) \ge s_\spincu,
\end{cases} \\
&= \nonumber [\x_l, \max(i,j-s_\spincu)]
\end{align}
where $l = \AlNorm(\x) - s_\spincu$ as above.

The proof that $\Lambda^\infty_\spincu$ is a chain isomorphism uses the fact that
\[
\Lambda^\infty_\spincu = \tilde \Lambda^\infty_\spincu + \text{higher order terms}
\]
with respect to an energy filtration, as proved by Ozsv\'ath and Szab\'o. We just need to check that every triangle contributing to $\Lambda^\infty_\spincu$ decreases or preserves both filtrations, and that every triangle contributing to $\tilde \Lambda^\infty_\spincu$ preserves both filtrations. The proof for $\II_\spincu$ is obvious from the definition, so we focus on $\JJ_\spincu$.

Thus, consider any term in \eqref{eq: Lambda-u-def}, corresponding to a triangle $\psi \in \pi_2(\a, \Theta_{\delta\beta},\x)$ admitting holomorphic representatives. If $\a = \y_l$, where $l = \AlNorm(\y) - s_\spincu$, observe that
\begin{align*}
\AlNorm_{w,z'}([\y_l,i]) &=
\begin{cases}
\frac{d(2 s_\spincu - 1)}{2(k+dm)} + \frac{1}{2} + i-1 & l < 0 \\
\frac{d(2 s_\spincu - 1)}{2(k+dm)}  + \frac{1}{2} + i & l \ge 0.
\end{cases} \\
\JJ_\spincu([\x, i-n_w(\psi), i-n_z(\psi) + s_\spincu]) &= \frac{d(2 s_\spincu - 1)}{2(k+dm)}   + \frac12 + i + \max(-n_w(\psi)-1, -n_z(\psi))
\end{align*}
and therefore
\begin{multline} \label{eq: Lambda-u-filt}
\AlNorm_{w,z'}([\y_l,i]) -\JJ_\spincu([\x, i-n_w(\psi), i-n_z(\psi) + s_\spincu])  = \\
\begin{cases}
\min(n_w(\psi), n_z(\psi)-1)  & l<0 \\
\min(n_w(\psi)+1, n_z(\psi)) & l\ge 0.
\end{cases}
\end{multline}
In the case where $n_z(\psi)=0$ and $l<0$, equations \eqref{eq: adb-c1-a}, \eqref{eq: xl-alex} (with $\y$ in place of $\x$), and \eqref{eq: su-def} imply that
\[
2d s_\spincu - k-dm = \gen{c_1(\spincx_\spincu), [P_\delta]} = 2d s_\spincu + (k+dm)(2n_w(\psi) - 2n_{z'}(\psi) -1),
\]
and hence $n_w(\psi) = n_{z'}(\psi)$. However, the multiplicity of $\psi$ in the fourth region abutting $\Theta_{\delta\beta}$ in the winding region must then be $-1$, a contradiction. Thus, the left-hand side of \eqref{eq: Lambda-u-filt} is always nonnegative, as required.

Finally, the small triangles $\psi_{\x,l}$ (which contribute to $\tilde\Lambda^\infty_\spincu$) all have either $n_w(\psi)=0$ or $n_z(\psi)=0$, and hence the difference \eqref{eq: Lambda-u-filt} vanishes for those terms.
\end{proof}

\begin{example} \label{ex: lens}
As a sanity check, consider the example where $Y = S^3$ and $K$ is the unknot, so that $Y_m(K)$ is the lens space $L(m,1)$. Denote the induced knot by $O_m$. In this case, we may assume that Figure \ref{fig: twist} (with its left and right edges glued together) is the entire Heegaard diagram, and we may take $k=0$, $d=1$, and $b=m$. The unique generator $\x \in \T_\alpha \cap \T_\beta$ has $\AlNorm(\x)=0$. For $l=0, \dots, m-1$, let $\spincu_l$ be the generator represented by $\x_l$. Then $s_{\spincu_l} = -l$, and $\spincu_l$ is characterized by the following property: if $B_n$ denotes the Euler number $m$ disk bundle over $S^2$, and $\spincv$ is any extension of $\spincu_l$, then
\[
\gen{c_1(\spincv), [S^2]} + m \equiv -2l \pmod {2m}.
\]
Equation \eqref{eq: xl-alex-norm} shows that $\HFKa(L(m,1), O_m, \spincu_l) \cong \F$, supported in Alexander grading $\frac{m-2l-1}{2m}$. It is easy to check that the symmetry \eqref{eq: HFK-sym-alex} holds. Moreover, by equation \eqref{eq: xl-gr}, the Maslov grading of $\x_l$ is
\begin{equation} \label{eq: lens-gr}
\absgr(\x_l)
=  \frac{ (2l-m)^2 -m }{4m},
\end{equation}
which agrees with the computation of $d$ invariants for lens spaces in \cite[Proposition 4.8]{OSzAbsolute}.
\end{example}

\subsection{Maslov gradings on the large surgery}

We will now prove some bounds on the Maslov gradings on the complexes $\CF^t(\Sigma, \bm\alpha, \bm\delta^{m,b}, w)$ as a function of $m$, provided that $b$ is within a bounded distance of $\frac{m}{2}$. These bounds will be used in Section \ref{sec: exact-sequence} to control the spin$^c$ structures in the surgery exact triangle. The technical statement is as follows:

\begin{proposition} \label{prop: maslov-bound}
Let $(\Sigma, \bm\alpha, \bm\beta, \bm\gamma, w, z, z')$ be a Heegaard triple diagram adapted to $(Y,K,\lambda)$ as above. For any integer $e \ge 0$, there is a constant $C \ge 0$ such that for all $m$ sufficiently large, and any $b$ with $\frac{m-e}{2} \le b \le \frac{m+e}{2}$,\footnote{In the terminology of \cite[Definition 4.4]{OSzRational}, this condition is equivalent to saying that the meridian $\beta_g$ is $e$-centered in $(\Sigma, \bm\alpha, \bm\delta^{m,b}, w)$.} every generator $\a \in \T_\alpha \cap \T_{\delta^{m,b}}$ satisfies
\begin{equation} \label{eq: maslov-bound}
-C \le \absgr(\a) \le \frac{m}{4} + C.
\end{equation}
\end{proposition}

\begin{remark}
In \cite[Corollary 4.7]{OSzRational}, which is stated when $Y$ is an integer homology sphere (hence $d=1$), Ozsv\'ath and Szab\'o proved the upper bound from \eqref{eq: maslov-bound} but gave an incorrect lower bound, asserting that the gradings on $\CF^t(\Sigma, \bm\alpha, \bm\delta, w)$ are always within a bounded distance of $m/4$. Equation \eqref{eq: lens-gr} above shows that only a constant lower bound is possible. Additionally, the statement of \cite[Corollary 4.7]{OSzRational} requires looking at different Heegaard diagrams for different spin$^c$ structures on $Y_{\lambda+m\mu}(K)$ (namely, arranging for the chosen spin$^c$ structure to be supported in the winding region), whereas Proposition \ref{prop: maslov-bound} applies simultaneously to all of the generators in the same Heegaard diagram.
\end{remark}

To prove Proposition \ref{prop: maslov-bound}, there are two types of generators in $\T_\alpha \cap \T_\delta$ to consider: interior generators $\x_l^{m,b}$ (for $\x \in \T_\alpha \cap \T_\beta$) and exterior generators $\q^{m,b}$ (for $\q \in \T_\alpha \cap \T_\gamma$). These two types of generators will require separate arguments.

\begin{lemma} \label{lemma: gr(xl)}
Fix $e \ge 0$. There is a constant $C_1 \ge 0$ such that for any $m$ sufficiently large, if we take $\frac{m-e}{2} \le b \le \frac{m+e}{2}$, then for any $\x \in \T_\alpha \cap \T_\beta$ and any $l = b-m, \dots, b-1$, the grading of the generator $\x_l$ satisfies
\[
-C_1 \le \absgr(\x_l^{m,b}) \le \frac{m}{4} + C_1.
\]
\end{lemma}

\begin{proof}
By our hypothesis on $b$, we may assume that $-\frac{m+e}{2} \le l \le \frac{m+e}{2}$. Using Lemma \ref{lemma: ad-spinc}, for each $\x \in \T_\alpha \cap \T_\beta$, define
\begin{equation} \label{eq: g(l)}
g_\x(l) := \absgr(\x_l^{m,b}) - \absgr(\x) + \frac14 = \frac{ ( 2d \AlNorm(\x) - 2dl - k-dm )^2}{4d(k+dm)} - \max(0,2l),
\end{equation}
which we may view as a function of all real numbers $l$.

Assuming $m$ is sufficiently large, one can easily check that the local minima of $g_\x$ occur at
$l = \AlNorm(\x) \pm \frac{k+dm}{2d}$, with values of $0$ and $- 2\AlNorm(\x)$. The local maxima of $g_\x$ on the interval $[-\frac{m+e}{2}, \frac{m+e}{2}]$ occur for $l \in \{-\frac{m+e}{2}, 0, \frac{m+e}{2}\}$, with values given by:
\begin{align*}
g_\x(-\tfrac{m+e}{2}) 
&= \frac{ ( 2d \AlNorm(\x) - k + de  )^2}{4d(k+dm)}   \\
g_\x(0) &= \frac{ ( 2d \AlNorm(\x) - k-dm )^2}{4d(k+dm)}   \\
&= \frac{m}{4}  - \AlNorm(\x)  + \frac{k}{4d} + \frac{ d \AlNorm (\x)^2}{k+dm} \\
g_\x(\tfrac{m+e}{2}) 
 &= \frac{ ( 2d \AlNorm(\x) +k - de - 2(k+dm))^2}{4d(k+dm)}  - m-e \\
 &= \frac{ ( 2d \AlNorm(\x) +k - de )^2 + 4(k+dm) (-2\Al(\x)-k + de + k+dm ) }{4d(k+dm)}  - m-e \\
 &= \frac{ ( 2d \AlNorm(\x) +k - de )^2}{4d(k+dm)}  -2\AlNorm(\x)
\end{align*}
Both $g_\x(-\frac{m+e}{2})$ and $g_\x(\frac{m+e}{2})$ are bounded above independent of $m$, whereas $g(0) = \absgr(\x_0)$ is not, so $g$ attains its global maximum on the interval $[-\frac{m+e}{2}, \frac{m+e}{2}]$ at $l=0$. Thus, there is a constant $C_\x$ such that for any $b$ satisfying $\frac{m-e}{2} \le b \le \frac{m+e}{2}$ and any $l = b-m, \dots, b-1$, we have
\[
-C_\x \le \absgr(\x_l) \le \frac{m}{4} + C_\x.
\]
Maximizing $C_\x$ over all $\x \in \T_\alpha \cap \T_\beta$ gives the desired result.
\end{proof}

%
%

%
%

%

\begin{lemma} \label{lemma: gr(qmb)}
Fix $e \ge 0$. Then for all $m$ for which $k+dm >0$, and  all $b$ with $\frac{m-e}{2} \le b \le \frac{m+e}{2}$, the absolute gradings of all the generators $\q_{m,b}$ are bounded by a constant independent of $m$.
\end{lemma}

To prove this lemma, we will work inductively on $m$ and $b$. Note that the Heegaard quadruple diagrams $(\Sigma, \bm\alpha, \bm\beta, \bm\delta^{m,b}, \bm\delta^{m+1,b})$ and $(\Sigma, \bm\alpha, \bm\beta, \bm\delta^{m,b}, \bm\delta^{m+1,b+1})$ are well-adapted, where now we are treating $\lambda+m\mu$ as the ``original'' longitude and $\lambda+(m+1)\mu$ as the new one. The cases correspond to $m=1,b=0$ and $m=1,b=1$ respectively. For each of these quadruple diagrams, an analogue of Proposition \ref{prop: alex-triangle} holds, where we plug in $1$ for $m$, $k+dm$ for $k$, and $1$ for $\nu$ in each of the formulas. Let $R_{m,b}^-$ and $R_{m,b}^+$ respectively denote the triply periodic domains that are analogous to $R$ in the two cases; note that $[R_{m,b}^{\pm}]^2  = -(k+dm)(k+dm+d)$. We may also refer to the original $\bm\gamma$ as $\bm\delta^{0,0}$.

For any $\q \in \T_\alpha \cap \T_\gamma$, there are triangles
\[
\psi_{\q,m,b}^- \in \pi_2(\q_{m,b}, \Theta, \q_{m+1,b}) \quad\text{and} \quad \psi_{\q,m,b}^+ \in \pi_2(\q_{m,b}, \Theta, \q_{m+1,b+1})
\]
with Maslov index $0$, which satisfy
\[
n_w(\psi_{\q,m,b}^\pm) = n_z(\psi_{\q,m,b}^\pm) = n_{z'}(\psi_{\q,m,b}^\pm) = 0, \qquad n_u(\psi_{\q,m,b}^+) = 1, \qquad n_u(\psi_{\q,m,b}^-) = 0.
\]

\begin{lemma} \label{lemma: A(qmb)-induct}
Let $m_0$ be an integer for which $k+dm_0 > 0$. For any $\q \in \T_\alpha \cap \T_\gamma$, any $m \ge m_0$, and any $0 \le b_0 \le m_0$ and $0 \le b \le m$, we have
\begin{equation} \label{eq: A(qmb)-induct}
\Al(\q_{m,b}) = \Al(\q_{m_0,b_0}) + \frac{d((m-2b) - (m_0-2b_0))}{2}.
\end{equation}
\end{lemma}

\begin{proof}
By applying the analogue of \eqref{eq: agd-alex} to the triangles $\psi_{\q,m,b}^-$ and $\psi_{\q,m,b}^+$, we have:
\begin{gather*}
\Al(\q_{m,b}) - \Al(\q_{m+1,b}) = -\frac{d}{2} \\
\Al(\q_{m,b}) - \Al(\q_{m+1,b+1}) = \frac{d}{2}
\end{gather*}
The lemma follows by induction.
\end{proof}

\begin{proof}[Proof of Lemma \ref{lemma: gr(qmb)}]
Choose some $m_0$ for which $k+dm_0>0$, and any $b_0$ with $\frac{m_0 -e}{2} \le b_0 \le \frac{m_0+e}{2}$. We will inductively obtain a bound (as in the statement of the lemma) which applies for all of the pairs $(m_0+2i, b_0+i)$. We can then repeat the argument for each of the finitely many choices for $b_0$, obtaining a different bound each time, and repeat it again with $m_0$ replaced by $m_0+1$ (and all possible corresponding values of $b_0$). The largest of the resulting bounds will then apply to all pairs $(m,b)$ with $\frac{m-e}{2} \le b \le \frac{m+e}{2}$.

The induction proceeds as follows. Suppose $m = m_0+2i$ and $b = b_0+i$. By applying \eqref{eq: agd-c1-q} to the triangle $\psi_{\q,m,b}^+$, we obtain:
\begin{align*}
\gen{c_1(\psi_{\q,m,b}^+), [R_{m,b}^+]} &= 2\Al(\q_{m,b}) + k+dm + d \\
&= 2\Al(\q_{m_0,b_0}) + d((m-2b) - (m_0-2b_0)) + k+dm+d \\
&= 2\Al(\q_{m_0,b_0}) + k+dm + d
\end{align*}
\begin{align*}
\absgr(\q_{m+1,b+1}) - \absgr(\q_{m,b})
&= \frac{c_1(\spincs_w(\psi_{\q,m,b}^+))^2+1}{4} \\
&= - \frac{ \gen{c_1(\psi_{\q,m,b}^+), [R_{m,b}^+]}^2}{4(k+dm)(k+dm+d)} + \frac14 \\
&= - \frac{ (2\Al(\q_{m_0,b_0}) + k+dm + d) ^2}{4(k+dm)(k+dm+d)} + \frac14 \\
&= - \frac{ \Al(\q_{m_0,b_0})^2 }{(k+dm)(k+dm+d) } - \frac{\Al(\q_{m_0,b_0})}{k+dm} - \frac{k+dm+d}{4(k+dm)} + \frac14 \\
&= - \frac{ \Al(\q_{m_0,b_0})^2 }{(k+dm)(k+dm+d) } - \frac{\Al(\q_{m_0,b_0})}{k+dm} - \frac{d}{4(k+dm)}.
\end{align*}

Similarly, applying \eqref{eq: agd-c1-q} to $\psi_{\q,m+1,b+1}^-$:
\begin{align*}
\gen{c_1(\psi_{\q,m+1,b+1}^-), [R_{m+1,b+1}^-]} &= 2\Al(\q_{m+1,b+1}) - (k+dm+d) + d \\
&= 2\Al(\q_{m_0,b_0}) + d((m+1-2b-2) - (m_0-2b_0)) - k-dm  \\
&= 2\Al(\q_{m_0,b_0})  - k-dm-d
\end{align*}
\begin{align*}
\absgr(\q_{m+2,b+1}) &- \absgr(\q_{m+1,b+1}) \\
&= \frac{c_1(\spincs_w(\psi_{\q,m+1,b+1}^-))^2+1}{4} \\
&= - \frac{ \gen{c_1(\psi_{\q,m+1,b+1}^-), [R_{m+1,b+1}^-]}^2}{4(k+dm+d)(k+dm+2d)} + \frac14 \\
&= - \frac{ (2\Al(\q_{m_0,b_0}) -k-dm-d)^2}{4(k+dm+d)(k+dm+2d)} + \frac14 \\
&= - \frac{ \Al(\q_{m_0,b_0})^2 }{(k+dm+d)(k+dm+2d) } + \frac{\Al(\q_{m_0,b_0})}{k+dm+2d} - \frac{k+dm+d}{4(k+dm+2d)} + \frac14 \\
&= - \frac{ \Al(\q_{m_0,b_0})^2 }{(k+dm+d)(k+dm+2d) } + \frac{\Al(\q_{m_0,b_0})}{k+dm+2d} + \frac{d}{4(k+dm+2d)}
\end{align*}
Combining these statements,
\begin{align*}
\absgr(\q_{m+2,b+1}) &- \absgr(\q_{m,b})  \\
&= - \frac{\Al(\q_{m_0,b_0})^2}{k+dm+d}\left( \frac{1}{k+dm} + \frac{1}{k+dm+2d} \right) \\
& \qquad    + \left(\Al(\q_{m_0,b_0}) + \frac{d}{4} \right) \left( \frac{1}{k+dm+2d} - \frac{1}{k+dm} \right) \\
&= - \frac{2 \Al(\q_{m_0,b_0})^2}{(k+dm)(k+dm+2d)} + \left(\Al(\q_{m_0,b_0}) + \frac{d}{4} \right) \left( \frac{1}{k+dm+2d} - \frac{1}{k+dm} \right) \\
&= \left(\frac{\Al(\q_{m_0,b_0})^2}{d} + \Al(\q_{m_0,b_0}) + \frac{d}{4} \right) \left( \frac{1}{k+dm+2d} - \frac{1}{k+dm} \right) \\
&= \frac{1}{d} \left(\Al(\q_{m_0,b_0}) + \frac{d}{2} \right)^2 \left( \frac{1}{k+dm+2d} - \frac{1}{k+dm} \right)
\end{align*}
Therefore, we may compute $\absgr(\q_{m,b}) - \absgr(\q_{m_0,b_0})$ by a telescoping sum:
\begin{align*}
\absgr(\q_{m,b}) - \absgr(\q_{m_0,b_0}) &= (\absgr(\q_{m,b}) - \absgr(\q_{m-2, b-1})) + \cdots + (\absgr(\q_{m_0+2,b_0+1}) - \absgr(\q_{m_0,b_0})) \\
&= \frac{1}{d} \left(\Al(\q_{m_0,b_0}) + \frac{d}{2} \right)^2 \left( \frac{1}{k+dm} - \frac{1}{k+dm_0} \right)
\end{align*}
Since this has a finite limit as $m \to \infty$, the values of $\absgr(\q_{m_0+2i,b_0+i})$, ranging over all $\q \in \T_\alpha \cap \T_\gamma$, are globally bounded by constants, as required.
\end{proof}

\begin{proof}[Proof of Proposition \ref{prop: maslov-bound}]
Apply Lemmas \ref{lemma: gr(xl)} and \ref{lemma: gr(qmb)}.
\end{proof}

\section{The surgery exact sequence} \label{sec: exact-sequence}

In this section, we will examine the construction of the long exact sequence relating the Floer homologies of $Y$, $Y_\lambda(K)$, and $Y_{\lambda+m\mu}(K)$ for $m$ large. In fact, we will make all statements on the level of chain complexes, rather than discussing the resulting exact sequence on homology. Ozsv\'ath and Szab\'o's original proof of the surgery formula \cite{OSzSurgery, OSzRational} does not explicitly discuss the maps that count holomorphic rectangles and pentagons (first used in \cite{OSzDouble}), so we will need to describe these maps in more detail, based on the description given by Mark and the first author \cite{HeddenMarkFractional}. \footnote{Our cyclic indexing of the groups and maps is shifted from that of \cite{HeddenMarkFractional}: our $\bm\beta$, $\bm\gamma$, and $\bm\delta$ respectively correspond to $\bm\gamma^2$, $\bm\gamma^0$, and $\bm\gamma^1$ there, and our $f^\circ_j$, $h^\circ_j$, and $g^\circ_j$ (defined below) correspond to $f_{j+2}$, $h_{j+2}$, and $g_{j+2}$ (indices mod $3$). Also, our $T$ corresponds to $\zeta$ in \cite{HeddenMarkFractional}.}

Throughout the proof, we will use a well-adapted diagram $(\Sigma, \bm\alpha, \bm\beta, \bm\gamma, \bm\delta^{b,m}, w, z, z')$, as described above. We will make a series of statements about ``all $m$ sufficiently large.'' To be precise, this means that we fix some integer $e \ge 0$, and consider pairs $(m,b)$ for which $\frac{m-e}{2} \le b \le \frac{m+e}{2}$, as in Proposition \ref{prop: maslov-bound}. This will be implicit throughout; we will generally suppress $b$ from the notation.

\subsection{Construction of the exact sequence}

\def\GR {\Gamma_m}

We begin by defining the twisted chain complex associated to $(\Sigma, \bm\alpha, \bm\beta, w)$. Let $\GR$ denote the group ring $\F[\Z/m\Z]$, which we realize as the quotient $\F[T]/(T^m-1)$. It is convenient to think of $\GR$ as a subring of $\F[\Q/m\Z]$, which is the ring of rational-exponent polynomials in $T$ (i.e., sums $\sum_{r \in \Q} a_r T^r$ with only finitely many $a_r \ne 0$) modulo the relation $T^m=1$. In particular, for any $r \in \Q$, the coset $T^r \GR$ depends only on the fractional part of $r$ (and is isomorphic to $\F^m$ as a vector space).

The twisted complex $\CFit(\bm\alpha, \bm\beta, w; \GR)$ is generated over $\GR$ by all pairs $[\x,i]$ as usual, with differential\footnote{In \cite{OSzSurgery, OSzRational}, the exponent of $T$ is described as the intersection number between $\partial(\phi)$ and a subvariety of $\Sym^g(\Sigma)$ determined by a marked point $p$ that sits on $\beta_g$ between $w$ and $z$, which is the same as our formulation.}
\begin{equation}
\partial(T^s \cdot [\x,i]) = \sum_{\y \in \T_\alpha \cap \T_\beta} \sum_{\substack{ \phi \in \pi_2(\x,\y) \\ \mu(\phi)=1}} \# \widehat\MM(\phi) \, T^{s+ n_w(\phi) - n_z(\phi)} [\y, i-n_w(\phi)].
\end{equation}
The other versions $\ul\CF^-$, $\ul\CF^+$, and $\ul\CF^t$ (for $t \in \N$) are derived from the infinity version in accordance with their definitions in the untwisted setting.

As in \cite{OSzSurgery, OSzRational}, the complex $\CFit(\bm\alpha, \bm\beta, w; \GR)$ is isomorphic to a direct sum of $m$ copies of $\CFi(\Sigma, \bm\alpha, \bm\beta, w)$, but we define the isomorphism slightly differently. Let
\begin{equation} \label{eq: theta}
\theta\co \CFpt(\bm\alpha, \bm\beta, w; \GR) \to
\bigoplus_{\spincs \in \Spin^c(Y)} \CFp(\bm\alpha, \bm\beta, \spincs, w) \otimes T^{-\AlNorm(\spincs)} \GR
\end{equation}
be defined by
\begin{equation} \label{eq: theta-def}
\theta(T^s [\x,i]) = [\x,i] \otimes T^{s - \AlNorm(\x)}.
\end{equation}
It is simple to check that this is an isomorphism of chain complexes.

Let us look more closely at the right-hand side of \eqref{eq: theta}. For each $\spincs \in \Spin^c(Y)$,
there are $m$ different powers of $T$ occurring in $\CFp(\bm\alpha, \bm\beta, \spincs, w) \otimes T^{-\AlNorm(\spincs)} \GR$, with exponents in $\Q/m\Z$. We will frequently need to lift these exponents to $\Q$; we do so by choosing the $m$ values of $r$ satisfying
\begin{equation} \label{eq: r-bounds}
r \equiv -\AlNorm_{Y,K}(\spincs) \pmod \Z \quad \text{and} \quad \frac{-k-dm}{2d} \le r < \frac{-k+dm}{2d}.
\end{equation}
We may thus write
\begin{equation} \label{eq: ab-twisted-decomp}
\CFi(\bm\alpha, \bm\beta, \spincs, w) \otimes T^{-\AlNorm_{w,z}(\spincs)} \GR =
\bigoplus_{r \in \Q \text{ satis.~\eqref{eq: r-bounds}}}  \CFi(\bm\alpha, \bm\beta, \spincs, w) \otimes T^r.
\end{equation}


Define chain maps
\begin{align}
\label{eq: f0}
f_0^+\co & \CFpt(\bm\alpha, \bm\beta, w; \GR) \to \CFp(\bm\alpha, \bm\gamma, w) \\
\label{eq: f1}
f_1^+ \co & \CFp(\bm\alpha, \bm\gamma, w) \to \CFp(\bm\alpha, \bm\delta, w) \\
\label{eq: f2}
f_2^+\co & \CFp(\bm\alpha, \bm\delta, w) \to \CFpt(\bm\alpha, \bm\beta, w; \GR)
\end{align}
by the following formulas:
\begin{align}
\label{eq: f0-def}
f_0^+(T^s \cdot [\x,i]) &= \sum_{\q \in \T_\alpha \cap \T_\gamma}  \sum_{\substack{\psi \in \pi_2(\x, \Theta_{\beta\gamma}, \q) \\ \mu(\psi)=0 \\ \mathclap{s + n_w(\psi) - n_z(\psi) \equiv 0 \pmod m}}}   \#\MM(\psi) \, [\q, i-n_w(\psi)] \\
\label{eq: f1-def}
f_1^+([\q,i]) &= \sum_{\a \in \T_\alpha \cap \T_\delta} \sum_{\substack{\psi \in \pi_2(\q, \Theta_{\gamma\delta}, \a) \\ \mu(\psi)=0}} \#\MM(\psi) \, [\a, i-n_w(\psi)] \\
\label{eq: f2-def}
f_2^+([\a,i]) &= \sum_{\x \in \T_\alpha \cap \T_\beta} \sum_{\substack{\psi \in \pi_2(\a, \Theta_{\delta\beta}, \x) \\ \mu(\psi)=0}} \#\MM(\psi) \, T^{n_w(\psi) - n_z(\psi)} \cdot [\x, i-n_w(\psi)].
\end{align}
Let $f_0^t$, $f_1^t$, $f_2^t$ denote the analogous maps on $\CF^t$, which will play a critical role in our argument below. (There are also corresponding chain maps on the $U$-completed complexes $\mathbf{CF}^-$ and $\mathbf{CF}^\infty$, but not on the ordinary $\CFm$ and $\CFi$ because the sums may fail to be finite.)

Following \cite{HeddenMarkFractional}, the quadrilateral-counting maps
\begin{align}
\label{eq: h0}
h_0^+\co & \CFpt(\bm\alpha, \bm\beta, w; \GR) \to \CFp(\bm\alpha, \bm\delta, w) \\
\label{eq: h1}
h_1^+ \co & \CFp(\bm\alpha, \bm\gamma, w) \to \CFpt(\bm\alpha, \bm\beta, w; \GR) \\
\label{eq: h2}
h_2^+\co & \CFp(\bm\alpha, \bm\delta, w) \to \CFp(\bm\alpha, \bm\gamma, w)
\end{align}
are defined by the following formulas:
\begin{align}
\label{eq: h0-def}
h_0^+(T^s \cdot [\x,i]) &= \sum_{\a \in \T_\alpha \cap \T_\delta}  \sum_{\substack{\rho \in \pi_2(\x, \Theta_{\beta\gamma}, \Theta_{\gamma\delta}, \a) \\ \mu(\rho)=-1 \\ \mathclap{s + n_w(\rho) - n_z(\rho) \equiv 0 \pmod m}}}   \#\MM(\rho) \, [\a, i-n_w(\rho)] \\
\label{eq: h1-def}
h_1^+([\q,i]) &= \sum_{\x \in \T_\alpha \cap \T_{\beta}} \sum_{\substack{\rho \in \pi_2(\q, \Theta_{\gamma\delta}, \Theta_{\delta\beta'}, \x) \\ \mu(\rho)=-1}} \#\MM(\rho) \, T^{n_w(\rho) - n_z(\rho)} \cdot [\x, i-n_w(\rho)] \\
\label{eq: h2-def}
h_2^+([\a,i]) &= \sum_{\q \in \T_\alpha \cap \T_\gamma} \sum_{\substack{\rho \in \pi_2(\a, \Theta_{\delta\beta}, \Theta_{\beta\gamma}, \q) \\ \mu(\rho)=-1 \\ n_w(\rho) \equiv n_z(\rho) \pmod m}} \#\MM(\rho) \, [\q, i-n_w(\rho)]
\end{align}

A standard argument shows that for each $j \in \Z/3$, the following holds:
\begin{itemize}
\item $h_j^+$ is a null-homotopy of $f^+_{j+1} \circ f^+_j$;
\item $h_{j+1}^+ \circ f_j^+ + f_{j+2}^+ \circ h_j^+$ is a quasi-isomorphism.
\end{itemize}
(The second statement is proven by introducing pentagon-counting maps, which we discuss in Section \ref{ssec: pent-filt}.) Therefore, the exact triangle detection lemma \cite[Lemma 4.2]{OSzDouble} implies an exact sequence on homology. Again, using the same formulas, one can likewise define such maps on $\CF^t$, $\CFmc$, and $\CFic$.

Each of the complexes $\CFpt(\bm\alpha, \bm\beta, w; \GR)$, $ \CFp(\bm\alpha, \bm\gamma, w)$, and $\CFp(\bm\alpha, \bm\delta, w)$ has a decomposition according to the evaluations of spin$^c$ structures on elements of $H_2$ of the corresponding $3$-manifolds. Remark \ref{rmk: H2-evals} implies that the maps $f_j^+$ and $h_j^+$ all respect that decomposition. In particular, the maps respect the subgroup of each complex consisting only of the groups in torsion spin$^c$ structures. Henceforth, by abuse of notation, we will disregard all non-torsion spin$^c$ structures; that is, whenever we refer to the Heegaard Floer complexes, we actually mean the subgroups consisting of only the torsion spin$^c$ structures.

Each of the three complexes discussed above is naturally filtered by the $i$ coordinate, with respect to which the maps $f_j$ and $h_j$ are obviously filtered. We define a second filtration on each complex as follows:
\begin{definition} \label{def: J-filtrations}
\begin{itemize}
\item
The filtration $\JJ_{\alpha\gamma}$ on $\CFp(\bm\alpha, \bm\gamma, w)$ is simply the Alexander filtration:
\begin{equation} \label{eq: ag-filt}
\JJ_{\alpha\gamma}([\q,i]) = \AlNorm_{w,z'}(\q) + i.
\end{equation}

\item
The filtration $\JJ_{\alpha\delta}$ on $\CFp(\bm\alpha, \bm\delta, w)$ is the Alexander filtration shifted by a constant on each spin$^c$ summand. To be precise, for each spin$^c$ structure $\spincu$, and each generator $\a$ with $\spincs_w(\a) = \spincu$, we define
\begin{equation} \label{eq: ad-filt}
\JJ_{\alpha\delta}([\a,i]) = \AlNorm_{w,z'}(\a) + i + \frac{  d^2 m (2 s_\spincu - 1) }{2k(k+dm)},
\end{equation}
where $s_\spincu$ is the number from Definition \ref{def: xu}.

\item
The filtration $\JJ_{\alpha\beta}$ on $\CFpt(\bm\alpha, \bm\beta, w; \GR)$ is defined via the identification $\theta$ and the decomposition \eqref{eq: ab-twisted-decomp}. For any $\x \in \T_\alpha \cap \T_\beta$ with $\spincs_w(\x) = \spincs$, and any $r$ satisfying \eqref{eq: r-bounds},
\begin{equation} \label{eq: ab-twisted-filt}
\JJ_{\alpha\beta}([\x,i] \otimes T^r) = i - \frac{2dr +k+d}{2k}.
\end{equation}
That is, $\JJ_{\alpha\beta}$ is the trivial filtration shifted by a constant that depends linearly on the exponent of $T$, and it does not depend on $\x$ except via its the associated spin$^c$ structure. We transport this back to $\CFpt(\bm\alpha, \bm\beta, w; \GR)$ via $\theta$.
\end{itemize}
\end{definition}

It would be tempting to try to prove that the maps $f_j^+$ and $h_j^+$ defined above are all filtered with respect to the filtrations $\JJ_{\alpha\beta}$, $\JJ_{\alpha\gamma}$, and $\JJ_{\alpha\delta}$, but this turns out not to be the case. To understand the reason for this failure, we must look at spin$^c$ structures. Each of the maps $f_j^+$, $h_j^+$ decomposes as a sum of terms corresponding to spin$^c$ structures on the relevant cobordisms: for instance, we may write
\[
f_0^+ = \sum_{\spincv \in \Spin^c_0(X_{\alpha\beta\gamma})} f^+_{0,\spincv},
\]
where $f^+_{0,\spincv}$ counts only the terms in \eqref{eq: f0-def} for which $\spincs_w(\psi) = \spincv$, and likewise for the other maps. (Recall that $\Spin^c_0(X_{\alpha\beta\gamma})$ denotes the set of spin$^c$ structures which restrict to the canonical torsion spin$^c$ structure on $Y_{\beta\gamma}$, which is represented by the generator $\Theta_{\beta\gamma}$.) As we will see, for each triangle $\psi$ contributing to $f^+_{j,\spincv}$, the filtration shift of $\psi$ is given by $n_{z'}(\psi)$ plus a term that is given by a linear step function of the evaluation of $c_1(\spincv)$ on the relevant triply periodic domain ($P_\gamma$, $R$, or $P_\delta$). As a result, $f^+_{j,\spincv}$ is filtered only when $\spincv$ lies within a certain range.

However, the maps on the truncated complexes $\CF^t$ (for $t \in \N$) are better behaved. Since the Maslov grading shift of $f^t_{j, \spincv}$ is given by a quadratic function of $c_1(\spincv)$, only finitely many terms of the terms $f^t_{j,\spincv}$ can be nonzero for any fixed $t$. (If $b_1(Y)>0$, this is why we only consider the torsion spin$^c$ structures.) By looking closely at how the Maslov gradings interact with the filtration shifts described above, we will prove:
\begin{proposition} \label{prop: tri-filt}
Fix $t \in \N$. For all $m$ sufficiently large, the maps $f_0^t$, $f_1^t$, and $f_2^t$ are all filtered with respect to the filtrations $\JJ_{\alpha\beta}$, $\JJ_{\alpha\gamma}$, and $\JJ_{\alpha\delta}$. Moreover, for any triangle $\psi$ contributing to any of these maps, the filtration shift of the corresponding term equals $n_{z'}(\psi)$.
\end{proposition}

The situation with the rectangle-counting maps is even more complicated. Unlike with the triangle maps, the Maslov grading alone does not guarantee that the only nonzero terms $h^t_{j,\spincv}$ are filtered. However, it turns out that we can simply throw away the bad terms. To be precise, we will define ``truncated'' versions $\tilde h^t_j$, each of which is a sum of terms $h^t_{j,\spincv}$ satisfying certain constraints. We will prove:
\begin{proposition} \label{prop: rect-filt}
Fix $t \in \N$. For all $m$ sufficiently large, the maps $\tilde h^t_0$, $\tilde h^t_1$, and $\tilde h^t_2$ have the following properties:
\begin{itemize}
\item $\tilde h^t_j$ is a filtered null-homotopy of $f^t_{j+1} \circ f^t_j$.
\item $\tilde h^t_{j+1} \circ f^t_j +  f^t_{j+2} \circ \tilde h^t_j$ is a filtered quasi-isomorphism.
\end{itemize}
\end{proposition}
The challenging part is to choose the spin$^c$ constraints appropriately so as to make $\tilde h^t_j$ a filtered map but still preserve the degeneration arguments needed to prove the other properties. Moreover, the filtered quasi-isomorphism property requires defining spin$^c$-truncated versions of the pentagon-counting maps, which we will discuss in Section \ref{ssec: pent-filt}.

\def\newgr {\overline\gr}

We also define a modified version of the Maslov (homological) grading on each of the three complexes.

\begin{definition} \label{def: newgr}
Let $\absgr$ denote the standard Maslov grading on each complex; in particular, on $\ul\CF^\circ(\bm\alpha, \bm\beta; \GR)$, it is simply an extension of the ordinary Maslov grading on the untwisted complex, without reference to the twisting variable $T$. The new grading $\newgr$ is defined as follows:
\begin{itemize}
\item On $\CF^\circ(\bm\alpha, \bm\gamma, w)$, define $\newgr = \absgr$.

\item For each $\spincu \in \Spin^c(Y_{\lambda+m\mu}(K))$, we define $\newgr$ on $\CF^\circ(\bm\alpha, \bm\delta, w, \spincu)$ by
\begin{equation} \label{eq: newgr-ad}
\newgr = \absgr + \frac{d^2m s_\spincu^2}{k(k+dm)} - \frac{m+1+3\sign(k)}{4}.
\end{equation}

\item For each $\spincs \in \Spin^c(Y)$ and each $r$ satisfying \eqref{eq: r-bounds}, we define $\newgr$ on $\CF(\bm\alpha, \bm\beta, w) \otimes T^r$ by
\begin{equation} \label{eq: newgr-ab}
\newgr = \absgr + \frac{(2dr+k)^2}{4kd} - \frac{2+3\sign(k)}{4}
\end{equation}
and then transport this grading to $\ul\CF^\circ(\bm\alpha, \bm\beta, w; \GR)$, via $\theta$.
\end{itemize}
\end{definition}

(Throughout the discussion below, we will use $\absgr(f)$ and $\newgr(f)$ to denote the grading shift of any map $f$ with respect to the appropriate grading: for instance, $\absgr(f) = \absgr(f(x)) - \absgr(x))$ for any homogeneous element $f$.)

\begin{proposition}
Fix $t \in \N$. For all $m$ sufficiently large, the maps $f_0^t$, $f_1^t$, $f_2^t$, $\tilde h_0^t$, $\tilde h_1^t$, and $\tilde h_2^t$ are all homogeneous with respect to $\newgr$, with respective degrees $0$, $-1$, $0$, $0$, $0$, and $1$.
\end{proposition}

Combining Propositions \ref{prop: tri-filt} and \ref{prop: rect-filt} with Lemma \ref{lemma: mapping-cone}, we deduce:
\begin{theorem} \label{thm: CFt-cone-f2}
Fix $t \in \N$. For all $m$ sufficiently large, the map
\[
\begin{pmatrix} f_1^t \\ h_1^t \end{pmatrix}  \co \CF^t(\Sigma, \bm\alpha, \bm\gamma, w, z') \to \Cone(f_2^t)
\]
is a filtered homotopy equivalence that preserves the grading $\newgr$.
\end{theorem}

\subsection{Triangle maps} \label{ssec: tri-filt}

In this section, we prove Proposition \ref{prop: tri-filt}. We consider the maps $f_0^t$, $f_1^t$, and $f_2^t$ individually. (Throughout, we will write $f_j^\circ$ when making statements that apply all the flavors of Heegaard Floer homology.)

\subsubsection{The map $f_0^t$} \label{sssec: f0-filt}

To begin, we look at how the spin$^c$ decomposition of $f_0^\circ$ interacts with the trivializing map $\theta$. For any $\spincv \in \Spin^c_0(X_{\alpha\beta\gamma})$, any $\x \in \T_\alpha \cap \T_\beta$ with $\spincs_w(\x) = \spincv|_Y$, and any $r \in \Q/m\Z$ congruent mod $\Z$ to $-\AlNorm_{w,z}(\x)$, we have:
\begin{align*}
f_{0,\spincv}^\circ(\theta^{-1}([\x,i] \otimes T^r))
&= f_{0,\spincv}^\circ (T^{r+\AlNorm(\x)} \cdot [\x,i])  \\
&= \sum_{\q \in \T_\alpha \cap \T_\gamma} \ \sum_{\substack{\psi \in \pi_2(\x, \Theta_{\beta\gamma}, \q) \\ \mu(\psi)=0 \\ \spincs_w(\psi) = \spincv \\ \mathclap{r + \AlNorm(\x) + n_w(\psi) - n_z(\psi) \equiv 0 \pmod m} }} \#\MM(\psi) \, [\q, i-n_w(\psi)].
\end{align*}
By \eqref{eq: abg-c1-x}, note that
\[
\AlNorm(\x) + n_w(\psi) - n_z(\psi) = \frac{\gen{c_1(\spincv), [P_\gamma]}+k}{2d}.
\]
Thus, $f_{0,\spincv}^\circ \circ \theta^{-1}$ is nonzero only on the summand
\[
\CF(\Sigma, \bm\alpha, \bm\beta, \spincs, w) \otimes T^r \GR,
\]
where $\spincs = \spincv |_Y$ and
\[
r \equiv - \frac{1}{2d} (\gen{c_1(\spincv), [P_\gamma]} + k)  \pmod m.
\]
On this summand, neglecting the power of $T$, the composition equals the untwisted cobordism map $F^\circ_{W_\lambda(K),\spincv}$.

\begin{lemma} \label{lemma: f0-trunc}
Fix $t \in \N$ and $\epsilon>0$. For all $m$ sufficiently large, if $\spincv$ is a spin$^c$ structure for which $f_{0,\spincv}^t \ne 0$, then
\begin{equation} \label{eq: f0-trunc}
\abs{ \gen{c_1(\spincv), [P_\gamma]} } < \epsilon dm.
\end{equation}
In particular, if we take $\epsilon<1$, then for any $\spincs \in \Spin^c(Y)$ and any $r \in \Q$ which satisfies \eqref{eq: r-bounds}, there is at most one $\spincv \in \Spin^c(W_\lambda(K))$ for which $f_{0,\spincv}^t \circ \theta^{-1}$ may restrict nontrivially to $\CF(\bm\alpha, \bm\beta, \spincs, w) \otimes T^r \GR$; this spin$^c$ structure must satisfy
\begin{equation} \label{eq: f0-c1-r}
\gen{c_1(\spincv), [P_\gamma]} = - 2dr - k.
\end{equation}
\end{lemma}

\begin{proof}
The grading shift of the term $f^\circ_{0,\spincv}$ is given by
\begin{align}
\nonumber \absgr (f^\circ_{0,\spincv}) &= \frac{c_1(\spincv)^2 - 2\chi(W_\lambda(K)) - 3\sigma(W_\lambda(K))}{4} \\
\label{eq: f0-gr} &= \frac{\gen{c_1(\spincv), [P_\gamma]}^2}{4kd} - \frac{2+3\sign(k)}{4}.
\end{align}
For fixed $t$, gradings of nonzero elements of $\ul\CF^t(\bm\alpha, \bm\beta, w; \GR)$ and $\CF^t(\bm\alpha,\bm\gamma, w)$ are bounded by a constant independent of $m$. Thus, for $m$ sufficiently large, the terms $f^t_{0,\spincv}$ with $\abs{ \gen{c_1(\spincv), [P_\gamma]} } \ge \epsilon dm$ must vanish.

For the second statement, note that the values of $\gen{c_1(\spincv), [P_\gamma]}$ for which $f_{0,\spincv}^t \circ \theta^{-1}$ may restrict nontrivially to $\CF^t(\bm\alpha, \bm\beta, w, \spincs) \otimes T^r $ form a single coset in $\Z/2dm$. If $\epsilon<1$, there is at most one such value within the permitted range.
\end{proof}

\begin{proposition} \label{prop: f0-filt}
Fix $t \in \N$. For all $m$ sufficiently large, the map
\[
f_0^t \co \ul\CF^t(\bm\alpha, \bm\beta, w; \GR) \to \CF^t(\bm\alpha, \bm\gamma, w)
\]
is filtered with respect to the filtrations $\JJ_{\alpha\beta}$ and $\JJ_{\alpha\gamma}$ and is homogeneous of degree $-1$ with respect to $\newgr$.
\end{proposition}

\begin{proof}
Assume $m$ is large enough to satisfy Lemma \ref{lemma: f0-trunc} (with $\epsilon \le 1$). Let $\spincv$ be a spin$^c$ structure for which $f^t_{0,\spincv} \ne 0$, and let $\spincs = \spincv|_Y$, which must therefore satisfy \eqref{eq: f0-c1-r}.

For any $\x \in \T_\alpha \cap \T_\beta$ and $\q \in \T_\alpha\cap \T_\gamma$ with $\spincs_w(\x) = \spincs = \spincv|_{Y_{\alpha\beta}}$ and $\spincs_w(\q) = \spincv|_{Y_{\alpha\gamma}}$, and any triangle $\psi \in \pi_2(\x, \Theta_{\beta\gamma}, \q)$ contributing to $f^t_{0,\spincv}$, we then have:
\begin{align*}
\JJ_{\alpha\beta} ([\x,i] &\otimes T^r) - \JJ_{\alpha\gamma}([\q,i-n_w(\psi)])   \\
&= - \frac{2dr + k +d}{2k} - \frac{\Al(\q)}{k}  + n_w(\psi)   \\
&= -\frac{2dr + k + d}{2k} - \frac{\gen{c_1(\spincv), [P_\gamma]} - 2kn_{z'}(\psi) + 2k n_w(\psi) - d}{2k} + n_w(\psi) \\
&= n_{z'}(\psi) \\
&\ge 0
\end{align*}
as required. The final statement follows from \eqref{eq: newgr-ab}, \eqref{eq: f0-c1-r}, and \eqref{eq: f0-gr}.
\end{proof}

\subsubsection{The map $f_1^t$} \label{sssec: f1-filt}

The map $f_1^\circ$ decomposes as a sum
\begin{equation} \label{eq: f1-decomp}
f_1^\circ = \sum_{\spincv \in \Spin^c_0(X_{\alpha\gamma\delta})} f_{1,\spincv}^\circ.
\end{equation}
By a result of Zemke \cite{ZemkeDuality}, each term $f^\circ_{1,\spincv}$ has an alternate description, as follows. Let $\spinct = \spincv|_{Y_\lambda(K)}$. As in Section \ref{ssec: cobordisms}, let $W_{\alpha\gamma\delta}$ be the $2$-handle cobordism from $Y_\lambda(K) \conn L(m,1)$ to $Y_{\lambda+m\mu}$ obtained by drilling out an arc from $\bar X_{\alpha\gamma\delta}$. Then $\spincv$ induces a spin$^c$ structure on $W_{\alpha\gamma\delta}$ whose restriction to $Y_\lambda(K) \conn L(m,1)$ is $\spinct \conn \spincs_0$, where $\spincs_0$ is the canonical spin$^c$ structure on $L(m,1)$. Moreover, since $L(m,1)$ is an L-space, $\CF^\circ(Y_\lambda(K) \conn L(m,1), \spinct \conn \spincs_0)$ is naturally identified with $\CF^\circ(Y_\lambda(K), \spinct)$ with grading shifted up by $\frac{m-1}{4}$ (which is the grading of the generator of $\HFa(L(m,1), \spincs_0)$). Under this identification, \cite[Theorem 9.1] {ZemkeDuality} implies that $f^\circ_{1,\spincv}$ is naturally identified with the map $F^\circ_{W_{\alpha\gamma\delta}, \spincv}$, as originally defined by Ozsv\'ath and Szab\'o \cite{OSz4Manifold}. As a consequence of this identification, we deduce that $f_{1,\spincv}^\circ$ is homogeneous of degree
\begin{align}
\nonumber \absgr (f_{1,\spincv}^\circ) &= \absgr(F^\circ_{W_{\alpha\gamma\delta}, \spincv}) + \frac{m-1}{4} \\
\label{eq: f1-gr}
&= - \frac{\nu^2 \gen{c_1(\spincv), [R]}^2}{4 m k (k+dm)} + \frac{m-3+3\sign(k)}{4}.
\end{align}

%
%

\begin{lemma} \label{lemma: f1-trunc}
Fix $t \in \N$. For all $m$ sufficiently large, if $f_{1,\spincv}^t \ne 0$, then
\begin{equation} \label{eq: f1-trunc}
\abs{ \gen{c_1(\spincv), [R]} } < \frac{m(k+dm)}{\nu}.
\end{equation}
In particular, for any $\spincu \in \Spin^c(Y_{\lambda+m\mu}(K))$, there is at most one nonzero term landing in $\CF^t(\Sigma, \bm\alpha, \bm\delta, \spincu, w)$, corresponding to a spin$^c$ structure $\spincv$ with
\begin{equation} \label{eq: f1-trunc-su}
\gen{c_1(\spincv), [R]} = \frac{2 d m s_\spincu}{\nu} .
\end{equation}
\end{lemma}

\begin{proof}
By Proposition \ref{prop: maslov-bound}, there is a constant $C$ such that for all $m$ sufficiently large, if $f^t_{1,\spincv}\ne 0$, then
\[
-C < \absgr(f^t_{1,\spincv})  < \frac{m}{4} + C.
\]
(Here, we used the fact that the gradings on $\CF^t(\Sigma,\bm\alpha,\bm\gamma, w)$ are bounded above and below independent of $m$.) Thus, for another constant $C'$,
\[
-C' < -\frac{\nu^2 \gen{c_1(\spincv), [R]}^2}{4 m k (k+dm)} + \frac{m}{4} < \frac{m}{4} + C',
\]
and hence
\[
-C' < \frac{\nu^2 \gen{c_1(\spincv), [R]}^2}{4 m k (k+dm)} < C'+ \frac{m}{4}.
\]
Thus, we obtain an upper bound on $\abs{\gen{c_1(\spincv), [R]}}$ which is asymptotic to $m^{3/2}$ if $k>0$ and to $m$ if $k<0$. In either case, for $m$ sufficiently large, we immediately deduce the weaker bound \eqref{eq: f1-trunc}.

For the second statement, note that for a fixed $\spincu \in \Spin^c(Y_{\lambda+m\mu}(K))$, the values of $\gen{c_1(\spincv), [R]}$, ranging over all $\spincv \in \Spin^c_0(X)$ with $\spincv|_{Y_{\alpha\delta}} = \spincu$,  form a single coset in $\Z/(\frac{2m(k+dm)}{\nu} \Z)$. Indeed, for any triangle $\psi \in \pi_2(\q, \Theta_{\gamma\delta}, \a)$ representing $\spincv$, \eqref{eq: agd-c1-a} and \eqref{eq: su-A(a)-cong} imply:
\begin{align*}
\gen{c_1(\spincv), [R]}
&= \frac{m}{\nu} \left( 2 \Al(\a)  + (k+dm) \left( 2n_u(\psi) - 2n_w(\psi) - 1 \right) +d \right) \\
&\equiv \frac{m}{\nu} \left( 2 \Al(\a)  + k+dm + d  \right) \pmod {\frac{2m(k+dm)}{\nu}}  \\
&\equiv \frac{2dms_\spincu}{\nu} \pmod {\frac{2m(k+dm)}{\nu}}.
\end{align*}
This congruence, combined with the bounds on $s_\spincu$ from \eqref{eq: su-bound}, implies that \eqref{eq: f1-trunc-su} holds.
\end{proof}

\begin{proposition} \label{prop: f1-filt}
For $m$ sufficiently large, the map
\[
f_1^t \co \CF^t(\bm\alpha, \bm\gamma, w) \to \CF^t(\bm\alpha, \bm\delta, w)
\]
is filtered with respect to the filtrations $\JJ_{\alpha\gamma}$ and $\JJ_{\alpha\delta}$ and is homogeneous of degree $-1$ with respect to $\newgr$.
\end{proposition}

\begin{proof}
By Proposition \ref{lemma: f1-trunc}, it suffices to consider only terms $f^t_{1,\spincv}$ with
\[
\gen{c_1(\spincv), [R]} = \frac{2 d m s_\spincu}{\nu},
\]
where $\spincu = \spincv|_{Y_{\alpha\delta}}$. For any $[\q,i] \in \CF^\circ(\Sigma, \bm\alpha, \bm\gamma, w)$ and any term $[\a, i-n_w(\psi)]$ occurring in $f^t_{1,\spincv}([\q,i])$, we first observe:
\begin{align*}
\AlNorm(\q) &- \AlNorm(\a) \\
&= \frac{\Al(\q)}{k} - \frac{\Al(\a)}{k+dm}   \\
&= \frac{k(\Al(\q) - \Al(\a)) + dm \Al(\q)}{k(k+dm)} \\
&= \frac{1}{k+dm} \left( k n_{z'}(\psi)  + dm n_u(\psi) - (k+dm) n_w(\psi) - \frac{dm}{2} \right) \\
 & \qquad + \frac{dm}{k(k+dm)} \left( \frac{\nu}{2m} \gen{c_1(\spincs_w(\psi)), [R]} - k n_u(\psi) + k n_{z'}(\psi) + \frac{k}{2} -  \frac{d}{2} \right) \\
&= n_{z'}(\psi) - n_w(\psi)  + \frac{d^2m}{2k(k+dm)} \left( \frac{\nu}{dm} \gen{c_1(\spincs_w(\psi)), [R]}   -  1 \right) \\
&= n_{z'}(\psi) - n_w(\psi) + \frac{d^2m ( 2s_\spincu -  1 )}{2k(k+dm)}.
\end{align*}
Therefore,
\[
\JJ_{\alpha\gamma}([\q,i]) - \JJ_{\alpha\delta}([\a, i-n_w(\psi)]) = n_{z'}(\psi) \ge 0
\]
as required. The final statement follows from equations \eqref{eq: newgr-ad}, \eqref{eq: f1-gr}, and \eqref{eq: f1-trunc-su}.
\end{proof}

\subsubsection{The map $f_2^t$} \label{sssec: f2-filt}

We start by examining how the spin$^c$ decomposition of $f_2^\circ$ interacts with the trivializing map $\theta$. For any $\a \in \T_\alpha \cap \T_\delta$, using \eqref{eq: adb-c1-x}, we have:
\begin{align*}
\theta \circ f_2^\circ([\a,i])
&= \sum_{\x \in \T_\alpha \cap \T_\beta} \sum_{\substack{\psi \in \pi_2(\a, \Theta_{\delta\beta}, \x) \\ \mu(\psi)=0}} \#\MM(\psi) \, [\x, i-n_w(\psi)] \otimes T^{n_w(\psi) - n_z(\psi) - \AlNorm_{w,z}(\x)} \\
&= \sum_{\x \in \T_\alpha \cap \T_\beta} \sum_{\substack{\psi \in \pi_2(\a, \Theta_{\delta\beta}, \x) \\ \mu(\psi)=0}} \#\MM(\psi) \, [\x, i-n_w(\psi)] \otimes T^{-\frac{1}{2d} (\gen{c_1(\spincs_w(\psi)), [P_\delta]} +k+dm) }.
\end{align*}
In other words, the term $\theta \circ f_{2,\spincv}^\circ$ lands in the summand
\begin{equation} \label{eq: f2-t-target}
\CF^\circ(\bm\alpha, \bm\beta, \spincv|_Y, w) \otimes T^{-\frac{1}{2d} (\gen{c_1(\spincv), [P_\delta]}+k+dm) }
\end{equation}
and agrees with the untwisted map $F^\circ_{W'_m, \spincv}$ in the first factor. The Maslov grading shift of the term $f^\circ_{2,\spincv}$ is given by
\begin{equation} \label{eq: f2-gr}
\absgr (f^\circ_{2,\spincv}) = \frac{c_1(\spincv)^2 + 1}{4} = -\frac{\gen{c_1(\spincv), [P_\delta]}^2}{4d(k+dm)} + \frac14.
\end{equation}

\begin{lemma} \label{lemma: f2-trunc}
Fix $t \in \N$ and $\epsilon>0$. For all $m$ sufficiently large, if $f_{2,\spincv}^t \ne 0$, then
\begin{equation} \label{eq: f2-trunc}
\abs{ \gen{c_1(\spincv), [P_\delta]} } < (1+\epsilon)(k+dm).
\end{equation}
In particular, if we assume that $(1+\epsilon)(k+dm) < 2dm$, the only spin$^c$ structures that may contribute to $f^t_2$ are those denoted by $\spincx_\spincu$ and $\spincy_\spincu$ in Definition \ref{def: xu}.
\end{lemma}

\begin{proof}
Assume that $m$ is large enough to satisfy Proposition \ref{prop: maslov-bound}. Suppose, toward a contradiction, that $f_{2,\spincv}^t \ne 0$ and $\abs{ \gen{c_1(\spincv), [P_\delta]} } \ge (1+\epsilon)(k+dm)$, Then for some constants $C, C'$ independent of $m$, for any homogeneous element $a \in \CF^t(\bm\alpha, \bm\delta, w)$ with $f_{2,\spincv}^t(a) \ne 0$,  we have:
\begin{align*}
\absgr(f_{2,\spincv}^t(a)) &= \absgr(a)  - \frac{\gen{c_1(\spincv), [P_\delta]}^2}{4d(k+dm)} + \frac14 \\
&\le \frac{m}{4}  - \frac{(1+\epsilon)^2 (k+dm)^2}{4d(k+dm)} + C + \frac14 \\
&= \frac{dm - (1+\epsilon)^2 (k+dm)}{4d} +C + \frac14 \\
&= \frac{1- (1+\epsilon)^2 }{4} \cdot m + C'
\end{align*}
Since the grading on $\ul\CF^t(\bm\alpha, \bm\beta, w; \GR)$ is bounded below independent of $m$, while $\absgr(f_{2,\spincv}^t(a))$ is bounded above by a negative multiple of $m$, we obtain a contradiction if $m$ is large enough.
\end{proof}

\begin{proposition} \label{prop: f2-filt}
Fix $t \in \N$. For all $m$ sufficiently large, the map
\[
f_2^t \co \CF^t(\bm\alpha, \bm\delta, w) \to \ul\CF^t(\bm\alpha, \bm\beta, w; \GR)
\]
is filtered with respect to the filtrations $\JJ_{\alpha\delta}$ and $\JJ_{\alpha\beta}$ and is homogeneous of degree $0$ with respect to $\newgr$.
\end{proposition}

\begin{proof}
Choose $m$ sufficiently large to satisfy Lemma \ref{lemma: f2-trunc}, where we assume that $(1+\epsilon)(k+dm) < 2dm$. Suppose $\spincv$ is a spin$^c$ structure on $W'_m$ for which $f^t_{2, \spincv} \ne 0$, and let $\spincu = \spincv|_{Y_{\lambda+m\mu}(K)}$ and $\spincs = \spincv|_Y$. We thus have
\[
-2dm < -(1+\epsilon)(k+dm) < \gen{c_1(\spincv), [P_\delta]} < (1+\epsilon)(k+dm) < 2dm.
\]

Let $r$ denote the rational number satisfying
\[
-k-dm \le 2dr < -k+dm \quad \text{and} \quad 2dr \equiv -(\gen{c_1(\spincv), [P_\delta]}+k+dm) \pmod {2dm}.
\]
Note that $r$ is one of the exponents appearing in \eqref{eq: ab-twisted-decomp}. By \eqref{eq: f2-t-target}, $f^t_{2,\spincv}$ lands in $\CF^\circ(\bm\alpha, \bm\beta, \spincv|_Y, w) \otimes T^r$. At the same time, by \eqref{eq: su-bound} and \eqref{eq: su-A(a)-cong}, the number $s_\spincu$ satisfies:
\[
-k-dm < 2ds_\spincu \le k+dm \quad \text{and} \quad 2ds_\spincu \equiv \gen{c_1(\spincv), [P_\delta]} + k+dm \pmod {2(k+dm)}.
\]

There are two possibilities to consider. If $-(1+\epsilon)(k+dm) < \gen{c_1(\spincv), [P_\delta]} \le 0$, then the above inequalities and congruences imply that
\begin{equation} \label{eq: f2-filt-c1-neg}
\gen{c_1(\spincv), [P_\delta]} = -2dr-k-dm = 2ds_\spincu - k-dm.
\end{equation}
On the other hand, if $0 < \gen{c_1(\spincv), [P_\delta]} < (1+\epsilon)(k+dm)$, we obtain
\begin{equation} \label{eq: f2-filt-c1-pos}
\gen{c_1(\spincv), [P_\delta]} = -2dr-k+dm = 2ds_\spincu + k+dm.
\end{equation}
%
%
%

Suppose $\psi \in \pi_2(\a, \Theta_{\delta\beta}, \x)$ is any triangle that counts for $f^t_{2,\spincv}$,
so that $\theta( f^t_{2,\spincv}([\a,i]))$ includes the term $[\x,i-n_w(\psi)] \otimes T^r$. We compute:
\begin{align*}
\JJ_{\alpha\delta}([\a,i]) &- \JJ_{\alpha\beta} ([\x,i-n_w(\psi)]  \otimes T^r)   \\
&= \AlNorm_{w,z'}(\a) + \frac{2dr + k + d}{2k} +  \frac{d^2m (2 s_\spincu-1)}{2k(k+dm)} + n_w(\psi)  \\
&= \frac{\gen{c_1(\spincv), [P_\delta]} - 2(k+dm) n_w(\psi) + 2(k+dm) n_{z'}(\psi) -d}{2(k+dm)}  \\
& \qquad + \frac{2dr + k + d}{2k} +  \frac{d^2m (2 s_\spincu-1)}{2k(k+dm)} + n_w(\psi)  \\
&= \frac{\gen{c_1(\spincv), [P_\delta]} -d }{2(k+dm)}  + \frac{2dr + k + d}{2k} +  \frac{d^2m (2 s_\spincu-1)}{2k(k+dm)} + n_{z'}(\psi).
\end{align*}
Depending on the sign of $\gen{c_1(\spincv), [P_\delta]}$, we may use either \eqref{eq: f2-filt-c1-neg} or \eqref{eq: f2-filt-c1-pos} to put the first two fractions in terms of $s_\spincu$, and deduce:
\begin{align*}
\JJ_{\alpha\delta}([\a,i]) &- \JJ_{\alpha\beta} ([\x,i-n_w(\psi)]  \otimes T^r)   \\
&= \frac{2ds_\spincu -d }{2(k+dm)}  - \frac{2ds_\spincu  - d}{2k} +  \frac{d^2m (2 s_\spincu-1)}{2k(k+dm)} + n_{z'}(\psi) \\
&= n_{z'}(\psi) \ge 0.
\end{align*}
For the final statement, equations \eqref{eq: newgr-ad}, \eqref{eq: newgr-ab}, and \eqref{eq: f2-gr} together with either \eqref{eq: f2-filt-c1-neg} or \eqref{eq: f2-filt-c1-pos} show that
\begin{align*}
\newgr(f^t_{2,\spincv})
&= -\frac{\gen{c_1(\spincv), [P_\delta]}^2}{4d(k+dm)} + \frac{(2dr+k)^2}{4kd} - \frac{d^2m s_\spincu^2}{k(k+dm)} + \frac{m}{4} = 0
\end{align*}
as required.
\end{proof}

\subsection{Rectangle maps} \label{ssec: rect-filt}

Next, we turn to the rectangle-counting maps. We will introduce the truncated maps $\tilde h^t_0$, $\tilde h^t_1$, and $\tilde h^t_2$, and use them to prove the first part of Proposition \ref{prop: rect-filt}.

\subsubsection{The map $h_0^t$} \label{sssec: h0-filt}

Just as with $f_0^\circ$ in Section \ref{sssec: f0-filt}, for each $\spincv \in \Spin^c_0(X_{\alpha\beta\gamma\delta})$, the composition $h_{0,\spincv}^\circ \circ \theta^{-1}$ is nonzero only on the summand $\CF^\circ(\bm\alpha, \bm\beta, \spincs, w) \otimes T^r \GR$, where $\spincs = \spincv |_Y$ and
\[
r \equiv - \frac{1}{2d} (\gen{c_1(\spincv), [P_\gamma]} + k)  \pmod m,
\]
on which $h_{0,\spincv}^\circ \circ \theta^{-1}$ equals an untwisted count of rectangles.

The cobordism $X_{\alpha\beta\gamma\delta}$ is always indefinite. Specifically, if $k>0$, then $X_{\alpha\beta\gamma}$ is positive-definite and $X_{\alpha\gamma\delta}$ is negative-definite while the reverse is true if $k<0$. Define $\spincv_{\alpha\beta\gamma} = \spincv|_{W_{\alpha\beta\gamma}}$, and likewise for other subsets of $\alpha,\beta,\gamma,\delta$.

Each summand $h^\circ_{0,\spincv}$ is homogeneous, with grading shift given by
\begin{align*}
\absgr h^\circ_{0,\spincv} &= \absgr(f^\circ _{0,\spincv_{\alpha\beta\gamma}}) + \absgr(f^\circ _{1,\spincv_{\alpha\gamma\delta}}) + 1 \\
&= \frac{c_1^2(\spincv) + m-1}{4}.
\end{align*}
%
%
%
We can break this down in two ways. The first is:
\begin{align}
\nonumber \absgr h^\circ_{0,\spincv}
&= \frac{c_1(\spincv_{\alpha\beta\gamma})^2 + c_1(\spincv_{\alpha\gamma\delta})^2 + m-1}{4} \\
\label{eq: h0-gr-Pg-R} &= \frac{\gen{c_1(\spincv), [P_\gamma]}^2}{4dk} - \frac{\nu^2 \gen{c_1(\spincv), [R]}^2}{4mk(k+dm)}  + \frac{m-1}{4}.
\end{align}
This expression alone does not allow us to simultaneously control $\gen{c_1(\spincv), [P_\gamma]}$ and $\gen{c_1(\spincv), [R]}$ as in Lemmas \ref{lemma: f0-trunc} and \ref{lemma: f1-trunc}; there could be nonzero summands $h^t_{0,\spincv}$ for which the evaluations of $c_1(\spinct)$ on $[P_\gamma]$ and $[R]$ are both large in magnitude while $\gr(h^t_{0,\spincv})$ is small. Instead, it will be more useful to write:
\begin{align}
\nonumber \absgr h^\circ_{0,\spincv} &= \frac{c_1(\spincv_{\alpha\beta\delta})^2 + c_1(\spincv_{\beta\gamma\delta})^2 + m-1}{4} \\
\label{eq: h0-gr-Pd-Q} &= \frac{ \gen{c_1(\spincv), [P_\delta]}^2 }{4d(k+dm)} - \frac {\gen{c_1(\spincv), [Q]}^2} {4m} + \frac{m-1}{4}.
\end{align}
By \eqref{eq: abgd-c1-Q}, note that $\gen{c_1(\spincv), [Q]} = em$, where $e$ is an odd integer.\footnote{Not to be confused with the $e$ from Proposition \ref{prop: maslov-bound}.} We will mostly be interested in the cases where $e=\pm 1$, but we can state the following lemma in more generality:

\begin{lemma} \label{lemma: h0-trunc}
Fix $t \in \N$ and $\epsilon>0$. For all $m$ sufficiently large, the following holds: If $\spincv$ is a spin$^c$ structure with $\gen{c_1(\spincv), [Q]} = em$, and $h^t_{0,\spincv} \ne 0$, then
\[
\abs{\gen{c_1(\spincv), [P_\delta]}} < (\abs{e}+\epsilon)(k+dm).
\]
\end{lemma}

\begin{proof}
By Proposition \ref{prop: maslov-bound}, there is a constant $C$ (independent of $m$) such that the gradings of all elements of $\CF^t(\bm\alpha, \bm\delta, w)$ are bounded above by $C+\frac{m}{4}$, while the grading on $\ul\CF^t(\bm\alpha, \bm\beta, w; \GR)$ is bounded below by $-C$. Suppose that $h^t_{0,\spincv} \ne 0$ and $\abs{\gen{c_1(\spincv), [P_\delta]}} \ge (\abs{e}+\epsilon)(k+dm)$. For any homogeneous element $x \in \ul\CF^t(\bm\alpha, \bm\beta, w; \GR)$, if $h^t_{0,\spincv}(x) \ne 0$, we have:
\begin{align*}
2C  &\ge \absgr(h^t_{0,\spincv}(x)) - \absgr(x) - \frac{m}{4} \\
&= \frac{ \gen{c_1(\spincv), [P_\delta]}^2 }{4d(k+dm)} - \frac {\gen{c_1(\spincv), [Q]}^2} {4m} + \frac{m-1}{4} - \frac{m}{4} \\
&\ge \frac{ (\abs{e}+\epsilon)^2 (k+dm) }{4d} - \frac {e^2m^2}{4m} - \frac{1}{4} \\
&= \frac{ \left((\abs{e}+\epsilon)^2- e^2\right) dm + (\abs{e}+\epsilon)^2 k - d }{4d}  \\
\end{align*}
The right hand side tends to infinity as $m \to \infty$, which gives a contradiction.
\end{proof}

We now define the ``truncated'' version of $h_0$. Fix a small real number $\epsilon>0$.

\begin{definition} \label{def: h0-trunc}
For any $\epsilon>0$, let $\tilde h_{0,\epsilon}^t$ be the sum of all terms $h^t_{0,\spincv}$ corresponding to spin$^c$ structures $\spincv$ which either satisfy both
\begin{align}
\label{eq: h0-filt-Pg-bound} \abs{\gen{c_1(\spincv), [P_\gamma]}} &< \epsilon dm \\
\label{eq: h0-filt-R-bound} \abs{ \gen{c_1(\spincv), [R]}} &< \frac{m(k+dm)}{\nu}
\end{align}
or satisfy
\begin{equation}
\label{eq: h0-filt-Q-bound} \abs{ \gen{c_1(\spincv), [Q]}} = \pm m.
\end{equation}
We will often suppress the dependence on $\epsilon$ from the notation and just write $\tilde h^t_0$.
\end{definition}

\begin{lemma} \label{lemma: h0-trunc-nulhtpy}
Fix $t \in \N$ and $\epsilon>0$. For all $m$ sufficiently large, $\tilde h^t_0$ is a null-homotopy of $f^t_1 \circ f^t_0$.
\end{lemma}

\begin{proof}
Let $\rho \in \pi_2(\x, \Theta_{\alpha\beta}, \Theta_{\beta\gamma}, \a)$ be any rectangle such that $\mu(\rho)=0$, $0 \le n_w(\rho) \le m$, and $\MM(\rho) \ne \emptyset$. The possible ends of $\MM(\rho)$ correspond to the following possible decompositions:
\begin{enumerate} [label=(R-\arabic*)]
\item \label{it: h0-decomp-rect}
A concatenation of a rectangle $\rho'$ (with $\spincs_w(\rho') = \spincs_w(\rho)$) and either an $(\alpha,\beta)$, $(\beta,\gamma)$, $(\gamma,\delta)$, or $(\alpha,\delta)$ bigon.

\item \label{it: h0-decomp-abg-agd}
A decomposition $\rho = \psi_1 * \psi_2$, where $\psi_1 \in \pi_2(\x, \Theta_{\beta\gamma}, \q)$ and $\psi_2 \in \pi_2(\q, \Theta_{\gamma\delta}, \a)$.

\item \label{it: h0-decomp-bgd-abd}
A decomposition $\rho = \phi_1 * \phi_2$, where $\phi_1 \in \pi_2(\Theta_{\beta\gamma}, \Theta_{\gamma\delta}, \Theta')$ and $\phi_2 \in \pi_2(\x, \Theta', \a)$ (for some $\Theta' \in \T_\beta \cap \T_\delta$.
\end{enumerate}

We claim that if $\spincv = \spincs_w(\rho)$ fails to satisfy either [\eqref{eq: h0-filt-Pg-bound} and \eqref{eq: h0-filt-R-bound}] or \eqref{eq: h0-filt-Q-bound}, then $\MM(\rho)$ has only ends of type \ref{it: h0-decomp-rect}. Indeed, if $\MM(\rho)$ has an end of type \ref{it: h0-decomp-abg-agd}, then $\psi_1$ counts for $f^t_{0,\spincv_{\alpha\beta\gamma}}$ and $\psi_2$ counts for $f^t_{1,\spincv_{\alpha\gamma\delta}}$. Hence, by Lemmas \ref{lemma: f0-trunc} and \ref{lemma: f1-trunc}, we deduce that $\spincv = \spincs_w(\rho)$ satisfies \eqref{eq: h0-filt-Pg-bound} and \eqref{eq: h0-filt-R-bound}. Similarly, if $\MM(\rho)$ has an end of type \ref{it: h0-decomp-bgd-abd}, then $\Theta' = \Theta_{\beta\delta}$ and $\phi_1$ must be one of the triangles $\tau_l^{\pm}$ from Lemma \ref{lemma: bgd-triangles}. Moreover, if we assume that $m > t$, then since $n_w(\phi_1) \le m$, $\phi_1$ must in fact be $\tau_0^{\pm}$. Therefore $\spincv$ satisfies \eqref{eq: h0-filt-Q-bound}. This proves the claim.

It follows that the map $h^t_0 - \tilde h^t_{0,\epsilon}$ (which counts rectangles which satisfy neither [\eqref{eq: h0-filt-Pg-bound} and \eqref{eq: h0-filt-R-bound}] nor \eqref{eq: h0-filt-Q-bound}) commutes with the differentials:
\[
(h^t_0 - \tilde h^t_{0,\epsilon}) \circ \partial_{\alpha\beta} + \partial_{\alpha\delta} \circ (h^t_0 - \tilde h^t_{0,\epsilon}) = 0.
\]
Since $h^t_0$ is a null-homotopy of $f^t_1 \circ f^t_0$, it follows that $\tilde h^t_{0,\epsilon}$ is as well.
\end{proof}

\begin{proposition} \label{prop: h0-filt}
Fix $t \in \N$ and $0 <\epsilon <1$. For all $m$ sufficiently large, the map $\tilde h^t_{0,\epsilon}$ is filtered with respect to $\JJ_{\alpha\beta}$ and $\JJ_{\alpha\delta}$ and is homogeneous of degree $0$ with respect to $\newgr$.
\end{proposition}

\begin{proof}
We will start by trying to understand the filtration shift of an arbitrary summand $h^t_{0,\spincv}$, and then specialize by imposing the conditions required for $h^t_{0,\spincv}$ to be included in $\tilde h^t_{0,\epsilon}$. Write $\gen{c_1(\spincv), [Q]} = em$, where $e$ is an odd integer.

Suppose $\x \in \T_\alpha \cap \T_\beta$ and $\a \in \T_\alpha \cap \T_\delta$ are generators with $\spincs := \spincs_w(\x) = \spincv|_Y$ and $\spincu := \spincs_w(\a) = \spincv|_{Y_{\lambda+m\mu}(K)}$. Suppose that $\rho \in \pi_2(\x,\Theta_{\beta\gamma}, \Theta_{\gamma\delta}, \a)$ is a rectangle which contributes to $h^\circ_{0,\spincv}$. That is,
\[
h^t_{0,\spincv} \circ \theta^{-1} ([\x,i] \otimes T^r) = [\a, i-n_w(\rho)] + \text{other terms}
\]
where $r$ is the unique number satisfying
\[
\frac{-k-dm}{2d} \le r < \frac{-k+dm}{2d} \quad \text{and} \quad r \equiv -\frac{1}{2d}( \gen{c_1(\spincv), [P_\gamma]} +k) \pmod m.
\]
Note that $r$ is one of the exponents appearing in \eqref{eq: ab-twisted-decomp}, and hence it appears in the definition of $\JJ_{\alpha\beta}([\x,i] \otimes T^r)$ by \eqref{eq: ab-twisted-filt}. Our goal is to show that
\[
\JJ_{\alpha\beta}([\x,i] \otimes T^r) - \JJ_{\alpha\delta}([\a,i-n_w(\rho)])  = n_{z'}(\rho) \ge 0.
\]

Let us write
\begin{equation} \label{eq: h0-filt-p-def}
\gen{c_1(\spincv), [P_\gamma]} = -2dr-k + 2pdm,
\end{equation}
where $p \in \Z$. In other words, $p$ is the unique integer for which
\begin{equation} \label{eq: h0-filt-p-bounds}
(2p-1) dm < \gen{c_1(\spincv), [P_\gamma]} \le (2p+1)dm.
\end{equation}
In particular, if $\spincv$ satisfies \eqref{eq: h0-filt-Pg-bound}, then $p=0$.

Associated to the spin$^c$ structure $\spincu$, we have the number $s_\spincu$, which by \eqref{eq: su-bound} and \eqref{eq: su-A(a)-cong} satisfies
\[
-(k+dm) < 2d s_\spincu \le k+dm \quad \text{and} \quad  2d s_\spincu \equiv 2 \Al(\a) + k + dm + d    \pmod {2(k+dm)}.
\]
At the same time, by \eqref{eq: abgd-c1-R}, we have
\[
\frac{\nu}{m} \gen{c_1(\spincv), [R]}   = 2 \Al(\a)  + (k+dm) (2n_u(\rho) - 2n_w(\rho)) - (k+dm) +d,
\]
and hence
\[
\frac{\nu}{m} \gen{c_1(\spincv), [R]}   \equiv 2d s_\spincu \pmod{2(k+dm)}.
\]
Write
\begin{equation} \label{eq: h0-filt-q-def}
\frac{\nu}{m} \gen{c_1(\spincv), [R]} = 2d s_\spincu + 2q(k+dm),
\end{equation}
where $q \in \Z$, so that
\begin{equation} \label{eq: h0-filt-q-bounds}
(2q-1)(k+dm) < \frac{\nu}{m} \gen{c_1(\spincv), [R]} \le (2q+1)(k+dm).
\end{equation}
Again, if $\spincv$ satisfies \eqref{eq: h0-filt-R-bound}, then $q=0$. Since $\frac{\nu}{m}[R] = [P_\gamma] + \frac{k}{m}[Q]$, we also have
\begin{equation}\label{eq: h0-filt-pq-bounds-Q}
2(q-p-1)dm + (2q-1)k < ke < 2(q-p+1)dm + (2q+1)k.
\end{equation}

We compute:
\begin{align*}
\JJ_{\alpha\delta}([\a, &i-n_w(\rho)]) - \JJ_{\alpha\beta}([\x,i] \otimes T^r) \\
&= \frac{\Al(\a)}{k+dm} -n_w(\rho) + \frac{d^2m(2s_\spincu-1)}{2k(k+dm)} + \frac{2dr + k+d}{2k} \\
&= \frac{\frac{\nu}{m} \gen{c_1(\spincv), [R]}  +k+dm -d}{2(k+dm)} - n_u(\rho) + \frac{dm (\frac{\nu}{m} \gen{c_1(\spincv), [R]} - 2q(k+dm)) }{2k(k+dm)} \\
& \qquad  + \frac{-\gen{c_1(\spincv), [P_\gamma]} - k + 2pdm}{2k} - \frac{d^2m}{2k(k+dm)} + \frac{k+d}{2k} \\
&= \frac{\nu \gen{c_1(\spincv), [R]} }{2m(k+dm)} + \frac{dm \nu \gen{c_1(\spincv), [R]} ) }{2km(k+dm)} + \frac{-\gen{c_1(\spincv), [P_\gamma]}  }{2k} + \frac{(p-q)dm}{k}\\
& \qquad   - \frac{d^2m}{2k(k+dm)} + \frac{k+d}{2k} - \frac{d}{2(k+dm)}  - n_u(\rho) \\
&= \frac{1}{2k} \gen{c_1(\spincv), \frac{\nu}{m}[R]- [P_\gamma]} + \frac{(p-q)dm}{k} - \frac{d}{2k} + \frac{k+d}{2k}   - n_u(\rho) \\
&= \frac{1}{2m} \gen{c_1(\spincv),  [Q] } + \frac{(p-q)dm}{k}  + \frac12   - n_u(\rho) \\
&= \frac{1}{2} (2n_u(\rho) - 2n_{z'}(\rho) - 1)  + \frac{(p-q)dm}{k}  + \frac12   - n_u(\rho) \\
&= - n_{z'}(\rho) + \frac{(p-q)dm}{k}
\end{align*}
Thus, to show that $\tilde h^t_{0,\epsilon}$ is filtered, we must simply show that $p-q = 0$ whenever $\spincv$ satisfies the conditions from Definition \ref{def: h0-trunc}.

If $\spincv$ satisfies \eqref{eq: h0-filt-Pg-bound} and \eqref{eq: h0-filt-R-bound}, we immediately deduce that $p=q=0$. Thus, suppose that $\spincv$ satisfies \eqref{eq: h0-filt-Q-bound}, i.e. $e=\pm1$. By Lemma \ref{lemma: h0-trunc}, we may also assume that $\abs{c_1(\spincv), [P_\delta]} < (1+\epsilon)(k+dm)$. Recall that $[P_\gamma] = [P_\delta] + d[Q]$ and $\frac{\nu}{m}[R] = [P_\delta] + \frac{k+dm}{m}[Q]$. Therefore,
\begin{align*}
(e-1-\epsilon) dm -(1+\epsilon)k  < \gen{c_1(\spincv), [P_\gamma]} &<  (e+1+\epsilon)dm + (1+\epsilon)k \\
(e-1-\epsilon)(k+dm) < \frac{\nu}{m} \gen{c_1(\spincv), [R]} &< (e+1+\epsilon)(k+dm)
\end{align*}
Assuming $m$ is sufficiently large, this implies that $p,q \in \{0,e\}$. It then follows from \eqref{eq: h0-filt-pq-bounds-Q} that $p=q$, as required.

We now turn to the statement about $\newgr$, which we check in each of the two cases in the previous paragraph. In the case where $p=q=0$, equations \eqref{eq: newgr-ad}, \eqref{eq: newgr-ab}, \eqref{eq: h0-gr-Pg-R}, \eqref{eq: h0-filt-p-def}, and \eqref{eq: h0-filt-q-def} immediately imply that $\newgr(h^t_{0,\spincv}) = 0$. If $p=q=e = \pm 1$, then let $\spincv' = \spincv + e \PD[Q]$, which has the same restrictions to $Y$ and $Y_{\lambda+m\mu}(K)$ as $\spincv$ and the same corresponding values of $r$ and $s_\spincu$. We may easily check the following: 
\begin{align*}
\gen{c_1(\spincv'), [P_\gamma]} &= \gen{c_1(\spincv), [P_\gamma]} - 2edm = -2dr-k \\
\frac{\nu}{m}\gen{c_1(\spincv'), [R]} &= \frac{\nu}{m} \gen{c_1(\spincv), [R]} - 2e(k+dm) = 2d s_\spincu \\
\gen{c_1(\spincv'), [Q]} &= \gen{c_1(\spincv), [Q]} - 2em =  -em \\
\gen{c_1(\spincv'), [P_\delta]} &= \gen{c_1(\spincv), [P_\delta]}  =  -2dr-k-dem
\end{align*}
Thus, the analogues of $e$, $p$, and $q$ associated to $\spincv'$ are $e'=-e$ and $p'=q'=0$, which implies that $\newgr(h^t_{0,\spincv'}) = 0$ by the previous case. At the same time, \eqref{eq: h0-gr-Pd-Q} implies that $c_1(\spincv)^2 = c_1(\spincv')^2$, and hence $\newgr(h^t_{0,\spincv}) = 0$ as required.
\end{proof}

\subsubsection{The map $h_1^t$} \label{sssec: h1-filt}

As in the previous section, we will need to define a truncated version of $\tilde h_1^t$ that uses only certain spin$^c$ structures. To begin, for any $\spincv \in \Spin^c_0(X_{\alpha\gamma\delta\beta})$ and any $\q \in \T_\alpha \cap \T_\gamma$, we have
\begin{align*}
\theta \circ h^\circ_{1,\spincv}([\q,i])
&= \sum_{\x \in \T_\alpha \cap \T_{\beta}} \sum_{\substack{\rho \in \pi_2(\q, \Theta_{\gamma\delta}, \Theta_{\delta\beta'}, \x) \\ \mu(\rho)=-1 \\ \spincs_w(\rho) = \spincv}} \#\MM(\rho) \, [\x, i-n_w(\rho)] \otimes T^{n_w(\rho) - n_z(\rho) - \AlNorm_{w,z}(\x)}  \\
&= \sum_{\x \in \T_\alpha \cap \T_{\beta}} \sum_{\substack{\rho \in \pi_2(\q, \Theta_{\gamma\delta}, \Theta_{\delta\beta'}, \x) \\ \mu(\rho)=-1 \\ \spincs_w(\rho) = \spincv}} \#\MM(\rho) \, [\x, i-n_w(\rho)] \otimes T^{-\frac{1}{2d} \left( \gen{c_1(\spincv), [P_\delta]} +k+dm\right)}
\end{align*}
Thus, the summand of \eqref{eq: ab-twisted-decomp} in which $\theta \circ h^\circ_{1,\spincv}$ lands is determined by the value of $\gen{c_1(\spincv), [P_\delta]}$ modulo $2dm$.


\begin{definition} \label{def: h1-trunc}
Let $\tilde h^t_1$ denote the sum of all terms $h^t_{1,\spincv}$ for which $\spincv$ satisfies
\begin{equation} \label{eq: h1-trunc-Q-bound}
\gen{c_1(\spincv), [Q]} = \pm m.
\end{equation}
\end{definition}

\begin{remark} \label{rmk: h1-trunc-simpler}
The definition of $\tilde h^t_1$ appears considerably somewhat simpler than that of $\tilde h^t_0$ (Definition \ref{def: h0-trunc}) in the previous section. In fact, however, the two definitions are parallel. Suppose that $0 < \epsilon< 1$ and that $\spincv$ is any spin$^c$ structure which satisfies
\begin{align}
\label{eq: h1-trunc-R-bound}
\abs{\gen{c_1(\spincv), [R]}} &< \frac{m(k+dm)}{\nu} \\
\label{eq: h1-trunc-Pd-bound}
\abs{\gen{c_1(\spincv), [P_\delta]}} &< (1+\epsilon)(k+dm)
\end{align}
(which are the conditions suggested by Lemmas \ref{lemma: f1-trunc} and \ref{lemma: f2-trunc}, analogous to \eqref{eq: h0-filt-Pg-bound} and \eqref{eq: h0-filt-R-bound} in Definition \ref{def: h0-trunc}). We then have
\begin{align*}
\abs{\gen{c_1(\spincv), [Q]}}
&\le \frac{\nu}{k+dm} \abs{\gen{c_1(\spincv), [R]}} + \frac{m}{k+dm} \abs{\gen{c_1(\spincv), [P_\delta]}}   \\
&\le \frac{\nu}{k+dm} \frac{m(k+dm)}{\nu} + \frac{m}{k+dm} (1+\epsilon) (k+dm)   \\
&= m(2+\epsilon)
\end{align*}
which implies \eqref{eq: h1-trunc-Q-bound}. Thus, it is not necessary to include \eqref{eq: h1-trunc-R-bound} and \eqref{eq: h1-trunc-Pd-bound} in the definition of $\tilde h^t_1$.
\end{remark}

\begin{lemma} \label{lemma: h1-trunc-nulhtpy}
Fix $t \in \N$. For all $m$ sufficiently large, $\tilde h_1^t$ is a null-homotopy of $f_2^t \circ f_1^t$.
\end{lemma}

\begin{proof}
This follows just like Lemma \ref{lemma: h0-trunc-nulhtpy}, taking Remark \ref{rmk: h1-trunc-simpler} into account.
\end{proof}

The cobordism $X_{\alpha\gamma\delta\beta}$ is indefinite if $k<0$ and negative-definite if $k>0$. As in the previous section, the grading shift $h^t_{1,\spincv}$ is given by
\begin{align}
\nonumber \absgr(h^t_{1,\spincv})
&= \frac{c_1(\spincv)^2 + m+2+3\sign(k)}{4} \\
\label{eq: h1-gr-R-Pd}
&= -\frac{\nu^2 \gen{c_1(\spincv), [R]}^2}{4mk(k+dm)} - \frac{\gen{c_1(\spincv), [P_\delta]}^2}{4d(k+dm)} + \frac{m+2+3\sign(k)}{4} \\
\label{eq: h1-gr-Pg-Q}
&= -\frac{\gen{c_1(\spincv), [P_\gamma]}^2}{4dk} - \frac{\gen{c_1(\spincv), [Q]}^2}{4m} + \frac{m+2 + 3 \sign(k)}{4}.
\end{align}
The following lemma is analogous to Lemma \ref{lemma: h0-trunc}:

\begin{lemma} \label{lemma: h1-trunc}
Fix $t \in \N$. For all $m$ sufficiently large, if $\spincv$ is any spin$^c$ structure with $\abs{\gen{c_1(\spincv), [Q]}} = em$ (where $e$ is an odd integer), and $h^t_{1,\spincv} \ne 0$, then $\abs{\gen{c_1(\spincv), [P_\gamma]}} < \abs{e}dm$.
\end{lemma}

\begin{proof}
The gradings on $\ul\CF^t(\bm\alpha, \bm\beta, w; \GR)$ and $\CF^t(\bm\alpha, \bm\gamma, w)$ are bounded above and below by constants independent of $m$. Therefore, for some constant $C$, if $\spincv$ is any spin$^c$ structure for which $h^t_{0,\spincv} \ne 0$, then $-C < \absgr(h^t_{1,\spincv}) < C.$
Using \eqref{eq: h1-gr-Pg-Q}, we have:
\[
-C < -\frac{\gen{c_1(\spincv), [P_\gamma]}^2}{4dk} - \frac{\gen{c_1(\spincv), [Q]}^2}{4m} + \frac{m+2-3\sign(k)}{4} < C.
\]
Thus, for some other positive constant $C'$,
\[
-C' < \frac{\gen{c_1(\spincv), [P_\gamma]}^2}{4dk} + \frac{(e^2-1)m}{4}  < C'.
\]
Setting $C'' = 4d\abs{k}C'$, we have
\[
C'' > \gen{c_1(\spincv), [P_\gamma]}^2 + (e^2-1)dkm.
\]
If $\abs{\gen{c_1(\spincv),[P_\gamma]}} \ge \abs{e}dm$, we obtain
\[
C'' > e^2d^2m^2 + (e^2-1) dkm  = dm ( e^2(k+dm) - k  ) \ge d^2m^2,
\]
which is a contradiction for $m$ sufficiently large. Hence $\abs{\gen{c_1(\spincv),[P_\gamma]}} < \abs{e}dm$ as required.
\end{proof}

When $k>0$, there is an even stronger statement:

\begin{lemma} \label{lemma: h1-trunc-k-pos}
Fix $t \in \N$, and assume $k>0$. For all $m$ sufficiently large, if $\spincv$ is any spin$^c$ structure for which $h^t_{1,\spincv} \ne 0$, then $\spincv$ satisfies \eqref{eq: h1-trunc-R-bound} and \eqref{eq: h1-trunc-Pd-bound} (and hence \eqref{eq: h1-trunc-Q-bound}). In particular, $\tilde h^t_1 = h^t_1$.
\end{lemma}

\begin{proof}
Let $C$ be as in the proof of Lemma \ref{lemma: h0-trunc}.  By \eqref{eq: h1-gr-R-Pd}, we have
\[
-C < \frac{\nu^2 \gen{c_1(\spincv), [R]}^2}{4mk(k+dm)} + \frac{\gen{c_1(\spincv), [P_\delta]}^2}{4d(k+dm)} - \frac{m+5}{4} < C.
\]
The first two terms are both nonnegative since $k>0$, so each one is less than $\frac {m+5}{4} + C$. Just as in the proofs of Lemmas \ref{lemma: f1-trunc} and \ref{lemma: f2-trunc}, for $m$ sufficiently large, both \eqref{eq: h1-trunc-R-bound} and \eqref{eq: h1-trunc-Pd-bound} must hold.
\end{proof}


\begin{proposition} \label{prop: h1-filt}
Fix $t \in \N$. For all $m$ sufficiently large, the map $\tilde h_1^t$ is filtered with respect to the filtrations $\JJ_{\alpha\gamma}$ and $\JJ_{\alpha\beta}$ and is homogeneous of degree $0$ with respect to $\newgr$. (When $k>0$, the same is true for $h_1^t$.)
\end{proposition}

\begin{proof}
Let $\spincv$ be a spin$^c$ structure for which $h^t_{1,\spincv} \ne 0$. Write $\gen{c_1(\spincv), [Q]} = em$. By \eqref{eq: h1-trunc-Q-bound}, we will eventually assume that $e = \pm 1$, but for now let us treat $e$ as an arbitrary odd integer (which will motivate the definition of $\tilde h_1^t$).

Suppose that $\rho \in \pi_2(\q, \Theta_{\gamma\delta}, \Theta_{\delta\beta}, \x)$ is a rectangle which contributes to $h^t_{1,\spincv}$. Let $r$ be the value with
\[
\frac{-k-dm}{2d} \le r < \frac{-k+dm}{2d} \quad \text{and} \quad r \equiv -\frac{1}{2d}( \gen{c_1(\spincv), [P_\delta]} +k+dm) \pmod m,
\]
which is one of the exponents appearing in \eqref{eq: ab-twisted-decomp}. Let $p$ be the integer for which
\[
\gen{c_1(\spincv), [P_\delta]} = -2dr-k + (2p-1) dm,
\]
which implies that
\begin{equation} \label{eq: h1-filt-p-bounds}
(2p-2)dm < \gen{c_1(\spincv), [P_\delta]} \le (2p)dm.
\end{equation}
It also follows from \eqref{eq: agdb-c1-Pd} that
\begin{gather*}
-2dr + 2pdm  = 2\Al(\x) - 2d n_w(\rho) + 2dn_z(\rho).
\end{gather*}

By Lemma \ref{lemma: h1-trunc}, we may assume that $\abs{\gen{c_1(\spincv), [P_\gamma]}} < dm$. Recall that $[P_\gamma] = [P_\delta] + d[Q]$, and therefore
\[
\gen{c_1(\spincv), [P_\delta]} = \gen{c_1(\spincv), [P_\gamma]} - edm.
\]
If $e \ge 1$, this gives
\[
-2e dm < \gen{c_1(\spincv), [P_\delta]} < 0.
\]
Therefore, $-e < p$ and $(2p-2) < 0$, so $-e < p \le 0$. Similarly, if $e \le -1$, then
\[
0 < \gen{c_1(\spincv), [P_\delta]} < -2edm,
\]
so $0 < p \le -e$. Specializing to the cases where $e=\pm 1$, if $e=1$, then $p=0$, and if $e=-1$, then $p=1$. In either case, note that $e+2p=1$. By \eqref{eq: agdb-c1-Q}, this means that
\[
p = \frac{1-e}{2} = n_w(\rho) - n_u(\rho) + 1.
\]
Now, we compute:
\begin{align*}
\JJ_{\alpha\gamma}(&[\q,i]) - \JJ_{\alpha\beta}([\x,i-n_w(\rho)] \otimes T^r) \\
&= \AlNorm_{w,z'}(\q) +  \frac{2d r +k+d}{2k} +n_w(\rho) \\
&= \frac{1}{2k} \left( 2\Al(\q) + 2dr + k + d \right) +n_w(\rho) \\
&= \frac{1}{2k} \left( 2\Al(\q) -2\Al(\x) + 2d n_w(\rho) - 2dn_z(\rho) + 2pdm + k + d \right) +n_w(\rho) \\
&= \frac{1}{2k} \Big( 2d n_z(\rho) + 2k n_{z'}(\rho) + 2dm n_u(\rho) - 2(k+dm+d)n_w(\rho) - k-d-2dm \\
& \qquad \quad + 2d n_w(\rho) - 2dn_z(\rho) + 2dm n_w(\rho) - 2dm n_u(\rho) + 2dm + k + d \Big) +n_w(\rho) \\
&= n_{z'}(\rho) \ge 0.
\end{align*}

For the statement about $\newgr$, we first note that
\[
\gen{c_1(\spincv), [P_\gamma]} = \gen{c_1(\spincv), [P_\delta]} + dem =  -2dr-k + (2p+e-1) dm = -2dr-k.
\]
Equations \eqref{eq: newgr-ab} and \eqref{eq: h1-gr-Pg-Q} then immediately imply that $\newgr(h^t_{1,\spincv})=0$, as required.
\end{proof}

\begin{remark}
In the proof above, without the simplifying assumption that $e+2p=1$, we would have found that
\[
\JJ_{\alpha\gamma}([\q,i]) - \JJ_{\alpha\beta}([\x,i-n_w(\rho)] \otimes T^r) = n_{z'}(\rho) + \frac{e+2p-1}{2k}.
\]
Thus, the map $h^+_1$ (which incorporates all spin$^c$ structures) is not necessarily filtered.
\end{remark}

%

\subsubsection{The map $h_2^t$} \label{sssec: h2-filt}

Next, we consider the map $h_2^t$. Let $W_{\alpha\delta\beta\gamma}$ be the associated cobordism, which is indefinite if $k>0$ and negative-definite if $k<0$. According to \eqref{eq: h2-def}, the map $h_2^t$ counts only holomorphic rectangles $\rho$ with the property that $n_w(\rho) \equiv n_z(\rho) \pmod m$. By \eqref{eq:  adbg-c1-Q}, this is equivalent to the condition that $\gen{c_1(\spincs_w(\rho)), [Q]} = em$, where $e$ is an odd integer. As usual, consider the decomposition into terms of the form $h^t_{2,\spincv}$.

Just as in the previous section, let $\tilde h^t_2$ denote the sum of all terms $h^t_{2,\spincv}$ for which \begin{equation} \label{eq: h2-trunc-Q-bound}
\gen{c_1(\spincv), [Q]} = \pm m.
\end{equation}
The analogue of Remark \ref{rmk: h1-trunc-simpler} holds here as well: for any $0 < \epsilon < 1$, if $\spincv$ satisfies
\begin{align}
\label{eq: h2-trunc-Pd-bound}
\abs{\gen{c_1(\spincv), [P_\delta]}} &< (1+\epsilon)(k+dm) \\
\label{eq: h2-trunc-Pg-bound}
\abs{\gen{c_1(\spincv), [P_\gamma]}} &< \epsilon dm,
\end{align}
then it satisfies \eqref{eq: h2-trunc-Q-bound} as well.
The following lemmas are left as an exercise:

\begin{lemma} \label{lemma: h2-trunc-nulhtpy}
Fix $t \in \N$. For all $m$ sufficiently large, $\tilde h_2^t$ is a null-homotopy of $f_0^t \circ f_2^t$.  \qed
\end{lemma}

\begin{lemma} \label{lemma: h2-trunc}
Fix $t \in \N$ and $\epsilon > 0$. For all $m$ sufficiently large, if $\spincv$ is any spin$^c$ structure with $\abs{\gen{c_1(\spincv), [Q]}} = em$ (where $e$ is an odd integer), and $h^t_{1,\spincv} \ne 0$, then
\begin{equation} \label{eq: h2-trunc-R-bound}
\abs{\gen{c_1(\spincv), [R]}} < \frac{m(k+dm) \abs{e}}{\nu}. 
\end{equation}
\end{lemma}

\begin{lemma} \label{lemma: h2-trunc-k-neg}
Fix $t \in \N$ and $\epsilon > 0$, and assume $k<0$. For all $m$ sufficiently large, if $\spincv$ is any spin$^c$ structure for which $h^t_{1,\spincv} \ne 0$, then $\spincv$ satisfies \eqref{eq: h2-trunc-Pd-bound}, \eqref{eq: h2-trunc-Pg-bound}, and  hence \eqref{eq: h2-trunc-Q-bound}. Therefore, $\tilde h^t_2 = h^t_2$. \qed
\end{lemma}

\begin{proposition} \label{prop: h2-filt}
Fix $t \in \N$. For all $m$ sufficiently large, the map $\tilde h_2^t$ is filtered with respect to the filtrations $\JJ_{\alpha\delta}$ and $\JJ_{\alpha\gamma}$ and is homogeneous of degree $1$ with respect to $\newgr$. When $k<0$, the same is true for $h_2^t$.
\end{proposition}

\begin{proof}
Let $\a \in \T_\alpha \in \T_\delta$ and $\q \in \T_\alpha \cap \T_\gamma$, and suppose $\rho \in \pi_2(\a, \Theta_{\delta\beta}, \Theta_{\beta\gamma}, \q)$ contributes to $h^t_{2,\spincv}([\a,i])$. Let $\spincu = \spincs_w(\a)$. The definition of $\JJ_{\alpha\delta}$ involves the number $s_\spincu$, which by \eqref{eq: su-bound} and \eqref{eq: su-A(a)-cong} satisfies
\[
-(k+dm) < 2d s_\spincu \le k+dm \quad \text{and} \quad  2d s_\spincu \equiv 2 \Al(\a) + k + dm + d    \pmod {2(k+dm)}.
\]
Combining these facts with \eqref{eq: adbg-c1-Pd}, we have:
\begin{align*}
\gen{c_1(\spincs_w(\rho)), [P_\delta]}
&= 2\Al(\a)  +2(k+dm) n_w(\rho) - 2(k+dm) n_{z'}(\rho) + d \\
&\equiv 2\Al(\a)  + d \pmod{2(k+dm)} \\
&\equiv 2d s_\spincu -k-dm \pmod{2(k+dm)}
\end{align*}
Let $q$ be the integer for which
\[
\gen{c_1(\spincs_w(\rho)), [P_\delta]} = 2d s_\spincu + (2q-1)(k+dm),
\]
so that
\[
(2q-2)(k+dm) < \gen{c_1(\spincs_w(\rho)), [P_\delta]} \le 2q(k+dm).
\]

Suppose that $\gen{c_1(\spincv), [Q]} = em$ and $h^t_{2,\spincv} \ne 0$, so that $\spincv$ satisfies \eqref{eq: h2-trunc-R-bound}. If $e \ge 1$, we have:
\begin{gather*}
\gen{c_1(\spincv), [P_\delta]} = \frac{\nu}{m} \gen{c_1(\spincv), [R]} - \frac{k+dm}{m} \gen{c_1(\spincv), [Q]} \\
- 2e(k+dm) < \gen{c_1(\spincv), [P_\delta]} < 0,
\end{gather*}
so $-2e < 2q$ and $2q-2 < 0$, so $-e < q \le 0$; and if $e=1$, then $q=0$. Likewise, if $e \le -1$, then $0 < q \le -e$; and if $e=-1$, then $q=1$. In either case where $e=\pm 1$, we see that $e+2q = 1$.

Assuming $e=\pm 1$, we now compute:
\begin{align*}
\JJ_{\alpha\delta}&([\a,i]) - \JJ_{\alpha\gamma}([\q,i-n_w(\rho)] ) \\
&= \frac{\Al(\a)}{k+dm} + i + \frac{d^2m(2s_\spincu-1)}{2k(k+dm)} - \frac{\Al(\q)}{k} - i + n_w(\rho) \\
&= \frac{2k \Al(\a) - 2(k+dm)\Al(\q) + d^2m(2s_\spincu-1)}{2k(k+dm)}  + n_w(\rho) \\
&= \frac{2(k+dm)( \Al(\a) - \Al(\q)) - 2dm \Al(\a) + d^2m(2s_\spincu-1)}{2k(k+dm)}  + n_w(\rho) \\
&= \frac{1}{k} \left( d n_z(\rho) + k n_{z'}(\rho) + dm n_u(\rho) - (k+dm+d)n_w(\rho) - \frac{dm}{2} \right) \\
& \qquad - \frac{ dm\left( \gen{c_1(\spincs_w(\rho)), [P_\delta]} -2(k+dm) n_w(\rho) + 2(k+dm)n_u(\rho) - d \right) } {2k(k+dm)} \\
& \qquad + \frac{d^2m(2s_\spincu-1)}{2k(k+dm)}  + n_w(\rho) \\
&= \frac{1}{k} \left( d n_z(\rho) + k n_{z'}(\rho) - d n_w(\rho) - \frac{dm}{2} \right)  - \frac{ dm\left( \gen{c_1(\spincs_w(\rho)), [P_\delta]}   - 2ds_\spincu \right) } {2k(k+dm)} \\
&= n_{z'}(\rho) + \frac{d}{2k} \left(  2n_z(\rho)  -  2n_w(\rho) - m  - m(2q-1) \right)   \\
&= n_{z'}(\rho) + \frac{d}{2k} \left( -\gen{c_1(\spincv), [Q]} - m(2q-1) \right)  \\
&= n_{z'}(\rho) - \frac{dm(e+2q-1)}{2k} \\
&= n_{z'}(\rho) \\
&\ge 0.
\end{align*}
The statement about $\newgr$ follows just as in the proof of Proposition \ref{prop: h1-filt}.
\end{proof}

\subsection{Pentagon maps} \label{ssec: pent-filt}

We now turn to the proof of the second part of Proposition \ref{prop: rect-filt}: showing that each map
$\tilde h^t_{j+1} \circ f^t_j +  f^t_{j+2} \circ \tilde h^t_j$ (where $j \in \Z/3$) is a filtered quasi-isomorphism. This relies on the standard strategy of counting holomorphic pentagons, used by Ozsv\'ath and Szab\'o in \cite{OSzDouble} and then adapted by Mark and the first author in \cite{HeddenMarkFractional}. We will only discuss the case of $j=0$, which is the most technically difficult because of the twisted coefficients. The arguments for $j=1$ and $j=2$ are similar and are left to the reader as an exercise.

Let $\tilde{\bm\beta} = (\tilde \beta_1, \dots, \tilde\beta_g)$ be obtained from $\bm\beta$ by a small Hamiltonian isotopy, such that each $\tilde \beta_i$ meets $\beta_i$ in a pair of points, and assume that $\tilde \beta_g$ is as shown in Figure \ref{fig: beta-tilde}. Let $v$ be a point that is in the same region of $\Sigma \minus (\bm\alpha \cup \bm\beta)$ as $w$ and in the same region of $\Sigma \minus (\bm\alpha \cup \tilde{\bm\beta})$ as $z$. Finally, let $\Theta_{\beta\tilde\beta} \in \T_\beta \cap \T_{\tilde\beta}$ denote the canonical top-dimensional generator. (The twisted chain complex associated to $(\bm\beta, \tilde{\bm\beta})$ is somewhat subtle; see \cite[p.~36]{HeddenMarkFractional}.)

\begin{figure}
\labellist
 \pinlabel $z$ at 93 43
 \pinlabel $w$ at 153 43
 \pinlabel $z'$ at 144 116
 \pinlabel $u$ at 146 74
 \pinlabel $v$ at 110 28
 \pinlabel {{\color{red} $\alpha_g$}} [l] at 223 21
 \pinlabel {{\color{blue} $\beta_g$}} [Bl] at 116 126
 \pinlabel {{\color{blue} $\tilde \beta_g$}} [Br] at 101 126
 \pinlabel {{\color{darkgreen} $\gamma_g$}} [l] at 223 82
 \pinlabel {{\color{purple} $\delta_g$}} [l] at 223 50
 \tiny
 \pinlabel $\bullet$ at 116 82
 \pinlabel $\Theta_{\beta\gamma}$ [br] at 116 82
 \pinlabel $\bullet$ at 162 82
 \pinlabel $\Theta_{\gamma\delta}$ [br] at 165 82
 \pinlabel $\bullet$ at 133 66
 \pinlabel $\Theta_{\delta\tilde\beta}$ [tl] at 133 66
\endlabellist
\includegraphics{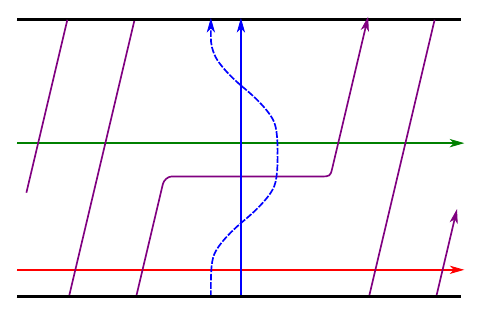}
\caption{Close-up of the winding region with the added curve $\tilde\beta_g$.}
\label{fig: beta-tilde}
\end{figure}
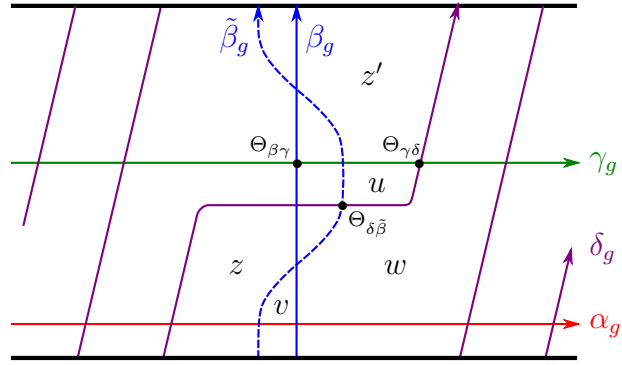

For each $\x \in \T_\alpha \cap \T_\beta$, let $\tilde\x \in \T_\alpha \cap \T_{\tilde\beta}$ be the nearest point. Indeed, every generator in $\T_\alpha \cap \T_{\tilde\beta}$ is of this form, and clearly $\AlNorm(\tilde\x) = \AlNorm(\x)$ and $\absgr(\tilde\x) = \absgr(\x)$. The differentials on the complexes $\ul\CF^t(\bm\alpha, \bm\beta, w; \GR)$ and $\ul\CF^t(\bm\alpha, \tilde{\bm\beta}, w; \GR)$ are given by:
\begin{align*}
\partial_{\alpha\beta}([\x,i]) &= \sum_{\y \in \T_\alpha \cap \T_\beta} \sum_{\substack{ \phi \in \pi_2(\x,\y) \\ \mu(\phi)=1}} \# \widehat\MM(\phi) \, T^{n_w(\phi) - n_v(\phi)} [\y, i-n_w(\phi)] \\
\partial_{\alpha\tilde\beta}([\tilde \x,i]) &= \sum_{\tilde\y \in \T_\alpha \cap \T_{\tilde\beta}} \sum_{\substack{ \phi \in \pi_2(\tilde\x,\y) \\ \mu(\phi)=1}} \# \widehat\MM(\phi) \, T^{n_v(\phi) - n_z(\phi)} [\tilde\y, i-n_w(\phi)]
\end{align*}
Note that in the exponents of $T$, we now use $v$ in place of whichever basepoint ($w$ or $z$) is contained within the same region. Let $\theta_{\alpha\beta}$ and $\theta_{\alpha\tilde\beta}$ be the trivializations defined by \eqref{eq: theta-def} and its $\tilde\beta$ analogue, and let $\JJ_{\alpha\beta}$ and $\JJ_{\alpha\tilde\beta}$ be the corresponding filtrations defined by \eqref{eq: ab-twisted-filt}.

Henceforth, we will treat $f_2^t$ and $h_1^t$ as mapping into $\ul\CF^t(\bm\alpha, \tilde{\bm\beta}, w; \GR)$, with \eqref{eq: f2-def} and \eqref{eq: h1-def} modified accordingly. Of course, all of the results of Sections \ref{sssec: f2-filt} and \ref{sssec: h1-filt} continue to hold. Thus, we may define maps
\[
\Psi_0, \tilde \Psi_0 \co \ul\CF^t(\Sigma, \bm\alpha, \bm\beta, w; \GR) \to \ul\CF^t(\Sigma, \bm\alpha, \tilde{\bm\beta}, w; \GR)
\]
by
\[
\Psi_0 = h_1 \circ f_0 + f_2 \circ h_0 \quad \text{and} \quad \tilde \Psi_0 = \tilde h_1 \circ f_0 + f_2 \circ \tilde h_0.
\]
In \cite{HeddenMarkFractional}, it is verified that $\Psi_0$ is a quasi-isomorphism. We must prove the analogous filtered statement:
\begin{proposition} \label{prop: Psi0-filt-QI}
Fix $t \in \N$ and $\epsilon>0$. For all $m$ sufficiently large, $\tilde \Psi^t_0$ is a filtered quasi-isomorphism with respect to the filtrations $\JJ_{\alpha\beta}$ and $\JJ_{\alpha\tilde\beta}$.
\end{proposition}
It is immediate from our previous results that $\tilde \Psi^t_0$ is a filtered map (since it is a sum of compositions of filtered maps), but more work will be required to see that is a filtered quasi-isomorphism.

The key to understanding $\tilde \Psi^t_0$ is to relate it to the chain isomorphism 
\begin{equation} \label{eq: Phi0}
\Phi_0^\circ \co \ul\CF^\circ(\Sigma, \bm\alpha, \bm\beta, w; \GR) \to \ul\CF^\circ(\Sigma, \bm\alpha, \tilde{\bm\beta}, w; \GR)
\end{equation}
given by
\begin{equation} \label{eq: Phi0-def}
\Phi_0^\circ(T^s \cdot [\x, i]) =
\sum_{\tilde\y \in \T_\alpha \cap \T_{\tilde\beta}} \sum_{\substack{\psi \in \pi_2(\x, \Theta_{\beta\tilde\beta}, \tilde\y) \\ \mu(\psi)=0 }} \#\MM(\psi) \, T^{s + n_w(\psi)-n_z(\psi)} \cdot [\tilde\y, i-n_w(\psi) ].
\end{equation}
The verification that $\Phi_0^\circ$ is an isomorphism uses a standard energy filtration argument, as in \cite{HeddenMarkFractional}: we have $\Phi_0^\circ(T^s \cdot [\x,i]) = T^s \cdot [\tilde x, i] + {}$ lower order terms. Moreover, for any $\psi \in \pi_2(\x, \Theta_{\beta\tilde\beta}, \tilde\y)$, we have $\AlNorm(\x) - \AlNorm(\tilde\y) = n_z(\psi) - n_w(\psi)$. It follows easily that $\Phi_0^\circ$ is a filtered isomorphism with respect to $\JJ_{\alpha\beta}$ and $\JJ_{\alpha\tilde\beta}$.

%
%
%

Consider the map
\begin{equation} \label{eq: g0-Phi0}
g_0^\circ \co \ul\CF^\circ(\Sigma, \bm\alpha, \bm\beta, w; \GR) \to \ul\CF^\circ(\Sigma, \bm\alpha, \tilde{\bm\beta}, w; \GR)
\end{equation}
defined by: 
\begin{equation}
\label{eq: g0-def}
g_0^\circ(T^s \cdot [\x,i]) = \sum_{\y \in \T_\alpha \cap \T_\beta'}  \sum_{\substack{\sigma \in \pi_2(\x, \Theta_{\beta\gamma}, \Theta_{\gamma\delta}, \Theta_{\delta\beta'}, \y) \\ \mu(\sigma)=-2 \\ \mathclap{s + n_w(\sigma) - n_v(\sigma) \equiv 0 \pmod m}}}   \#\MM(\sigma) \, T^{n_v(\sigma) - n_z(\sigma)} [\y, i-n_w(\sigma)].
\end{equation}
The following lemma is a slight refinement of the statement from \cite{HeddenMarkFractional} and immediately implies that $\Psi_0^\circ$ is a quasi-isomorphism:
\begin{lemma} \label{lemma: g0-htpy}
The map $g_0^\circ$ is a chain homotopy between $\Psi_0^\circ$ and $\Phi_0^\circ + U^m \Phi'$, where $\Phi'$ is some other chain map. In particular, when $m > t$, $g_0^t$ is in fact a chain homotopy between $\Psi_0^t$ and $\Phi_0^t$.
\end{lemma}

\begin{proof}
Just as in \cite[p.~122]{OSzSurgery} and \cite[p.~36]{HeddenMarkFractional}, this comes down to a model computation of $F^\circ_{\beta\gamma\delta\tilde\beta}(\Theta_{\beta\gamma} \otimes \Theta_{\gamma\delta} \otimes \Theta_{\delta\tilde\beta})$, which is made only slightly more complicated by the presence of twisted coefficients. The key observation is that there are $m$ distinguished holomorphic rectangles in $\pi_2(\Theta_{\beta\gamma}, \Theta_{\gamma\delta}, \Theta_{\delta\tilde\beta}, \Theta_{\beta\tilde\beta})$, which are the only classes with $n_w=0$. By looking more closely at the computation there (specifically \cite[Figure 2]{HeddenMarkFractional}), one can verify that all other holomorphic rectangles have $n_w$ divisible by $m$. (Compare Lemma \ref{lemma: bgd-triangles} above.) Therefore, one of the terms that arises in the count of degenerations of holomorphic pentagons is of the form $F^\circ_{\alpha\beta\tilde\beta}(- \otimes (\Theta_{\beta\tilde\beta} + U^m \Theta'))$, where $\Theta'$ is some element of $\ul{\mathbf{CF}}^{\le0} (\Sigma, \bm\beta, \tilde{\bm\beta}, w)$, so it has the form described in the lemma.
\end{proof}

\begin{remark} \label{rmk: pentagon-signs}
To generalize Lemma \ref{lemma: g0-htpy} to $\Z$ coefficients, one would need to show that the $m$ distinguished holomorphic rectangles mentioned above all count with the same sign. This statement is implicitly asserted, but without justification, in \cite{OSzSurgery}. Intuitively, Ozsv\'ath--Stipsicz--Szab\'o's approach to sign assignments from \cite{OSSzSign} could be useful here; because the boundaries of domains of these rectangles all interact with the orientations of the $\beta$, $\gamma$, $\delta$, and $\tilde \beta$ curves in the same way, the rectangles should all count with the same sign. However, that argument is far from rigorous; to our knowledge, it is not known whether the combinatorial sign assignments from \cite{OSSzSign} actually agree with the orientations of moduli spaces, even for bigons.
\end{remark}

Before we discuss the filtration shifts, we need to state analogues of the results of Section \ref{ssec: grading-shift} for pentagons. To begin, let $V$ be the $(\beta,\tilde\beta)$ periodic domain with $\partial V = \beta_g - \tilde\beta_g$, $n_v(V) = 1$, and $n_w(V) = n_z(V) = n_{z'}(V) = n_u(V) = 0$. Let $\tilde P_\gamma$, $\tilde P_\delta$, and $\tilde Q$ be the analogues of $P_\gamma$, $P_\delta$, and $Q$ with $\beta$ circles replaced by $\beta'$ circles: to be precise, up to thin domains, we have
\[
[\tilde P_\gamma] = [P_\gamma] + k[V], \quad [\tilde P_\delta] = [P_\delta]+(k+dm)[V], \quad \text{and} \quad
[\tilde Q] = [Q]-m[V].
\]

The Heegaard diagram determines a $4$-manifold $X_{\alpha\beta\gamma\delta\tilde\beta}$, which admits various decompositions into the pieces described in Section \ref{ssec: cobordisms}; for instance, we have
\[
X_{\alpha\beta\gamma\delta\tilde\beta} = X_{\alpha\beta\gamma\delta} \cup_{Y_{\alpha\delta}} X_{\alpha\delta\tilde\beta} = X_{\alpha\beta\gamma} \cup_{Y_{\alpha\gamma}} X_{\alpha\gamma\delta\tilde\beta} .
\]
In the intersection pairing form on $H_2(X_{\alpha\beta\gamma\delta\tilde\beta})$, we have
\[
[V] \cdot [V] = [V] \cdot [Q] = [V] \cdot [\tilde Q] = 0,
\]
and all other intersection numbers can be deduced accordingly.


\begin{lemma}
For any $\x \in \T_\alpha \cap \T_\beta$, $\tilde\y \in \T_\alpha \cap \T_{\tilde\beta}$, and $\sigma \in \pi_2(\x, \Theta_{\beta\gamma}, \Theta_{\gamma\delta}, \Theta_{\delta\tilde\beta}, \tilde \y)$, we have:
\begin{align}
\label{eq: abgdb-alex}
\Al(\x) - \Al(\tilde \y) &= d n_z(\sigma) + k n_{z'}(\sigma) + dm n_u(\sigma) - (k+dm+d) n_w(\sigma) -dm \\
\label{eq: abgdb-c1-Pg}
\gen{c_1(\spincs_w(\sigma)), [P_\gamma]} &= 2 \Al(\x) + 2dn_w(\sigma) - 2d n_v(\sigma) - k \\
\label{eq: abgdb-c1-Pdt}
\gen{c_1(\spincs_w(\sigma)), [\tilde P_\delta]} &= 2 \Al(\tilde\y) + 2dn_z(\sigma) - 2d n_v(\sigma) - (k+dm) \\
\label{eq: abgdb-c1-Q}
\gen{c_1(\spincs_w(\sigma)), [Q]} &= m(2n_u(\sigma) - 2n_{z'}(\sigma) - 1) \\
\label{eq: abgdb-c1-Qt}
\gen{c_1(\spincs_w(\sigma)), [\tilde Q]} &= m(2n_u(\sigma) - 2n_w(\sigma)-1) \\
\label{eq: abgdb-c1-V}
\gen{c_1(\spincs_w(\sigma)), [V]} &= 2n_w(\sigma) - 2n_{z'}(\sigma)
\end{align}
\end{lemma}

\begin{proof}
This follows from Propositions \ref{prop: alex-triangle} and \ref{prop: alex-rectangle} in much the same way as Proposition \ref{prop: alex-rectangle} follows from Proposition \ref{prop: alex-triangle}.
\end{proof}

Just as with the other maps, $g_0^\circ$ decomposes as a sum.
\[
g_0^\circ = \sum_{\spincv \in \Spin^c_0(X_{\alpha\beta\gamma\delta\tilde\beta})} g^\circ_{0,\spincv},
\]
where $g^\circ_{0,\spincv}$ counts pentagons $\sigma$ with $\spincs_w(\sigma) = \spincv$. To be precise, for any $\x \in \T_\alpha \cap \T_\beta$ and any $r \equiv -\AlNorm_{w,z}(\x) \pmod \Z$, we have: 
\begin{align*}
(\theta_{\alpha\tilde\beta} \circ g^\circ_0 & \circ \theta_{\alpha\beta}^{-1}) ([\x,i] \otimes T^r) \\
&=
\sum_{\tilde\y \in \T_\alpha \cap \T_{\tilde\beta}}  \sum_{\substack{\sigma \in \pi_2(\x, \Theta_{\beta\gamma}, \Theta_{\gamma\delta}, \Theta_{\delta\tilde\beta}, \tilde\y) \\ \mu(\sigma)=-2 \\ \mathclap{r+\AlNorm_{w,z}(\x) + n_w(\rho) - n_v(\rho) \equiv 0 \pmod m}}}   \#\MM(\sigma) \, [\tilde\y, i-n_w(\sigma)] \otimes T^{-\AlNorm_{w,z} (\tilde\y) + n_v(\sigma) - n_z(\sigma)} \\
&=
\sum_{\tilde\y \in \T_\alpha \cap \T_{\tilde\beta}}  \sum_{\substack{\sigma \in \pi_2(\x, \Theta_{\beta\gamma}, \Theta_{\gamma\delta}, \Theta_{\delta\tilde\beta}, \tilde\y) \\ \mu(\sigma)=-2 \\ \mathclap{r+ \frac{1}{2d} \left( \gen{c_1(\spincs_w(\sigma)), [P_\gamma]} \right) \equiv 0 \pmod m}}}   \#\MM(\sigma) \, [\tilde\y, i-n_w(\sigma)] \otimes T^{-\frac{1}{2d} \left(\gen{c_1(\spincs_w(\sigma)), [\tilde P_\delta]} + k+dm \right)}
\end{align*}
In particular, given a spin$^c$ structure $\spincv$, let $r$ and $s$ be the numbers satisfying
\begin{equation} \label{eq: g0-r-s-def}
\begin{aligned}
\frac{-k-dm}{2d} &\le r < \frac{-k+dm}{2d} & r &\equiv -\frac{1}{2d}( \gen{c_1(\spincv), [P_\gamma]} +k) \pmod m \\
\frac{-k-dm}{2d} &\le s < \frac{-k+dm}{2d} & s &\equiv -\frac{1}{2d}( \gen{c_1(\spincv), [\tilde P_\delta]} +k+dm) \pmod m
\end{aligned}
\end{equation}
which are both among the exponents appearing in the decompositions of $\ul\CF^\circ(\bm\alpha, \bm\beta, w; \GR)$ and $\ul\CF^\circ(\bm\alpha, \tilde{\bm\beta}, w; \GR)$ given by \eqref{eq: ab-twisted-decomp}. Then the composition $\theta \circ g_{0,\spincv}^\circ \circ \theta^{-1}$ takes
\[
\CF^\circ(\bm\alpha, \bm\beta, \spincs, w) \otimes T^r \GR \to \CF^\circ(\bm\alpha, \tilde{\bm\beta}, \tilde\spincs, w) \otimes T^s \GR,
\]
where $\spincs = \spincv |_{Y_{\alpha\beta}}$ and $\tilde\spincs = \spincv |_{Y_{\alpha\tilde\beta}}$. (Analogous statements hold for $\Phi_0$.)

The grading shift of $g^\circ_{0,\spincv}$ is
\[
\absgr(g^\circ_{0,\spincv})
= \frac{c_1(\spincv)^2  +m + 4}{4}.
\]
This can be expressed in terms of $c_1$ evaluations in various ways. For instance:
\begin{align}
\label{eq: g0-gr-Pg-R-Pdt}
\absgr(g^\circ_{0,\spincv})
&= \frac{\gen{c_1(\spincv), [P_\gamma]} ^2}{4dk} - \frac{\nu^2 \gen{c_1(\spincv), [R]}^2}{4km(k+dm)} - \frac{\gen{c_1(\spincv), [\tilde P_\delta] }^2}{4d(k+dm)} + \frac{m+4}{4} \\
\label{eq: g0-gr-Pg-Pgt-Qt}
&= \frac{\gen{c_1(\spincv), [P_\gamma]} ^2 - \gen{c_1(\spincv), [\tilde P_\gamma]}^2}{4dk} - \frac{\gen{c_1(\spincv), [\tilde Q] }^2}{4m} + \frac{m+4}{4} \\
\label{eq: g0-gr-Pd-Pdt-Q}
&= \frac{\gen{c_1(\spincv), [P_\delta]} ^2 - \gen{c_1(\spincv), [\tilde P_\delta]}^2}{4d(k+dm)} - \frac{\gen{c_1(\spincv), [Q] }^2}{4m} + \frac{m+4}{4}
\end{align}

\begin{lemma} \label{lemma: g0-trunc}
Fix $t \in \N$ and $0<\epsilon < \epsilon' < 1$. For all $m$ sufficiently large, if $\spincv$ is any spin$^c$ structure for which $g^t_{0,\spincv} \ne 0$, then the following implications hold:
\begin{enumerate}
\item \label{it: g0-trunc-Pg-Pgt}
If $\abs{\gen{c_1(\spincv), [P_\gamma]}} < \epsilon dm$ and $\gen{c_1(\spincv), [\tilde Q] } = \pm m$, then
\begin{equation} \label{eq: g0-trunc-Pgt}
\abs{\gen{c_1(\spincv), [\tilde P_\gamma]}} < \epsilon' dm.
\end{equation}
\item \label{it: g0-trunc-Pdt-Pd}
If $\gen{c_1(\spincv), [Q]} = \pm m$ and $\abs{\gen{c_1(\spincv), [\tilde P_\delta]}} < (1+\epsilon)(k+dm)$, then \begin{equation} \label{eq: g0-trunc-Pd}
\abs{\gen{c_1(\spincv), [P_\delta]}} < (1+\epsilon')(k+dm).
\end{equation}
\item \label{it: g0-trunc-V-Q}
If $\gen{c_1(\spincv), [V]}=0$, then
\[
\gen{c_1(\spincv), [Q]} = \gen{c_1(\spincv), [\tilde Q]}  = \pm m.
\]
\end{enumerate}
\end{lemma}

\begin{proof}
The gradings on $\ul\CF^t(\bm\alpha, \bm\beta, w; \GR)$ and $\ul\CF^t(\bm\alpha, \tilde{\bm\beta}, w; \GR)$ are bounded (above and below) independently of $m$. Thus, for some constant $C>0$, if $g^t_{0,\spincv} \ne 0$, then
\begin{equation} \label{eq: g0-trunc-gr}
-C \le \absgr(g^t_{0,\spincv}) \le C.
\end{equation}

To prove \eqref{it: g0-trunc-Pg-Pgt}, let us assume $k>0$; the case where $k<0$ proceeds almost identically. By \eqref{eq: g0-gr-Pg-Pgt-Qt}, we have:
\begin{align*}
-C &\le \frac{\gen{c_1(\spincv), [P_\gamma]} ^2 - \gen{c_1(\spincv), [\tilde P_\gamma]}^2}{4dk} - \frac{\gen{c_1(\spincv), [\tilde Q] }^2}{4m} + \frac{m+4}{4} \\
&< \frac{\epsilon^2 d^2 m^2  - \gen{c_1(\spincv), [\tilde P_\gamma]}^2}{4dk} - \frac{m^2}{4m} + \frac{m+4}{4},
\end{align*}
so
\[
\gen{c_1(\spincv), [\tilde P_\gamma]}^2 < \epsilon^2 d^2 m^2 + C',
\]
where $C' = 4dk(C+1)$. Therefore,
if $m > \sqrt{\frac{C'}{d^2({\epsilon'}^2 - \epsilon^2)}}$, we obtain
\[
\abs{\gen{c_1(\spincv), [\tilde P_\gamma]}} < \epsilon' dm,
\]
as required.

The proof of \eqref{it: g0-trunc-Pdt-Pd} proceeds similarly using \eqref{eq: g0-gr-Pd-Pdt-Q}.

For \eqref{it: g0-trunc-V-Q}, assuming that $\gen{c_1(\spincv), [V]}=0$, we have $\gen{c_1(\spincv), [P_\gamma]} = \gen{c_1(\spincv), [\tilde P_\gamma]}$ and $\gen{c_1(\spincv), [Q]} = \gen{c_1(\spincv), [\tilde Q]}  = em$ for some odd integer $e$. By \eqref{eq: g0-gr-Pg-Pgt-Qt}, we have
\[
(e^2-1)m   \le 4(C+1).
\]
Thus, for $m$ sufficiently large, we deduce that $e=\pm 1$, as required.
\end{proof}

\begin{definition} \label{def: g0-trunc}
Fix $0 < \epsilon < \frac13$. Call a spin$^c$ structure $\spincv$ \emph{good} if it satisfies both
\begin{align}
\label{eq: g0-trunc-Pg} \abs{\gen{c_1(\spincv), [P_\gamma]}} &< \epsilon dm  \\
\label{eq: g0-trunc-Qt} \gen{c_1(\spincv), [\tilde Q]} &= \pm m, \\
\intertext{or it satisfies both}
\label{eq: g0-trunc-Pdt} \abs{\gen{c_1(\spincv), [\tilde P_\delta]}} &< (1+\epsilon)(k+dm)  \\
\label{eq: g0-trunc-Q} \gen{c_1(\spincv), [Q]} &= \pm m,  \\
\intertext{or it satisfies}
\label{eq: g0-trunc-V} \gen{c_1(\spincv), [V]} &= 0.
\end{align}
Let $\tilde g^t_0$ denote the sum of all terms $g^t_{0,\spincv}$ for which $\spincv$ is good.
\end{definition}

\begin{lemma} \label{lemma: g0-filt}
Fix $t \in \N$ and $\epsilon>0$. For all $m$ sufficiently large, the map $\tilde g^t_0$ is filtered with respect to the filtrations $\JJ_{\alpha\beta}$ and $\JJ_{\alpha\tilde\beta}$.
\end{lemma}

\begin{proof}
It suffices to show that each nonzero term $g^t_{0,\spincv}$ in the definition of $\tilde g^t_0$ is filtered. We will start by looking at an arbitrary term $g^t_{0,\spincv}$, and then specialize to the case where $\spincv$ is good (which will justify our definition).

Let $r$ and $s$ be as in \eqref{eq: g0-r-s-def}. Write
\begin{align*}
\gen{c_1(\spincv), [P_\gamma]} &= -2dr -k + 2pdm \\
\gen{c_1(\spincv), [\tilde P_\delta]} &= -2ds -k + (2q-1) dm
\end{align*}
so that we have
\begin{align}
\label{eq: g0-filt-Pg-bounds} (2p-1)dm &< \gen{c_1(\spincv), [P_\gamma]} \le (2p+1)dm \\
\label{eq: g0-filt-Pdt-bounds} (2q-2)dm &< \gen{c_1(\spincv), [\tilde P_\delta]} \le 2q dm.
\end{align}
Let us assume that $\gen{c_1(\spincv), [\tilde Q]} = em$, where $e$ is an odd integer.

For any $\sigma \in \pi_2(\x, \Theta_{\beta\gamma}, \Theta_{\gamma\delta}, \Theta_{\delta\tilde\beta}, \tilde \y)$ contributing to $g^t_{0,\spincv}$, we compute:
\begin{align*}
\JJ_{\alpha\beta}([\x,i] \otimes T^r) &- \JJ_{\alpha\tilde\beta}([\tilde\y,i-n_w(\sigma)] \otimes T^s) \\
&= i - \frac{2d r +k+d}{2k} - (i-n_w(\sigma)) + \frac{2d s +k+d}{2k} \\
&= \frac{d(s-r)}{k} + n_w(\sigma) \\
&= \frac{\gen{c_1(\spincv), [P_\gamma]-[\tilde P_\delta]} + (2q-2p-1)dm }{2k} + n_w(\sigma) \\
&= \frac{\gen{c_1(\spincv), [P_\delta] + d[Q] -[\tilde P_\delta]} + (2q-2p-1)dm }{2k} + n_w(\sigma) \\
&= \frac{\gen{c_1(\spincv), -(k+dm)[V] + d[Q]} + (2q-2p-1)dm }{2k} + n_w(\sigma) \\
&= \frac{2(k+dm)(n_{z'}(\sigma)-n_w(\sigma)) + dm(2n_u(\sigma) - 2n_{z'}(\sigma) -1) }{2k} \\
& \qquad + \frac{(2q-2p-1)dm }{2k} + n_w(\sigma) \\
&= \frac{2k n_{z'}(\sigma)- 2 dm n_w(\sigma) + 2dm n_u(\sigma)  -dm  }{2k} + \frac{(2q-2p-1)dm }{2k}  \\
&= n_{z'}(\sigma) + \frac{d \gen{c_1(\spincv), [\tilde Q]} + (2q-2p-1)dm }{2k}  \\
&= n_{z'}(\sigma) + \frac{(e+2q-2p-1)dm }{2k}
\end{align*}
Thus, it suffices to show that $e+2q-2p-1 = 0$.

Note that
\begin{align*}
\gen{c_1(\spincv), [P_\gamma]} &= \gen{c_1(\spincv), [\tilde P_\gamma]} - k\gen{c_1(\spincv), [V]} \\
&= \gen{c_1(\spincv), [\tilde P_\delta]} + d\gen{c_1(\spincv), [\tilde Q]} - k\gen{c_1(\spincv), [V]} \\
\intertext{so}
\gen{c_1(\spincv), [P_\gamma] - [\tilde P_\delta]} &= edm  - k\gen{c_1(\spincv), [V]}.
\end{align*}
Combining this with the bounds \eqref{eq: g0-filt-Pg-bounds} and \eqref{eq: g0-filt-Pdt-bounds}, we obtain:
\begin{equation} \label{eq: g0-filt-V-bounds}
2p-2q-1 < e   - \frac{k\gen{c_1(\spincv), [V]}}{dm} < 2p-2q+3
\end{equation}
In particular, if $\gen{c_1(\spincv), [V]}$ is small compared to $m$ (i.e., if $\abs{k \gen{c_1(\spincv), [V]}} < 2dm$), we immediately deduce that $e = 2p-2q+1$, which is precisely what we need. To ensure this, we now apply the hypotheses for $\tilde g^t_0$:
\begin{itemize}
\item
Suppose that $\spincv$ satisfies \eqref{eq: g0-trunc-Pg} and \eqref{eq: g0-trunc-Qt}, and hence also \eqref{eq: g0-trunc-Pgt} by Lemma \ref{lemma: g0-trunc}\eqref{it: g0-trunc-Pg-Pgt}, where we take $\epsilon' = 2\epsilon$. Then
\begin{align*}
\abs{k \gen{c_1(\spincv), [V]}} &= \abs{ \gen{c_1(\spincv), [P_\gamma] - [\tilde P_\gamma]}} \\
&\le \abs{\gen{c_1(\spincv), [P_\gamma]}} + \abs{\gen{c_1(\spincv), [\tilde P_\gamma]}}  \\
&< 3\epsilon dm \\
&< 2dm
\end{align*}
as required.

\item
Suppose that $ \spincv$ satisfies \eqref{eq: g0-trunc-Pdt} and \eqref{eq: g0-trunc-Q}. By Lemma \ref{lemma: g0-trunc}\eqref{it: g0-trunc-Pdt-Pd}, again taking $\epsilon'=2\epsilon$, we also have
\[
\abs{\gen{c_1(\spincv), [P_\delta]}} < (1+2\epsilon) dm.
\]
Therefore,
\begin{align*}
(k+dm) \abs{\gen{c_1(\spincv), [V]}} = \abs{\gen{c_1(\spincv), [\tilde P_\delta] - [P_\delta]}} &< (2+3\epsilon)(k+dm).
\end{align*}
Since $\gen{c_1(\spincv), [V]}$ is an even integer, it must equal either $-2$, $0$, or $2$. For $m$ sufficiently large, we again obtain $\abs{k \gen{c_1(\spincv), [V]}} <2dm$, as required.

\item
Finally, if $\spincv$ satisfies \eqref{eq: g0-trunc-V}, then the conclusion is obvious. \qedhere
\end{itemize}
\end{proof}

\begin{lemma} \label{lemma: g0-trunc-htpy}
Fix $t \in \N$. For all $m$ sufficiently large, the map $\tilde g^t_0$ is a (filtered) chain homotopy between $\tilde \Psi^t_0$ and $\Phi^t_0$.
\end{lemma}

\begin{proof}
We begin by reviewing the proof that $g_0$ gives a homotopy between $\Psi_0$ and $\Phi_0$, and then see how to modify it to incorporate the notion of good spin$^c$ structures. (See \cite{HeddenMarkFractional} for the most complete treatment.)

For any class $\sigma \in \pi_2(\x, \Theta_{\beta\gamma}, \Theta_{\gamma\delta}, \Theta_{\delta\tilde\beta}, \tilde\y)$ with $\mu(\sigma) = -1$ (and hence $\dim \MM(\sigma)=1$), the possible ends of $\MM(\sigma)$ correspond to the following six types of degenerations:
\begin{enumerate} [label=(P-\arabic*)] \setcounter{enumi}{-1}
\item \label{it: g0-ends-pent}
A concatenation of a pentagon $\sigma'$ (with $\spincs_w(\sigma') = \spincs_w(\sigma)$) with either an $(\alpha,\beta)$, $(\beta,\gamma)$, $(\gamma,\delta)$, $(\delta, \tilde\beta)$, or $(\alpha,\tilde\beta)$ bigon.
\item \label{it: g0-ends-abgd-adb}
A decomposition $\sigma = \rho_1 * \psi_1$, where $\rho_1 \in \pi_2(\x, \Theta_{\beta\gamma}, \Theta_{\gamma\delta}, \a)$ and $\psi_1 \in \pi_2(\a, \Theta_{\delta\tilde\beta}, \tilde\y)$ for some $\a \in \T_\alpha \cap \T_\delta$.
\item \label{it: g0-ends-abg-agdb}
A decomposition $\sigma = \psi_2 * \rho_2$, where $\psi_2 \in \pi_2(\x, \Theta_{\beta\gamma}, \q)$ and $\rho_2 \in \pi_2(\q, \Theta_{\gamma\delta}, \Theta_{\delta\tilde\beta}, \tilde\y)$ for some $\q \in \T_\alpha \cap \T_\gamma$.
\item \label{it: g0-ends-bgd-abdb}
A decomposition $\sigma = \psi_3 * \rho_3$, where $\psi_3 \in \pi_2(\Theta_{\beta\gamma}, \Theta_{\gamma\delta}, \Theta_{\beta\delta})$ and $\rho_3 \in \pi_2(\x, \Theta_{\beta\delta}, \Theta_{\delta\tilde\beta}, \tilde\y)$.
\item \label{it: g0-ends-gdb-abgb}
A decomposition $\sigma = \psi_4 * \rho_4$, where $\psi_4 \in \pi_2(\Theta_{\gamma\delta}, \Theta_{\delta\tilde\beta}, \Theta_{\gamma\tilde\beta})$ and $\rho_4 \in \pi_2(\x, \Theta_{\beta\gamma}, \Theta_{\gamma\tilde\beta}, \tilde\y)$.
\item \label{it: g0-ends-bgdb-abb}
A decomposition $\sigma = \rho_5 * \psi_5$, where $\rho_5 \in \pi_2(\Theta_{\beta\gamma}, \Theta_{\gamma\delta}, \Theta_{\delta\tilde\beta}, \tilde \Theta)$ and $\psi_5 \in \pi_2(\x, \Theta, \tilde\y)$, where $\tilde \Theta \in \T_\beta \cap \T_{\tilde\beta}$.
\end{enumerate}

If we look at all such classes $\sigma$ with no restrictions on the spin$^c$ structures, ends of type \ref{it: g0-ends-pent} correspond to $\partial_{\alpha\tilde\beta} \circ g_0 + g_0 \circ \partial_{\alpha\beta}$ (since $\Theta_{\beta\gamma}$, $\Theta_{\gamma\delta}$, and $\Theta_{\delta\tilde\beta}$ are all cycles. Ends of types \ref{it: g0-ends-abgd-adb} and \ref{it: g0-ends-abg-agdb} correspond to $f_2 \circ h_0$ and $h_1 \circ f_0$, respectively. Ends of types \ref{it: g0-ends-bgd-abdb} and \ref{it: g0-ends-gdb-abgb} cancel in pairs, as seen in the proof of Lemma \ref{lemma: bgd-triangles}. Finally, ends of type \ref{it: g0-ends-bgdb-abb} correspond to $\Phi_0$. The fact that each $\MM(\sigma)$ has algebraically zero ends implies that
\[
f_2 \circ h_0 + h_1 \circ f_0 + \Phi_0 =  \partial_{\alpha\tilde\beta} \circ g_0 + g_0 \circ \partial_{\alpha\beta}
\]
as required.

To prove that $\tilde g^t_0$ gives a homotopy between $\tilde \Psi^t_0$ and $\Phi^t_0$, we must see what happens when we restrict our attention to good spin$^c$ structures.

\begin{itemize}
\item
If $\MM(\sigma)$ has an end of type \ref{it: g0-ends-abgd-adb}, then the evaluations of $c_1(\spincv)$ on $[P_\gamma]$, $[P_\delta]$, $[Q]$, and $[R]$ are determined solely by $\rho_1$, while $\gen{c_1(\spincv), [\tilde P_\delta]}$ is determined solely by $\psi_1$. In particular, since $\psi_1$ counts for $f_2^t$, Lemma \ref{lemma: f2-trunc} implies that \eqref{eq: g0-trunc-Pdt} holds (i.e. $\abs{\gen{c_1(\spincv), \tilde P_\delta}} < (1+\epsilon) (k+dm)$). We claim that $\rho_1$ counts for $\tilde h_0^t$ (see Definition \ref{def: h0-trunc}) iff $\spincv$ is good (see Definition \ref{def: g0-trunc}).

Because \eqref{eq: g0-trunc-Pdt} holds, one of the criteria for $\rho_1$ counting for $\tilde h_0^t$ (namely \eqref{eq: h0-filt-Q-bound}) coincides precisely with one of the criteria for $\spincv$ being good (namely \eqref{eq: g0-trunc-Q}). We thus must simply show that the remaining criteria in each definition are equivalent.
\begin{itemize}
\item
If $\abs{\gen{c_1(\spincv), [P_\gamma]}} < \epsilon dm$ and $\abs{\gen{c_1(\spincv), [R]}} < \frac{m(k+dm)}{\nu}$ as in Definition \ref{def: h0-trunc}, then $\gen{c_1(\spincv), [\tilde Q]} = \pm m$ by Remark \ref{rmk: h1-trunc-simpler}, and hence $\spincv$ is good.

\item
If $\abs{\gen{c_1(\spincv), [P_\gamma]}} < \epsilon dm$ and $\gen{c_1(\spincv), [\tilde Q]} = \pm m$ as in Definition \ref{def: g0-trunc}, then Maslov grading considerations show that $\abs{\gen{c_1(\spincv), [\tilde P_\gamma]}} < 2 \epsilon dm$ (just as in Lemma \ref{lemma: g0-trunc}\eqref{it: g0-trunc-Pg-Pgt}). Therefore,
\begin{align*}
\abs{\gen{c_1(\spincv), [R]}} &= \abs{\gen{c_1(\spincv), \frac{m}{\nu}[\tilde P_\gamma] + \frac{k}{\nu}[Q]}} \\
&\le \frac{m}{\nu} \abs {\gen{c_1(\spincv), [\tilde P_\gamma]}} + \frac{\abs{k}}{\nu} \abs{\gen{c_1(\spincv), [Q]}} \\
&< \frac{m}{\nu} 2\epsilon dm + \frac{\abs{k}}{\nu} m \\
&< \frac{m(\abs{k}+2\epsilon dm)}{\nu} \\
&< \frac{m(k+dm)}{\nu}.
    \end{align*}
(If $k>0$, the last inequality is automatic; if $k<0$, it holds provided that $m > \frac{-2k}{d(1-2\epsilon)}$.) Therefore, $\rho_1$ counts for $\tilde h^t_0$ via \eqref{eq: h0-filt-R-bound}.

\item
If $\gen{c_1(\spincv), [V]}=0$, then Maslov grading considerations as in Lemma \ref{lemma: g0-trunc}\eqref{it: g0-trunc-V-Q} show that $\gen{c_1(\spincv), [Q]} = \pm m$, and hence $\rho_1$ counts for $\tilde h^t_0$ via \eqref{eq: h0-filt-Q-bound}.
\end{itemize}

The claim is thus proved. The case where $\MM(\sigma)$ has an end of type \ref{it: g0-ends-abg-agdb} is handled similarly, with fewer cases to check.

\item
Next, suppose $\MM(\sigma)$ has an end of type \ref{it: g0-ends-bgd-abdb}. If we assume that $m \ge t$, we see that $\psi_3$ equals one of the classes $\tau_0^{\pm}$ from the proof of Lemma \ref{lemma: bgd-triangles}; without loss of generality, assume that $\psi_3 = \tau_0^+$, so that $\gen{c_1(\spincv), [Q]} = m$. Let $\sigma' = \tau_0^- * \rho_3$, which is the class that provides the canceling end, and let $\spincv' = \spincs_w(\sigma')$. Then $\DD(\sigma') = \DD(\sigma) + Q$, so $\spincv' = \spincv + \PD[Q]$, so $\gen{c_1(\spincv'), [Q]} = -m$. We claim that $\spincv$ is good iff $\spincv'$ is good. It will follow that when considering contributions coming from only good spin$^c$ structures, ends of type \ref{it: g0-ends-bgd-abdb} cancel in pairs.

Since
\[
\gen{c_1(\spincv), [\tilde P_\delta]} = \gen{c_1(\spincv'), [\tilde P_\delta]} = \gen{c_1(\spincs_w(\rho_3)), [\tilde P_\delta]},
\]
$\spincv$ satisfies \eqref{eq: g0-trunc-Pdt} iff $\spincv'$ does. For the other two criteria from Definition \ref{def: g0-trunc}, suppose that $\spincv$ is good; the converse follows similarly. There are two cases to consider:

\begin{itemize}
\item If $\gen{c_1(\spincv), [V]}=0$ , then (equivalently) $\abs{\gen{c_1(\spincv), [\tilde Q]}} = m$. We have:
\begin{align*}
\gen{c_1(\spincv'), [V]} &= \gen{c_1(\spincv) + 2\PD[Q], [V]} \\
 &= \gen{c_1(\spincv), [V]} + 2[Q]\cdot [V] \\
 &=0
\end{align*}
Thus, $\spincv'$ is good.

\item
If $\gen{c_1(\spincv), [V]} \ne 0$, then we must be in the case where $\gen{c_1(\spincv), [\tilde Q]} = -m$ (and hence $\gen{c_1(\spincv), [V]} = 2$) and $\abs{\gen{c_1(\spincv), [P_\gamma]}} < \epsilon dm$.

As in Lemma \ref{lemma: g0-trunc}\eqref{it: g0-trunc-Pdt-Pd}, we have
\[
-C < \frac{\gen{c_1(\spincv),[P_\delta]}^2 - \gen{c_1(\spincv), [\tilde P_\delta]}^2}{4d(k+dm)} < C,
\]
where $C$ is a constant that is independent of $m$. Note that
\begin{align*}
\gen{c_1(\spincv), [P_\delta]} &= \gen{c_1(\spincv), [\tilde P_\delta] - (k+dm)[V]} 
= \gen{c_1(\spincv), [\tilde P_\delta]} - 2(k+dm).
\end{align*}
Therefore, we have:
\begin{gather*}
-C < \frac{\left(  \gen{c_1(\spincv), [\tilde P_\delta]} - 2(k+dm)\right)^2 - \gen{c_1(\spincv), [\tilde P_\delta]}^2}{4d(k+dm)} < C \\
-C < \frac{ -4(k+dm)\gen{c_1(\spincv), [\tilde P_\delta]} + 4(k+dm)^2 }{4d(k+dm)} < C \\
-Cd <  -\gen{c_1(\spincv), [\tilde P_\delta]} + k+dm  < Cd \\
k+dm -Cd <  \gen{c_1(\spincv), [\tilde P_\delta]}   < k+dm + Cd
\end{gather*}
If $m$ is sufficiently large, it follows that $\abs{\gen{c_1(\spincv), [\tilde P_\delta]}} < (1+\epsilon)(k+dm)$. And since $\gen{c_1(\spincv), [\tilde P_\delta]} = \gen{c_1(\spincv'), [\tilde P_\delta]}$, we deduce that $\spincv'$ is good.
\end{itemize}
The claim is thus proved. The case of where $\MM(\sigma)$ has an end of type \ref{it: g0-ends-gdb-abgb} is handled similarly.

\item
Finally, ends of type \ref{it: g0-ends-bgdb-abb} correspond to $\Phi_0^t$, as in Lemma \ref{lemma: g0-htpy}. If $\rho$ has an end of this type, then $n_w(\rho) = n_{z'}(\rho)$, since $w$ and $z'$ are in the same region of both $(\Sigma, \bm\beta, \bm\gamma, \bm\delta, \bm\beta')$ and $(\Sigma, \alpha, \bm\beta, \bm\beta')$ (see Figure \ref{fig: beta-tilde}). By \eqref{eq: abgdb-c1-V}, $\gen{c_1(\spincs_w(\rho)), [V]}=0$, so $\spincv$ must be good.
\end{itemize}
\end{proof}

We have thus concluded the proof of Proposition \ref{prop: Psi0-filt-QI}, and hence of Theorem \ref{thm: CFt-cone-f2}. \qed

\section{Proof of the filtered mapping cone formula} \label{sec: proof}

We now turn to the proof of the filtered mapping cone formula. As noted in Section \ref{sec: mapping-cone}, it suffices to prove the mapping cone formula for $\CF^t$, namely Proposition \ref{prop: mapping-cone-CFt}.


\begin{proof}
Let us assume that $k>0$; the case where $k<0$ is similar and is left to the reader. Fix $t \in \N$. By Theorem \ref{thm: CFt-cone-f2}, for sufficiently large $m$ and a well-adapted diagram $(\Sigma, \bm\alpha, \bm\beta, \bm\delta, w, z, z')$, the map
\[
\begin{pmatrix} f_1^t \\ h_1^t \end{pmatrix}  \co \CF^t(\bm\alpha, \bm\gamma, w) \to \operatorname{Cone}(f_2^t)
\]
is a filtered homotopy equivalence.

Let us start by looking closely at how these maps interact with the spin$^c$ decomposition of $\CF^t(\Sigma, \bm\alpha, \bm\gamma, w)$.


Fix a spin$^c$ structure $\spinct \in \Spin^c(Y_\lambda(K))$. As in the Introduction, let $\{s_l \mid l \in \Z\}$ be the arithmetic sequence (with step $k/d$) characterized by
\begin{equation}\label{eq: sl-def}
d s_l \equiv   \frac{2k \AlNorm_{Y_\lambda, K_\lambda}(\spinct) + d-k}{2} \pmod k \quad \text{and} \quad \frac{(2l-1)k}{2}  < ds_l \le \frac{(2l+1)k}{2}
\end{equation}
These numbers are precisely the values of $\AlNorm_{Y,K}(\xi)$ for all $\xi \in G_{Y_\lambda, K_\lambda}^{-1}(\spinct)$.

First, consider the restriction of $f_1^t$ to $\CF^t(\bm\alpha, \bm\gamma, w, \spinct)$. For all spin$^c$ structures $\spincv \in \Spin^c_0(W_{\alpha\gamma\delta})$ extending $\spinct$, the values of $\gen{c_1(\spincv), [R]}$ are congruent modulo $\frac{2mk}{\nu} = 2kp$. Indeed, let $\spincv_l$ be the spin$^c$ structure extending $\spinct$ with
\[
\frac{(2l-1) mk}{\nu} < \gen{c_1(\spincv_l), [R]} \le \frac{(2l+1)mk}{\nu}.
\]
For any triangle $\psi \in \pi_2(\q, \Theta_{\gamma\delta}, \a)$ representing $\spincv$, \eqref{eq: agd-c1-q} gives
\begin{align*}
\frac{\nu}{2m} \gen{c_1(\spincv_l), [R]} &=   \Al(\q) + k n_u(\psi) - k n_{z'}(\psi) + \frac{d-k}{2}  \\
&\equiv k \AlNorm_{Y_\lambda, K_\lambda}(\spinct) + \frac{d-k}{2} \pmod k \\
&\equiv ds_l \pmod k
\end{align*}
and therefore
\[
\gen{c_1(\spincv_l), [R]} =  \frac{2dm s_l}{\nu}.
\]
Moreover, by Lemma \ref{lemma: f1-trunc}, if $f^t_{1,\spincv_l} \ne 0$, then
\[
-\frac{m(k+dm)}{\nu} < \gen{c_1(\spincv_l), [R]} < \frac{m(k+dm)}{\nu},
\]
so it follows that
\[
\abs{l} < \frac{k+dm}{k} + \frac12.
\]
Let $L$ be the largest integer satisfying this constraint; thus, the only possibly nonzero terms in the restriction of $f_1^t$ to $\CF^t(\bm\alpha, \bm\gamma, w, \spinct)$ are $f^t_{1, \spincv_l}$ for $l = -L, \dots, L$. For each such $l$, let $\spincu_l = \spincv_l | _{Y_{\alpha\delta}}$. By \eqref{eq: f1-trunc-su}, we have
\[
\gen{c_1(\spincv_l), [R]} = \frac{2dm s_{\spincu_l}}{\nu}
\]
and therefore $s_{\spincu_l} = s_l$.

Now, look at the spin$^c$ structures on $W_{\alpha\delta\beta}$ extending $\spincu_l$. For each $l=-L, \dots, L$,
recall that we have spin$^c$ structures $\spincx_{\spincu_l}, \spincy_{\spincu_l}$ on $W_{\alpha\delta\beta}$ which satisfy
\begin{align*}
\gen{c_1(\spincx_{\spincu_l}), [P_\delta]} &= 2ds_l - k-dm \\
\gen{c_1(\spincy_{\spincu_l}), [P_\delta]} &= 2ds_l + k+dm.
\end{align*}
For ease of notation, let us write $\spincx_l = \spincx_{\spincu_l}$ and $\spincy_l = \spincy_{\spincu_l}$. For $l=-L+1, \dots, L$, $\spincy_{l-1}$ and $\spincx_l$ have the same restriction to $Y$; denote this by $\spincs_l$. Note that $\spincs_l = \spincs_{l-1} + \PD[K]$, so the list $(\spincs_{-L+1}, \dots, \spincs_L)$ is cyclic with period $d$.
Moreover, we have
\begin{align*}
\gen{c_1(\spincx_l), [P_\delta]} &= 2ds_l - k-dm \\
&= 2ds_{l-1} + k-dm \\
&= \gen{c_1(\spincy_{l-1}), [P_\delta]} - 2dm.
\end{align*}
In particular, the images of the maps $\theta \circ f^t_{2, \spincx_l}$ and $\theta \circ f^t_{2, \spincy_{l-1}}$ both lie in the summand $\CF^t(\bm\alpha, \bm\beta, \spincs_l, w) \otimes T^{-s_l}$, using the decomposition \eqref{eq: ab-twisted-decomp}. Finally, Lemma \ref{lemma: f2-trunc} implies that the maps $\theta \circ f^t_{2, \spincx_{-L}}$ and $\theta \circ f^t_{2, \spincy_{L}}$ vanish.

Next, suppose $\spincv \in \Spin^c_0(X_{\alpha\gamma\delta\beta})$ is a spin$^c$ structure restricting to $\spinct$ for which $\gen{c_1(\spincv), [Q]} = \pm m$ and $h^t_{1,\spincv} \ne 0$. By Lemma \ref{lemma: h1-trunc-k-pos}, we have $\abs{c_1(\spincv), [R]} < \frac{m(k+dm)}{\nu}$ and $\abs{c_1(\spincv), [P_\delta]} < (1+\epsilon)(k+dm)$. This implies that for some $l \in \{-L, \dots, L\}$, $\spincv$ restricts to $\spincv_l$ on $X_{\alpha\gamma\delta}$ and to either $\spincx_l$ or $\spincy_l$ on $X_{\alpha\delta\beta}$ (and not $\spincx_{-L}$ or $\spincy_L$). Therefore, the image of $h^t_{1,\spincv}$ lies in one of the summands of \eqref{eq: ab-twisted-decomp} mentioned in the previous paragraph.


Thus we see that $\CF^t(\Sigma, \bm\alpha, \bm\gamma, w, \spinct)$, equipped with its two filtrations $\II_{\alpha\gamma}$ and $\JJ_{\alpha\gamma}$, is doubly-filtered quasi-isomorphic to the doubly-filtered complex
\begin{equation} \label{eq: mapping-cone-CFt}
\xymatrix@R=0.6in{
\CF^t(\bm\alpha, \bm\delta, \spincu_{-L})  \ar[dr]|{F^t_{W'_m, \spincy_{-L}}} & \CF^t(\bm\alpha, \bm\delta, \spincu_{-L+1}) \ar[d]|{F^t_{W'_m, \spincx_{-L+1}}} \ar[dr]|{F^t_{W'_m, \spincy_{-L+1}}} & \cdots  \ar[d] \ar[dr]|{F^t_{W'_m, \spincy_{L-1}}} &\CF^t(\bm\alpha, \bm\delta, \spincu_{L}) \ar[d]|{F^t_{W'_m, \spincx_{L}}} \\
& \CF^t(\bm\alpha, \bm\beta, \spincs_{-L+1}) \otimes T^{-s_{-L+1}} & \cdots & \CF^t(\bm\alpha, \bm\beta, \spincs_{L}) \otimes T^{-s_L}
}
\end{equation}
which inherits its filtrations from those on $\CF^t(\bm\alpha, \bm\delta, w)$ and $\ul\CF^t(\bm\alpha, \bm\beta, w; \GR)$.

By Theorem \ref{thm: large-surgery}, there are doubly-filtered quasi-isomorphisms
\[
\Lambda^t_{\spincu_l}\co \CF^t(\Sigma, \bm\alpha, \bm\delta, \spincu_l) \to A^t_{\spincs_l, s_l}
\]
where the Alexander filtration on $\Lambda^t_{\spincu_l}$ is identified with the filtration $\JJ_{\spincu_l}$ from \eqref{eq: Ju-def}. Moreover, each $\CF^t(\bm\alpha, \bm\beta, \spincs_l) \otimes T^{-s_l}$ can be identified with $B^t_{\spincs_l} = C_{\spincs_l}\{0 \le i \le t\}$, so that the complex in \eqref{eq: mapping-cone-CFt} is quasi-isomorphic to
\begin{equation} \label{eq: mapping-cone-AtBt}
\xymatrix@C=0.6in@R=0.6in{
A^t_{\xi_{-L}}  \ar[dr]|{h^t_{\xi_{-L}}} & A^t_{\xi_{-L+1}} \ar[d]|{v^t_{\xi_{-L+1}}} \ar[dr]|{h^t_{\xi_{-L+1}}} & \cdots  \ar[d] \ar[dr]|{h^t_{\xi_{L-1}}} & A^t_{\xi_L} \ar[d]|{v^t_{\xi_L}} \\
& B^t_{\spincs_{-L+1}}  & \cdots & B^t_{\spincs_{L}}
}
\end{equation}
By definition, this is precisely the complex $X^t_{\lambda,\spinct,-L,L}$ from Section \ref{sec: mapping-cone}.

To complete the proof, we must check that the $\JJ$ filtration and the  absolute grading on \eqref{eq: mapping-cone-AtBt} agree with the descriptions given in the Introduction. Let us denote these by $\JJ_{\operatorname{mc}}$ and $\gr_{\operatorname{mc}}$, respectively.
\begin{itemize}
\item
On each summand $\CF^t(\bm\alpha, \bm\delta, \spincu_l)$, $\JJ_{\alpha\delta}$ is defined in \eqref{eq: ad-filt} as the Alexander filtration plus $\frac{d^2m(2s_{\spincu_l}-1)}{2k(k+dm)}$, so $\JJ_{\operatorname{mc}}$ on $A^t_{\xi_l}$ is obtained by shifting $\JJ_{\spincu_l}$ by the same amount:
\begin{align*}
\JJ_{\operatorname{mc}}([\x, i, j]) &= \JJ_{\spincu_l}([\x,i,j]) + \frac{d^2m(2s_{\spincu_l}-1)}{2k(k+dm)} \\
&= \max\{i-1, j-s_l\} + \frac{d(2s_l - 1)}{2(k+dm)}  + \frac12 + \frac{d^2m(2s_l-1)}{2k(k+dm)} \\
&= \max\{i-1, j-s_l\} + \frac{2ds_l +k-d}{2k}
\end{align*}
This agrees with \eqref{eq: Jt-def-A}.

The absolute grading on $\CF^t(\bm\alpha, \bm\delta, \spincu_l)$ in \eqref{eq: mapping-cone-CFt} is the original Maslov grading, plus the shift from \eqref{eq: newgr-ad}, plus $1$ (by the definition of a mapping cone). Thus, the induced grading on $A^t_{\xi_l}$ in \eqref{eq: mapping-cone-AtBt} (taking into account the grading shift from Theorem \ref{thm: large-surgery}) is given by
\begin{align*}
\gr_{\operatorname{mc}}([\x,i,j]) &= \absgr(\x)-2i + \frac{d^2 m s_l^2}{k(k+dm)} - \frac{m+1+3\sign(k)}{4} + 1 - \Delta_{\spincu_l} \\
&= \absgr(\x)-2i + \frac{d^2 m s_l^2}{k(k+dm)} - \frac{m+1+3\sign(k)}{4} + 1 + \frac{(2ds_l - k - dm)^2}{4d(k+dm)} - \frac14 \\
&= \absgr(\x)-2i + \frac{4d^3 m s_l^2 + k(4d^2s_l^2 - 4ds_l(k+dm) + (k+dm)^2)}{4dk(k+dm)} + \frac{2-m-3\sign(k)}{4} \\
&= \absgr(\x)-2i + \frac{4(k+dm) d^2 s_l^2  - 4kds_l(k+dm) + k(k+dm)^2}{4dk(k+dm)} + \frac{2-m-3\sign(k)}{4}    \\
&= \absgr(\x)-2i + \frac{4 d^2 s_l^2  - 4kds_l + k(k+dm)}{4dk} + \frac{2-m-3\sign(k)}{4}    \\
&= \absgr(\x)-2i + \frac{4 d^2 s_l^2  - 4kds_l + k^2 }{4dk} + \frac{2-3\sign(k)}{4}   \\
&= \absgr(\x)-2i + \frac{(2ds_l - k)^2 }{4dk} + \frac{2-3\sign(k)}{4}
\end{align*}
which agrees with \eqref{eq: grt-def-A}.

\item
On each summand $\CF^t(\bm\alpha, \bm\beta, \spincs_l) \otimes T^{-s_l}$, \eqref{eq: ab-twisted-filt} gives
\begin{align*}
\JJ_{\operatorname{mc}}([\x, i, j]) &= \JJ_{\alpha\beta}([\x, i] \otimes T^{-s_l}) \\
&= i - \frac{2d(-s_l) + k + d}{2k} \\
&= i-1 + \frac{2ds_l + k - d}{2k}
\end{align*}
which agrees with \eqref{eq: Jt-def-B}. Moreover, by \eqref{eq: newgr-ab}, the absolute grading on $\CF^t(\bm\alpha, \bm\beta, \spincs_l) \otimes T^{-s_l}$ is
\begin{align*}
\gr_{\operatorname{mc}}([\x,i,j]) &= \absgr(\x)-2i + \frac{(2ds_l - k)^2}{4kd} - \frac{2+3\sign(k)}{4},
\end{align*}
which agrees with \eqref{eq: grt-def-B}.
\end{itemize}

Thus, we have shown that $\CF^t(\bm\alpha,\bm\gamma, w, \spincu)$ is filtered homotopy equivalent to $X^t_{\lambda, \spinct, -L,L}$.

A priori, the value of $L$ may increase with $t$, since we needed larger values of $m$ to prove the results of Section \ref{sec: exact-sequence}. However, by Lemma \ref{lemma: Xab-indep-ab}, $X^t_{\lambda, \spinct, -L,L}$ is filtered homotopy equivalent to $X^t_{\lambda, \spinct,a,b}$ for any $a \le -L$ and $b \ge L$, independent of $t$, as required.
\end{proof}

\section{Rational surgeries} \label{sec: rational-surgery}

In \cite[Section 7]{OSzRational}, Ozsv\'ath and Szab\'o derived a mapping cone formula for rational surgeries as well as integral ones. Here, we show how this computation interacts with the second filtration discussed in this paper. For simplicity, we show the details only in the case of $1/n$ surgery on a null-homologous knot $K \subset Y$ for $n>0$, but it is not hard to generalize to the case of arbitrary rational surgeries on rationally null-homologous knots.

%

Let $K_{1/n}$ denote the knot in $Y_{1/n}(K)$ obtained from a left-handed meridian of $K$. (Unlike in the case of integral surgery, we emphasize that $K_{1/n}$ is not isotopic to the core circle of the surgery solid torus.) Observe that $Y_{1/n}(K)$ is obtained by a certain surgery on $K' = K \conn O_n$ in $Y' = Y \conn {-L(n,1)}$, where $O_n \subset -L(n,1)$ is the Floer simple knot from Example \ref{ex: lens}, and the induced knot (coming from the meridian of $K'$) is precisely $K_{1/n}$.

There is a natural one-to-one correspondence between $\Spin^c(Y)$ and $\Spin^c(Y_{1/n}(K))$. To be completely precise, let $W$ be the $2$-handle cobordism from $Y'$ to $Y_{1/n}(K)$. For each $\spincs \in \Spin^c(Y)$ and each $q \in \{0, \dots, n-1\}$, let $\spincs_{q} = \spincs \conn \spincu_q \in \Spin^c(Y \conn -L(n,1))$, where $\spincu_q$ is as described in Example \ref{ex: lens}. Then then all the $\spincs_q$ are cobordant to the same spin$^c$ structure $\spinct \in \Spin^c(Y_{1/n}(K))$ through $W$.

The computation in Example \ref{ex: lens} shows that $\HFKa(-L(n,1), O_n, \spincu_q)$ is supported in Alexander grading $-\frac{n-2q-1}{2n}$. (The $-$ comes from the orientation reversal on $L(n,1)$.) The K\"unneth principle for connected sums (\cite[Theorem 7.1]{OSzKnot}, \cite[Theorem 5.1]{OSzRational}) then implies that $\CFKi(Y',K', \spincs_q)$ is isomorphic to $\CFKi(Y,K)$, with the Alexander grading (and hence all values of $j$) shifted by $-\frac{n-2q-1}{2n}$.

We now apply the mapping cone formula to $(Y',K')$ to compute $\CFKi(Y_{1/n}(K), K_{1/n}, \spinct)$. Using the terminology from Section \ref{ssec: rel-spinc}, we take $d=n$, and the framing on $K'$ corresponds to $k=1$. Label the elements of $G^{-1}_{Y_{1/n}, K_{1/n}}(\spinct)$ by $(\xi_l)$, where
\[
\frac{2l-1}{2n} < \AlNorm_{Y',K'}(\xi_l) \le \frac{2l+1}{2n},
\]
and set $s_l = \AlNorm_{Y',K'}(\xi_l)$.

For each $l \in \Z$, it is clear that $G_{Y',K'}(\xi_l) = \spincs_{[q]}$ for some $q \in \{0, \dots, n-1\}$. To determine $q$, the Alexander grading satisfies
\[
\AlNorm_{Y',K'}(\xi_l) \equiv \frac{-n+2q+1}{2n} \pmod \Z,
\]
so write
\[
\AlNorm_{Y',K'}(\xi_l) = \frac{-n+2q+1}{2n} + r
\]
for then $r \in \Z$. Then $2l-1 < (2r-1)n + 2q+1 \le 2l+1$, so $(2r-1)n \le 2(l-q) < (2r-1)n + 2$. Therefore, if $n$ is even, we deduce that $2(l-q) = (2r-1)n$ and $s_l = \frac{2l+1}{2n}$, while if $n$ is odd, then $2(l-q) = (2r-1)n+1$ and $s_l= \frac {l}{n}$.

By definition, the complexes $A^\infty_{\xi_l}$ and $B^\infty_{\xi_l}$ are each copies of $\CFKi(Y',K', \spincs_q)$. As noted above, this is isomorphic to $\CFKi(Y,K,\spincs)$, with the $j$ coordinate shifted by $-\frac{n-2q-1}{2n}$. Under this identification, the filtrations on $A^\infty_{\xi_l}$ and $B^\infty_{\xi_l}$ given by formulas \eqref{eq: It-def-A}, \eqref{eq: Jt-def-A}, \eqref{eq: It-def-B}, and \eqref{eq: Jt-def-B} can each be expressed in terms of $\CFKi(Y,K,\spincs)$, as follows. The quantity $j-s_l$ in \eqref{eq: It-def-A} and \eqref{eq: Jt-def-A} is replaced with $j + \frac{-n+2q+1}{2n} - s_l$, which by the above discussion is equal to $j-r$ (in both the $n$ even and $n$ odd cases), and the quantity $\frac{2ds_l+k-d}{2k}$ simplifies to $l - \floor{\frac{n-1}{2}}$. Thus, we have:
\begin{align}
\intertext{On $A^\infty_{\xi_l}$,}
\label{eq: It-ratl-A} \II_\spinct([\x,i,j]) &= \max\{i,j-r\} \\
\label{eq: Jt-ratl-A} \JJ_\spinct([\x,i,j]) &= \max\{i-1,j-r\} + l - \floor{\frac{n-1}{2}} \\
\intertext{On $B^\infty_{\xi_l}$,}
\label{eq: It-ratl-B} \II_\spinct([\x,i,j]) &= i \\
\label{eq: Jt-ratl-B} \JJ_\spinct([\x,i,j]) &= i-1 + l - \floor{\frac{n-1}{2}} 
\end{align}
It is easy to check that the $\II_t$ filtration agrees with Ozsv\'ath and Szab\'o's description of the $A$ and $B$ complexes: namely, the mapping cone contains $n$ copies of each of the $A_s$ and $B_s$ complexes for $K$. The $\JJ_t$ filtration takes the same form on each copy, with some shifts as necessary.

\bibliographystyle{amsplain}
\bibliography{bibliography}

\end{document}